\documentclass[12pt,a4paper]{amsart}

\usepackage[top=3cm, bottom=3cm, left=2.5cm, right=2.5cm]{geometry}

\usepackage{fullpage}

\usepackage{etex}
\usepackage{amssymb}
\usepackage{amsmath}
\usepackage{mathabx}
\usepackage{amscd}
\usepackage[latin1]{inputenc}
\usepackage{wrapfig}
\usepackage{psfrag}
\usepackage{epsfig}
\usepackage{epic}
\usepackage{eepic}
\usepackage{pifont}
\usepackage[usenames,dvipsnames]{color}
 \usepackage[abs]{overpic}
 \usepackage{bbold}
\usepackage{pdflscape}   
\usepackage[normalem]{ulem}

\usepackage{wasysym}

\usepackage[all]{xy}

\usepackage{hyperref}

\usepackage{caption}

\usepackage{setspace}

\usepackage{enumitem}

\usepackage{mathrsfs}

\usepackage{subcaption}

\usepackage{rotating}
\usepackage{mathtools}

\usepackage{setspace}

\usepackage[section]{algorithm}
\usepackage{algorithmicx,algpseudocode}
\usepackage{tikz}
\usetikzlibrary{shapes, quotes}

\usepackage[all]{xy}

\def \R  {{\mathbb R}}

\def \Z  {{\mathbb Z}}
\def \Q  {{\mathbb Q}}
\def \C  {{\mathbb C}}
\def \N  {{\mathbb N}}
\def \CP {{{\mathbb C}{\mathbb P}}}

\def \CP {{\mathbb C}{\mathbb P}}
\def \t  {{\mathfrak t}}

\def \O {{\mathcal G}}

\def    \FS  {\text{FS}}

\newcommand{\algt}{\ensuremath{\mathfrak{t}}}

\newcommand{\Id}{\mathrm{Id}}
\renewcommand{\emph}{\textbf}
\renewcommand{\min}{\mathrm{min}}
\renewcommand{\max}{\mathrm{max}}

\DeclareMathOperator{\grad}{grad}
\DeclareMathOperator{\prr}{pr}
\def \pr {{I^*}}
\DeclareMathOperator{\last}{last}
\DeclareMathOperator{\iso}{\sharp iso}
\DeclareMathOperator{\fat}{\sharp fat}
\DeclareMathOperator{\rank}{rank}
\DeclareMathOperator{\Ann}{Ann}
\DeclareMathOperator{\sspan}{span}

\def \bl {{\mathbb{B}\hskip -0.015in \mathbb{L}}}

\DeclareMathOperator \Hirz {Hirz}

\DeclareMathOperator \id {id}
\DeclareMathOperator \inc {inc}

\DeclareMathOperator \pt {pt}

\newcommand{\acts}{\mathbin{\raisebox{-.5pt}{\reflectbox{\begin{sideways}$\circlearrowleft$\end{sideways}}}}}

\numberwithin{figure}{section}
\numberwithin{table}{section}
\numberwithin{equation}{section}
\swapnumbers

\makeatletter
\let\c@equation\c@figure
\makeatother

\makeatletter
\let\c@table\c@figure
\makeatother

\makeatletter
\let\c@algorithm\c@figure
\makeatother

\newtheorem{Lemma}[equation]{Lemma}
\newtheorem{Theorem}[equation]{Theorem}
\newtheorem*{thm*}{Theorem}

\newtheorem{Question*}{Question}

\newtheorem{Proposition}[equation]{Proposition}
\newtheorem*{Proposition*}{Proposition}
\newtheorem{Corollary}[equation]{Corollary}

\newtheorem*{Lemma*}{Lemma}
\newtheorem*{Corollary*}{Corollary}

\theoremstyle{definition}
\newtheorem{Example}[equation]{Example}

\newtheorem{Notation}[equation]{Notation}
\newtheorem*{Notation*}{Notation}
\newtheorem{Fact}[equation]{Fact}
\newtheorem{Facts}[equation]{Facts}
\newtheorem{Definition}[equation]{Definition}
\newtheorem{Remark}[equation]{Remark}
\newtheorem*{Remark*}{Remark}

\newtheorem{noTitle}[equation]{}

\newif\ifdebug                                                      %
\newcommand{\printname}[1]
   {\smash{\makebox[0pt]{\hspace{-1.0in}\raisebox{8pt}{\tiny #1}}}}

\newcommand{\labell}[1] {\label{#1}{\ifdebug{\printname{#1}}\fi}}

\newcommand{\mute}[1] {}
\begin{document}

\title[Equivariant cohomology of $S^1\acts (M^4,\omega)$]{Equivariant cohomology of a complexity-one four-manifold is determined by combinatorial data 
}

\author{Tara S. Holm}
\address{Department of Mathematics, Cornell University, Ithaca, NY  14853-4201, USA}
\email{tsh@math.cornell.edu}

\author{Liat Kessler}
\address{Department of Mathematics, Physics, and Computer Science, University of Haifa,
at Oranim, Tivon 36006, Israel}
\email{lkessler@math.haifa.ac.il}

\begin{abstract}
 For Hamiltonian circle actions on compact, connected 
 four-dimensional manifolds, 
we give a generators and relations description 
for the even part of the equivariant cohomology, 
as an algebra over 
the equivariant cohomology of a point.
 This description depends on 
combinatorial data encoded in the decorated graph of the manifold.
We then give an explicit combinatorial description of all weak algebra isomorphisms. 
We use this description to prove that the even parts of the equivariant cohomology algebras are weakly isomorphic 
and the odd groups have the same ranks if and only if the labeled graphs obtained from the decorated graphs 
by forgetting the height and area labels are isomorphic.

As a consequence, we give an example of an isomorphism of equivariant cohomology algebras 
that cannot be induced by an equivariant diffeomorphism of manifolds preserving a compatible almost complex structure. We also provide a soft proof that there are finitely many maximal Hamiltonian circle actions 
on a fixed compact, connected, four-dimensional symplectic manifold.
 \end{abstract}



\maketitle

\setcounter{secnumdepth}{1}
\setcounter{tocdepth}{1}
\tableofcontents

\section{Introduction}

Beginning with work of Masuda \cite{masuda}, there have been a number of questions posed, and some answered,
probing the extent to which equivariant cohomology is a complete invariant \cite{CMS,Masuda-Suh}.  
For  toric manifolds (in other words, smooth compact toric varieties),
Masuda proved that when the equivariant cohomology algebras of two toric manifolds are isomorphic,
the manifolds must be equivariantly diffeomorphic \cite{masuda}, \cite[Remark~2.5(1)]{Ma2}.
Moreover, if the equivariant cohomology algebra isomorphism preserves the first equivariant Chern class, then
the spaces are isomorphic as varieties \cite[Remark~2.5(3)]{Ma2}.
A special case of toric manifolds are \emph{toric symplectic manifolds}: compact, connected symplectic manifolds with a Hamiltonian action of a torus of half the dimension.

 In this paper we look at the equivariant cohomology of a four-dimensional \emph{Hamiltonian $S^1$-manifold}:
 a compact, connected symplectic manifold equipped with a Hamiltonian $S^1$-action. 
 Building on work of Audin \cite{audin} and Ahara and Hattori \cite{ah},
 Karshon \cite{karshon} showed that 
a four-dimensional Hamiltonian $S^1$-manifold is determined by its \emph{decorated graph}: a labeled graph indicating the isolated fixed points as thin vertices and the fixed surfaces as fat vertices; the vertices are labeled by the moment map value, a fat vertex is also labeled by its symplectic area and genus; for a natural number $n>1$, an edge labeled $n$ between vertices indicates that the fixed points are connected by an invariant sphere whose stabilizer is the cyclic subgroup of $S^1$ of order $n$; see \S \ref{decorated}. 
We call the labeled graph obtained from the decorated graph by forgetting the moment map value  and area labels, and adding a vertex label indicating when an isolated vertex is extremal and a fat vertex label indicating its self intersection, the \emph{dull graph} of the Hamiltonian $S^1$-manifold. We give a complete definition
and analyze isomorphisms of dull graphs in Section~\ref{se:dull}.

Equivariant cohomology in the sense of Borel is a generalized cohomology theory in the equivariant category.
 For a torus $T^k = (S^1)^k$,
the equivariant cohomology (over $\Z$) is defined to be
$$
H_{T^k}^*(M;\Z ): = H^*((M\times (S^{\infty})^k)/T^k; \Z),
$$ 
where 
 $S^{\infty}$ is the unit sphere in $\C^{\infty}$, the circle acts freely
$S^1 \acts S^{\infty}$ by coordinate multiplication, and
$T^k \acts (M\times (S^{\infty})^k)$ diagonally. 
In particular,
$$
H^*_{T^k}(\pt) =H^*((S^{\infty})^k/T^k; \Z) =H^*((\CP^\infty)^k;\Z) = \Z[u_1,\dots,u_k], \,\,\, \deg(u_i)=2.
$$
The constant map $\pi \colon M \to \pt$ induces a map $\pi^{*} \colon H^{*}_{T^k}(\pt)  \to H_{T^k}^*(M)$ which endows $H_{T^k}^*(M)$ with an $H^*_{T^k}(\pt)$-algebra structure. 
We let $\cup$ denote the cup product in equivariant cohomology.  
We say that $H_{T^k}^{*}(M)$ and $ H_{T^k}^{*}(N)$ are
 \emph{weakly isomorphic as algebras} if there is a ring isomorphism $f \colon H_{T^k}^{*}(M) \to H_{T^k}^{*}(N)$ and an automorphism $\gamma$ of $T^k$ such that $f(\pi^*(u)\cup w)=\pi^*(\gamma^{*}(u))\cup f(w)$ for any $u \in H_{T^k}^{*}(\pt)$ and $w \in H_{T^k}^{*}(M)$.
 If $\gamma$ is the identity automorphism then $f$ is an \emph{isomorphism of algebras}.

 First, we obtain a generators and relations description of $H_{S^1}^{2*}(M)$ 
 from the decorated graph. 
 The generators are the equivariant Poincar\'e dual classes to $S^1$-invariant submanifolds
 that correspond to edges and fat vertices in the decorated graph.
 Moreover we express the algebra structure over $H_{S^1}^*(\pt) = \Z[t]$ in terms of these generators. 
 See Theorem \ref{ThNew} for the explicit statement. 
In the proof of Theorem \ref{ThNew}, we apply our previous results in \cite[Theorem 1.1]{KesslerHolm2} that 
 the inclusion of the fixed points set $i:M^{S^1}\hookrightarrow M$ induces an injection in
integral equivariant cohomology
$$
i^*:H^{*}_{S^1}(M;\Z) \hookrightarrow H^{*}_{S^1}\left(M^{S^1};\Z\right)
$$
and our characterization of the image of  $i^*$ in equivariant cohomology with 
rational coefficients. 
 We use the generators and relations description to relate the 
 algebraic and combinatorial structures of the Hamiltonian $S^1$-action.

\begin{Theorem}\labell{thm:unique}
Let  $S^1 \acts (M,\omega_M)$ and $S^1 \acts (N,\omega_N)$ be compact, connected, four-dim\-en\-sion\-al 
Hamiltonian $S^1$-manifolds.
The following are equivalent.
\begin{enumerate}
\item The dull graphs of $M$ and $N$ are isomorphic as labeled graphs.
\item  $H_{S^1}^{2*}(M)$ and $H_{S^1}^{2*}(N)$ are isomorphic as algebras over $H_{S^1}^{*}(\pt)$ \newline and 
 $\rank H_{S^1}^{i}(M)=\rank H_{S^1}^{i}(N)$ for all odd $i$. 
 
\item  $H_{S^1}^{2*}(M)$ and $H_{S^1}^{2*}(N)$ are weakly  isomorphic as algebras over $H_{S^1}^{*}(\pt)$ \newline and 
 $\rank H_{S^1}^{i}(M)=\rank H_{S^1}^{i}(N)$ for all odd $i$. 
  \end{enumerate}
\end{Theorem}

In fact, more is true.  We say that an isomorphism of algebras
$$\Lambda \colon H_{S^1}^{2*}(M;\Z) \to H_{S^1}^{2*}(\widetilde{M};\Z)$$
is {\bf orientation-preserving} if the induced isomorphism 
in ordinary cohomology is or\-ien\-ta\-tion-preserving, where the orientations on $M$ and $N$ are the ones induced by the symplectic forms.  
Otherwise, it is {\bf orientation-reversing}. 
In Theorem~\ref{thm:strong}, we prove that an abstract
orientation-preserving isomorphism of cohomology algebras must arise from an 
isomorphism of the dull graphs, which induces the given abstract isomorphism by 
way of the generators-and-relations presentation from Theorem~\ref{ThNew}.

We will prove Theorem \ref{thm:unique} in Section \ref{sec5b} and explore its consequences.
We will list the isomorphisms of the even-dimensional equivariant cohomology as 
$H_{S^1}^*(pt)$-algebras.  
We will check which of these algebra isomorphisms send the first equivariant 
Chern class $c_1^{S^1}(TM)$ to $c_1^{S^1}(TN)$.
The isomorphisms with this property are  induced by  equivariant biholomorphisms of the $S^1$-manifolds, equipped with invariant complex structures that are compatible with the symplectic forms; 
see Corollary~\ref{cor:uniquec2} and Remark \ref{rem:orient-rev}.

More generally, there are orientation-preserving isomorphisms of equivariant cohomology rings, as $H_{S^1}^*(pt)$-algebras, 
that do not send  $c_1^{S^1}(TM)$ to $c_1^{S^1}(TN)$ or to $-c_1^{S^1}(TN)$, namely 
the \emph{chain flips}, defined in Section \ref{sec5b}. 
We discovered the
chain flip when trying to emulate Masuda's work on toric manifolds \cite{masuda}.
In that case, there are 
equivariant cohomology
generators that are supported on $T$-invariant codimension $2$ submanifolds.  These generators in the toric context are, 
in a certain sense, unique.  
Trying to establish similar uniqueness properties 
for our generators has led us to some alternative generators, which are linear combinations 
of the originals 
and which are not Poincar\'e dual to $S^1$-invariant submanifolds.
This led us to discover the chain flip isomorphism.

\begin{Example} \labell{first pass}
Consider the two Hamiltonian $S^1$-manifolds $S^1 \acts (M,\omega_M)$ and $S^1 \acts (N,\omega_N)$ with extended decorated graphs shown in Figure~\ref{fig:chain flip}. 
\begin{center}
\begin{figure}[h]
\includegraphics[height=5cm]{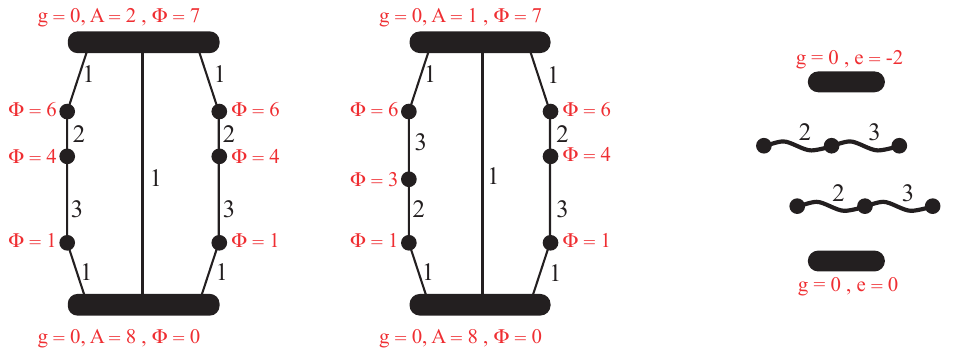} 
\caption[.]{On the left are two extended decorated graphs that differ by a chain flip. 
On the right is the dull graph
that is the dull graph of each.
}
\label{fig:chain flip}
\end{figure}
\end{center}
The two $S^1$-manifolds do have the same dull graph, thus, by Theorem \ref{thm:unique},  there is an algebra-isomorphism between
their equivariant cohomology rings.  We will see that 
this
isomorphism, explicitly, is a chain flip.

On the other hand, the decorated graphs are not isomorphic, so the two manifolds 
are not $S^1$-equivariantly symplectomorphic. 
Nevertheless, they are diffeomorphic: they are both 
$7$-fold blowups of $\CP^2$.  We will see that the blowup forms determining the symplectic structures 
are not equal, so there is no symplectic diffeomorphism \cite[Thm.~1.8]{KK:blowups}.  
Furthermore, by calculating the equivariant first Chern classes, we will show that
there can be no equivariant diffeomorphism preserving
an $S^1$-invariant, compatible almost complex structure.  We will work through this example
in full detail in Section \ref{sec5b}.
It is true that there is an orientation-preserving, equivariant diffeomorphism
between the two manifolds.  This
assertion follows from ongoing joint work with Susan Tolman  \cite{hkt}, where we construct
an explicit equivariant diffeomorphism inducing a chain flip.
 \hfill $\diamondsuit$
\end{Example}

Our analysis has several consequences to questions other than the rigidity of equivariant cohomology. 
As a byproduct of our characterization of isomorphisms of dull graphs in Section \ref{se:dull}, 
we deduce in  Corollary~\ref{cor:rei} that every compact, connected, simply connected four-dimensional 
Hamiltonian $S^1$-manifold
is equivariantly diffeomorphic to one that admits a toric action extending the circle action.

As a further application of our generators and relations description, we deduce
that there is a finite number of inequivalent maximal Hamiltonian circle actions on a fixed 
compact, connected, four-dimensional symplectic manifold $(M,\omega)$. 
A Hamiltonian torus action is \emph{maximal}  if it does not extend to a Hamiltonian action of a strictly larger torus on $(M,\omega)$. 
Karshon gives necessary and sufficient conditions for a Hamiltonian circle action on a four-dimensional symplectic manifold to extend to a toric one
\cite[Prop.~5.21]{karshon}. In Figure~\ref{fig:no extension}, we show the extended decorated
graph and dull graph for a maximal Hamiltonian
circle action.
We call two torus actions  \emph{equivalent} if and only if they differ by an equivariant symplectomorphism composed 
with a reparametrization of the group ${(S^1)}^k$.

\begin{center}
\begin{figure}[h]
\includegraphics[height=3.5cm]{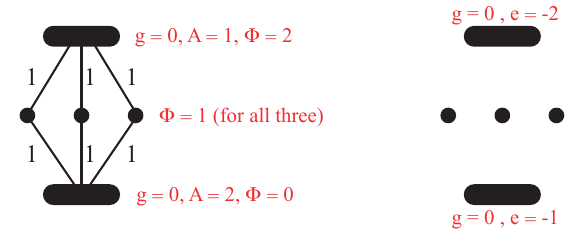} 
\caption[.]{The extended decorated graph and dull graph for a Hamiltonian circle action on a 
four-dimensional manifold
that does not extend to a toric action.  The manifold is diffeomorphic to a $4$-point blowup of
$\CP^2$. According to Corollary ~\ref{cor:rei}, the dull graph indicates that the manifold 
is equivariantly diffeomorphic to one that does
admit a toric action.  To wit, if the symplectic size of one of the blowups is decreased by $\varepsilon$,
then the manifold will admit a toric action.
}
\label{fig:no extension}
\end{figure}
\end{center}

\begin{Theorem} \labell{finite}
Let $(M,\omega)$ be a compact, connected, four-dimensional symplectic manifold. 
The number of inequivalent 
maximal Hamiltonian torus actions on $(M,\omega)$  is finite. 
\end{Theorem}

We prove Theorem \ref{finite} in Section \ref{sec8}.
The proof is analogous to the proof of McDuff and Borisov \cite[Proposition 3.1]{mcduff:finite} 
 establishing that there are finitely many toric actions on a given symplectic manifold.
The key application of the Hodge index theorem is similar. We use the fact that a compact, connected, four-dimensional Hamiltonian $S^1$-manifold admits an invariant integrable complex structure that is compatible with the symplectic form, with respect to which the $S^1$-action is holomorphic.
However some of the steps 
require extra work. In particular, in Lemma \ref{c1c2}, we give a formula for the classes $c_{1}(TM)$ and $c_{1}(TM)^2-2c_2({TM})$, 
using our generators.
Our proof of Theorem~\ref{finite} is ``soft": it does not use ``hard" pseudo-holomorphic tools. 
By contrast, 
pseudo-holomorphic
curves play a key role in
the proof of the finiteness of maximal Hamiltonian torus actions in \cite[Theorem 1.1]{pinso-max}, and in
the  algorithm to list all homologically trivial actions of a given compact Lie group $G$ 
on a given symplectic $k$-blowup of a symplectic ruled surface in \cite[Theorem 1.2]{kkp} and \cite[Theorem 2.13]{KesslerHolm1}.

\subsection*{Acknowledgements}
We thank Susan Tolman for sustained discussion about Hamiltonian $S^1$-manifolds.
We are grateful to Yael Karshon, Allen Knutson, Jason Liu, and Daniele Sepe for helpful conversations;
to Mikiya Masuda for clarifying remarks; and to
Rei Henigman for observing that Corollary \ref{cor:rei} follows from Proposition \ref{scale1} and Karshon's work \cite{karshon}.
We are indebted to the anonymous referee for illuminating observations and corrections, and
to the editors for their patience as we implemented them.
The first author was supported in part
by the National Science Foundation under Grants DMS--1711317 and DMS--2204360. 
The second author was supported in part by the Israel Science Foundation, Grant  570/20, and
an NSF-BSF Grant 2021730. 
Any opinions, 
findings, and conclusions or recommendations expressed in this material are 
those of the authors and do
not necessarily reflect the views of the National Science Foundation.

\section{Classification of Hamiltonian circle actions\\ on symplectic four-manifolds} \labell{sec2}

We record here the details we will need resulting from the classification of Hamiltonian circle actions on
symplectic four-manifolds \cite{ah,audin,karshon}.
An \emph{effective} 
action of a torus $T={(S^1)}^{k}$ on a symplectic manifold $(M,\omega)$ is \emph{Hamiltonian} if there exists a \emph{moment map}, that is, a smooth map 
$\Phi \colon M \to \t^{*}\cong {\R}^k$  that satisfies Hamilton's equation
$$d \Phi_j = - \iota(\xi_j) \omega $$
for all $j=1,\dots,k$, where $\xi_1,\ldots,\xi_k$ are the vector fields
that generate the torus action.

In this paper we consider compact, connected, four-dimensional 
Hamiltonian $S^1$-man\-i\-folds
$S^1\acts(M,\omega)$.
Hamilton's equation guarantees that the set of fixed points of the $S^1$-action coincides 
with the critical set of the moment map $\Phi:M\to \R$, which is a Morse-Bott function. 
Moreover, the indices and dimensions of its critical submanifolds are all even, hence they	
can only consist
of isolated points (with index $0$ or $2$ or $4$) and two-dimensional submanifolds (with index $0$ or $2$).  
The latter
can only occur at the extrema of $\Phi$.  By Morse-Bott theory (and since the manifold is connected), the maximum and minimum of the moment map is each attained on exactly one component of the fixed point set.

\begin{noTitle}{\bf Gradient spheres.} \label{gradient}
An $S^1$-invariant Riemannian metric $\langle \cdot, \cdot \rangle$ on $(M, \omega, \Phi)$ is called \emph{compatible} 
with $\omega$ if the automorphism $J \colon TM \to TM$ defined by $\langle \cdot, \cdot \rangle=\omega(\cdot,J\cdot)$ 
is an almost complex structure, i.e., $J^2 = -\Id$. Such a $J$ is $S^1$-invariant.
With respect to a compatible metric, 
the gradient vector field of the moment map, characterized by $\langle v,\grad \Phi \rangle=d \Phi(v)$, is
\begin{equation*} 
\grad \Phi = -J \xi_M,
\end{equation*}
where $J$ is the corresponding almost complex structure and $\xi_M$ is the vector field that generates the $S^1$ action.
The vector fields $\xi_M$ and $J\xi_M$ generate a $\C^\times = (S^1)^\C$ action.
The closure of a non-trivial $\C^\times$ orbit is a topological sphere, called a \emph{gradient sphere}; it need not be a smooth submanifold. Gradient spheres are collections
of gradient flow lines for $\Phi$.  On a gradient sphere, $S^1$ acts by rotation with two fixed points at the north and south poles; 
all other points on the sphere have the same stabilizer. 
We say that a gradient sphere is \emph{free} if the stabilizer of a generic point in the sphere is trivial; otherwise it is \emph{non-free}. 
 For a smooth gradient sphere $S$ we have $J TS=TS$.
 Hence, since in dimension $2$ any almost complex structure is integrable, i.e., arises from an underlying complex atlas on the manifold, $S$ is a $J$-holomorphic sphere.

In dimension four, the existence of non-free gradient spheres does not depend on the compatible metric.
This is because the non-free gradient spheres are precisely the isotropy  $\Z_n$-spheres (with various $n>1$), 
i.e.\ connected components of the closure of the set of points in $M$ whose stabilizer is equal to the 
(non-trivial) cyclic subgroup of $S^1$ of order $n$ \cite[Lemma 3.5]{karshon}.

 All but a finite number of gradient spheres are free gradient spheres whose north and south poles are 
 maximum and minimum of the moment map; we call the latter spheres \emph{trivial}.
A {\bf generic} metric is one for which there exists no free gradient sphere whose north and south poles are both interior fixed points.
For a generic compatible metric, the arrangement of gradient spheres is determined by the decorated graph defined below \cite[Lemma 3.9]{karshon}.
 \end{noTitle}

\begin{noTitle} {\bf The decorated graph: conventions.}  \labell{decorated}
 To each connected four-dimensional Hamiltonian $S^1$-manifold, 
 we associate a \textbf{decorated graph} as in \cite{karshon}.
We translate the moment map by a constant, if necessary, to fix the minimum value of the moment map to be $0$.  
For each isolated fixed point, there is a vertex
labeled by its moment map value.
For each fixed surface $\Sigma$, there is
a \emph{fat} vertex labeled by its moment map value,
its symplectic area $\frac{1}{2\pi} \int_\Sigma \omega$,
and its genus $g$. If there are two fat vertices, they necessarily have the same genus.
If there is one fat vertex, it must have genus $0$ and the manifold $M$ is simply connected.
If there are no fat vertices, the manifold is again simply connected and we say the genus is $0$.
In this way, the genus is an invariant associated to the manifold $M$. 

In our figures depicting decorated graphs, the moment map value determines the 
vertical placement of a (fat or isolated) vertex. 
The horizontal placement is for convenience and does not carry any significance.
For each $\Z_n$-sphere, $n>1$, there is an edge connecting the vertices corresponding to its fixed points and labeled by the integer $n$. The  \emph{edge-length} is the difference of the moment map values of its vertices.
\end{noTitle}

We recall the following facts.
\begin{itemize}
\item The sphere $S$ corresponding to an edge of label $n$ is a symplectic sphere whose size, $\frac{1}{2\pi} \int_S \omega$, 
is $\frac1n$  times the {edge-length}.
\item For $n>1$, a fixed point has an isotropy weight $-n$ exactly when it is the north pole of a $\Z_n$-sphere, corresponding to a downward edge labeled $n$, and a weight $n$
exactly when it's the south pole of a $\Z_n$-sphere.
\item In particular, two edges incident to the same vertex have relatively prime edge labels, since the action is effective.
\item An interior fixed point has one positive weight and one negative
weight, a maximal fixed point has both weights non-positive, a
minimal fixed point has both weights non-negative.
\item A fixed point has an isotropy weight $0$ if and only if it lies on a fixed surface.
\end{itemize}
We denote 
$$
\fat:=\# \text{  fixed surfaces}; \, \,
\iso:=\# \text{  isolated fixed points}; \,\,
k:=\# \text{ chains of edges}.
$$

\begin{noTitle} {\bf The extended decorated graph: conventions.} \labell{decorated1}
The \emph{extended decorated graph} with respect to a compatible metric  is the graph 
obtained from the decorated graph as follows.
We add edges labeled $1$ for each non-trivial free gradient sphere. 
We also add an edge labeled $1$ for a trivial gradient sphere when $\fat \geq 1$ and as long as $k<2$.
 In the resulting graph, every interior vertex is attached to one edge from above and one edge from below; 
 the moment map labels remain monotone along each chain of edges; 
 and there are at least two chains of edges.
The \emph{length} of one of the new edges is the difference of the moment map values of its vertices.
In what follows, when we say 
extended decorated graph, we mean with respect to a generic compatible K\"ahler metric.
When the ``generic" property is crucial, we will reiterate it.
\end{noTitle}

\begin{Proposition}\labell{label1}
In an extended decorated graph with respect to a generic metric,
we have the following.
\begin{enumerate}
\item If an edge has label $1$ then it is either the first or the last in a chain from $\min$ to $\max$ (or both).
\item For every interior fixed point that is not connected to top or bottom, there is exactly one edge from above and 
one edge from below, both with label $>1$.
\item Only edges of label $1$ can emanate from a fat vertex.
\end{enumerate}
\end{Proposition} 

\begin{proof}
The first item is a consequence of having a {generic} compatible metric on $(M,\omega, \Phi)$.  This
implies there exists no free gradient sphere whose north and south poles are both interior fixed points \cite[Corollary 3.8]{karshon}.  
The second item then follows from the first item and the construction of the extended
decorated graph.
The third item is a consequence of the action being effective.
\end{proof}

\begin{noTitle} {\bf Topological invariants.} \labell{factsi}
Let $F_{\min}$ and $F_{\max}$ be the extremal
critical sets of the moment map $\Phi$.  
For $*=\min,\max$, we define
$$
a_{*} = \int_{F_{*}} \omega \mbox{ ; } y_{*} = \Phi(F_{*}) \mbox{ ; and }
e_* =\left\{ \begin{array}{ll}
F_* \cdot F_* & \mbox{when } \dim(F_*)=2\\
-\frac{1}{mn} & \mbox{when } \dim(F_*)=0
\end{array}\right.,
$$
where $m$ and $n$ are the isotropy weights at $F_*$ when it is an isolated fixed point. 
In this case, $|m|$ and $|n|$ are the two largest labels emanating from the vertex corresponding 
to the point in an extended decorated graph. If $F_{*}$ is of $\dim 2$ we denote it by $\Sigma_{*}$.
For an interior isolated fixed point $p$, we define
$
y_p = \Phi(p),
$
and let  $m_p$ and $n_p$ be the absolute values of the isotropy weights at $p$; these are the labels of the edges emanating from the vertex corresponding to $p$ in an extended decorated graph.
We let $e_p=\frac{1}{m_pn_p}$.

These parameters are related by the following formul\ae.  
\begin{equation}\labell{eq:emin}
e_{\min} = \frac{\left(\sum_p y_pe_p\right)+a_{\min}-\left(\sum_p e_p\right)\cdot y_{\max}-a_{\max}}{y_{\max}-y_{\min}}
\end{equation}
and
\begin{equation}\labell{eq:emax}
e_{\max} = \frac{\left(\sum_p e_p\right)\cdot y_{\min}+a_{\max}-\left(\sum_p y_pe_p\right)-a_{\min}}{y_{\max}-y_{\min}},
\end{equation}
where $p$ runs over the interior fixed points.
 Formul\ae\ \eqref{eq:emin} and \eqref{eq:emax} can
be deduced from \cite[Proof of Lemma~2.18]{karshon}, which has a missing term
that we have restored (the missing term is the $a_{\max}$; its absence
does not affect the validity of Karshon's proof). 
\end{noTitle}

The proofs in Sections~\ref{se:dull}, \ref{sec3} and \ref{sec4} require recalling the details of the 
characterization of compact, connected, four-dimensional Hamiltonian $S^1$-manifolds.  We recall these here.

\begin{noTitle} {\bf Circle actions that extend to toric actions.} \labell{projection}
Consider a compact, connected, toric symplectic four-manifold $T \acts (M,\omega)$.
An inclusion $$\inc \colon S^1 \hookrightarrow T={(S^1)}^{2}; \,\, s \mapsto (s^{m},s^{n})$$ induces a projection on the duals of the Lie algebras ${\R}^{2} \to \R$ defined by $(m,n) \in {\Z}^{2}$, explicitly ${\R}^2 \ni (x,y) \mapsto m x+n y$.
Composing this projection on the moment map of the torus action yields the moment map of the circle action.
\end{noTitle}

\begin{Notation}\label{gcd}
For the $S^1$-action to be effective, the pair $(m,n)$ is either  $\pm(1,0)$, 
$\pm (0,1)$, or satisfies $\gcd(m,n)=1$.  In what follows, we will frequently need to use this fact.  We fix
$a,b \in \Z$ are such that 
\begin{equation}\label{eq:gcd}
am-bn=1.
\end{equation}
When $(m,n)=\pm (1,0)$, we take $a=\pm 1$ and $b=0$; when $(m,n)=\pm (0,1)$, we take $a=0$ and $b=\mp 1$; and
when $(m,n)=\pm (1,1)$, we take $a=0$ and $b=\mp 1$.   These then still satisfy \eqref{eq:gcd}.
\end{Notation}

The fixed surfaces are the preimages, under the $T$-moment map, of the edges of the Delzant polytope 
parallel to $(-n,m)$.
Such a surface has genus zero and its normalized symplectic 
area equals the affine length of the corresponding edge. 
The isolated fixed points are the preimages of the vertices of the polygon that do not lie on such edges. 
To determine the $S^1$ isotropy of a $T$-invariant sphere, if the image under the $T$-moment map
is parallel to the primitive vector $e=(\alpha,\beta)$, then relative to the circle action, the sphere
is a $\Z_\ell$ sphere for $\ell=|m\alpha+n\beta|$.  For further details, see \cite[\S 2.2]{karshon}.
An example is shown in Figure~\ref{fig:toric projection0}.

\begin{center}
\begin{figure}[h]
\includegraphics[height=5cm]{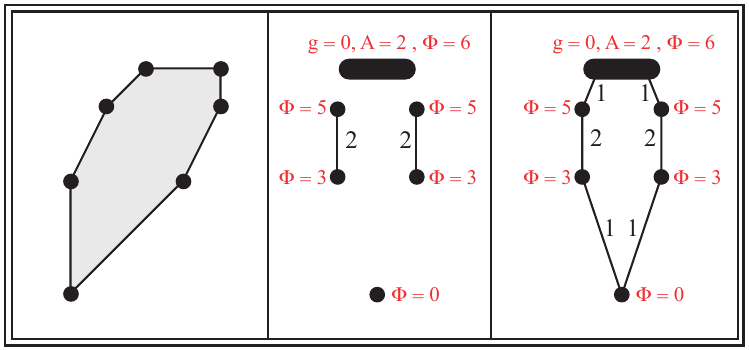} 
\caption[.]{A moment polytope for a four-dimensional toric manifold on the left, together with the decorated graph and
extended decorated graph for the second coordinate circle action.
}
\label{fig:toric projection0}
\end{figure}
\end{center}

\begin{Example}\labell{e:cp2}
The complex projective plane
$\CP^2$ with  a multiplication of the Fubini-Study form $\omega_{\FS}$ by $\lambda>0$
admits the toric action
$$(t_1,t_2) \cdot [z_0;z_1;z_2]=[z_0;t_1 z_1; t_2 z_2]$$
whose moment map is the Delzant triangle of edge-length $\lambda$.
Denote by $L$ the homology class of a line $\CP^1$ in $\CP^2$.
For each of the edges of the Delzant triangle, its preimage is an invariant embedded holomorphic and symplectic sphere in $L$.
For $(m,n)\in {\Z}^2$ with $\gcd(m,n)=1$ the inclusion $s \mapsto (s^{m},s^{n})$
induces the circle action
 $$s \cdot  [z_0;z_1;z_2]= [z_0;s^m z_1;s^n z_2].$$
For the $T$-moment map, the edges of the Delzant triangle lie on the lines $x=0$, $y=0$, and  $x+y =1$, as
shown in Figure~\ref{fig:circle on cp2}(a).
Thus there is an $S^1$-fixed sphere exactly when $(m,n)$ is $\pm (1,0)$, $\pm (0,1)$ and $\pm (1,1)$.
Otherwise, there are a $\Z_{|m-n|}$-sphere, a ${\Z}_{|n|}$-sphere, and a ${\Z}_{|m|}$-sphere: $[0;z_1;z_2], \, \, [z_0;0;z_2], \, \, [z_0,z_1,0]$, respectively. See Figure \ref{fig:circle on cp2} for the corresponding labeled graphs. 
In all cases, the K\"ahler metric is generic.

\begin{center}
\begin{figure}[h]
\includegraphics[height=5cm]{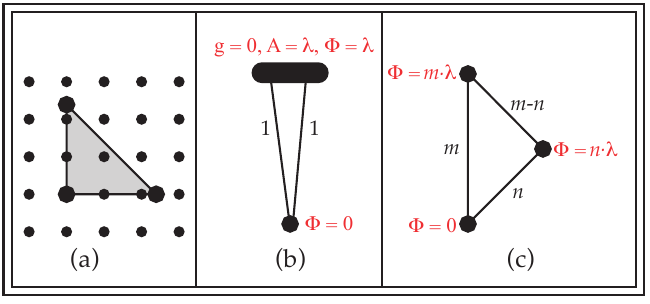} 
\caption{The $T^2$ moment map image and  extended decorated graphs for a Hamiltonian circle action on $\CP^2$. In
(c), we are assuming $m>n>0$.}
\label{fig:circle on cp2}
\end{figure}
\end{center}
 \hfill $\diamondsuit$
\end{Example}

\begin{Example}\labell{e:hirz}
We denote by $\Hirz_{N}$ the Hirzebruch surface that is the algebraic submanifold of
$\CP^1 \times \CP^2$ defined in homogeneous coordinates by
\[
\{([w_1;w_2], [z_0;z_1;z_2]) \in \CP^1 \times \CP^2 \mid w_{1}^{N}z_2=w_{2}^{N}z_1\}.
\]
For $\beta,f > 0$, we denote by $\omega_{\beta,f}$ the sum of the Fubini-Study form on $\CP^1$ multiplied by $\beta+\frac{N}{2}f$ and the Fubini-Study form on $\CP^2$ multiplied by $f$. 
We shall use the same notation for its restriction to the Hirzebruch surface.
The {\bf zero section} $S_0$ is the sphere   $\{([w_1;w_2], [1;0;0])\}$, the {\bf section at infinity} $S_{\infty}$ is the sphere $\{([w_1;w_2], [0, w_1^N,;w_2^N])\}$, and the {\bf fiber at zero} $F_0$ is the sphere $\{([1; 0], [z_0;z_1;0])\}$.

The Hirzebruch surface $(\Hirz_N, \omega_{\beta,f})$ admits the toric action
\begin{equation*} 
(s,t)\cdot ([w_1;w_2],[z_0;z_1;z_2]) = ([w_1;sw_2],[tz_0;z_1;s^N z_2])
\end{equation*}
whose moment map image is the trapezoid in Figure \ref{fig:hirzebruch}.
The parameter $\beta$ is the length of the top edge of the trapezoid, $f$ is the height, $\frac{1}{N}$ is the slope of the right edge if $N>0$, and the right edge is vertical if $N=0$. 
For $(m,n)\in {\Z}^2$ with $\gcd(m,n)=1$ the inclusion $s \mapsto (s^{m},s^{n})$
induces the circle action
\[
\xi \cdot ([w_1;w_2], [z_0;z_1,;z_2]) = ([w_1; \xi^m w_2], [\xi^n z_0; z_1;\xi^{Nm} z_2]).
\]
There are two fixed spheres if $(m,n)=\pm (0,1)$.
The circle action has exactly one fixed sphere and one ${\Z}_{N}$-sphere if $(m,n)=\pm(1,0)$.
Otherwise, there are two ${\Z}_{|m|}$-spheres, a $\Z_{|mN-n|}$-sphere, and a ${\Z}_{|n|}$-sphere. 
The decorated graphs for these actions are shown in Figure \ref{fig:circle on hirzebruch}. 

In these examples, the K\"ahler metric is generic except for graph (d) with $m=1$. 
However, these non-generic Hamiltonian $S^1$-manifolds are  each isomorphic to a manifold whose graph is 
of type (c) with $N'=N-2|n|$ (and $m,|n|$ as before); the K\"ahler metric on that isomorphic manifold is generic.
For further details, see \cite[Remark 6.12]{karshon}.  \hfill $\diamondsuit$

\begin{center}
\begin{figure}[h]
\includegraphics[height=2.5cm]{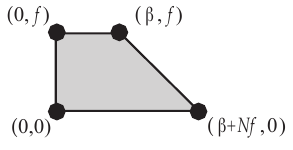} 
\caption{The standard Hirzebruch trapezoid}\labell{trapezoid}
\label{fig:hirzebruch}
\end{figure}
\end{center}

\begin{center}
\begin{figure}[h]
\includegraphics[height=5cm]{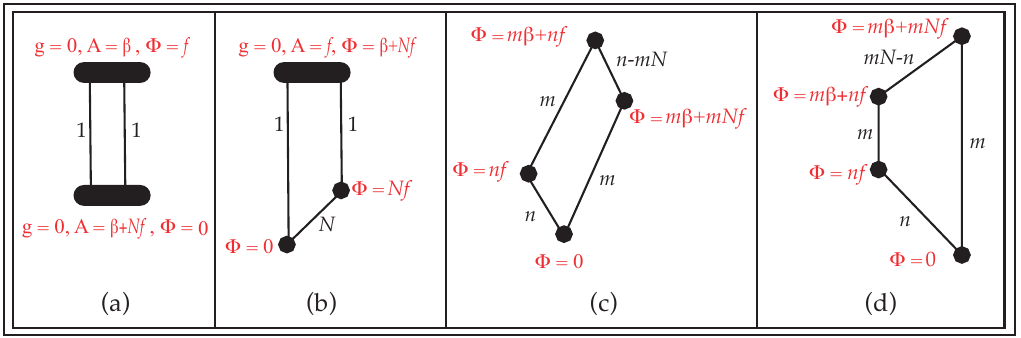} 
\caption{Extended decorated graphs for a Hamiltonian circle action on a Hirzebruch surface.  In (c) and (d), we are
assuming $m,n>0$.  In (c), we have $n-mN>0$ and in (d), $mN-n>0$. The graph in (a) with genus label $g \geq 0$ is the extended decorated graph of a symplectic $S^1$-ruled surface.
}
\label{fig:circle on hirzebruch}
\end{figure}
\end{center}

\end{Example}

\newpage

\begin{Notation}\label{basic notions}
Consider an  $S^2$-bundle over a compact Riemann surface $\Sigma$.
We fix base points $\ast\in S^2$ and $\ast\in \Sigma$.
For the trivial $S^2$-bundle $\Sigma \times S^2$ over $\Sigma$, we denote 
$F:=[\ast \times S^2], \,\,\, B^{*}:=B=[\Sigma \times \ast],$
 classes in the homology group  $H_2(\Sigma \times S^2;\Z)$. 
When we consider the non-trivial $S^2$-bundle $M_{\Sigma}\stackrel{\pi}{\to}\Sigma$, denote the homology class of the fiber by 
$F=[\pi^{-1}(\ast)]\in H_2(M_\Sigma;\Z)$.
For each $\ell$, the trivial bundle admits a section ${\sigma}_{2\ell}: \Sigma\to  \Sigma \times S^2$ 
whose image $\sigma_{2\ell}(\Sigma)$ has 
even self intersection number $2\ell$. Similarly, for each $\ell$, 
the non-trivial bundle admits a section $\sigma_{2\ell+1}:\Sigma\to M_{\Sigma}$ 
whose image $\sigma_{2\ell+1}(\Sigma)$ has odd self intersection number $2\ell+1$.  
We denote $B_N:=[\sigma_N(\Sigma)] \in H_2(M_\Sigma;\Z)$.
For every $N \in \Z$, we have  $B_{N}=B_{-N}+N\cdot F.$ 
When we consider the non-trivial $S^2$-bundle denote $B^*:=\frac{1}{2}(B_1+B_{-1})$ in $H_2(M_\Sigma;\Q)$. Note that 
\begin{equation}\label{eq:b*}
B^{*}=\frac{1}{2}(B_{N}+B_{-N})
\end{equation} 
for all even $N$ in the trivial case, and for all odd $N$ in the non-trivial case.
A Hirzebruch surface is an $S^2$-bundle over $S^2$. In $\Hirz_{N}$, we have $S_0=B_{-N}$, $S_{\infty}=B_{N}$ and $F_0=F$.

An $S^2$-bundle over $\Sigma$  with a circle action that fixes the basis and rotates each fiber, an invariant symplectic form $\omega$ and a moment map is called a 
 \emph{symplectic $S^1$-ruled surface}. 
 It admits a \emph{ruled} complex analytic structure as an $S^2$-bundle over $(\Sigma,j)$, that is compatible with the $S^1$-action and with $\omega$, such that the K\"ahler metric is generic.
 Its extended decorated graph is as in Figure~\ref{fig:circle on hirzebruch}(a), with genus labels $g \geq 0$. The parameter $\beta$ is $\frac{1}{2\pi}$ times the symplectic area of $B_{-N}$ and $f$ is $\frac{1}{2\pi}$ times the symplectic area of $F$.
 \end{Notation}

\begin{noTitle} \labell{nbup} {\bf  The effect of a blowup on the decorated graph.} 
Let $J$ be  an  integrable $\omega$-compatible complex structure on a Hamiltonian $S^1 \acts (M^4,\omega)$ 
with respect to which the $S^1$-action is holomorphic. 
Let $p$ be a fixed point in $M$.
Let $U \subset M$ be an invariant open ball centered around the fixed point $p$, 
small enough such that the $S^1$-action on $U$ is linear (in holomorphic coordinates). In particular it induces an $S^1$-action on the manifold $\widetilde{U}=\{(z,l) \colon z \in l\} \subset U \times \CP^{1}$.
The \emph{equivariant complex blowup} $(\widetilde{M},\widetilde{J})$ of $(M,{J})$ at $p$ is the complex $S^1$-manifold 
\begin{equation}\labell{eq:bupman} 
\widetilde{M}=M \smallsetminus \{p\}\cup \widetilde{U},
\end{equation}
obtained by adjoining $M \smallsetminus \{p\}$ with $\widetilde{U}$ via the 
equivariant isomorphism $\widetilde{U} \smallsetminus (z=0) \cong U \smallsetminus \{p\}$ given by $(z,l) \mapsto z$.  There is a natural equivariant projection 
\begin{equation} \labell{eq:bupmap}
\bl \colon \widetilde{M} \to M
\end{equation}
extending the identity on $M  \smallsetminus \{p\}$. The inverse image ${\bl}^{-1}(p)$ is naturally isomorphic to $\CP^{1}$ and is called the \emph{exceptional divisor} of the blowup; it is $S^1$-invariant. Let $E$ be the homology class in $H_{2}(\widetilde{M};\Z)$ of the exceptional divisor. In $k$-fold complex blowup we denote the classes of the exceptional divisors by $E_1,\ldots,E_k$.

 For $\varepsilon >0$, define an $\varepsilon$-blowup of the decorated graph, 
 according to the location of the blowup, as in 
 Figure \ref{fig:blowup effect}. 
 We say that the obtained graph is \emph{valid} 
 if the (fat or not) vertices created 
 in the $\varepsilon$-blowup do not surpass the other pre-existing (fat or not) vertices 
 in the same chain of edges, 
 and the fat vertices after the $\varepsilon$-blowup have positive size labels.
By \cite[Theorem 7.1 and its proof]{karshon},
if the $\varepsilon$-blowup of the decorated graph is valid, 
 then there exists an invariant K\"ahler form $\widetilde{\omega}$  on the equivariant complex 
 blowup in the cohomology class $$\bl^{*}[\omega]-\varepsilon \Xi,$$ where  $\Xi$ is the Poincar\'e 
 dual of the exceptional divisor class $E$, and the graph of the blowup $S^1 \acts (\widetilde{M},\widetilde{\omega}, \widetilde{J})$ 
 is this $\varepsilon$-blown up graph.
  Moreover,
if the K\"ahler metric on $S^1 \acts (M,\omega)$ is generic, then so is 
 the resulting K\"ahler metric on $S^1 \acts (\widetilde{M},\widetilde{\omega})$.

\begin{center}
\begin{figure}[h]
\includegraphics[height=3cm]{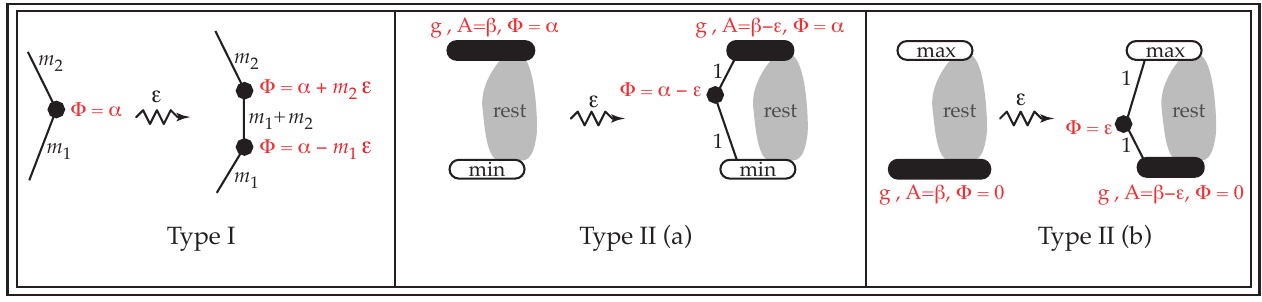} 

\includegraphics[height=3cm]{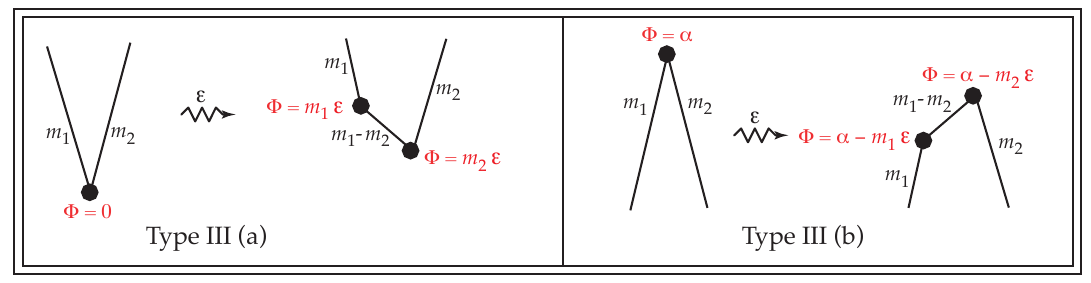} 

\includegraphics[height=3cm]{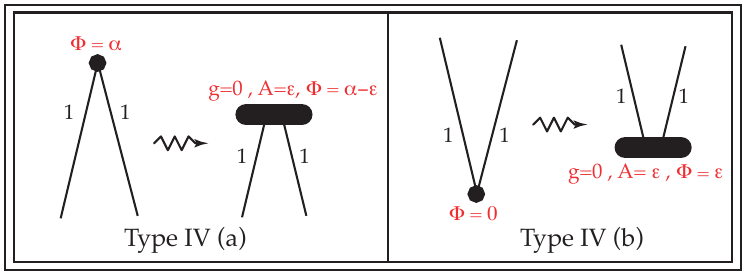} 

\caption[.]{
The effect on the decorated graph of Types I--IV blowups.
Note that in Type III(a) and IV(b), the final moment map should be corrected with a vertical translation
to comply with the convention that the minimum moment map value is $\Phi=0$.
}
\label{fig:blowup effect}
\end{figure}
\end{center}
\end{noTitle}

\begin{noTitle}\labell{nbd}{\bf Equivariant blow down}.
By the Castelnuovo-Enriques criterion \cite[p.\ 476]{GH}, if a holomorphic sphere $S \cong \CP^1$ is embedded in a complex manifold $M$ of dimension two and 
$[S] \cdot [S]=-1$, then 
one can blow down along $S$, replacing it with a point $p$,  to get a complex manifold  $\overline{M}$. If $S^1 \acts M$ and $S$ is invariant then $\overline{M}$ admits an $S^1$-action with the point $p$ fixed.
If $M$ admits a K\"ahler form $\omega$ such that $S^1 \acts (M,\omega)$ is Hamiltonian, then 
by the equivariant tubular neighbourhood theorem, a neighbourhood of $S$ is equivariantly symplectomorphic to a neighbourhood of the exceptional divisor in an $\varepsilon$-blowup of $C^2$ with a linear $S^1$-action. 
By removing a neighbourhood of $S$ and gluing in a standard ball, we get 
an invariant K\"ahler form $\overline{\omega}$  on the equivariant complex blow down $\overline{M}$ such that $S^1 \acts (\overline{M},\overline{\omega})$ is Hamiltonian.
 Its equivariant K\"ahler $\varepsilon$-blowup is isomorphic to $(M,\omega,J)$ with the given $S^1$-action. 
The effect on the decorated graph is the reverse of the effect of the blowup.
\end{noTitle}

\begin{noTitle} \labell{preface} {\bf Characterization of Hamiltonian $S^1$-manifolds.}
By \cite[Theorem 6.3, Lemma 6.15, and proof of Theorem 7.1]{karshon}, every 
compact, connected, four-dimensional Hamiltonian $S^1$-manifold is, up to an equivariant symplectomorphism, obtained by finitely many $S^1$-equivariant K\"ahler blowups starting from one of the following {minimal models}:
\begin{itemize}
\item  The complex projective plane $\CP^2$ with the Fubini-Study form with a
 circle action 
 $$s \cdot  [z_0;z_1;z_2]=[z_0;s^m z_1;s^n z_2].$$ 
 This is the projection of a toric action, as in Example \ref{e:cp2}.
\item The Hirzebruch surface $\Hirz_N$ with the form $ \omega_{\beta,f}$ with a circle action
$$
\xi \cdot ([w_1;w_2], [z_0;z_1,;z_2]) = ([w_1; \xi^m w_2], [\xi^n z_0; z_1;\xi^{Nm} z_2]).
$$
This example is the projection of a toric action,
as in Example \ref{e:hirz}.

\item A {symplectic $S^1$-ruled surface}, with a ruled compatible integrable complex structure, as in Notation~\ref{basic notions}.
 \end{itemize}
\end{noTitle}

\begin{Remark} \labell{rem:spheres}
  For $S^1 \acts (M^4,\omega)$, let $J$ be an 
integrable complex structure on $(M,\omega)$ such that the $S^1$-action is holomorphic and $\omega(\cdot,J\cdot)$ is a generic Riemannian metric. For the existence of such $J$ see \cite[Theorem 7.1]{karshon}.
By  \cite[Lemma~4.9]{ah}, a gradient sphere is smooth at its poles except 
when the gradient sphere is free 
and the pole in question is an isolated minimum (or maximum) of $\Phi$ with both isotropy weights $> 1$ (or $< -1$).
In particular, a non-free gradient sphere is an embedded $J$-holomorphic sphere.
By \cite[Lemma 2.4]{counter}, the preimages of fat vertices and the non-free gradient spheres whose moment-map images are
edges of label $>1$ are
embedded complex (hence symplectic) curves.

If $M$ is an $S^2$-bundle over $\Sigma$,
the fiber class $F$ is represented by an embedded complex sphere. 
Therefore the edge labeled $N$ in Figure \ref{fig:circle on hirzebruch}(b); 
the edges labeled $n$ and $n-mN$ in Figure \ref{fig:circle on hirzebruch}(c);
and the edges labeled $n$ and $mN-n$ in Figure \ref{fig:circle on hirzebruch}(d) are 
each the image of an embedded complex sphere, even if $N=1$, $n=1$, $n-mN=1$, or
$mN-n=1$, respectively.
The same is true for a trivial edge with label $1$ in an extended decorated graph with two fat vertices.

Similarly, the classes $B_N$ and $B_{-N}$ are represented by embedded complex spheres, 
hence the edges labeled $1$ in Figure \ref{fig:circle on hirzebruch}(b) and the edges labeled 
$m$ in the Figure \ref{fig:circle on hirzebruch}(c) and (d) are each the image of an 
embedded complex sphere, even if $m=1$.
An edge labeled one in Figure \ref{fig:circle on cp2}(a) is the image of an embedded 
complex sphere in the class $L$ of a line in $\CP^2$.
Other labeled one edges
that are represented by embedded complex spheres are edges that are the images of an 
exceptional divisor or of the proper transform of a complex sphere in an equivariant K\"ahler blowup.
\end{Remark}

\begin{Definition}
In the extended decorated graph of $S^1 \acts (M,\omega)$ with respect to a generic metric $\omega(\cdot,J \cdot)$, we call an edge \emph{robust} if there is an embedded $J$-holomorphic  sphere 
corresponding to the edge by the moment map. 
Otherwise, we call the edge \emph{ephemeral}.
\end{Definition}

Note that by \cite[Lemma~4.9]{ah}, robustness does not depend on the choice of generic metric. Thus, if an edge is robust, there is an integrable complex structure on $(M,\omega)$ such that the edge corresponds to an embedded holomorphic sphere with respect to this structure.
It follows from  \S \ref{preface} and Remark \ref{rem:spheres}  that
if there are ephemeral edges in the extended decorated graph for a generic metric,  
then the number of 
fat vertices is exactly one. 
By our convention, it is the maximal vertex. Moreover, 
the edges in such a graph that are possibly
ephemeral are the first edges in the chains.
However, for $i=1,2$, the first edge in the  $i^{\mathrm{th}}$ 
chain is never ephemeral. 
The following characterization of ephemeral edges is thus a consequence of \S \ref{preface} and 
Remark \ref{rem:spheres}.

\begin{Proposition} \labell{claim:eph}
Let $S^1 \acts (M,\omega)$ be a compact, connected, four-dimensional
Hamiltonian $S^1$-manifold. 
In the extended decorated graph with respect to a generic metric, we
order the chains
so that the labels satisfy $m_{1,1} \geq m_{2,1} \geq 1=m_{3,1}=\cdots=m_{k,1}$, where $m_{i,1}$ is the label of the first edge 
from the minimum in the $i^{\mathrm{th}}$ chain.
An edge is ephemeral if and only if the following hold.
\begin{itemize}
\item The graph has exactly
one fat vertex, which by our conventions is the maximum
under the momentum map;
\item the edge is  the first edge in the 
 $i^{\mathrm{th}}$ chain; and
 \item $i \geq 3$ and $m_{2,1}\geq 2$.
\end{itemize}
\end{Proposition}

\section{Dull graphs and their isomorphisms}\label{se:dull}

We now turn from decorated graphs to a combinatorial structure
with less information.  We will show that it retains enough to recover equivariant cohomology.

\begin{Definition}\label{def:dull}
The {\bf dull graph} of a compact, connected, four-dimensional Hamiltonian $S^1$-manifold $S^1 \acts (M,\omega)$ 
is the labeled graph $\O$ obtained from the decorated graph by 
\begin{itemize}
\item forgetting the height and area labels;
\item adding a vertex label to an extremal isolated vertex to indicate it is extremal; and 
\item adding a vertex label to each fat vertex to indicate its self intersection. 
\end{itemize}
\end{Definition}


\begin{center}
\begin{figure}[h]
\includegraphics[height=4cm]{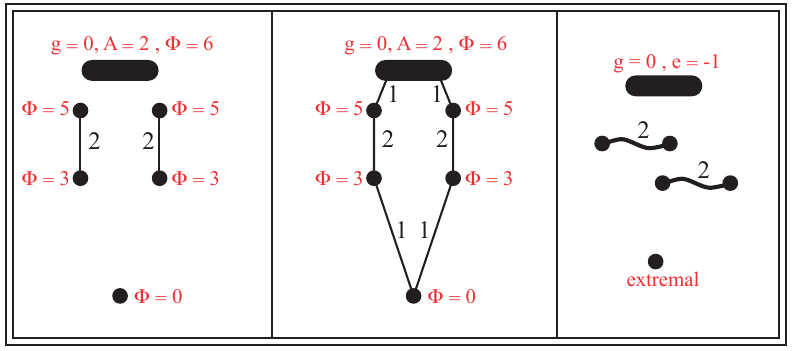} 
\caption[.]{The decorated graph, extended decorated graph, and dull graph for a Hamiltonian circle action on a four-dimensional manifold.
}
\label{fig:extended and dull}
\end{figure}
\end{center}

We first investigate when two decorated
graphs can have the same dull graph.  
For the blowup and blow down procedures
described in Section~\ref{sec2}, the dull graph
does not record the symplectic size $\varepsilon$
of the blowup or blow down.  This leads us to
the following definition.

\begin{Definition}
A {\bf multi-rescaling} is a map between two decorated graphs
that can change moment map labels and area labels while preserving
\begin{itemize}
\item the $\min$ and $\max$ vertices,
    \item the edge labels, 
     \item the order by height of the vertices in each chain, and
    \item for fat vertices, the genus labels and the $e_{\min}$ and $e_{\max}$ values
    in \eqref{eq:emin} and \eqref{eq:emax}.
    
\end{itemize}
We use the same term for such a map between extended decorated graphs.
\end{Definition}

\begin{Lemma} \labell{equiv}
For two compact, connected, four-dimensional Hamiltonian $S^1$-manifolds,  
$$S^1 \acts (M,\omega) \mbox{ and } S^1 \acts (M',\omega'),$$
the following are equivalent.
\begin{enumerate}
\item The dull graphs of $M$ and $M'$ are isomorphic as labeled graphs. 
\item The extended decorated graphs of $M$ and $M'$ with respect to a generic compatible metric differ by a composition of finitely many of the following maps:
\begin{itemize}
\item[(a)] a flip of the whole graph; 
\item[(b)] a multi-rescaling; and 
\item[(c)] a flip of a chain 
 that begins and ends with an edge of label $1$.
\end{itemize}
\end{enumerate}
Moreover, any isomorphism of the dull graphs
of $M$ and $M'$ is induced by a map on the
extended decorated graphs that is the
composition of finitely many of the 
maps (a)-(c).
\end{Lemma}

\begin{Remark}
Note that the maps of type (a) and (c) flip at least some edges upside down.  The
map of type (b), a multi-rescaling, is the only one that keeps all edges
right-side up.
\end{Remark}

We call a map of type (c) is a {\bf chain flip}, as discussed in the Introduction.  
Note that the scaling factors in a multi-rescaling (b) can differ on each fat vertex area label and edge length.  
The effect of the
maps in (a)-(c) are shown in 
Figure~\ref{fig:chain + flip + rescale}.  


\begin{center}
\begin{figure}[h]
\includegraphics[height=5cm]{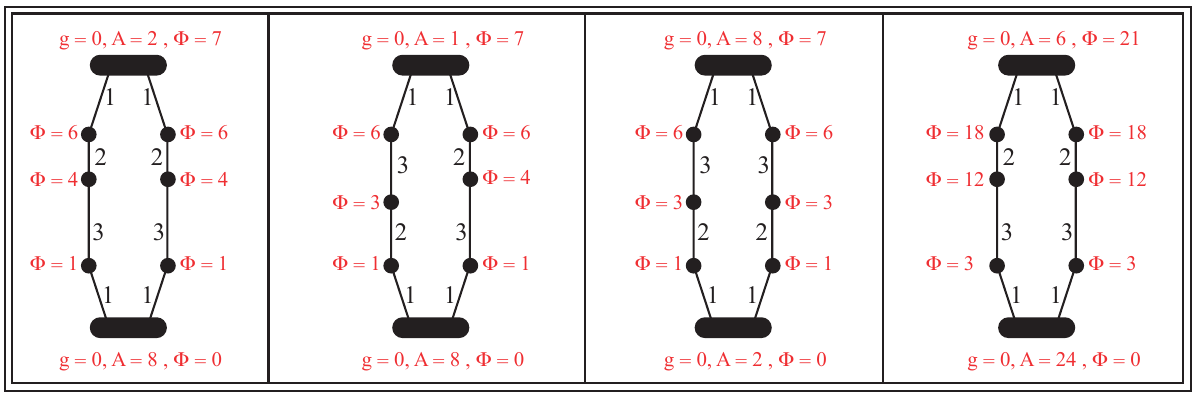} 
\caption[.]{An extended decorated graph on the left; a chain flip of the extended decorated graph; a flip
of the whole extended decorated graph; and a multi-rescaling of the extended decorated graph.
}
\label{fig:chain + flip + rescale}
\end{figure}
\end{center}

\begin{proof}[Proof of Lemma \ref{equiv}]
The implication (2) $\Rightarrow$ (1) is straight forward because the dull graph simply forgets some of the information
from the extended decorated graph.
So the main content is to show (1) $\Rightarrow$ (2),
keeping track that the isomorphism of dull graphs
is induced by the map of extended decorated graphs.
Consider dull graphs that are obtained from the decorated graphs associated to four-dimensional Hamiltonian $S^1$-manifolds.

Recall that there are exactly two vertices that are either fat or labeled as extremal. 
By Proposition~\ref{label1}, the connected components of a dull graph are: the 
component of one extremal/fat vertex; the component of the other; and, if they exist, 
chains of edges in which all the vertices are interior.          
Note that such a chain might consist of a single isolated vertex, otherwise each endpoint
of a chain is adjacent to exactly one edge. In the extended decorated graph there might 
be an edge 
labeled one between an end-point of a chain and the $\min$ vertex and an edge labeled 
one between the other end-point and the $\max$ vertex.

An isomorphism of dull graphs preserves the labels on extremal and fat vertices. 
So it sends 
the set of vertices that are $\max$ and $\min$ in the decorated graph of 
$S^1 \acts (M,\omega)$ to the set of vertices that are $\max$ and $\min$ in the 
decorated graph of $S^1 \acts (M',\omega')$. Assume first that it sends the vertex that is $\max$  in $M$ to the 
vertex that is $\max$ in $M'$ and $\min$ to $\min$. Then, on the vertices and edges of the
connected components of an extremal vertex it coincides with 
the map induced from 
permutations of identical chains.
Such a permutation only affects the chains in the connected component of the extremal vertex.
This can be seen  by induction on the number of edges in the shortest path to an extremal vertex.
Similarly, if it sends $\max$ to $\min$ and $\min$ to $\max$, then on the vertices and edges in the connected 
components of an extremal vertex, it coincides with the 
 map induced from a flip of the extended decorated graph, 
 possibly followed by permutations of identical chains.
  In both cases, since the adjacency relation and the edge labels are preserved, all other connected components are
  chains not containing an extremal vertex.
  On these remaining components, 
  the dull graph isomorphism coincides either with the 
map induced from the identity;  a permutation of identical chains; or  with the map induced from a  flip, possibly followed by a permutation of identical chains.

Now assume that the isomorphism of the dull graphs is the identity map.
Then the two extended 
graphs must agree on the 
type of the maximal (minimal) vertex (i.e., fat or isolated), and on its genus and self intersection 
when it is a fat vertex, and 
on the arrangement of the edges and their labels in each of the chains. 
Thus, up to chain flips of type (c), the extended decorated graphs must differ by rescaling the heights of the moment map values 
and scaling the areas of
the fat vertices, precisely a multi-rescaling.
Finally, we must verify that
 $e_{\min}$ and $e_{\max}$ are preserved. 
If the maximal vertex is fat then $e_{\max}$ is one of its labels in the dull graph hence does not change.  
If the maximal vertex is isolated then 
the number $e_{\max}$ in both decorated graphs is the reciprocal of the product of the 
labels adjacent to the vertex, hence it does not change.
By the same argument, $e_{\min}$ does not change.  Moreover, in this case, the identity map on the dull graphs is in fact induced by the composition of maps of type (b) and (c) just described on the extended dull graphs.

Thus, we have established that (2) holds, together with the fact that the isomorphism of dull graphs
is induced by a map on the extended decorated graphs that is a composition of maps of type (a),(b), and (c).
This completes the proof.
\end{proof}

\begin{Remark} 
Let $S^1 \acts (M,\omega)$ and $S^1 \acts (M',\omega')$ be
compact, connected, four-dimensional Hamiltonian $S^1$-manifolds.
An isomorphism of extended decorated graphs of $M$ and $M'$ restricts to an isomorphism of the
decorated graphs.
Moreover, by \cite[Proof of Lemma 3.9]{karshon}, an isomorphism of extended decorated graphs with respect to generic compatible metrics
 is determined by an isomorphism of the
decorated graphs. 
The maps (a), (b), and (c) between extended decorated graphs in the lemma above translate naturally to maps between decorated graphs.
\end{Remark}

\begin{Remark}\labell{flipandscale}
If the decorated graphs of  $(S^1 \acts M',\omega',\Phi')$ and $(S^1 \acts M,\omega,\Phi)$ 
differ by a map of type (a), {\em i.e.} a flip of the whole graph, 
then there is an equivariant diffeomorphism from $S^1 \acts M'$ to $S^1 \acts M$ that 
induces the map between the graphs.
To prove this, we first note that equipping $M$ with $-\omega$ and the given $S^1$-action on $M$, 
we get a Hamiltonian $S^1$-manifold  $(S^1 \acts M,-\omega,-\Phi)$ whose decorated graph also 
differs from the decorated graph of $(S^1 \acts M,\omega,\Phi)$ by a map of type (a)  
and a vertical translation.  Thus, the decorated graphs of $(S^1 \acts M,-\omega,-\Phi)$ and $(S^1 \acts M',\omega',\Phi')$ are isomorphic, up to a vertical translation.  A vertical translation corresponds to adding a global constant to the moment map. By the uniqueness of the decorated graph \cite[Theorem 4.1]{karshon}, there is an equivariant symplectomorphism from $(S^1 \acts M,-\omega,-\Phi)$ to $(S^1 \acts M',\omega',\Phi')$ that induces the isomorphism between their decorated graphs. This map is as an equivariant diffeomorphism $h \colon S^1 \acts M \to S^1 \acts M'$ that is orientation-preserving with respect to the orientations induced by the symplectic forms. Moreover, for any $S^1$-invariant, $\omega$-compatible almost complex structure $J$, the almost complex structure $-h_{*}J$ on $M'$ is $S^1$-invariant and $\omega'$-compatible.

We note that the decorated graph of   $(S^1 \acts M',\omega',\Phi')$ also coincides with the decorated graph of the $S^1$-manifold obtained from  $S^1 \acts (M,\omega_M)$ by precomposing with the non-trivial automorphism of the circle. Hence the map (a) also indicates a strictly weakly equivariant diffeomorphism.
\end{Remark}

A map of type (b) corresponds to an equivariant diffeomorphism as well. 
 We will show that, geometrically, it corresponds
to changing the sizes of the symplectic blowups, which changes the symplectic form on the resulting
manifold but not the equivariant diffeomorphism type.

Let ${M_k}$ be the smooth manifold underlying a complex 
$k$-blowup at $k$ distinct points of 
$M_0$ that is
either the complex projective plane $\CP^2$, or the Hirzebruch surface $\Hirz_N$, or a ruled surface, i.e., an $S^2$-bundle over $(\Sigma,j)$ with a ruled  integrable complex structure, with $g(\Sigma)>0$.
A \emph{blowup form} on $M_k$ is 
 a symplectic form for which there exist disjoint 
embedded symplectic spheres (oriented by the symplectic form)  in the 
homology classes 
\begin{itemize}
\item $L,E_1,\ldots,E_k$ if $M_0=\CP^2$;
\item $S_0,F_0,E_1,\ldots,E_k$ if $M_0$ is $\Hirz_N$;
\item $F,E_1,\ldots,E_k$ if $M_0$ is a ruled surface with $g(\Sigma)>0$.
\end{itemize}
See Example \ref{e:cp2}, Example \ref{e:hirz}, Notation \ref{basic notions} and \S \ref{nbup} for the notation of the homology classes.
We say that the blowup form is \emph{compatible} with an $S^1$-action on $M_k$ if the embedded symplectic spheres in the above classes can be chosen to be invariant.

Given $M_k$, we say that a blowup form $\omega'$ is a \emph{positive rescaling} of a blowup form $\omega$ if $\omega'$ differs form $\omega$ by \begin{itemize}
\item rescaling the sizes $\frac{1}{2\pi}\langle \omega,E_i \rangle$ for $i=1,\ldots,k$ and  $\frac{1}{2\pi}\langle \omega,F \rangle$, $\frac{1}{2\pi}\langle \omega,B^{*} \rangle$,
 if $(M_k,\omega)$ is a $k$-fold blowup of a ruled surface and $B^*$ is as in \eqref{eq:b*};
\item rescaling the sizes $\frac{1}{2\pi}\langle \omega,E_i \rangle$ for $i=1,\ldots,k$ and  $\frac{1}{2\pi}\langle \omega,S_0 \rangle$, $\frac{1}{2\pi}\langle \omega,F_0 \rangle$,
  if $(M_k,\omega)$ is a $k$-fold blowup of a Hirzebruch surface;
or
\item  rescaling the sizes 
$\frac{1}{2\pi}\langle \omega,E_i \rangle$ for $i=1,\ldots,n$ and  $\frac{1}{2\pi}\langle \omega,L \rangle$, 
if $(M_k,\omega)$ is a $k$-fold blowup of  $\CP^2$. 
\end{itemize}

For two actions $S^1 \acts (M_k,\omega)$ and $S^1 \acts (M'_k,\omega')$, we say that an $S^1$-equivariant diffeomorphism $S^1 \acts M_k \to S^1 \acts M'_k$ is  an \emph{$S^1$-compatible positive rescaling} of the $S^1$-compatible blowup form 
$\omega$ if it 
pulls back $\omega'$ to  an $S^1$-compatible blowup form that is a positive rescaling of $\omega$.

\begin{Proposition} \labell{scale1}
Let $S^1 \acts (M,\omega)$ and $S^1 \acts (M',\omega')$ be compact, connected, four-di\-men\-sion\-al Hamiltonian $S^1$-manifolds, equipped with the orientations induced by the symplectic forms.
Given a map $\psi$ from the extended decorated graph of $S^1 \acts (M,\omega)$
to that of $S^1 \acts (M',\omega')$  (with respect to generic compatible metrics) that is of type (b), 
i.e.\ a multi-rescaling, then $\psi$ is induced by 
an equivariant diffeomorphism $h \colon S^1 \acts M \to S^1 \acts M'$
that is orientation-preserving.  Moreover, for any $S^1$-invariant, $\omega$-compatible almost complex structure $J$, the structure $h_{*}J$ is $S^1$-invariant and $\omega'$-compatible.
\end{Proposition}

\begin{proof}
By  \cite[Theorem 6.3 and Lemma 6.15]{karshon} every compact, connected Hamiltonian $S^1$-manifold of dimension four is obtained by a sequence of equivariant symplectic blowups from either $S^1 \acts (\CP^2,\lambda \omega_{\FS})$, or $S^1 \acts (\Hirz_{N},\omega_{\beta,f})$, or a symplectic $S^1$-ruled surface. So the symplectic form is an $S^1$-compatible blowup form.   We will prove that $\psi$ is induced by an $S^1$-compatible positive rescaling of the symplectic blowup form,  composed with an $S^1$-equivariant symplectomorphism.
 The proof is by induction on the sum 
$\ell$ of the number of fat vertices and the number of edges in the extended decorated graph.  (Note that it is the same number in both graphs.)
Our base cases are the minimal models, which have $\ell=3$ and $4$ and are exhibited in Figures~\ref{fig:circle on cp2} and
\ref{fig:circle on hirzebruch}.

In the case $\ell=3$, $(M,\omega)$ and $(M',\omega')$ are $(\CP^2,\lambda \omega_{\FS})$ and $(\CP^2,\lambda' \omega_{\FS})$, and the circle actions come from the toric action on $\CP^2$. 
Thus each of  the extended decorated graphs is (up to a flip)
 either Figure \ref{fig:circle on cp2}  (b) or (c).
 In (b) and its flip, 
the area label of the fat vertex and the lengths of the edges all equal the same positive number.
In (c), dividing the edge-length by the edge-label gives  the same number for all three edges. 
Since the graphs of $S^1 \acts (M,\omega)$ and of $S^1 \acts (M',\omega')$ differ by a multi-rescaling, they must be of the same type, with the same edge-labels. So the $S^1$-action on $\CP^2$ is the same.
The value $\lambda$ is $\frac{1}{2\pi}$ times the area of the line $\CP^1$ in the homology class $L$ in $(\CP^2,\lambda \omega_{\FS})$, and $\lambda'$ is $\frac{1}{2\pi}$ times the area of the line $\CP^1$ in $L$ in $(\CP^2,\lambda' \omega_{\FS})$. Hence, $\lambda$ ($\lambda'$) equals the area label of the fat vertex if the  graph is of type (b) (or its flip), and the edge-length over the edge-label for each of the edges if the  graph is of type (c). 
Thus $\psi$ is induced from the identity diffeomorphism of $\CP^2$, which is an $S^1$-compatible positive rescaling $\lambda \omega_{\FS} \mapsto \lambda' \omega_{\FS}$.

If $\ell=4$, then 
$(M,\omega)$ is a symplectic ruled surface, 
either rational or 
irrational. In the first case $(M,\omega)$ is a Hirzebruch surface 
$(\Hirz_N,\omega_{\beta,f})$.
If $N=1$, it is also a blowup of $\CP^2$ at one point. The extended decorated graph 
for the circle action is thus (up to a flip) one of the graphs in Figure \ref{fig:circle on hirzebruch}. 
In the second case,  $S^1 \acts (M,\omega)$ is a symplectic $S^1$-ruled surface 
of positive genus and its extended decorated graph is as in (a) in 
Figure \ref{fig:circle on hirzebruch} with the genus label $g>0$.
By Example \ref{e:hirz} and Notation \ref{basic notions},
 the symplectic size $f$ of the fiber at zero $F_0$ (respectively, the fiber $F$) is $2 \pi$ times
 $$\begin{cases} 
\text{the length of each edge} & \text{ in  (a) }\\
\left[ \text{\begin{tabular}{l}the area label of the fat vertex =\\ 
$\frac{1}{N}$(the length of the $N$-labeled edge)\end{tabular}}\right.& \text{ in   (b) }\\
 \frac{1}{n} (  \text{the length of the $n$-labeled edge} )& \text{ in  (c) }\\
   \frac{1}{n} ( \text{the length of the $n$-labeled edge}) & \text{ in      (d) }
 \end{cases},$$  
 and the symplectic size $\beta$ of the zero section $S_0$ (respectively, the section $B_{-N}$)
 is $2 \pi$  times
  $$\begin{cases} 
 \text{the area label of the max fat vertex } & \text{ in    (a) }\\
 \text{the length of the right edge emanating from max}
& \text{ in   (b) }\\
\frac{1}{m} (  \text{the length of the left $m$ edge}) 
& \text{ in  (c) }\\
\frac{1}{m}  (  \text{the length of the left $m$ edge}) 
& \text{ in  (d) }
 \end{cases}.$$  
 Note that, if $M$ is rational, $N$ is determined by the edge labels in the decorated graph and the labels $e_{\max}$ and $e_{\min}$: it is $e_{\min}$ in (a),  the label of the edge between isolated vertices in (b), $ \frac{1}{2m^2}\left(\frac{1}{e_{\max}}-\frac{1}{e_{\min}}\right)$ 
 (where $m$ is the duplicate label) in (d) and (c). So $N$ is not affected by a multi-rescaling of the graph.
 Also note that the edge-lengths of the remaining edges and the area label of the remaining fat vertex 
 are determined by the above sizes and by equations \eqref{eq:emin} and \eqref{eq:emax}.
 Therefore, the multi-rescaling $\psi$ of the graph only replaces $\beta,f$ with positive $\beta',f'$, respectively. By \cite[Theorem 6.3 and Lemma 6.15]{karshon}, $(M',\omega')$ is also a symplectic ruled surface. Moreover, it has the same genus and parameter $N$ (if rational) as $M$, so $M=M'$. Moreover, since the graphs of $S^1 \acts (M,\omega)$ and $S^1 \acts (M,\omega')$ differ only by the values of the parameters $\beta,f$, the $S^1$-actions are the same and the identity map is an $S^1$-compatible positive rescaling of the form. This map induces $\psi$.

For the induction step, let $S^1\acts M$ with $\ell \geq 5$. 
By \cite[Theorem 7.1 and its proof]{karshon},
there is an equivariant symplectomorphism $\rho$ from $S^1 \acts (M,\omega)$ to an $S^1$-manifold $S^1 \acts (\widetilde{M},\widetilde{\omega})$ on which
there is a compatible integrable complex structure $\widetilde{J}$ such that $S^1 \acts (\widetilde{M},\widetilde{\omega},\widetilde{J})$  is obtained by a single, 
equivariant, K\"ahler blowup from a manifold $S^1 \acts (\overline{M},\overline{\omega},\overline{J})$, with $\ell$ now decreased by one,  and the metrics determined by $(\widetilde{\omega},J)$ and by $(\overline{\omega},\overline{J})$ are generic.
The image of the exceptional divisor under the moment map is either a fat $\min$-vertex, a fat $\max$-vertex, or a
robust edge in the extended  decorated graph for a generic metric. See Figure~\ref{fig:blowup effect} for the effect of the blowup on the decorated graph.
In the first two cases, the label $e_{\min/\max}$ is $-1$ and the genus label $g$ is $0$.
In the third case,
if the edge is the $j^{\mathrm{th}}$ from the bottom on the $i^{\mathrm{th}}$ chain, then by \eqref{rem:int}, the \emph{combinatorial intersection number} $\frac{m_{i,j-1}-m_{i,j+1}}{m_{i,j}}$ equals $-1$.
Recall that here $m_{i,k}$ is the label of the $k^{\mathrm{th}}$ edge in the $i^{\mathrm{th}}$ 
chain for $1 \leq k \leq \ell_i$, and $m_{i,0}$ and $m_{i,\ell_{i+1}}$ are as in \eqref{eq:label0} and \eqref{eq:label+1}.

Denote by $\rho$ also the induced isomorphism on the extended decorated graphs. Then $\psi \circ \rho^{-1}$ is a multi-rescaling of the extended decorated graph of $S^1 \acts (\widetilde{M},\widetilde{\omega},J)$. The resulting graph 
is the extended decorated graph of $S^1 \acts (\widetilde{M}',\widetilde{\omega}')$ with respect to  a generic metric. It
is isomorphic to the extended decorate graph of $S^1 \acts (M',\omega')$ with respect to  a generic metric. By Karshon's uniqueness Theorem \cite[Theorem 4.1]{karshon}, there is an isomorphism $\rho' \colon S^1 \acts (M',\omega') \to S^1 \acts (\widetilde{M}',\widetilde{\omega}')$ that induces this isomorphism; we call the graph isomorphism also $\rho'$.
The multi-rescaling $\psi \circ \rho^{-1}$ does not change $e_{\min/\max}$, $g$, the $m_{i,k}$ labels,
the adjacency relation, or the thickness of the extremal vertices. 
So the resulting extended decorate graph also contains a fat vertex or edge, respectively, 
with the same $e_{\min/\max}$ and $g$ labels and combinatorial intersection numbers. 
Moreover, 
Proposition \ref{claim:eph} guarantees that if the exceptional divisor corresponds to an edge, it is 
robust.  
It is the image 
under the moment map of an embedded invariant complex (and symplectic) sphere, complex with respect
to a compatible integrable complex structure $\widetilde{J}'$ on $S^1 \acts (\widetilde{M}',\widetilde{\omega}')$ such that the metric $\widetilde{\omega}'(\cdot,\widetilde{J}'\cdot)$ is generic. The preimage of a fat 
vertex with genus label $0$ is also such a sphere; see Remark \ref{rem:spheres}. Note that the symplectic areas of 
the corresponding complex spheres in $(\widetilde{M},\widetilde{\omega},\widetilde{J})$ and in $(\widetilde{M}',\widetilde{\omega}',\widetilde{J}')$ might differ by a positive factor.
Blowing down equivariantly, along the corresponding embedded invariant complex  (hence symplectic) spheres in $S^1 \acts (\widetilde{M},\widetilde{\omega},\widetilde{J})$ and in $S^1 \acts (\widetilde{M}',\widetilde{\omega}',\widetilde{J}')$ yields $S^1 \acts (\overline{M},\overline{\omega})$ and $S^1 \acts (\overline{M'},\overline{\omega'})$  with extended decorated graphs  (with respect to generic metrics).
The map $\rho' \circ \psi \circ \rho^{-1}$ between the graphs of the blown down manifolds is a composition of an isomorphism and a multi-rescaling. It induces a composition of an isomorphism $\bar{\rho}$ and a multi-rescaling $\overline{\psi}$ between extended decorated graphs (with respect to a generic metric) of Hamiltonian $S^1$-manifolds.
By the induction hypothesis, $\overline{\psi}$ is induced by a map between the blown down manifolds that is an $S^1$-compatible positive rescaling of the symplectic blowup form, composed with an equivariant symplectomorphism.
Therefore,
 $\rho' \circ \psi \circ \rho^{-1}$ is induced by an $S^1$-compatible positive rescaling of the symplectic blowup form, composed with an equivariant symplectomorphism. 
 Hence so is $\psi$. 
This completes the proof.
\end{proof}

\begin{Remark}
We have thus shown that the first two combinatorial maps (a) and (b) in Lemma \ref{equiv} are induced by orientation-preserving equivariant diffeomorphisms.
The third is as well, but the proof is much more subtle and will be addressed in 
\cite{hkt}.
\end{Remark}

The multi-rescaling (b) is powerful: we may use it to prove that a simply connected Hamiltonian $S^1$-manifold is
equivariantly diffeomorphic to a toric one.
This is reminiscent of the example explored
in \cite[Theorem~3.2]{HauKn}: equilateral pentagon
space admits no periodic Hamiltonian function, but it
is diffeomorphic to a toric four-manifold.

\begin{Corollary}\label{cor:rei}
Every compact, connected, simply connected four-dimensional Hamiltonian $S^1$-manifold 
is equivariantly diffeomorphic to one that extends to a toric action. 
Moreover, the equivariant diffeomorphism can be chosen to be orientation-preserving, where the orientations are the ones induced by the symplectic forms. 
\end{Corollary}

\begin{proof}
By \cite[Proposition 5.21]{karshon}, a Hamiltonian circle action on a compact, connected symplectic four-manifold extends to a toric action if and only if 
\begin{itemize}
\item [(i)] each fixed surface has genus $0$, and 
\item [(ii)] each non-extremal level set for
the moment map contains at most two non-free orbits. 
\end{itemize}
Let $S^1 \acts (M,\omega)$ be a compact, connected, 
simply-connected, four-dimensional Hamiltonian 
$S^1$-manifold.
The ``simply connected'' hypothesis guarantees condition 
(i) is satisfied.
Moreover, by \cite[Proposition 5.13]{karshon}, if $M$
has only isolated fixed points, 
then condition (ii) is also satisfied.

Now assume that there is at least one fixed sphere. 
We consider the chains in the extended decorated graph with 
respect to a generic metric.
Assume, without loss of generality, 
that the maximal vertex is fat.
The minimal vertex could correspond to a fixed surface or to an isolated fixed point.
We fix an order of the chains so that the labels on edges 
emanating from the minimum 
are non-decreasing  $m_{1,1} \geq m_{2,1} \geq 1=m_{3,1}=\cdots=m_{k,1}$,
where $k\geq 2$.
When $k=2$, item (ii) is satisfied automatically.
For $k>2$, we will perform a multi-rescaling to adjust the heights
of the vertices in the $3^{\mathrm{rd}}$ through $k^{\mathrm{th}}$.  Let $a$ be
the maximum value of the moment map labels for the isolated vertices in first and
second chains, and let $b$ be moment map label for the maximal fixed
surface.  We perform a multi-rescaling so that 
\begin{itemize}
\item the spheres with labels
$m_{3,2},\dots,m_{3,\ell_3-1}$ in the $3^{\mathrm{rd}}$ chain have moment image
in the interval $\left(\ a\ , \ b-\frac{b-a}{2} \right)$;

\item the spheres with labels
$m_{4,2},\dots,m_{4,\ell_4-1}$ in the $4^{\mathrm{th}}$ chain have moment image
in the interval $\left(  \ b-\frac{b-a}{2} \ , \ b-\frac{b-a}{2^2} \ \right)$; 
and so forth, with finally

\item the spheres with labels
$m_{k,2},\dots,m_{k,\ell_k-1}$ in the $k^{\mathrm{th}}$ chain have moment image
in the interval $\left(  \ b-\frac{b-a}{2^{k-3}} \ , \ b-\frac{b-a}{2^{k-2}} \ \right)$.
\end{itemize}
The impact of this multi-rescaling on the
extended decorated graph is indicated 
in Figure~\ref{fig:rescale to toric}.
By \cite[\S5--7]{karshon}, there exists an invariant symplectic
form $\widetilde{\omega}$ on $M$ so that the associated 
moment map $\widetilde{\Phi}$ has this new image.We have constructed this $(S^1 \acts M,\widetilde{\omega},\widetilde{\Phi})$ precisely so that it satisfies item (ii).

\newpage

\begin{center}
\begin{figure}[h]
\includegraphics[height=8cm]{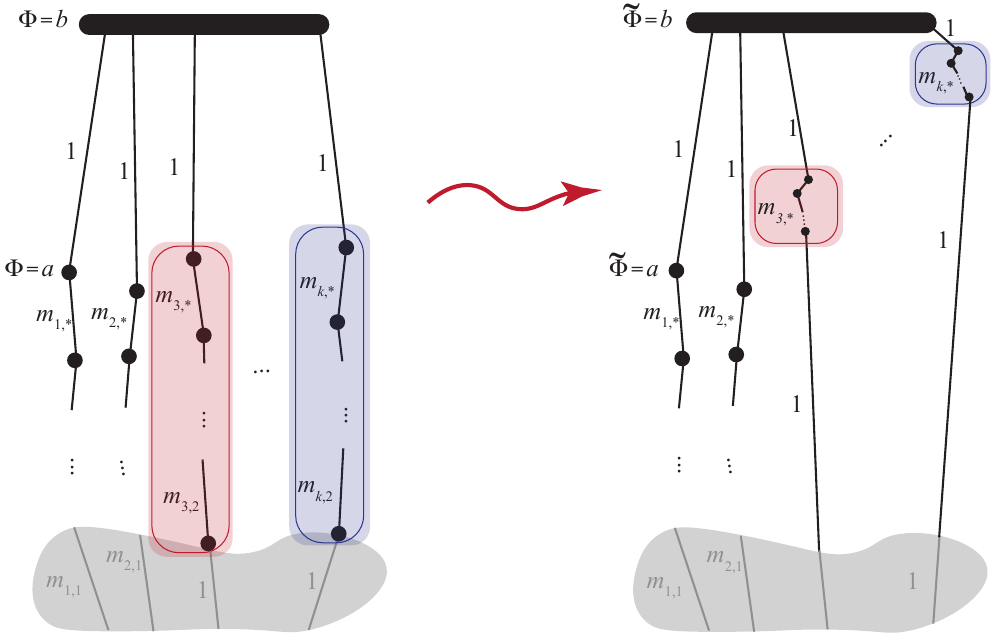} 
\caption[.]{This figure indicates the multi-rescaling that shifts moment images of spheres
with labels greater than $1$ so that any non-extremal level set of the moment map contains at
most two non-free orbits.  As indicated, the non-free orbits of the 
$3^{\mathrm{rd}}$ chain are above the non-free
orbits of the first two chains; and so forth until the non-free orbits in the $k^{\mathrm{th}}$ 
chain are above all other non-free orbits.
The area at the bottom of the figures is indicated in gray,
as the procedure applies both when 
the minimum corresponds to a
fixed surface and to when it is an isolated fixed  point.}
\label{fig:rescale to toric}
\end{figure}
\end{center}

 By Proposition \ref{scale1}, the multi-rescaling of the extended graph is induced by an or\-ien\-ta\-tion-preserving equivariant diffeomorphism of $S^1 \acts M$.
 This diffeomorphism takes  $\omega$ to $\widetilde{\omega}$
and $\Phi$ to $\widetilde{\Phi}$.
For values of $\widetilde{\Phi}$ above the minimum and up to $a$, 
there may be
up to two non-free orbits in the level set; and for values 
above $a$ and below $b$, there is is at most one non-free 
orbit in the level set.
That is, item (ii) holds.  
Therefore, the circle action on
$M$, equipped with the multi-rescaled symplectic form $\widetilde{\omega}$,
extends to a toric action, completing the proof.
\end{proof}

\section{A Generators-and-Relations description:\\ Notation, statements and corollaries} \labell{sec3}

The goal of this section is to give a generators and relations presentation for  
the even part $H_{S^1}^{2*}(M;\Z)$ of the equivariant cohomology of a 
Hamiltonian circle action on a compact, connected, four-dimensional symplectic manifold, 
as an algebra over $H_{S^1}^{*}(\pt)$.

\subsection*{Even degree equivariant cohomology}

\begin{Notation}[The Generators]\labell{not:gens}
Let $S^1\acts (M^4,\omega)$ be a compact, connected, four-dimensional symplectic manifold equipped with a 
Hamiltonian circle action. 
Consider the associated extended decorated graph with respect to a generic compatible K\"ahler metric.
Suppose that the extended decorated graph consists
of $k$ chains $C_1,\dots,C_k$
of edges between the maximum and minimum vertices. Note that $k \geq 2$, by our conventions in \ref{decorated1}. 
  If there are no fixed surfaces then the number of chains $k=2$. 
For each chain $1 \leq i \leq k$, let $\ell_i$ be the number of edges in the chain $C_i$; we enumerate the edges on $C_i$ by their order in the chain, starting from the bottom.
 Denote by $m_{i,j}$ the label of the $(i,j)$ edge.
 Without loss of generality, we assume that if there is exactly one fat vertex then it is maximal. 
We also fix an ordering of the chains so that $m_{1,1} \geq m_{2,1} \geq \cdots \geq m_{k,1}$.

If the minimal vertex is fat, denote its moment map preimage by $\Sigma_0$. If the maximal vertex is fat denote its preimage by $\Sigma_{\infty}$.
If the $j^{\mathrm{th}}$ edge on the $i^{\mathrm{th}}$ chain is not ephemeral  
(see Proposition \ref{claim:eph}), we fix  an 
$S^1$-invariant embedded symplectic sphere  $S_{i,j}$ corresponding to that edge  by the moment map.
By Remark \ref{rem:spheres}, each of the preimages of $\Sigma_0$ and $\Sigma_{\infty}$
is an invariant embedded symplectic surface, and the $S_{i,j}$s  exist.
If $m_{i,j}>1$ the sphere $S_{i,j}$ is a $\Z_{m_{i,j}}$-sphere. 
Every two distinct spheres $S_{i,j} \neq S_{i',j'}$ are either disjoint or intersect at
a single isolated fixed point.

We orient the manifold $M$ and any symplectic
submanifold using the symplectic form.  This allows us to use push-forward maps
to define the following degree $2$ classes
$$
\tau_0:=\iota_{\Sigma_0}^{!}(\mathbb{1}), \,\, \tau_{\infty}:=\iota_{\Sigma_{\infty}}^{!}(\mathbb{1}), \,\, \sigma_{i,j}:=\iota_{S_{i,j}}^{!}(\mathbb{1}).
$$
The restriction of one of these classes to ordinary cohomology is the class Poincar\'e dual to the surface, defined using the orientation on $M$ coming from the symplectic form.
If $\Sigma_0$ ($\Sigma_\infty$) does not exist, we set $\tau_0=0$ ($\tau_\infty=0$).
We then define
$$\tau_h:=\sum_{j=1}^{\ell_1} m_{1,j}\sigma_{1,j}.$$
Note that in a graph with two fat vertices and zero isolated vertices, corresponding to an $S^1$-ruled symplectic $S^2$-bundle over a compact surface, $\sigma_{1,1}=\tau_h=\sigma_{2,1}$ and $m_{1,1}=1=m_{2,1}$.  Thus,  in such a graph, $\tau_h=\iota_{S}^{!}(\mathbb{1})$ where $S$ in an invariant embedded symplectic sphere in the fiber class.
 In a graph with exactly one fat vertex, if the first edge in the $i^{\mathrm{th}}$ chain is ephemeral, denote 
$$
\sigma_{i,1}:=\tau_h-\sum_{j=2}^{\ell_i}m_{i,j}\sigma_{i,j}.
$$

Denote by $\max$ ($\min$) the fixed component of maximal (minimal) value of the moment map, it can be either a fixed surface $\Sigma_{\infty}$ ($\Sigma_{0}$) or an isolated vertex $v_{\infty}$ ($v_0$).

For $1 \leq i \leq k$ and $1\leq j < \ell_i$ denote by $v_{i,j}$ the south pole of the $S^1$-invariant embedded symplectic sphere $S_{i,j+1}$ whose moment map image is the $(i,j+1)$ edge in the extended decorated graph, i.e., the point on $S_{i,j+1}$ that is of minimal moment map value. 
We use the same notation $v_{i,j}$ for the corresponding (isolated) vertex of the decorated graph.                             
\end{Notation}

\begin{Notation}[The Relations]\labell{not:rels}
There are two types of relations among the generators defined above, multiplicative relations and  linear relations. 

The multiplicative relations can be verified using localization.  They hold because the submanifolds that are Poincar\'e dual to the classes can be chosen to be disjoint.  We define the {\bf multiplicative ideal} $\mathcal{I}$ to be generated by
\begin{enumerate}
\item[A.] $\tau_0\cup \tau_\infty$
\item[B1.] $\tau_0 \cup \sigma_{i,j}$ for every $2\leq j\leq \ell_i$
\item[B2.] $\tau_\infty \cup \sigma_{i,j}$ for every $1\leq j\leq \ell_i -1$
\item[C.] $\sigma_{i,j} \cup \sigma_{m,n}$ whenever the edges do not share a thin vertex
\item[D1.]  $\tau_h^2$ when there is both a fat minimum and fat maximum
\item[D2.] $\tau_{\infty} \cup  \sigma_{1,1} \cup \sigma_{2,1}$  when $M=\CP^2$ and there is a fat vertex
\item[D3.]   $\sigma_{1,1} \cup \sigma_{2,1} \cup \sigma_{2,2}$ when $M=\CP^2$ and there is no fat vertex
\end{enumerate}

\noindent For item A, this is redundant when there are not two fat vertices.  For items of type B, these are redundant
if there is no fat minimum or no fat maximum, respectively.  
Type C applies to classes where there are Poincar\'e dual embedded spheres.
Items of type D apply in special cases as indicated. Note that D1 follows from C, but we keep the relation for bookkeeping purposes.

The linear relations can also be verified by localization.  We define
the {\bf linear relation ideal} to be
$$
\mathcal{J} = \left \langle  \tau_h - \sum_{j=1}^{\ell_i} m_{i,j}\sigma_{i,j}\ \  \ \forall \ 1 \leq i \leq k\right\rangle.
$$
If the action $S^1\acts M$ extends to a toric $T^2\acts M$, 
the reader familiar with toric varieties will 
recognize these relations as part of a linear system on the $T^2$-equivariant cohomology of $M$.  
A full linear system describes the kernel of the
restriction map from $T^2$-equivariant to ordinary cohomology.  In this case, $\mathcal{J}$
describes the difference between the $T^2$- and $S^1$-equivariant cohomology
rings:
$$
H_{T^2}^*(M;\Z)/\mathcal{J} \cong H_{S^1}^*(M;\Z).
$$
\end{Notation}

We are now prepared to state our main theorem describing $H_{S^1}^{2*}(M;\Z)$ by generators and relations.

\begin{Theorem} \labell{ThNew}
Let $(M,\omega)$ be a compact, connected, four-dimensional symplectic manifold endowed with a Hamiltonian $S^1$-action.  Then
\begin{eqnarray*}
H_{S^1}^{2*}(M;\Z) &=& \frac{\Z[\tau_0,\tau_\infty,\tau_h,\sigma_{1,1},\dots,\sigma_{1,\ell_1},\dots, \sigma_{k,1},\dots,\sigma_{k,\ell_k}]}
{\mathcal{I}+\mathcal{J}
},
\end{eqnarray*}
where $\mathcal{I}$ is the multiplicative ideal and $\mathcal{J}$is  the linear relation ideal.
Moreover, the map $\pi^*:H_{S^1}^*(pt;\Z)\to H_{S^1}^{2*}(M;\Z)$ 
endows $H_{S^1}^{2*}(M;\Z)$
with the structure of an $H_{S^1}^*(pt;\Z)$-algebra.    This structure is determined
by the image of the generator 
\begin{equation} \labell{eq:algstr}
\pi^*(t)  =  \tau_\infty-\tau_0+(\Sigma_0 \cdot \Sigma_0) \tau_h +\sum_{i=1}^{k} \sum_{j=1}^{\ell_i} b_{i,j} \sigma_{i,j}, 
\end{equation}
where the $b_{i,j}$s are integers satisfying the properties listed in Lemma \ref{lem:comb} below.
\end{Theorem}
We note that we can omit the $\sigma_{i,j}$s that correspond to ephemeral edges from the list of generators, and moreover omit $\sigma_{i,1}$ for $i \geq 3$ in general, since they are linear combinations of the other $\sigma_{i,j}$s over $\Z$.
Some of the listed generators might be the zero element:  $\tau_\infty$ if $\sharp \text{fat vertices}=0$; $\tau_0$ if $\sharp \text{fat vertices}<2$.

\begin{Lemma} \labell{lem:comb}
For  $1\leq i\leq k$ and  $1\leq j\leq \ell_i$,
there are integers
$b_{i,j}$ so that for $j\geq 2$,
\begin{equation}\label{bsms}
b_{i,j}m_{i,j-1} - b_{i,j-1}m_{i,j}=1.
\end{equation}
Once we fix the first two $b_{i,1}$ and $b_{i,2}$, the $b_{i,j}$s are 
determined recursively for $j>2$ to satisfy $b_{i,2}m_{i,1}-b_{i,1}m_{i,2}=1$. 
We furthermore fix the $b_{i,j}$s so that if there is a maximal fixed surface,
then
\begin{equation} \label{bsf1}
\Sigma_{\infty} \cdot \Sigma_{\infty}+\Sigma_0 \cdot \Sigma_0+\sum_{i=1}^{k} b_{i,\ell_i}=0,
\end{equation}
and if additionally there is a minimal fixed surface, then
\begin{equation}\label{bsf2}
\sum_{i=1}^{k} b_{i,1}=0.
\end{equation}
Finally, we choose $b_{i,1}$ and $b_{i,2}$ as follows
so that the $b_{i,j}$s have the following additional properties depending on the nature of the dull graph.

\begin{itemize} 
\item[a.] Assume that there are two fixed surfaces.
We  choose $b_{i,1}=0$. For $\ell_i>1$ we  choose $b_{i,2}=1$; we then have $b_{i,\ell_i}=1$. 

\item[b.] Assume that there are no fixed surfaces (hence $k=2$). We  choose the $b_{i,j}$s such that they satisfy the gcd relation cyclically, i.e.,
$$
-b_{1,1}m_{2,1} - b_{2,1}m_{1,1}=1
\mbox{
and }
b_{2,\ell_2}m_{1,\ell_1} + b_{1,\ell_1}m_{2,\ell_2}=1.
$$
Moreover, if $m_{1,\ell_1}=1=m_{2,\ell_2}$ we  choose $b_{1,\ell_1}=0$ and $b_{2,\ell_2}=1$, such that if $\ell_i=1$ then $b_{i,\ell_i}=0$.
Alternatively, if $m_{1,1}=1=m_{2,1}$ we  choose $b_{1,1}=0$ and $b_{2,1}=-1$, such that if $\ell_i=1$ then $b_{i,1}=0$.
Note that we might not be able to make the latter two choices simultaneously.

\item[c.] Assume that there is exactly one fixed surface.  By convention, it is maximal 
and the chains are ordered such that  
$m_{1,1}\geq m_{2,1}\geq 1=m_{3,1}=\cdots = m_{k,1}$. 
If $\ell_i\geq2$ for some $i$, we  choose the $b_{i,j}$s in the first two chains 
so that the cyclic gcd relation
$$
-b_{1,1}m_{2,1} - b_{2,1}m_{1,1}=1
$$
is satisfied. 
For the remaining $k-2$ chains we  choose 
$b_{i,1}=0, \, b_{i,2}=1$, which yields $b_{i,\ell_i}=1$ as in the two-surface
case.
If $k=2$ and $\ell_1=\ell_2=1$, then we set  $b_{1,1}=0$ and $b_{2,1}=-1$.
\end{itemize}
\end{Lemma}

\begin{proof}
The existence of integers $b_{i,j}$s  that satisfy the basic property \eqref{bsms}  
is proved in \cite[Lemma 5.7]{karshon}; we include the proof here for completeness.
To prove that we can set the $b_{i,j}$s such that \eqref{bsms}, \eqref{bsf1}, \eqref{bsf2} hold, and verify items (a), (b) and (c), we apply straight forward induction arguments.

\noindent {\textsc{Base Case:  The minimal models.}}

\begin{center}
\begin{figure}[h]
\includegraphics[height=5cm]{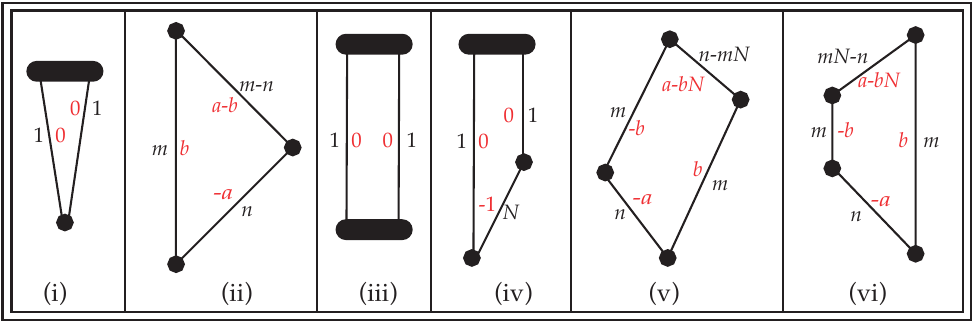} 

\caption[.]{
The base cases.}
\label{fig:toric projection1}
\end{figure}
\end{center}

In Figure~\ref{fig:toric projection1}, the $m_{i,j}$ are marked in black
and the $b_{i,j}$ are marked in red for the possible labeled graphs for $S^1\acts M$.  
Since $m$ and $n$ are relatively prime, we have fixed $a,b\in \Z$ such that  $am-bn=1$ as in \eqref{eq:gcd}.
We have also indicated in blue the self intersections of the fat vertices. The self intersections are calculated using \S \ref{factsi} , Figures \ref{fig:circle on cp2} and \ref{fig:circle on hirzebruch}, and Notation \ref{basic notions}.
It is then straight-forward to show that the relations \eqref{bsms}, \eqref{bsf1}, \eqref{bsf2}, and (a), (b), and (c) hold for the minimal
models.
For example, Figure~\ref{fig:toric projection1}(ii), 
for the length two chain, we verify \eqref{bsms} by
$$
(a-b)\cdot n - (-a)\cdot (m-n) = an-bn+am-an =am-bn=1;
$$
for the cyclic relation of part (b) of the Lemma, around the top, we have
$$
(a-b)\cdot m + b\cdot (m-n) = am-bm+bm-bn = am-bn = 1;
$$
and for the cyclic relation of part (b) of the Lemma, around the bottom
$$
-b\cdot n - (-a)\cdot m = am-bn = 1.
$$

\medskip

\noindent {\textsc{Inductive Step.}}  The effect of a blowup, by type.  In each case, the old and new
$m_{i,j}$, $b_{i,j}$ and self intersection of a fat vertex are marked in black, red and blue respectively.

\begin{center}
\begin{figure}[h]
\includegraphics[height=3.5cm]{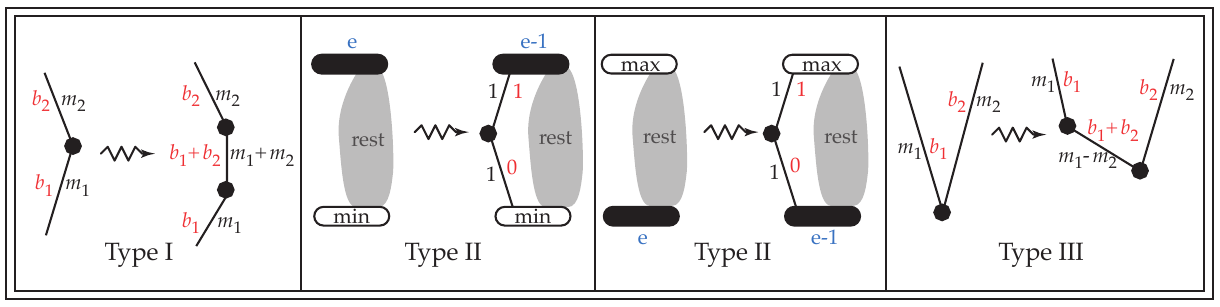} 

\includegraphics[height=3.5cm]{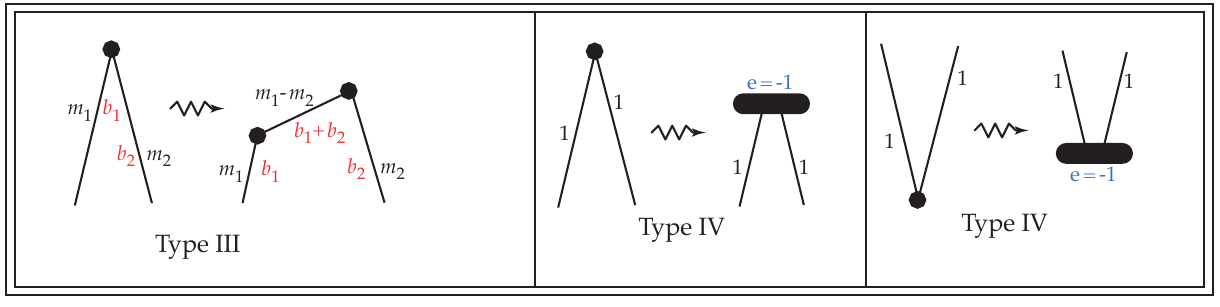} 

\caption[.]{
The inductive step.  In the last figure, we must be careful to flip if there was not a fat vertex at the top, and reassign the $b_{i,j}$s according to our convention.}
\label{fig:toric projection2}
\end{figure}
\end{center}

Again, straight forward computations show that the properties \eqref{bsms}, \eqref{bsf1}, \eqref{bsf2}, (a), (b) and (c) are maintained.  For example, in
Type I we assume by induction that
$$
b_2m_1-b_1m_2 = 1
$$
which then lets us deduce that
$$
b_2\cdot (m_1+m_2)-(b_1+b_2)\cdot m_2 = b_2m_1+b_2m_2 - b_1m_2 - b_2m_2 = b_2m_1-b_1m_2 = 1,
$$
and similarly
$$
(b_1+b_2)\cdot m_1-b_1\cdot (m_1+m_2) = b_1m_1+b_2m_1-b_1m_1-b_1m_2 = b_2m_1-b_1m_2 = 1.
$$

One must be careful when introducing a fat vertex in a blowup of Type IV.  
In case the fat vertex introduced at the blowup is the first one, and is maximal, we set $b_{1,\ell_1}=0$ and $b_{2,\ell_2}=1$, such that if $\ell_i=1$ then $b_{i,\ell_i}=0$, at the left graph (before the blowup); the blowup has no effect then on the $b_{i,j}$s, and in particular $b_{1,\ell_1}=0$ and $b_{2,\ell_2}=1$ at the right graph as well.
In case a second fat vertex is introduced at the blowup,
one may need to re-define all of the $b_{i,j}$ because what results from the cyclic convention does not agree with
the convention described in (a).  Again, one can argue inductively: once there are two fixed surfaces, it is possible to equivariantly blow
down to a minimal model with two fixed surfaces \cite[Lemma C.14]{karshon},
and re-start the inductive process from that minimal model, using
only blowups of Types I and II.
This completes the inductive step.

It is also possible to prove (b) directly, the case where there are isolated fixed points.  Indeed, 
here the action must extend to a toric one, with the
circle  action corresponding to $\begin{bmatrix}m\\n\end{bmatrix}$.  Again, we have fixed $a$ and $b$ so that
$
am-bn=1
$ as in \eqref{eq:gcd}.
The $m_{i,j}$ and $b_{i,j}$ can now be defined in terms of the toric action.  To see this, we let $e_1,\dots, e_{L}$ be the vectors
parallel to the edges of the toric polygon, as in Figure~\ref{fig:toric projection3}(i).  

\begin{center}
\begin{figure}[h]
\includegraphics[height=5cm]{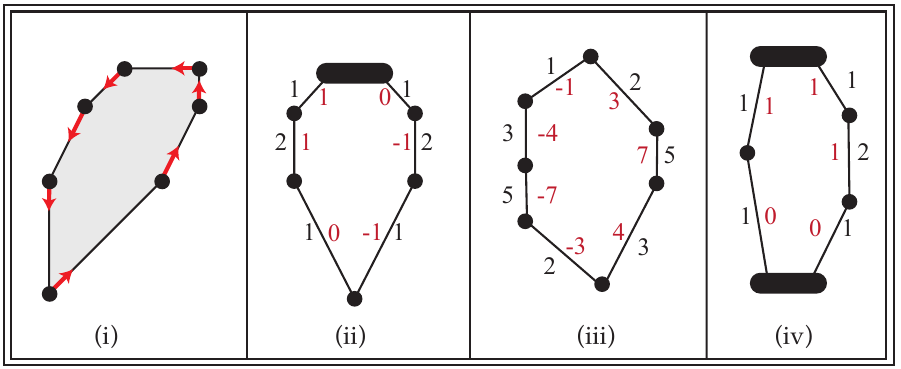} 

\caption[.]{
In (i) is Delzant polytope with edge directions $e_i$ marked as red vectors. In (ii), (ii) and (iv) are the labeled graphs
for three different choices of $\left(\begin{array}{cc}m\\ n\end{array}\right)$ and
with isotropy labels $m_{i,j}$ in black and {\color{red} $b_{i,j}$ in red}.}
\label{fig:toric projection3}
\end{figure}
\end{center}

Define 
$$
\mu_i = \left\langle \begin{bmatrix}m\\ n\end{bmatrix}, e_{i}\right\rangle  \mbox{ and }
\beta_i = \left\langle \begin{bmatrix}b\\ a\end{bmatrix}, e_{i}\right\rangle.
$$
That the circle action corresponding to $\begin{bmatrix}m\\n\end{bmatrix}$ has isolated fixed points
means that no $\mu_i=0$.
We check that
\begin{eqnarray*}
\beta_{i+1}\mu_i-\beta_i\mu_{i+1}
& = &
\left\langle \begin{bmatrix}b\\ a\end{bmatrix}, e_{i+1}
\right\rangle  \cdot \left\langle \begin{bmatrix}m\\ n\end{bmatrix}, e_{i}
\right\rangle -
\left\langle \begin{bmatrix}b\\ a\end{bmatrix}, e_{i}
\right\rangle  \cdot \left\langle \begin{bmatrix}m\\ n\end{bmatrix}, e_{i+1}
\right\rangle \\
& = &
\det \left(
\begin{bmatrix}
 m & n  \\
 b & a  \\
\end{bmatrix}\cdot
 \begin{bmatrix}
\top & \top \\
e_i & e_{i+1}\\
\bot & \bot
\end{bmatrix}
\right)\\
& = &
\det \left(
\begin{bmatrix}
 m & n  \\
 b & a  \\
\end{bmatrix}\right) \cdot
\det\left( \begin{bmatrix}
\top & \top \\
e_i & e_{i+1}\\
\bot & \bot
\end{bmatrix}
\right) = 1\cdot 1.\\
\end{eqnarray*}
This wraps around modulo $L$ where $L=\ell_1+\ell_2$ is the number of edges in the polygon. 

We note that $\mu_i>0$ on one chain and $\mu_i<0$ on the other chain.  The labels $m_{i,j}$ then correspond precisely to the
appropriate $|\mu_k|$.  The $b_{i,j}$ are exactly $\beta_k$.  One then has to carefully check that the sign changes exactly
cancel out and the relations described in (b) continue to hold.
\end{proof}

\begin{noTitle} \labell{NoTitinj} {\bf Localization in equivariant cohomology.}
By  \cite[Theorem 1.1(A)]{KesslerHolm2} the inclusion of the fixed point components
$$i=\bigoplus_{F\subset M^{S^1}} \iota_F \colon M^{S^1} \hookrightarrow M$$
induces an injection in equivariant cohomology
$$i^{*} \colon H_{S^1}^{*}(M;\Z) \hookrightarrow H_{S^1}^{*}(M^{S^1};\Z)=
\bigoplus_{F \subset M^{S^1}}H^{*}(F) \otimes \Z[t].
$$
 Let $F$ be a connected component of the fixed point set $M^{S^1}$.
Let $X$ be an invariant embedded symplectic (oriented) sphere in $S^1 \acts (M,\omega)$, and $\eta=\iota^{!}_{X \hookrightarrow M}(\mathbb{1}_X)$.
Consider the following diagram of inclusion maps:
\begin{equation} \labell{diagramrestrict}
\begin{array}{c}
\begin{xymatrix}{
 F \cap X \ar[rr]^{\iota_{F \cap X \hookrightarrow F}} \ar[d]_{\iota_{F \cap X \hookrightarrow X}}
 &{\color{white}.} &  F  \ar[d]^{\iota_{F \hookrightarrow M}} \\
 X\ar[rr]_{\iota_{X \hookrightarrow M}} & {\color{white}.} & M.}
\end{xymatrix}
\end{array}
\end{equation}
Then, by the push-pull property of the pushforward map, the \emph{restriction} of $\eta$ to $F$
\begin{equation} \labell{eqrestriction}
 \eta|_F:=\iota^{*}_{F  \hookrightarrow M} \circ \iota^{!}_{X \hookrightarrow M}(\mathbb{1}_{X})=\iota^{!}_{F \cap X \hookrightarrow F} \circ \iota^{*}_{F \cap X \hookrightarrow X}(\mathbb{1}_{X}).
 \end{equation}
In particular, if $X$ and $F$ do not intersect then $\eta|_{F}=0$.

The pushforward map is related to the Euler class by $\iota_X^{*}\left(\iota_{X}^{!}(\mathbb{1})\right)=e_{S^1}\big(\nu(X\subset M)\big)$. The equivariant Euler classes are computed in 
\cite[\S 4]{KesslerHolm2}.
As a result of this discussion, we establish Tables \ref{Table-rest-2}, \ref{Table-rest-0},
and \ref{Table-rest-1} of the 
restriction of the listed generators to the components of $M^{S^1}$. 
For a class $\eta \in H_{S^1}^{*}(M;\Z)$, the {\bf support} of $\eta$ is the set of
fixed point components on which $\eta|_F\neq 0$.
\end{noTitle}

\begin{proof}[Proof of the algebra structure \eqref{eq:algstr}]\labell{proof-algebra}
Since the map $i^{*} \colon H_{S^1}^{*}(M;\Z) \hookrightarrow H_{S^1}^{*}(M^{S^1};\Z)$
is injective  \cite[Theorem 1.1(A)]{KesslerHolm2}, it is enough to show  that for each connected component $F$, the restriction $\iota^{*}_{F \hookrightarrow M} $ of the right hand side to $F$ equals $\iota^{*}_{F  \hookrightarrow M} \circ \pi_M^{*}(t)=\pi_F^{*}(t)$, which equals $1 \otimes t$ if $F$ is a fixed surface and $t$ if $F$ is an isolated fixed point.
 This follows from Tables \ref{Table-rest-2}, \ref{Table-rest-0}, \ref{Table-rest-1}, justified in \S\ref{NoTitinj}, and Lemma \ref{lem:comb}.
\end{proof}

Let $\Sigma \subset M^T$ be a fixed surface. There is a unique orientation on the normal bundle $\nu(\Sigma\subset M)$ to $\Sigma$ so that $S^1$ acts on $\nu_p \Sigma$ with weight $+1$ for each $p \in \Sigma$.  We call this the 
{\bf positive-normal} orientation.
Note that we now possibly have two ways to orient the surface $\Sigma$ itself: 
\begin{itemize}
\item restricting the symplectic form on $M$ to an orientation on $\Sigma$, or 
\item combining the symplectic orientation on
$M$ and the positive-normal orientation on $\nu(\Sigma\subset M)$ to induce an orientation 
on $\Sigma$.
\end{itemize}
These two possibilities
are the same when $\Sigma$ is a minimal fixed surface and are different 
if $\Sigma$ is a maximal fixed surface.

We now define some additional classes that
will be key players in our understanding of isomorphisms of cohomology rings.
\begin{Definition}\labell{def:comp euler}
For each connected component $F\subset M^{S^1}$ of the fixed set, we define the {\bf component Euler class} $\varepsilon_F$ of $F$ 
in terms of its restrictions to the fixed components.
Specifically, we define it by
$$
\iota^*_{F'}(\varepsilon_F) = \left\{ \begin{array}{ll}
e_{S^1}(\nu(F\subset M)) \in H^{4-\dim(F)}_{S^1}(F;\Z) & \mbox{ when } F'=F; \\
0 & \mbox{ for all } F'\neq F,
\end{array}
\right.
$$
where $e_{S^1}(\nu(F\subset M))$ is defined using
the positive-normal orientation on $\nu(F\subset M)$ when $F$ is 
a fixed surface, and using the symplectic orientation on
$\nu(F\subset M) = T_FM$ when $F$ is an isolated fixed point.
Because this collection $\left( \iota^*_{F'}(\varepsilon_F)\right)_{F'\subset M^{S^1}}$ satisfies the compatibility conditions
in \cite[Theorem~1.1(C)]{KesslerHolm2}, 
it is a class in $i^*(H^*_{S^1}(M;\Z))$. The map $i^*$ is injective,
and hence the collection is the image of a well-defined class $\varepsilon_F\in H^{4-\dim(F)}_{S^1}(M;\Z)$.
\end{Definition}

\begin{Fact} \labell{fact1}
For the minimal and maximal fixed surfaces, the component Euler classes
are the classes 
$$
\varepsilon_{\Sigma_0} =\tau_0  \mbox{ and } \varepsilon_{\Sigma_\infty} = -\tau_{\infty}.
$$
    For an interior isolated fixed point $v_{i,j}$, the component Euler class
    $\varepsilon_{v_{i,j}}$ is $\sigma_{i,j} \cup \sigma_{i,j+1}$. 
For the extremal isolated fixed points $v_0$ and $v_{\infty}$, the component Euler classes are 
$\sigma_{1,1} \cup \sigma_{2,1}$ and $\sigma_{1,\ell_{1}} \cup \sigma_{2,\ell_{2}}$. 
Note that for $F\in \{\Sigma_0,\Sigma_{\infty}, v_0, v_{\infty}\}$, the component $F$ exists if and only if  $\varepsilon_{F}\neq 0$.
One can use the component Euler classes to distinguish the interior and extremal isolated vertices.  Indeed,
for an isolated vertex $F$, 
$$\iota^*_F(\varepsilon_F)\  \in H^4_{S^1}(F;\Z)=\Z\cdot t^2.$$  
Under the identifications we have made, the integer coefficient of this Euler class 
$\iota^*_F(\varepsilon_F)$ is precisely the product of the 
weights of the $S^1$-action on the tangent space $T_F M$, which can be equipped
with a complex structure compatible with the symplectic form.
Thus, for interior isolated vertices, the coefficient of $t^2$ in
$\iota^*_F(\varepsilon_F)$ is negative, while for extremal ones, the coefficient of $t^2$
in $\iota^*_F(\varepsilon_F)$ is positive.
\end{Fact}

\noindent We deduce the following corollary from
 \eqref{eq:algstr}, Lemma \ref{lem:comb} and Fact \ref{fact1}.
\begin{Corollary}\label{cor:pit}
$$\tau_0 \cup \pi^{*}(t)=-\tau_0 \cup \tau_0+(\Sigma_0 \cdot \Sigma_0)\, \tau_0 \cup \tau_h=-\tau_0 \cup \tau_0+e_{\min}\, \tau_0 \cup \tau_h.$$
$$\tau_\infty \cup \pi^{*}(t)=\tau_\infty \cup \tau_\infty-(\Sigma_\infty \cdot \Sigma_\infty) \,\tau_\infty \cup \tau_h=\tau_\infty \cup \tau_\infty-e_{\max} \,\tau_\infty \cup \tau_h.$$
$$\tau_h \cup \pi^{*}(t)=\begin{cases}
                                               \tau_h \cup \tau_\infty -\tau_h \cup \tau_0 & \text{ if }\fat=2\\
                                               \tau_h \cup \tau_\infty-\sigma_{1,1}\cup \sigma_{2,1} = \tau_h \cup \tau_\infty-\varepsilon_{v_0} & \text{ if } \fat=1\\
                                               \sigma_{1,\ell_1} \cup \sigma_{2,\ell_2} -\sigma_{1,1} \cup \sigma_{2,1}=\varepsilon_{v_\infty}-\varepsilon_{v_0} & \text{ if }\fat=0\\
                                               \end{cases}.$$
                                               In the case $\fat=1$ here we assume that $\tau_\infty \neq 0$ as we do in this section.
\end{Corollary}

\begin{Notation} \label{not:labels}
We set
$$
\sigma_{i,0}=\begin{cases}
\tau_0 & \text{ if }\fat=2\\
\sigma_{i^*,1}  & \text{ if } \fat=1, 0 \\
   \end{cases};
\mbox{ and }
m_{i,0}=\begin{cases}
0 & \text{ if }\fat=2 \\
-m_{i^*,1} & \text{ if } \fat=1,0 \\
\end{cases},
$$
and
$$
\sigma_{i,\ell_i+1}=\begin{cases}
\tau_\infty & \text{ if }\fat=2,1 \\
    \sigma_{i^*,\ell_{i^*}} & \text{ if } \fat=0 \\ 
\end{cases};
\mbox{ and }
m_{i,\ell_i+1}=\begin{cases}
0 & \text{ if }\fat=2,1 \\
    -m_{i^*,\ell_{i^*}} & \text{ if } \fat=0 \\ 
\end{cases}.
$$
Here, if $i \in \{1,2\}$ we set $i^{*}=1+i \ (\text{mod } 2)$; otherwise $i^*=1$.
\end{Notation}

We will need the following corollary of Theorem~\ref{ThNew} when we more closely investigate 
the algebra structure of $H_{S^1}^*(M;\Z)$. 

\begin{Corollary} \label{cor:relation}
We have the following relations among the generators.
\begin{enumerate}
\item For every $i$ and $1\leq j \leq \ell_i$ we have
\begin{eqnarray}\labell{eq:labelproducts}
m_{i,j} \sigma_{i,j} \cup \sigma_{i,j}&=&
-(m_{i,j-1} \sigma_{i,j} \cup \sigma_{i,j-1}+m_{i,j+1} \sigma_{i,j} \cup \sigma_{i,j+1}).
\end{eqnarray}
\item Moreover, for $\eta$ a linear combination over $\Z$ of classes of the form $\sigma_{r,s}$, 
with $1\leq s \leq \ell_r$, we have the following.
\begin{itemize}
\item If $\eta \cup \varepsilon_{v_{i,j}}=0$ for $1\leq j<\ell_i$ then the coefficients of $\sigma_{i,j}, \, \sigma_{i,j+1}$ are $\gamma m_{i.j},\, \gamma m_{i,j+1}$, respectively, with the same $\gamma \in \Z$.
\item  If $\varepsilon_{v_{0}}=\sigma_{1,1} \cup \sigma_{2,1} \neq 0$ and $\eta \cup \varepsilon_{v_{0}}=0$  then the coefficients of $\sigma_{1,1}, \, \sigma_{2,1}$ are $\gamma m_{1,1},\, -\gamma m_{2,1}$, respectively, with the same $\gamma \in \Z$.
\item  If $\varepsilon_{v_{\infty}}=\sigma_{1,\ell_1} \cup \sigma_{2,\ell_2} \neq 0$ and $\eta \cup \varepsilon_{v_{\infty}}=0$  then the coefficients of $\sigma_{1,\ell_1}, \, \sigma_{2,\ell_2}$ are $\gamma m_{1,\ell_1},\, -\gamma m_{2,\ell_2}$, respectively, with the same $\gamma \in \Z$.
  
  \item  If $\varepsilon_{\Sigma_{0}}=\tau_0  \neq 0$ and $\eta \cup (\varepsilon_{\Sigma_{0}} \cup \varepsilon_{\Sigma_{0}})=0$  then the coefficients of $\sigma_{i,1}, \, \sigma_{i',1}$ for $i\neq i'$ are $\gamma m_{i,1},\, -\gamma m_{i',1}$, respectively, with the same $\gamma \in \Z$.
   \item  If $\varepsilon_{\Sigma_{\infty}}=-\tau_\infty  \neq 0$ and $\eta \cup (\varepsilon_{\Sigma_{\infty}} \cup \varepsilon_{\Sigma_{\infty}})=0$  then the coefficients of $\sigma_{i,\ell_i}, \, \sigma_{i',\ell_{i'}}$ for $i\neq i'$ are $\gamma m_{i,\ell_i},\, -\gamma m_{i',\ell_{i'}}$, respectively, with the same $\gamma \in \Z$.
   \end{itemize}
   \end{enumerate}
\end{Corollary}

\begin{proof}
First we prove \eqref{eq:labelproducts}.
If $\fat=2$, or $\fat=1$ and $1<j$ or $\fat=0$ and $1<j<\ell_i$, we have 
\begin{eqnarray*}
0&=&\sigma_{i,j} \cup \tau_h=\sum_{s=1}^{\ell_i} m_{i,s}\sigma_{i,j} \cup \sigma_{i,s} \\
&=&m_{i,j-1}\sigma_{i,j} \cup \sigma_{i,j-1}+m_{i,j} \sigma_{i,j} \cup \sigma_{i,j}+m_{i,j+1}\sigma_{i,j} \cup \sigma_{i,j+1}.
\end{eqnarray*}
In the other cases the statement follows from  $\sigma_{i,j} \cup (\tau_h-\tau_h)=0$, e.g., if $\fat=0,1$,  and $\ell_i \geq 2$,
\begin{eqnarray*}
0&=&\sigma_{i,1} \cup (\tau_h-\tau_h) =\sigma_{i,1} \cup \Big(\sum_{s=1}^{\ell_i} m_{i,s}\sigma_{i,s} -\sum_{t=1}^{\ell_{i*}} m_{i*,t}\sigma_{i*,s}\Big) \\
&=&m_{i,1} \sigma_{i,1} \cup \sigma_{i,1}+m_{i,2} \sigma_{i,1} \cup \sigma_{i,2}-m_{i*,1} \sigma_{i,1} \cup \sigma_{i*,1}.
\end{eqnarray*}

Now we turn to a class $\eta$ that is a $\Z$-combination of classes $\sigma_{r,s}$ with $1\leq s \leq \ell_r$ and we consider the case
when  $\eta \cup \varepsilon_{v_{i,j}}=0$.
First assume $1\leq j<\ell_i$. The only classes $\sigma_{r,s}$ whose cup product with $\varepsilon_{v_{i,j}}=\sigma_{i,j} \cup \sigma_{i,j+1}$ 
is not zero are $\sigma_{i,j}$ and $\sigma_{i,j+1}$, so for $\eta=\sum_{r,s}a_{r,s} \sigma_{r,s}$,
$$\eta \cup \varepsilon_{v_{i,j}}=(a_{i,j}\sigma_{i,j}+a_{i,j+1}\sigma_{i,j+1}) \cup  \varepsilon_{v_{i,j}}.$$
Using \eqref{eq:labelproducts}, we have 
\begin{eqnarray*}
m_{i,j}\sigma_{i,j} \cup \varepsilon_{v_{i,j}}
&=& m_{i,j}\sigma_{i,j} \cup \sigma_{i,j} \cup \sigma_{i,j+1}\\
&=&-(m_{i,j-1} \sigma_{i,j} \cup \sigma_{i,j-1} \cup \sigma_{i,j+1}+m_{i,j+1} \sigma_{i,j} \cup \sigma_{i,j+1} \cup \sigma_{i,j+1})\\
&=&-m_{i,j+1} \sigma_{i,j+1} \cup \varepsilon_{v_{i,j}}.
\end{eqnarray*}
Thus, if $\eta \cup \varepsilon_{v_{i,j}}=0$, we have $a_{i,j}\sigma_{i,j} \cup \varepsilon_{v_{i,j}}=-a_{i,j+1}\sigma_{i,j+1} \cup \varepsilon_{v_{i,j}}$, and then either $a_{i,j}=0=a_{i,j+1}$ or $\frac{a_{i,j+1}}{a_{i,j}}=\frac{m_{i,j+1}}{m_{i,j}}$.
Hence, since $m_{i,j}$ and $m_{i,j+1}$ are relatively prime, $a_{i,j+1}=\gamma m_{i,j+1}$ and $a_{i,j} = \gamma m_{i,j}$ for $\gamma \in \Z$.
The proof of the other cases is similar.
\end{proof}

\subsection{Odd degree equivariant cohomology}
 By \cite{KesslerHolm2}, the equivariant Poincar\'e polynomial of $M$, over $\Z$, is
\begin{equation}\labell{eq:poincare}
\begin{array}{rcl}
P_{S^1}^{M}(t) & = & P^{M}(t) \cdot  \frac{1}{1-t^2}\\
& = & 1 + (\iso-1+2\fat)t^2 + 
(\iso+2\fat)t^4\left(\frac{1}{1-t^2}\right)\\
& &  + 
2g t +(\fat)2gt^3\left(\frac{1}{1-t^2} \right).
\end{array}
\end{equation}

It follows from \eqref{eq:poincare} that the odd degree ranks are determined 
by the genus of a fixed surface, if there is one and the number of fixed surfaces:
\begin{equation}\label{eq:firstbetti}
\beta_1 = 2g \mbox{ and } \beta_3 = (\fat) 2g.
\end{equation}
In particular, if $\fat=0$ or $g=0$ then the ranks of  $H_{S^1}^{2*+1}(M;\Z)$ 
are all zero.

\section{A generators and relations description:\\
 Proof of Theorem~\ref{ThNew}}\labell{sec4}

By \S \ref{preface}, to complete the proof of Theorem \ref{ThNew}, 
it is enough to 
 give a generators and relations presentation of the equivariant cohomology algebras in the minimal models, and describe the effect of an equivariant K\"ahler blowup on such a presentation.

\begin{noTitle} {\bf The case of circle actions that extend to toric actions.}
 A toric symplectic manifold yields a toric variety with fan defined by the moment map polytope.
  The equivariant cohomology of a toric variety $X$ is described by \cite[Proposition 2.1]{masuda}. 
 The generators are the $H_{T}^{2}(X)$-classes 
 \begin{equation*} 
\Upsilon_i=\iota_{X_i}^{!}(\mathbb{1}), \text{ where }\mathbb{1} \in H_{T}^{0}(X_i)
\end{equation*}
 and $X_1,\ldots,X_m$ are the $T$-invariant divisors.
 The relations correspond to the subsets of the $X_i$s that have an empty intersection. 
 For $\Upsilon_1,\ldots,\Upsilon_m$, the cup product $\Pi_{i \in I} \Upsilon_i$ ($I \subset \{1,\ldots,m\}$) is the Poincar\'e dual of the intersection $\displaystyle{\cap_{i \in I}{X_i}}$, hence  $\Pi_{i \in I} \Upsilon_i=0$ if and only if  $\cap_{i \in I}{X_i}=\emptyset$.
 By \cite[Proposition 2.2]{masuda}, to each $i \in \{1,\ldots,m\}$, there is a unique element $v_i \in H_{2}(BT)$ such that $\pi_{T}^{*}(u)=\sum_{i=1}^{m}\langle u,v_i\rangle \Upsilon_i$ for every $u \in H^{2}(BT)$.

Consider a circle action $S^1 \acts (M^4,\omega)$ that is obtained from a
toric action
 $T \acts (M,\omega)$
by the inclusion $$\inc \colon S^1 \hookrightarrow T={(S^1)}^{2}; \,\, s \mapsto (s^{m},s^{n}).$$ 

For a $T$-invariant symplectic sphere $X$, and the inclusions induced from $\iota_X \colon X \hookrightarrow M$, 
we have a commutative diagram:

\begin{equation} \labell{diagram1}
\begin{CD}
 ET {\times}_T X @> {\iota_X^T} >> ET {\times}_T M \\
@A \inc_X AA @A \inc AA \\
 ET {\times}_{S^1} X @> {\iota_X^{S^1}}>> ET {\times}_{S^1} M
\end{CD}
\end{equation}
 Here and later the vertical maps are defined using $ET {\times}_{S^1} M$ rather than $E{S^1} {\times}_{S^1} M$ and $ET/{S^1}$ rather then $E{S^1}/{S^1}$; these spaces are homotopy equivalent.
This commutative diagram is Cartesian in the sense that $ET {\times}_{S^1} X$ is the inverse image of $ET \times_{T} X$ under $\inc$. 
Hence the push-pull  formula
 \begin{equation*} 
 {\inc}^{*}\circ {\iota_X^T}^{!} = {\iota_X^{S^1}}^{!} \circ {\inc_X}^{*}
 \end{equation*}
 holds. Here ${\iota_X^T}^!$ and ${\iota_X^{S^1}}^{!}$ are the equivariant pushforward maps $H_{T}^{0}(X) \to H_{T}^{2}(M)$ and $H_{S^1}^{0}(X) \to H_{S^1}^{2}(M)$ induced by the inclusion of $X$ into $M$ and $\inc^*\colon H_{T}^{*}(M) \to H_{S^1}^{*}(M)$ and $\inc_X^*\colon H_{T}^{*}(X) \to H_{S^1}^{*}(X)$  are the pullback maps in equivariant cohomology induced by the inclusion of $S^1$ into $T$. 
Denote $$\eta:= \iota_X^{!}(\mathbb{1}^{S^1}_X).$$
We obtain the commutative diagram
 \begin{equation*} 
\begin{CD}
 \mathbb{1}^{T}_X \in H_{T}^{0}(X) @> {\iota_X^T}^! >> H_{T}^{2}(M) \ni \Upsilon \\
@V \inc_X^* VV @V \inc^* VV\\
  \mathbb{1}^{S^1}_X \in H_{S^1}^{0}(X) @> {\iota_X^{S^1}}^{!}>> H_{S^1}^{2}(M) \ni \eta
\end{CD}
\end{equation*}
where the vertical arrows are surjective.
Since  $\inc_X^{*}(\mathbb{1}^{T}_X)=1^{S^1}_X$, we have 
 \begin{equation*} 
{\inc}^{*}(\Upsilon)= {\inc}^{*}\circ {\iota_X^T}^{!}(\mathbb{1}^{T}_X) = {\iota_X^{S^1}}^{!}(\mathbb{1}^{S^1}_X)=\eta.
 \end{equation*}
  The commutative diagram \eqref{diagram1} also implies that the following diagram commutes.
 \begin{equation*} 
 \begin{CD}
ET {\times}_T M @>\pi_T >> BT\\
 @ A \inc AA @ A \inc_B AA \\
E{S^1} {\times}_{S^1} M @>\pi_{S^1} >> B{S^1}.
 \end{CD}
 \end{equation*}
  Consequently, the following diagram commutes.
 \begin{equation} \labell{diagram4}
 \begin{CD}
 H_{T}^{*}(M) @<\pi_{T}^* << H_{T}^{*}(\pt) @= R[x,y]\\
 @ V \inc^{*} VV @ V \inc_B^{*} VV \\
H_{S^1}^{*}(M) @<\pi_{S^1}^* << H_{S^1}^{*}(\pt)@=R[t].
 \end{CD}
 \end{equation}
  The induced map $\inc_B^{*} \colon H_{T}^{*}(\pt)=R[x,y] \to H_{S^1}^{*}(\pt)=R[t]$ is the map sending
$$x \mapsto m \cdot t,$$
$$y \mapsto n \cdot t.$$
By \eqref{diagram4},
\begin{equation*} 
H_{S^1}^{*}(M)=\frac{{\inc}^{*}H_T^{*}(M)}{{\inc}^{*}({\pi}_{T}^{*}(\ker \inc_B^{*}))},
\end{equation*}
and ${\pi}_{S^1}^{*}(t)$ equals
\begin{equation}\label{projection2} 
\begin{array}{ccc}
{\pi}_{S^1}^{*}(t) & = &  {\pi}_{S^1}^{*}\big((am-bn)t\big) = {\pi}_{S^1}^{*}(am \cdot t-bn \cdot t)\\
& =& a {\pi}_{S^1}^{*}(\inc_B^{*}(x))-b{\pi}_{S^1}^{*}(\inc_B^{*}(y))\\
& =& a \inc^{*}(\pi_T^{*}(x))-b \inc^{*}(\pi_T^{*}(y))
\end{array}
\end{equation}
up to the equivalence in $H_{S^1}^{*}(M)$, where $a,b$ are as in \eqref{eq:gcd}.
 \end{noTitle}

\begin{noTitle} {\bf Minimal models: $\CP^2$ and ruled rational surfaces.}
Let $(M,\omega)=(\CP^2,\lambda\omega_{\FS})$.
Consider the toric action on $(M,\omega)$, defined by
$$(t_1,t_2) \cdot [z_0;z_1;z_2]=[z_0;t_1 z_1; t_2 z_2].$$ 
By \cite{masuda}, 
the equivariant cohomology
$$H_{T}^{\star}(M)=\frac{\Z[\Upsilon_1,\Upsilon_2,\Upsilon_3]}{\langle \Upsilon_1 \cup \Upsilon_2 \cup \Upsilon_3 \rangle}$$ 
and
 $H_{T}^{*}(pt)=\Z[x,y]$, with
$${\pi}_{T}^{*}(x)=\Upsilon_3-\Upsilon_2,$$
$${\pi}_{T}^{*}(y)=\Upsilon_3-\Upsilon_1.$$

\begin{center}
  \begin{figure}[ht]
\begin{overpic}[
scale=1.0,unit=1mm]{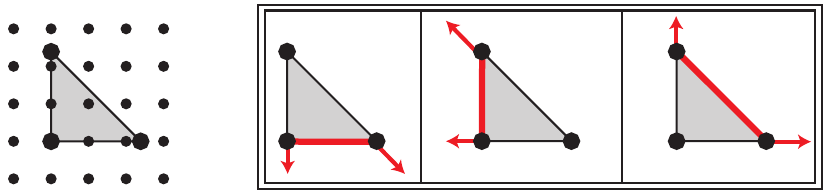}
   \put(55,5){\small{$\Upsilon_1$}}
   \put(76,14){\small{$\Upsilon_2$}}
   \put(122,18){\small{$\Upsilon_3$}}
\end{overpic}
\caption{$\CP^2$ moment image and generating classes (restricted to the fixed points, with non-zero
restrictions indicated by a red arrow in $\algt^*$).}
  \label{fig:classes for cp2}
\end{figure}
\end{center}

Now consider an effective $S^1$-action on $(M,\omega)$ obtained from an inclusion $\inc \colon S^1 \to T$.
It is defined by 
$$s \cdot [z_0; z_1; z_2]=[z_0;s^m z_1;s^n z_2]$$
for  $(m,n)\in {\Z}^2$ 
as in Example \ref{e:cp2}
with fixed $a,b \in \Z$ such that $am-bn=1$ as in \eqref{eq:gcd}.
 Following our convention that if there is one fat vertex, it must be a maximum value for the moment map, the 
relevant circle actions with one fat vertex correspond to $(m,n) \in  \{(-1,0),(0,-1),(1,1)\}$.

 We have 
 \begin{equation}\label{eq:linear}
 {{\inc}^{*}({\pi}_{T}^{*}(\ker {\inc_B}^{*}))}={\inc}^{*}\left({\pi}_{T}^{*}\big\langle nx-my\big\rangle\right)
 =\Big\langle n(\eta_3-\eta_2)-m(\eta_3-\eta_1)  \Big\rangle,
 \end{equation}
 and, by  \eqref{projection2}, 
\begin{equation} \label{eq:algebra}
{\pi}_{S^1}^{*}( t)=a (\eta_3-\eta_2)-b(\eta_3-\eta_1),
\end{equation}
where $\eta_i:=\inc^{*}(\Upsilon_i)$.
\end{noTitle}

\begin{Proposition} \labell{Procp2}
For an effective $S^1$-action on $(M,\omega)=(\CP^2,\lambda\omega_{\FS})$ that is obtained from a Delzant triangle of edge-length $\lambda$ by the projection $(x_1,x_2) \mapsto mx_1+nx_2$, 
we have
\begin{itemize}
\item
If $(m,n) \in  \{(-1,0),(0,-1),(1,1)\}$,
then $\fat=1$, $\tau_0=0$, and 
$$H_{S^1}^{*}(M)=\frac{\Z[\tau_{\infty},\sigma_{1,1},\sigma_{2,1}]}{\langle \tau_{\infty} \cup \sigma_{1,1} \cup \sigma_{2,1}, \sigma_{2,1}-\sigma_{1,1} \rangle}=\frac{\Z[\tau_{\infty},\sigma_{1,1}]}{\langle {\sigma_{1,1}}^2 \cup \tau_{\infty} \rangle},$$
and 
$${\pi}_{S^1}^{*}(t)=\tau_{\infty}-\sigma_{1,1}.$$

\item Otherwise for relatively prime $(m,n) \in \Z \times \Z \smallsetminus \{(\pm 1,0),(0,\pm 1),\pm(1,1)\}$, then $\fat=0$ and
$\tau_0=\tau_\infty=0$.  When $m>n>0$ as in Figure~\ref{fig:circle on cp2}(c), we have
$$H_{S^1}^{*}(M)=\frac{\Z[\sigma_{1,1},\sigma_{2,1},\sigma_{2,2}]}{\langle \sigma_{1,1} \cup \sigma_{1,2} \cup \sigma_{2,1}, n(\sigma_{2,2}-\sigma_{2,1})-m(\sigma_{2,2}-\sigma_{1,1})\rangle},$$
and 
$${\pi}_{S^1}^{*}(t)=a(\sigma_{2,2}-\sigma_{2,1})-b(\sigma_{2,2}-\sigma_{1,1}).$$
For other values of $m$ and $n$, this presentation is adjusted accordingly.  
\end{itemize}
\end{Proposition}

\begin{proof}
This follows immediately by restricting from $T$ to $S^1$.  
In the first bullet, the classes $\eta_j$, corresponding to the $T$-invariant spheres, are, in the notations of
 Theorem~\ref{ThNew}, as follows.
$$
\begin{array}{| c | c | c |}\hline
\mbox{For } (m,n)=(-1,0) & \mbox{For } (m,n)=(0,-1)  & \mbox{For } (m,n)=(1,1) \\ \hline
\eta_1  \mapsto  \sigma_{1,1} & \eta_1  \mapsto   \tau_\infty & \eta_1  \mapsto  \sigma_{1,1} \\
\eta_2 \mapsto \tau_\infty & \eta_2  \mapsto \sigma_{1,1}  & \eta_2  \mapsto \sigma_{2,1} \\
\eta_3  \mapsto  \sigma_{2,1}  & \eta_3  \mapsto  \sigma_{2,1}  & \eta_3  \mapsto   \tau_\infty\\ \hline
\end{array}
$$
To match Theorem~\ref{ThNew}, we must add a generator $\tau_h$ satisfying
$$
\tau_h = \sigma_{1,1} = \sigma_{2,1}.
$$
That $ \sigma_{1,1} = \sigma_{2,1}$ is a consequence of \eqref{eq:linear}.  This
is equivalent to having relations $\tau_h-\sigma_{1,1}$ and $\tau_h-\sigma_{2,1}$, the 
difference of which is exactly \eqref{eq:linear} for these $m$ and $n$:  $\sigma_{1,1} - \sigma_{2,1}$.
The formula for $\pi_{S^1}^*(t)$ follows immediately from \eqref{eq:algebra}.

In the second bullet above, in the case $m>n>0$, we have
\begin{eqnarray*}
\eta_1 & \mapsto & \sigma_{1,1} \\
\eta_2 & \mapsto &\sigma_{2,1} \\
\eta_3 & \mapsto &  \sigma_{2,2}
\end{eqnarray*}
Note that to match Theorem~\ref{ThNew}, we must add a generator $\tau_h$ which satisfies
$$
\tau_h = m\sigma_{1,1} = n\sigma_{2,1}+(m-n)\sigma_{2,2}.
$$
As above, the two terms on the right-hand side are equal because of \eqref{eq:linear}. and
this is equivalent to two linear relations whose difference is
$$
m\sigma_{1,1} - \left( n\sigma_{2,1}+(m-n)\sigma_{2,2}\right) = n(\sigma_{2,2}-\sigma_{2,1})-m(\sigma_{2,2}-\sigma_{1,1}),
$$
as desired.  The formula for $\pi_{S^1}^*(t)$ follows immediately from \eqref{eq:algebra}.

For other values of $m$ and $n$, the map on the $\eta_j$s is changed appropriately.
This completes the proof of the Proposition, and indeed the Proof of Theorem~\ref{ThNew} when $M=\CP^2$.
\end{proof}

We now turn to the Hamiltonian action of $T=(S^1)^2$ on the 
Hirzebruch surface $(\Hirz_N, \omega_{\beta,f})$ induced by
\begin{equation*} 
(s,t)\cdot ([w_1;w_2],[z_0;z_1;z_2]) = ([w_1;sw_2],[tz_0;z_1;s^N z_2]).
\end{equation*}
By \cite{masuda}, $$H_{T}^{\star}(\Hirz_N)=\frac{\Z[\Upsilon_1,\Upsilon_2,\Upsilon_3,\Upsilon_4]}{\langle \Upsilon_1 \cup \Upsilon_4,\Upsilon_2 \cup \Upsilon_3 \rangle}.$$
See Figure \ref{fig:classes for Hirz} for the $\Upsilon_i$s.
For the generators of $H_{T}^{*}(pt)=\R[x,y]$,
$${\pi}_{T}^{*}(x)=\Upsilon_3-\Upsilon_2,$$
$${\pi}_{T}^{*}(y)=\Upsilon_4-\Upsilon_1+N \Upsilon_3.$$

\begin{center}
  \begin{figure}[ht]
\begin{overpic}[
scale=0.78,unit=1mm]{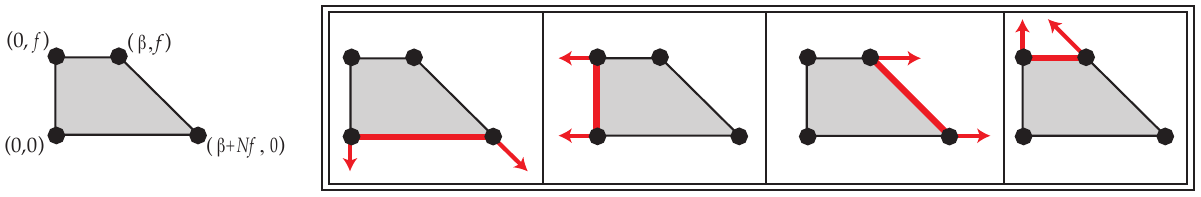}
   \put(54,3){\small{$\Upsilon_1$}}
   \put(73,12){\small{$\Upsilon_2$}}
   \put(123,12){\small{$\Upsilon_3$}}
   \put(135.5,20){\small{$\Upsilon_4$}}
\end{overpic}
\caption{The moment image for the Hirzebruch surface and generating classes (restricted to the fixed points, with non-zero
restrictions indicated by a red arrow in $\algt^*$).}\label{fig:classes for Hirz}
\end{figure}
\end{center}

Now we consider an effective $S^1$-action on $(\Hirz_N, \omega_{\beta,f})$ obtained from an inclusion $\inc \colon S^1 \to T$,
so that
\[
\xi \cdot ([w_1;w_2], [z_0;z_1,;z_2]) = ([w_1; \xi^m w_2], [\xi^n z_0; z_1;\xi^{Nm} z_2]).
\]
See Example \ref{e:hirz}.
We shall refer to a Hirzebruch surface with this $S^1$-action as $$\Hirz_N(m,n).$$
When $(m,n)=\pm(0,1)$, there are two fat vertices and the labeled graph is as in Figure~\ref{fig:circle on hirzebruch}(a).
Following our convention that if there is one fat vertex, it must be a maximum value for the moment map, the 
relevant circle actions with one fat vertex correspond to $(m,n) \in  \{(-1,0),(1,N)\}$
and the labeled graph is as in Figure~\ref{fig:circle on hirzebruch}(b).

We have 
\begin{equation}\label{eq:hirz linear}
{{\inc}^{*}({\pi}_{T}^{*}(\ker {\inc_B}^{*}))}=\langle  n(\eta_3-\eta_2)-m(\eta_4-\eta_1+N \eta_3) \rangle,
\end{equation}
and, by  \eqref{projection2}, 
\begin{equation}\label{eq: algebra t}
{\pi}_{S^1}^{*}(t)=a(\eta_3-\eta_2)-b(\eta_4-\eta_1+N \eta_3),
\end{equation}
where $\eta_i:=\inc^{*}(\Upsilon_i)$ and
 $a,b \in \Z$ are such that $am-bn=1$, using the conventions in Notation~\ref{gcd}.

\begin{Proposition} \labell{Prohirz}
For an effective $S^1$-action on  $\Hirz_{N}(m,n)$  on $(M,\omega)=(\Hirz_N, \omega_{\beta,f})$ 
with $m,n$ relatively prime, we have the following possibilities.
\begin{itemize}
\item For $(m,n)=\pm (0, 1)$, we have $\fat=2$, $e_{min}=\pm N$, and
$$H_{S^1}^{\star}(M)=\frac{\Z [\tau_0,\tau_{\infty},\sigma_{1,1},\sigma_{2,1}]}{\langle \sigma_{1,1} \cup \sigma_{2,1}\ ,\ \tau_{0} \cup \tau_{\infty}\ ,\  \sigma_{2,1}-\sigma_{1,1} \rangle}=
\frac{\Z [\tau_0,\tau_{\infty},\tau_h]}{\langle {\tau_h}^2\ ,\ \tau_0 \cup \tau_{\infty} \rangle},$$
and 
$${\pi}_{S^1}^{*}(t)=\tau_{\infty}-\tau_{0}+e_{min} \tau_h.$$
\item For $(m,n)\in \left\{ (-1,0),(1,N)\right\}$, we have 
$\fat=1$, $\tau_0=0$, the labeled graph is as in Figure~\ref{fig:circle on hirzebruch}(b), and
$$H_{S^1}^{\star}(M)=\frac{\Z [\tau_{\infty}, \sigma_{1,1},\sigma_{1,2}, \sigma_{2,1}]}{\langle \sigma_{1,1} \cup \tau_{\infty}\ ,\ \sigma_{2,1} \cup \sigma_{1,2}\ ,\  \sigma_{2,1}-\sigma_{1,2}-N \sigma_{1,1} \rangle}, $$
and
$${\pi}_{S^1}^{*}(t)=\tau_{\infty}-\sigma_{1,1}.$$

 \item 
For relatively prime $(m,n)$ in  $\Z \times \Z \smallsetminus \{\pm(1,0),\pm(0,1),\pm(1,N)\}$,  
we have $\fat=0$ and $\tau_0=\tau_\infty=0$. As in Figure~\ref{fig:circle on hirzebruch}(c) and (d), the possible configurations of chains
are two chains of length two; or one chain of length three and one of length one.
\begin{itemize}
\item When $n>mN>0$ as in Figure~\ref{fig:circle on hirzebruch}(c), we have
$$
H_{S^1}^{\star}(M)=\frac{\Z [\sigma_{1,1}\ ,\ \sigma_{1,2}\ ,\ \sigma_{2,1}\ ,\ \sigma_{2,2}]}
{\langle \sigma_{1,1} \cup \sigma_{2,2}\ ,\ \sigma_{1,2} \cup \sigma_{2,1}\ ,\ 
n(\sigma_{2,2}-\sigma_{1,1})-m(\sigma_{1,2}-\sigma_{2,1}+N \sigma_{2,2})\rangle},
$$
and 
$$ 
{\pi}_{S^1}^{*}(t)=a(\sigma_{2,2}-\sigma_{1,1})-b(\sigma_{1,2}-\sigma_{2,1}+N \sigma_{2,2}).
$$
\item When $mN>n>0$ and $n\geq m$ as in Figure~\ref{fig:circle on hirzebruch}(d), we have
$$
H_{S^1}^{\star}(M)=\frac{\Z [\sigma_{1,1}\ ,\ \sigma_{1,2}\ ,\ \sigma_{1,3}\ ,\ \sigma_{2,1}]}
{\langle \sigma_{1,1} \cup \sigma_{1,3}\ ,\ \sigma_{1,2} \cup \sigma_{2,1}\ ,\ 
n(\sigma_{1,3}-\sigma_{1,1})-m(\sigma_{1,2}-\sigma_{2,1}+N \sigma_{1,3})\rangle},
$$
and $$ {\pi}_{S^1}^{*}(t)=a(\sigma_{1,3}-\sigma_{1,1})-b(\sigma_{1,2}-\sigma_{2,1}+N \sigma_{1,3}).$$
\end{itemize}
For other values of $m$ and $n$, this presentation is adjusted accordingly.  
\end{itemize}
\end{Proposition}

\begin{proof}
This follows immediately by restricting from $T$ to $S^1$.  In the first bullet, the classes $\eta_j$, corresponding to the $T$-invariant spheres, are, in the notations of Theorem~\ref{ThNew}, as follows.
$$
\begin{array}{| c | c | }\hline
\mbox{For } (m,n)=(0,1) & \mbox{For } (m,n)=(0,-1)   \\ \hline
\eta_1  \mapsto  \tau_0 & \eta_1  \mapsto   \tau_\infty  \\
\eta_2 \mapsto\sigma_{1,1} & \eta_2  \mapsto \sigma_{1,1}   \\
\eta_3  \mapsto  \sigma_{2,1}  & \eta_3  \mapsto  \sigma_{2,1} \\ 
\eta_4  \mapsto  \tau_\infty  & \eta_4  \mapsto \tau_0 \\ \hline
\end{array}
$$
Again, to match Theorem~\ref{ThNew}, we must add a generator $\tau_h$ satisfying
$$
\tau_h = \sigma_{1,1} = \sigma_{2,1}.
$$
That $ \sigma_{1,1} = \sigma_{2,1}$ is a consequence of \eqref{eq:hirz linear}.  This
is equivalent to having relations $\tau_h-\sigma_{1,1}$ and $\tau_h-\sigma_{2,1}$, the 
difference of which is exactly \eqref{eq:hirz linear} for these $m$ and $n$:  $\sigma_{1,1} - \sigma_{2,1}$.
The formula for $\pi_{S^1}^*(t)$ follows immediately from \eqref{eq: algebra t}.

In the second bullet above, we have
$$
\begin{array}{| c | c | }\hline
\mbox{For } (m,n)=(-1,0) & \mbox{For } (m,n)=(1,N)   \\ \hline
\eta_1  \mapsto  \sigma_{2,1} & \eta_1  \mapsto  \sigma_{2,1} \\
\eta_2 \mapsto  \tau_\infty & \eta_2  \mapsto \sigma_{1,1}   \\
\eta_3  \mapsto  \sigma_{1,1}  & \eta_3  \mapsto  \tau_\infty  \\ 
\eta_4  \mapsto   \sigma_{1,2} & \eta_4  \mapsto  \sigma_{1,2}  \\ \hline
\end{array}
$$
To match Theorem~\ref{ThNew}, we must add a generator $\tau_h$ satirsfying
$$
\tau_h = N\cdot \sigma_{1,1}+\sigma_{1,2} = \sigma_{2,1},
$$
That the two terms on the right-hand side are equal is a consequence of \eqref{eq:hirz linear}. This is equivalent
to adding two linear relations whose difference is 
$N\cdot \sigma_{1,1}+\sigma_{1,2} = \sigma_{2,1}$.
The formula for $\pi_{S^1}^*(t)$ follows immediately from \eqref{eq: algebra t}.

In the third bullet above, again we have two cases and 
$$
\begin{array}{| c | c | }\hline
\mbox{For } n>mN>0 & \begin{array}{c} \mbox{For } mN>n>0 \\ \mbox{and } n>m\end{array}  \\ \hline
\eta_1  \mapsto  \sigma_{2,1} & \eta_1  \mapsto  \sigma_{2,1} \\
\eta_2 \mapsto \sigma_{1,1} & \eta_2  \mapsto \sigma_{1,1}   \\
\eta_3  \mapsto  \sigma_{2,2}  & \eta_3  \mapsto \sigma_{1,3}  \\ 
\eta_4  \mapsto   \sigma_{1,2} & \eta_4  \mapsto  \sigma_{1,2}  \\ \hline
\end{array}
$$
To match Theorem~\ref{ThNew}, we must add a generator $\tau_h$. First when $n_mN>0$, we let
$$
\tau_h = n\cdot \sigma_{1,1}+m\cdot \sigma_{1,2} = m\cdot \sigma_{2,1} + (n-mN)\cdot \sigma_{2,2},
$$
where the two terms on the right are equal by \eqref{eq:hirz linear}.
When $mN-n>0$ we let
$$
\tau_h = n\cdot \sigma_{1,1}+m\cdot \sigma_{1,2} + (mN-n)\cdot \sigma_{1,3} = m\cdot \sigma_{2,1}.
$$
where again the two terms on the right are equal by \eqref{eq:hirz linear}.
These equalities each give rise to two linear relations whose respective differences are
$$
n(\sigma_{2,2}-\sigma_{1,1})=m(\sigma_{1,2}-\sigma_{2,1}+N\sigma_{2,2}) \mbox{ or }
n(\sigma_{1,3}-\sigma_{1,1})=m(\sigma_{1,2}-\sigma_{2,1}+N\sigma_{1,3}),
$$
as desired.
Finally, the formula given for $\pi_{S^1}^*(t)$ follows from \eqref{eq: algebra t}.

For other values of $m$ and $n$, the map on the $\eta_j$s is changed appropriately.
This completes the proof of the Proposition, and indeed the Proof of Theorem~\ref{ThNew} when 
$(M,\omega)=(\Hirz_N, \omega_{\beta,f})$.
\end{proof}

\begin{noTitle} {\bf Minimal models: symplectic $S^1$-ruled surfaces}
Recall that a {symplectic $S^1$-ruled surface} is an $S^2$-bundle over a compact surface 
$\Sigma$ with a circle action that fixes the basis and rotates each fiber. This admits an 
invariant symplectic form, an invariant K\"ahler structure, 
and a moment map. 

\begin{Proposition} \labell{Proruled}
For an  $S^1$-ruled surface $S^1 \acts (M,\omega)$,
we have
$$H_{S^1}^{2 \star}(M)=\frac{\Z [\tau_0,\tau_{\infty},\tau_h]}{\langle \tau_0 \cup \tau_{\infty}, {\tau_h}^2\rangle}.$$
The algebra structure is given by $${\pi}^{*}(t)=\tau_{\infty}-\tau_0+e_{\min} \tau_h$$ as in \eqref{eq:algstr}.
\end{Proposition}
Note that in this case the decorated graph has two fat vertices and no isolated vertices, and in the extended decorated graph, there are also two edges labeled $1$ between the fat vertices.
We have $\sigma_{1,1}=\tau_h=\sigma_{2,1}$ and set
$b_{1,1}=b_{2,1}=0$.

\begin{proof}
The classes  $\tau_0,\tau_{\infty},\tau_h \in H_{S^1}^2(M)$ span a subring of the equivariant cohomology. We will show that this subring equals $H_{S^1}^{2*}(M)$.
The ruled $S^1$-action fixes the base $\Sigma$ and rotates each $S^2$-fiber, fixing the south pole $S$ and the north pole $N$.
Therefore the given fibration $M \to \Sigma$ with fiber $S^2$ yields a fibration $$\prr \colon (M \times ES^1) / {S^1} \to \Sigma$$  with fiber $(S^2 \times ES^1)/{S^1}$.  Recall that $$H^{*}((S^2 \times ES^1)/{S^1})=H_{S^1}^{*}(S^2)=\frac{\Z[\xi_N,\xi_S]}{\langle \xi_N \cup \xi_S \rangle}, \, \xi_N|_{S}=0, \, \xi_{N}|_{N}=-t, \, \xi_S|_{S}=t, \, \xi_{S}|_{N}=0,$$ and $H^{2*}(\Sigma)=\Z[[\Sigma]], \text{ with }[\Sigma]^2=0$. The inclusion of the fiber $\iota \colon S^2 \to M$ gives an inclusion $\iota \colon (S^2 \times ES^1) /{S^1} \to (M \times ES^1) / {S^1}$. We have $\iota^{*}\tau_0=\xi_S,\, \iota^{*}\tau_{\infty}=\xi_N$ and $\prr^{*}([\Sigma])=\tau_h$. Therefore, by the Leray-Hirsch Theorem, $$H_{S^1}^{2*}(M)\sim H^{2*}(\Sigma) \otimes H_{S^1}^{*}(S^2),$$ and $\tau_0,\tau_\infty,\tau_h$ generate 
$H_{S^1}^{2*}(M)$, with the specified relations.
\end{proof}

\end{noTitle}

\begin{noTitle} {\bf The effect of an equivariant blowup.}
For $S^1 \acts (M,\omega)$, let $J$ be an integrable $\omega$-compatible complex structure on $M$
 with respect to which the $S^1$-action is holomorphic. Let $p$ be an $S^1$-fixed point in $M$.
Recall the equivariant complex blowup $(\widetilde{M}, \widetilde{J})$ of $(M,J)$ at $p$ and the blowup map: the equivariant  projection 
$$
\bl \colon \widetilde{M} \to M
$$
extending the identity on $M  \smallsetminus \{p\}$, defined in \S \ref{nbup}.
Denote the pushforward of the blowup map 
$${{\bl}}^{!} \colon H_{S^1}^{i}\left(\widetilde{M}\right) \to H_{S^1}^{i}(M).$$
Denote the pushforward of the inclusion of the exceptional divisor $\bl^{-1}(p)$
$$\iota_{\bl^{-1}(p)}^{!} \colon H_{S^1}^{i}(\CP^{n-1}) \to H_{S^1}^{i+2}\left(\widetilde{M}\right).$$

We have the diagram
\begin{equation} \labell{diagram1}
\begin{CD}
 \bl^{-1}(p) @> {\iota_{\bl^{-1}(p)}} >> \widetilde{M} \\
@V q VV @V \bl VV \\
 p @> {\iota_{p}}>> M
\end{CD}
\end{equation}
Anderson and Fulton use this diagram in \cite{AF24} to determine the relationship between the equivariant 
cohomology of $M$ and of its equivariant blowup $\widetilde M$ as follows. 
 The proof 
from ordinary cohomology generalizes to equivariant cohomology because all subsets and maps are equivariant with 
respect to an ambient torus action. 
For the statement in ordinary cohomology for 
the blowup along a complex subvariety, see \cite[Theorem~7.31]{voisin}.   We need only the case of the
$S^1$-equivariant blowup of a point in a four-manifold, and so the statements from the literature 
simplify greatly.
\end{noTitle}

\begin{Lemma}\cite[Proposition 6.2.1]{AF24}. \label{lem:coh-of-blowup}
For $S^1 \acts (M^4,\omega,J)$, there is an isomorphism
\begin{equation*}
H_{S^1}^{k}(M) \oplus H_{S^1}^{k-2} {(p)} \to H_{S^1}^{k}(\widetilde{M})
\end{equation*}
defined by
$$(a,b) \mapsto \bl^{*}(a)+{\iota^{!}_{\bl^{-1}(p)}}(q^{*}(b)).$$
The product structure on $H_{S^1}^{*}(\widetilde{M})$ is determined by the following formulas:
\begin{enumerate}
    \item $\bl^{*}(a) \cup \bl^{*}(a')=\bl^{*}(a \cup a')$;
    \item ${\iota^{!}_{\bl^{-1}(p)}}(q^{*}(b)) \cup {\iota^{!}_{\bl^{-1}(p)}}(q^{*}(b'))=-{\iota^{!}_{\bl^{-1}(p)}}(q^{*}(b \cup b') \cup \zeta)$;
    \item ${\bl}^{*}(a) \cup {\iota^{!}_{\bl^{-1}(p)}}(q^{*}(b))={\iota^{!}_{\bl^{-1}(p)}}(q^{*}(\iota_{p}^{*}(a) \cup b)).$
\end{enumerate}
 Here $\zeta=c_{1}^{S^1}(\mathcal{O}(1))\in H_{S^1}^2(\CP^1;\Z)$.
\end{Lemma}
 { Assume that the invariant integrable $\omega$-compatible complex structure $J$ is such that $\omega(\cdot,J\cdot)$ is a generic Riemannian metric.}
The effect of an equivariant K\"ahler blowup on $S^1 \acts (M^4,\omega,J)$ on the decorated graph is described in \S \ref{nbup} and Figure \ref{fig:blowup effect}. 
We highlight the following facts.

\begin{Facts}\labell{facts}
Let $s$ ($s'$) be a robust edge (fat vertex) in the extended decorated graph of $S^1 \acts (M,\omega)$ with respect to a generic metric. By Remark \ref{rem:spheres}, there
is an
 invariant embedded complex  (hence symplectic) surface $S$ in $S^1 \acts (M,\omega,J)$ whose moment map image is $s$ ($s'$). 
 Then there
is an 
invariant embedded complex   surface $\widetilde{S}$  in $S^1 \acts (\widetilde{M},\widetilde{\omega},\widetilde{J})$ 
whose image is  the  edge (fat vertex) obtained from $s$ ($s'$) in the blowup decorated graph,
and  
$$\bl \circ \iota_{\widetilde{S}}=\iota_S.$$
The exceptional divisor $\bl^{-1}(p)$ is an  $S^1$-invariant embedded complex  sphere 
whose moment map image is  the new edge or fat vertex in the  extended decorated graph for $S^1 \acts (\widetilde{M},\widetilde{\omega})$  with respect to a generic metric.
\end{Facts}

\begin{Notation} \labell{notbl}
Denote 
\begin{equation}\labell{eq:e}
e=\iota_{\bl^{-1}(p)}^{!}(\mathbb{1}_{\CP^{1}})=\iota_{\CP^1}^{!}(\mathbb{1}_{\CP^{1}}) \in H_{S^1}^{2}(\widetilde{M}).
\end{equation}
For an invariant embedded complex  surface $S$ (of genus $\geq 0$) in $S^1 \acts (M,J)$, we set $\eta:=\iota_{S}^{!}(\mathbb{1}_{S})$. For
an invariant embedded complex  surface $\widetilde{S}$ in $S^1 \acts (\widetilde{M},\widetilde{J})$ such that 
\begin{equation}\label{eq:bliota}
\bl \circ \iota_{\widetilde{S}}=\iota_S,
\end{equation}
set
$$\widetilde{\eta}:=\iota_{\widetilde{S}}^{!}(\mathbb{1}_{\widetilde{S}}).$$ 
Note that by \eqref{eq:bliota} and  the functoriality of the pushforward:
\begin{equation*}
{\bl}^{!}(\widetilde{\eta})={\bl}^{!} \left(\iota_{\widetilde{S}}^{!}(\mathbb{1}_{\widetilde{S}})\right)=({\bl \circ \iota_{\widetilde{S}}})^{!}(\mathbb{1}_{{S}})=\iota_S^{!}(\mathbb{1}_{S})=:\eta.
\end{equation*}
For $\xi=\sum a_r \eta_r$ with $a_r \in \Z$ and $\eta_r=\iota_{S_r}^{!}(\mathbb{1}_{S_r})$ denote $\widetilde{\xi}:=\sum a_r \widetilde{\eta_r}.$
\end{Notation}

We use Facts~\ref{facts} and the effect of the blowup on the graph, shown in Figure~\ref{fig:blowup effect},
to deduce that
in $H_{S^1}^*(\widetilde{M};\Z)$ we have Table \ref{Table-blup}.

\newpage
 
\renewcommand{\arraystretch}{1.5}
\begin{table}[h]
$$
\begin{array}{c|c|c|c|c|c}
\text{Blowup type} & \text{at} & \widetilde{\tau_0} & \widetilde{\tau_{\infty}}& \widetilde{\sigma_{i,j}}& e \\ \hline\hline
I & v_{i^*,j^*} &{\tau}^{\widetilde{M}}_0 &{\tau}^{\widetilde{M}}_\infty & \mathrm{\begin{tabular}{c}${\sigma}^{\widetilde{M}}_{i,j}$  if $i \neq i^*$\\ ${\sigma}^{\widetilde{M}}_{i^*,j}$ if   $i=i^*$ and $j\leq j^*$\\ ${\sigma}^{\widetilde{M}}_{i^*,j+1}$ if   $i=i^*$ and  $j>j^*$ \end{tabular}}&{\sigma}^{\widetilde{M}}_{i^*,j^*+1} \\ \hline
I\! I & \mathrm{\begin{tabular}{c} max \end{tabular}} &{\tau}^{\widetilde{M}}_0  & {\tau}^{\widetilde{M}}_\infty & {\sigma}^{\widetilde{M}}_{i,j}&{\sigma}^{\widetilde{M}}_{k+1,2}\\ \hline
I\! I  & \mathrm{\begin{tabular}{c}min \end{tabular}} &{\tau}^{\widetilde{M}}_0  & {\tau}^{\widetilde{M}}_\infty & {\sigma}^{\widetilde{M}}_{i,j}&{\sigma}^{\widetilde{M}}_{k+1,1} \\ \hline
I\! I\! I & \mathrm{\begin{tabular}{c} max \end{tabular}} & {\tau}^{\widetilde{M}}_0 &{\tau}^{\widetilde{M}}_\infty=0 & {\sigma}^{\widetilde{M}}_{i,j}& {\sigma}^{\widetilde{M}}_{1,\ell_1+1}\\ \hline
   I\! I\! I &\mathrm{\begin{tabular}{c} min \end{tabular}}  & {\tau}^{\widetilde{M}}_0=0 &{\tau}^{\widetilde{M}}_\infty & \mathrm{\begin{tabular}{c}${\sigma}^{\widetilde{M}}_{i,j}$ if $i\neq 1$ \\ ${\sigma}^{\widetilde{M}}_{1,j+1}$ if $i=1$ \end{tabular}}& {\sigma}^{\widetilde{M}}_{1,1}\\ \hline
     I\! V & \mathrm{\begin{tabular}{c}max\end{tabular}} &{\tau}^{\widetilde{M}}_0  &0 &{\sigma}^{\widetilde{M}}_{i,j}&\tau^{\widetilde{M}}_{\infty} \\ \hline
I\! V &  \mathrm{\begin{tabular}{c}min\end{tabular}}   & 0 &{\tau}^{\widetilde{M}}_\infty &{\sigma}^{\widetilde{M}}_{i,j}&\tau^{\widetilde{M}}_{0} \\ 
 \hline
\end{array}
$$
\bigskip

\caption{The impact of a blowup on  $\tau_0,\tau_{\infty},\sigma_{i,j}$ and the class $e$, by type of blowup. The columns marked  $\widetilde{\alpha}$ indicate the image
of a class $\alpha\in H^*_{S^1}(M)$ in $H_{S^1}^*(\widetilde{M})$.  The column $e$ identifies the exceptional
class in $H_{S^1}^*(\widetilde{M})$. 
Note that in the extended decorated graph  with respect to a generic metric, 
the number of chains remains the same under a blowup, except in type II where there is one new chain of 
length $2$ created by the blowup.  
The length of any chain that is not touched by a blowup remains the same; and for blowups of type I and III,
the length of the chain in 
which the blowup occurs is increased by one.
}\labell{Table-blup}
\end{table}

\vskip 0.1in

\noindent We now proceed to prove the main theorem.

\begin{proof}[Proof of the generators and relations description in Theorem \ref{ThNew}]
By \S \ref{preface}, we can justify the description by induction on the number $n$ of $S^1$-equivariant  K\"ahler blowups, starting from a minimal model. 
The base case $n=0$ is contained in 
Proposition \ref{Procp2}, Proposition \ref{Prohirz}, and Proposition \ref{Proruled}.

For the induction step, let $S^1 \acts (M,\omega)$ be an $n$-fold $S^1$-equivariant symplectic
blowup of a minimal model.  Consider an $S^1$-equivariant  K\"ahler blowup $S^1 \acts (\widetilde{M},\widetilde{\omega})$.
We aim to describe the evolution of $H_{S^1}^*(M;\Z)$ to $H_{S^1}^*(\widetilde{M};\Z)$.

 Denote by $G^M$  the set that  consists of the $H_{S^1}^{2}(M)$-classes $\tau^M_0,\,\tau^M_{\infty}$, and the $\sigma^M_{i,j}$s that correspond to robust edges in the extended decorated graph of $S^1 \acts (M,\omega)$ with respect to a generic metric. Let $G^{\widetilde{M}}$ be the set defined similarly for $S^1 \acts (\widetilde{M},\widetilde{\omega})$. By the induction hypothesis, $H_{S^1}^{2*}(M)$ is generated by the elements of  $G$. We claim that $H_{S^1}^{2*}(\widetilde{M})$ is generated by the elements of  $G^{\widetilde{M}}$. 

We first note that Table~\ref{Table-blup} above establishes that the  set $\{\widetilde{\eta}\,|\, \eta \in G^M\} \bigcup \{e\}$ equals ${G}^{\widetilde{M}}$.
We aim to show that a class $\alpha\in H^{2k}_{S^1}(\widetilde M;\Z)$ is a linear combination of
products of classes in ${G}^{\widetilde{M}}$.  By Lemma~\ref{lem:coh-of-blowup}, 
$\alpha$ is the image of a class
$(a,b)\in H_{S^1}^{2k}(M) \oplus H_{S^1}^{2k-2} {(p)}$.  By the induction hypothesis, the class $a$ is a linear combination 
of products of the degree $2$ classes in $G^M$.  The formul\ae\ given in Lemma~\ref{lem:coh-of-blowup} 
for the product structure in $H^{2k}_{S^1}(\widetilde M;\Z)$ imply that $\bl^*(a)$ is also 
a linear combination of products of elements in $\{\widetilde{\eta}\,|\, \eta \in G^M\}$.  

Next, we note that $b\in H_{S^1}^{2k-2} {(p)}\cong \Z t^{2k-2}$ (when $k>0$). The pull-back
$q^*(b)$ restricts to the fixed points of $\CP^1$ as $b$ on each.  This is an equivariant constant class.
It is then straightforward to compute 
that 
$${\iota^{!}_{\bl^{-1}(p)}}(q^{*}(b))=\beta\cup e,$$ 
where $\beta$ is the equivariant constant class, 
the integer multiple of $\pi^*(t)^{2k-2}$ corresponding to $b\in \Z t^{2k-2}$;
and the $\cup$ denotes the $H_{S^1}^*(pt)$-algebra multiplication.  
Using the formula \eqref{eq:algstr} for $\pi^*(t)$, proved on page~\pageref{proof-algebra}, we deduce that 
$\beta$ is a linear combination of elements in  ${G}^{\widetilde{M}}$.
Therefore, the class
$\alpha$ itself is in the set generated by $\{\widetilde{\eta}\,|\, \eta \in G^M\} \bigcup \{e\}={G}^{\widetilde{M}}$, 
which is thus indeed a generating set for $H^{2k}_{S^1}(\widetilde M;\Z)$.

Now, we show that the linear relation \begin{equation}\labell{tauhblup}
\tau^{\widetilde{M}}_{h}=\sum_{j=1}^{\ell^{\widetilde{M}}_{i}}m^{\widetilde{M}}_{i,j}\sigma^{\widetilde{M}}_{i,j}\end{equation}
holds for all $1 \leq i \leq k^{\widetilde{M}}$. 
Since the map $i^{*} \colon H_{S^1}^{*}(M;\Z) \hookrightarrow H_{S^1}^{*}(M^{S^1};\Z)$
is injective  \cite[Theorem 1.1]{KesslerHolm2}, it is enough to show that equation 
\eqref{tauhblup} holds on every component of the fixed point set.  That is, we must 
check that the restrictions of the left-hand and right-hand classes to each of the 
fixed surfaces $\Sigma_0,\Sigma_{\infty}$ and isolated fixed points $v_{i,j}$ 
coincide. This is straight forward bookkeeping based on the localization formul\ae\ given
in Tables \ref{Table-rest-2}, \ref{Table-rest-0}, 
and \ref{Table-rest-1}.
If $\tau_h$ is one of the $\sigma_{i,j}$s that correspond to robust edges in the graph of $S^1 \acts (M,\omega)$ then 
$\tau_h=\iota_{S}^{!}(\mathbb{1}_{\CP^1})$ and  $\tau^{\widetilde{M}}_h=\widetilde{\tau_h}=\iota_{\widetilde{S}}^{!}(\mathbb{1}_{\CP^1})$. 

We must show that these are all the linear relations among the generators.  Suppose there is
another. By performing an equivariant K\"ahler blow down,
 this would give a linear relation in $H_{S^1}^{2*}({M})$.  By the induction hypothesis,
this must coincide with a combination of the relations in the blown down list \eqref{tauhblup} with
$M$ instead of $\widetilde{M}$.
When we blow up again, we will get a combination of linear relations in  \eqref{tauhblup}  
together with a multiple 
$\gamma e$ for $\gamma \in \Z$.
However, since the above relations hold, $\gamma$ must be zero. 

 Finally, we track the effect of the blowup on the multiplicative relations. Since the map $i^{*} \colon H_{S^1}^{*}(M;\Z) \hookrightarrow H_{S^1}^{*}(M^{S^1};\Z)$
is injective  \cite[Theorem 1.1]{KesslerHolm2}, we know that a product of classes $x_1\cup \cdots \cup x_j$ is
zero precisely when $(x_1\cup \cdots \cup x_j)|_F=0$ for each fixed component $F$.  But
$(x_1 \cup \cdots \cup x_j)|_F = (x_1|_F) \cup \cdots \cup (x_j|_F)$, so for each $F$, we must have
some $x_k|_F=0$ because we are working in a domain. In other words, the intersection of
the supports of the cohomology classes must be empty.
In particular, for $\eta \in G^M$, the product 
$$\widetilde{\eta} \cup e=\left[\widetilde{S_\eta}\right] \cup [\bl^{-1}p]$$ is zero if and only if $p \notin S_\eta$. 
Similarly, when $\eta \neq \theta$ in $G^M$, the product $$\widetilde{\eta} \cup \widetilde{\theta}$$ equals $0$ if and only if either 
\begin{itemize}
    \item $\eta \cup \theta=0$; or 
    \item $\eta \cup \theta \neq 0$ and $p \in S_{\eta} \bigcap S_{\theta}$. 
\end{itemize}
The effect of a blowup on $\tau_h^2$ is laid out in Table~\ref{Table-blup2}.

\newpage

\renewcommand{\arraystretch}{1.2}
\begin{table}[h]
$$
\begin{array}{c|c|c|c}
\text{Blowup type} & \text{at} & \text{with}& {\text{the restriction of }\tau^{\widetilde{M}}_{h}} \cup   {\tau^{\widetilde{M}}_{h} \text{ equals }}\\ \hline\hline
I & v_{i^*,j^*}& &{\tau_h \cup \tau_h}|_{v_{i,j}}  \text{ at }\tilde{v}_{i,j}\\
&&&\text{ if }i \neq i^{*} \text{ or }i=i^{*} \text{ and }j \leq j^{*}\\
&&&{\tau_h \cup \tau_h}|_{*} \text{ at }* \text{ for }*=\max,\min\\
&&& {\tau_h \cup \tau_h}|_{v_{i^*,j-1}} \text{ at }\tilde{v}_{i^*,j} \text{ if }j>j^{*}+1\\
  & & &\text{and $0$ at } \tilde{v}_{i^{*},j^{*}+1}\\ \hline
I\! I & \mathrm{\begin{tabular}{c} max or min  \end{tabular}} && {\tau_h \cup \tau_h}|_{v_{i,j}} \text{ at }\tilde{v}_{i,j} \text{ if } i \leq k \\
&&& {\tau_h \cup \tau_h}|_{{*}} \text{ at } {*} \text{ for }*=\min, \max\\
&&& 0 \text{ at }\tilde{v}_{k+1,1} \\ \hline
I\! I\! I & \mathrm{\begin{tabular}{c} max \end{tabular}} & \mathrm{\begin{tabular}{c}$m^M_{\sigma_{1,\ell_1}}>m^M_{\sigma_{2,\ell_2}}$\\$\geq \ldots \geq m^M_{\sigma_{k,\ell_k}}$ \end{tabular}} & (m^M_{\sigma_{1,\ell_1}}-m^M_{\sigma_{2,\ell_2}})^2 {m^M_{\sigma_{2,\ell_2}}}^2 t^2 \text{ at } \max\\ 
&&& 0 \text{ at } \tilde{v}_{1,\ell_1}\\
&&& {\tau_h \cup \tau_h}|_{v_{1,j}} \text{ at }\tilde{v}_{1,j}\text{ for }j<\ell_1\\
&&&{\tau_h \cup \tau_h}|_{v_{i,j}} \text{ at }\tilde{v}_{i,j} \text{ for }i>1\\
&&&{\tau_h \cup \tau_h}|_{\min} \text{ at }\min \\ \hline
   I\! I\! I &\mathrm{\begin{tabular}{c} min \end{tabular}} & \mathrm{\begin{tabular}{c} $m^M_{\sigma_{1,1}}>m^M_{\sigma_{2,1}}$\\ $\geq \ldots \geq m^M_{\sigma_{k,1}}$ \end{tabular}} & (m^M_{\sigma_{1,1}}-m^M_{\sigma_{2,2}})^2 {m^M_{\sigma_{2,2}}}^2 t^2 \text{ at } \min \\ 
   &&& 0 \text{ at } \tilde{v}_{1,1}\\
   &&& {\tau_h \cup \tau_h}|_{v_{1,j-1}} \text{ at }\tilde{v}_{1,j}\text{ for }j>1\\
&&&{\tau_h \cup \tau_h}|_{v_{i,j}} \text{ at }\tilde{v}_{i,j} \text{ for }i>1\\
&&&{\tau_h \cup \tau_h}|_{\max} \text{ at }\max\\ \hline
     I\! V & \mathrm{\begin{tabular}{c}max or min \end{tabular}}& &0 \text{ at $\max$ ($\min$)} \\ 
     & & &\text{and $\tau_h \cup \tau_h$ elsewhere} \\ \hline
   
\end{array}
$$
\bigskip
\caption{The impact of a blowup on $\tau_h \cup \tau_h$. In type III, if the blowup is at $\max$, the term $(m^M_{\sigma_{1,\ell_1}}-m^M_{\sigma_{2,\ell_2}}) m^M_{\sigma_{2,\ell_2}}$  is $>0$; if the blowup is at $\min$, the term $(m^M_{\sigma_{1,1}}-m^M_{\sigma_{2,2}}) m^M_{\sigma_{2,2}}$ is $>0$.
}
\labell{Table-blup2}
\end{table}

We argue that we have found all of the multiplicative relations.
Indeed, because the blowup is a local change, the only
changes to the list of multiplicative relations will be changes that involve the support of
the new exceptional divisor replacing the support of the point blown up.  These are  
precisely the changes we have identified.
\end{proof}


\section{Invariants under a weak isomorphism of algebras} \labell{sec5a}

Let $M$ and $N$ be compact, connected, four-dimensional symplectic manifolds, 
each equipped with a Hamiltonian circle action.  In this section, we establish what effect 
a weak isomorphism  $f\colon H_{S^1}^*(M;Z)\to H_{S^1}^*(N;Z)$ can have on the generators.
For that, we refine our understanding of the structure of $H_{S^1}^*(M;Z)$.

We assume that if $\fat^M=1$, then we are in the case when $\tau^M_\infty \neq 0$.
Otherwise, we compose the given weak isomorphism  $f\colon H_{S^1}^*(M;Z)\to H_{S^1}^*(N;Z)$ 
with the isomorphism $H_{S^1}^*(M;Z) \to H_{S^1}^*(M;Z)$ that sends 
\begin{equation*} 
\begin{array}{ll}
\tau^M_0 \mapsto -\tau^M_{\infty}; & \\
\tau^M_{\infty} \mapsto -\tau^M_0; &\\
 \sigma^M_{i,j} \mapsto -\sigma^M_{i,\ell_i-j+1} &  \forall 1\leq j \leq \ell_i ; \text{ and }\\
 \pi_{M}^{*}(t) \mapsto  \pi_{M}^{*}(t).
 \end{array}
\end{equation*}
This isomorphism is induced by changing the Hamiltonian $S^1$-manifold
from $(S^1 \acts M, \omega,\Phi)$ to $(S^1 \acts M,-\omega,-\Phi)$. That is,
we replace $\omega$ with $-\omega$, without changing the circle action; this gives
a new presentation for the cohomology.

\begin{noTitle}{\bf Annihilator subalgebras.}
For an equivariant cohomology class $\alpha\in H_{S^1}^{2*}(M)$, we denote by 
$$
\Ann(\alpha) = \left\{\beta \in H_{S^1}^{2*}(M) \big|\ \alpha \cup \beta = 0\right\}
$$
the {\bf annihilator} of $\alpha$ in $H_{S^1}^{2*}(M)$.  This is an $H_{S^1}^*(pt;\Z)$-subalgebra of the 
equivariant cohomology ring and it retains a grading.  We define
$$
\Ann^{2k}(\alpha) = \Ann(\alpha)\cap H_{S^1}^{2k}(M)
$$
and note that $\Ann^{2k}(\alpha)$ is  a $\Z$-module.
Localization makes annihilator submodules 
easier to compute: one does all the calculations in the equivariant cohomology of the fixed point sets where it
is easy to see what the zero divisors are.  The ranks of the graded pieces of $\Ann(\alpha)$ 
play a key role in the proof of Theorem \ref{thm:unique}.  
\end{noTitle}

\begin{Notation} \label{not:z012}
The following sets $Z_0$, $Z_1$, and $Z_2$ are key to understanding the
annihilator subalgebra.  
They are the primitive classes in $H^4_{S^1}(M)$ that are 
supported on exactly one fixed component.
Their definition depends on the number of fixed surfaces.
\begin{itemize}
\item If $\fat=2$ denote
$$Z_0:=\{\tau_h \cup \tau_0, \, \tau_h\cup \tau_\infty\},\, Z_1:=\{ \sigma_{i,j} \cup \sigma_{i,j+1} \,|\,  j < \ell_i\}, \, Z_2:=\{\tau_0 \cup \tau_0, \,, \tau_\infty \cup \tau_\infty \}.$$
\item If $\fat=1$, we have assumed that $\tau_\infty \neq 0$ and denote 
$$Z_0:=\{\tau_h\cup \tau_\infty\},\, Z_1:=\{ \sigma_{1,1} \cup \sigma_{2,1}, \, \sigma_{i,j} \cup \sigma_{i,j+1} \,|\,  j < \ell_i\}, \, Z_2:=\{\tau_\infty \cup \tau_\infty\}.$$
\item If $\fat=0$ denote
$$Z_0:=\emptyset, \, Z_1:=\{\sigma_{1,1} \cup \sigma_{2,1}, \, \sigma_{1,\ell_1} \cup \sigma_{2,\ell_2}, \, \sigma_{i,j} \cup \sigma_{i,j+1} \,|\, j < \ell_i\},\, Z_2:=\emptyset.$$
\end{itemize}
\end{Notation}

\begin{Lemma} \label{cl:ann} 
{\color{white}b.}
\begin{enumerate}
\item The sets $Z_0$, $Z_1$, and $Z_2$ are pairwise disjoint and
the elements of $Z_0 \bigcup Z_1 \bigcup Z_2$ are linearly independent over $\Z$. \label{ZlinIdept}
\item For $\alpha \in Z_0 \bigcup Z_1 \bigcup Z_2$, we have $\alpha \cup \alpha=0$ exactly if $\alpha \in Z_0$.
\label{Zselfzerodiv}
\item If $\alpha,\beta$ are distinct classes in $Z_0 \bigcup Z_1$, then $\alpha \cup \beta=0$. \label{Zzerodiv}
\item The rank of $\Ann^{2}(\alpha)$ is
$$
\rank\left(\Ann^{2}(\alpha)\right) = \begin{cases}
    \sharp iso+2 \sharp fat -2 & \mbox{ when } \alpha \in Z_0 \bigcup Z_1\\
    \sharp iso+2 \sharp fat -3 & \mbox{ when } \alpha \in Z_2\\
\end{cases}
$$ \label{rankAnn2}

\item For $\alpha, \, \beta \in Z_0 \bigcup Z_1$, such that  $\alpha \neq \beta$, we have 
$$\rank\left(\Ann^{2}(\alpha) \bigcap \Ann^{2}(\beta)\right) < \rank(\Ann^{2}(\alpha)).$$\label{rankInequality}
\end{enumerate}
\end{Lemma}

\begin{proof}
Items \eqref{ZlinIdept}--\eqref{Zzerodiv} are straight-forward calculations using localization and 
the generators and relations description in Theorem \ref{ThNew}.
Items \eqref{rankAnn2} and \eqref{rankInequality} follow from the calculation of $\Ann^{2}(\alpha)$ for $\alpha \in Z_0 \bigcup Z_1 \bigcup Z_2$, listed in Tables \ref{Table-ann-2-1a}, \ref{Table-ann-1-1a}, \ref{Table-ann-0}, and explained above the tables.
\end{proof}

\begin{Proposition}\label{pro:ann}
The non-zero classes $\alpha \in H_{S^1}^{4}(M)$ that  satisfy that $\Ann^{2}(\alpha)$ has maximal rank are precisely the 
non-zero integer
multiples of the classes in $Z_0 \bigcup Z_1$.
Moreover, such a class $\alpha$ is a multiple of a class in $Z_0$ when $\alpha \cup \alpha =0$ and a multiple of a class  in $Z_1$ when $\alpha \cup \alpha \neq 0$.
\end{Proposition}
\noindent By \emph{has maximal rank}, we mean maximal among the ideals for non-zero classes.
The proof of this proposition is inspired by Masuda's proof of \cite[Lemma 3.1]{masuda} in the toric setting.

\begin{proof}
For a class $\alpha \in H_{S^1}^{4}(M)$ that is supported on exactly one fixed component, the rank of $\Ann^{2}(\alpha)$ is given in Lemma \ref{cl:ann}\eqref{rankAnn2}.
If  $\alpha \in H_{S^1}^{4}(M)$ is supported on at least two fixed components, then there are degree $2$ classes  that vanish on
one of the components on which $\alpha$ is non-zero, 
and do not vanish on (at least one of) the other(s). 
So $\alpha \cup H_{S^1}^{2}(M)$ has rank at least two, and so its kernel in $H_{S^1}^{2}(M)$ has codimension at least two. Recall that $\dim H_{S^1}^{2}(M)$ equals $\iso+2\fat -1$, see \eqref{eq:poincare}. 
The second (``Moreover") statement is 
a consequence of the first part and Lemma \ref{cl:ann}\eqref{Zselfzerodiv}.
\end{proof}

\begin{Notation} \label{Not:zieta}
Let $\eta \in H_{S^1}^{2*}(M)$. 
For the sets $Z_0, \, Z_1, \, Z_2$ defined in Notation~\ref{not:z012}, 
we set 
$$Z_i(\eta):=\left\{\sigma \in Z_i \, \, \Big|\,\, \eta \cup \sigma=0\right\}$$
for $i=0,1,2$.
\end{Notation}

\begin{Lemma}\label{cl:ann2}
\phantom{.}
\begin{enumerate}
\item When $\fat=2$, the classes $\tau_0$ and $\tau_\infty$ are linearly independent over $\Z$. Moreover, the set
$\sspan_{\Z}(\tau_0)\bigcup \sspan_{\Z}(\tau_\infty) \smallsetminus \{0\}$ equals the set 
$$
S_{0,\infty}:=\left\{\eta \in H_{S^1}^{2}(M) \, \Big| \, Z_1(\eta)=Z_1, \, \eta \cup \eta \neq 0, \, \rank \Ann^{2}(\eta)=\iso+1\right\}.
$$

\vskip 0.1in

\noindent 
When $\fat=1$, we have assumed $\tau_\infty \neq 0$, and then
$$
\sspan_{\Z}(\tau_\infty)
= \left\{\eta \in H_{S^1}^{2}(M) \, \Big| \, Z_1(\eta)=Z_1\right\}.$$

\item The set 
$$
\sspan_{\Z}\left(\bigcup_{i=1}^{k}\bigcup_{j=1}^{\ell_i}\{\sigma_{i,j}\}\right)=\left\{\eta \in H_{S^1}^{2}(M)\,\Big|\, Z_0(\eta)=Z_0\right\}.
$$

\item 
The set
$$
\left \{\eta \in H_{S^1}^{2}(M) \, \Big| \, Z_0(\eta)=Z_0, \, |Z_1(\eta)|+|Z_{2}(\eta)|=\iso+\fat-2\right\}$$
consists of non-zero integer multiples of one of the following:
\begin{itemize}
\item $\sigma_{i,j}$ for $1 \leq i \leq k$ and $1\leq j \leq \ell_i$;
\item  $\displaystyle{\sum_{j=\alpha}^{\beta}m_{i,j}\sigma_{i,j}}$ for $1 \leq \alpha <\beta \leq \ell_i$;
\item $\displaystyle{\sum_{j=1}^{\beta_i}m_{i,j}\sigma_{i,j}-\sum_{j=1}^{\beta_{i'}}m_{i',j}\sigma_{i',j}}$ for $i \neq i'$ and $\beta_i < \ell_i, \, \beta_{i'}<\ell_{i'}$; or
\item $\displaystyle{\sum_{j=\alpha_{i}}^{\ell_i}m_{i,j}\sigma_{i,j}-\sum_{j=\alpha_{i'}}^{\ell_{i'}}m_{i',j}\sigma_{i',j}}$ for $i \neq i'$ and $1<\alpha_i , \, 1<\alpha_{i'}$.
\end{itemize}

\end{enumerate}

\end{Lemma}

\begin{proof}
We verify each item by a double inclusion argument.
\begin{enumerate}
\item Assume $\fat=2$.
Let $\eta$ be an element in the set $S_{0,\infty}$. By definition of the set, $\eta$ is in $\bigcap_{\sigma \in Z_1}{\Ann^{2}(\sigma)}$. Hence, by Table \ref{Table-ann-2-1a}, 
$$\eta=a_0 \tau_0+a_h \tau_h+a_\infty \tau_\infty \text{ for } a_0,a_h,a_\infty \in \Z.$$ Moreover, by Table~\ref{Table-ann-2-1a}, if at least two of $a_0,a_h,a_\infty$ are  $\neq 0$ then $\rank \Ann^{2}(a_0 \tau_0+a_h \tau_h+a_\infty \tau_\infty)<\iso+1$, and if exactly one of $a_0,a_h,a_\infty$ is $\neq 0$ then $\rank \Ann^{2}(a_0 \tau_0+a_h \tau_h+a_\infty \tau_\infty)=\iso+1$. Since the last property holds, we must have that $\eta$ is a non-zero multiple of either $\tau_0$ or $\tau_h$ or $\tau_\infty$. By Table \ref{Table-rest-2}, for $*=0,\infty$,
 $$\tau_h \cup \tau_h =0 \neq \tau_{*} \cup \tau_{*}.$$ Hence, since $\eta \cup \eta \neq 0$, the class $\eta$ must be in $\sspan_{\Z}(\tau_0)\bigcup \sspan_{\Z}(\tau_\infty) \smallsetminus \{0\}$.
 On the other hand, it follows immediately from Table \ref{Table-ann-2-1a} and Table \ref{Table-rest-2} that $\tau_0$ and $\tau_\infty$ are in $S_{0,\infty}$, and therefore so is $\sspan_{\Z}(\tau_0)\bigcup \sspan_{\Z}(\tau_\infty) \smallsetminus \{0\}$.

\noindent 
Now, when $\fat=1$, our convention is that $\tau_\infty\neq 0$ and $\tau_0=0$.  By Table \ref{Table-ann-1-1a},
 $$\bigcap_{\sigma \in Z_1}{\Ann^{2}(\sigma)}=\{a_\infty \tau_\infty \, | \, a_\infty \in \Z\}.$$
 \item The second item follows directly from Tables \ref{Table-ann-2-1a}, \ref{Table-ann-1-1a}, and \ref{Table-ann-0}.

 \item The third item follows from the previous one and from item (2) of Corollary \ref{cor:relation}.
 
\end{enumerate}
\noindent This completes the proof of the lemma.
\end{proof}

\begin{Notation}\labell{not:genus}
We set the \emph{genus $g$} of $S^1 \acts M^4$ to be the genus of a fixed surface, if one exists, and $0$ otherwise.
\end{Notation}

\begin{Lemma} \labell{iso+fat}
Assume that $H_{S^1}^{2*}(M)$ and $H_{S^1}^{2*}(N)$ are weakly isomorphic as algebras over $H_{S^1}^{*}(\pt)$. 
Then we have 
$$\fat_M=\fat_N \mbox{ and } \iso_M=\iso_N.$$
If, in addition, $\rank H_{S^1}^{1}(M) = \rank H_{S^1}^{1}(N)$, then
the genera $g_M=g_N$.
\end{Lemma}

\begin{proof}
By Proposition \ref{pro:ann}, a weak isomorphism $H_{S^1}^{2*}(M) \to H_{S^1}^{2*}(N)$ sends $Z_0^M \bigcup Z_1^M$ to $Z_0^N \bigcup Z_1^N$. Therefore
$$
\iso_M+\fat_M=\iso_N+\fat_N.
$$
The weak isomorphism of $H_{S^1}^{2*}(M)$ and $H_{S^1}^{2*}(N)$ implies that the even degree parts of the 
Poincar\'e polynomials $P_{S^1}^{M}(t)$ and $P_{S^1}^{N}(t)$ coincide.
Comparing with \eqref{eq:poincare}, we conclude that
$$
\iso_M+2\fat_M=\iso_N+2\fat_N.
$$
Combining the two equalities, we get  $\fat_M=\fat_N$ and $\iso_M=\iso_N$.

\noindent Finally, the equality of the ranks of $H_{S^1}^{1}$ implies 
that $g_M=g_N$.
\end{proof}

\begin{Notation}\label{not:f}
For the remainder of this section,  we let $f \colon H_{S^1}^{2*}(M) \to H_{S^1}^{2*}(N)$ be a weak isomorphism of 
algebras over $H_{S^1}^{*}(\pt)$, and we assume that the rank of $H_{S^1}^{3}(M)$ equals the rank of $H_{S^1}^{3}(N)$.
We will rely on Lemma \ref{iso+fat} and write $\fat$ for both terms $\fat_M$ and $\fat_N$, and write $\iso$ for 
both terms $\iso_M$ and $\iso_N$.
\end{Notation}

\begin{noTitle}\labell{restrictions} {\bf Restrictions to the fixed points.}
We have proved in \cite[Theorem~1.1]{KesslerHolm2} 
that the inclusion $i:  M^{S^1}\to M$ induces an injection called ``restriction to the fixed points" in equivariant cohomology: 
$$
i^* 
=\bigoplus_{F \subset M^{S^1}}\iota_F^{*} : H_{S^1}^*(M;\Z)
\to \bigoplus_{F \subset M^{S^1}} H_{S^1}^*(F;\Z).
$$
The weak isomorphism $f \colon H_{S^1}^{2*}(M) \to H_{S^1}^{2*}(N)$ induces 
$f \colon {i^{*}}(H_{S^1}^{2*}(M)) \to {i^{*}}(H_{S^1}^{2*}(N))$, such that the following diagram is commutative. 
\begin{equation} \labell{diagramisorestrict1}
\begin{gathered}
\xymatrix{
 H_{S^1}^{2*}(M) \ar[r]^-{i^{*}_M}\ar[d]_{f} & {i^{*}}(H_{S^1}^{2*}(M)) \mbox{ \Large$\subset$ } H_{S^1}^*(M^{S^1}) \ar@<-8ex>[d]^f \\ 
 H_{S^1}^{2*}(N) \ar[r]_-{i^{*}_N} & {i^{*}}(H_{S^1}^{2*}(N)) \mbox{ \Large$\subset$ } H_{S^1}^*(N^{S^1})
}
\end{gathered}
\end{equation}
For $\eta \in H_{S^1}^{2*}(M)$, extending \eqref{eqrestriction}, we denote its restriction to a fixed component
$F$ by $\eta|_{F} := \iota_F^*(\eta)\in H_{S^1}^*(F;\Z)$.  Thus,
for a class $\eta\in H_{S^1}^{2*}(M)$, we have
\begin{equation} \labell{eqrestrict2}
{f}\left(\bigoplus_{F \subset M^{S^1}}(\eta|_F)\right)={f}\left(\bigoplus_{F \subset M^{S^1}}\iota_F^{*}(\eta)\right)=
{f}(i_M^{*}(\eta))=i_N^{*}(f(\eta))=
\bigoplus_{F \subset N^{S^1}}f(\eta)|_{F}.
\end{equation}
\end{noTitle}

Our results on the ranks of the annihilator allow us to understand the map induced by a weak algebra isomorphism on the component Euler classes  defined in
Definition~\ref{def:comp euler}.

\begin{Corollary} \labell{cor:Euler} 
Let $M$ and $N$ be compact, connected, four-dimensional 
Hamiltonian $S^1$-manifolds and let 
$f \colon H_{S^1}^{2*}(M) \to H_{S^1}^{2*}(N)$ be a
weak algebra isomorphism.
Then $f$ induces a bijection
\begin{equation}\labell{eqhatf}
\widehat{f} \colon \{\text{connected components of }M^{S^1}\} \to  \{\text{connected components of }N^{S^1} \}
\end{equation}
that sends a component $F \subset M^{S^1}$ to the component $\widehat{f} (F) \subset N^{S^1}$ such that 
\begin{equation}\label{Euler100}
f(\varepsilon_F) =
\pm \varepsilon_{\widehat{f} (F)}.
\end{equation}
The map $\widehat{f}$ sends isolated points to isolated points and surfaces to surfaces.
Moreover, 
for $\eta \in H_{S^1}^{2*}(M)$, 
\begin{equation} \labell{eqzero}
\eta|_{F}=0 \Leftrightarrow {f(\eta)}|_{\widehat{f}(F)}=0.
\end{equation}
\end{Corollary}

\begin{proof}
The weak algebra isomorphism $f$ sends $\Ann^{2k}(\eta)$ to $\Ann^{2k}(f(\eta))$ for every $\eta \in H_{S^1}^{2*}(M)$ and $k \in \N$. Therefore, by Proposition \ref{pro:ann}, for $i=0,1$, $f$ sends $Z_i^M \bigcup -Z_i^M$ to $Z_i^N \bigcup -Z_i^N$ and 
\begin{equation} \label{eq:fz01}
Z_i^{N}(f(\eta)) \bigcup -Z_i^{N}(f(\eta)) =f(Z_{i}^{M}(\eta) \bigcup -Z_i^{M}(\eta)) \text{ for every }\eta \in H_{S^1}^{2}(M).
\end{equation}
Hence, by item (1) of Lemma \ref{cl:ann2}, 
\begin{equation} \label{eq:fz2}
\{\pm f(\tau^M_0),\pm f(\tau^M_\infty) \}=\{\pm\tau^N_0,\pm\tau^N_\infty\},
\end{equation}
and thus $f$ sends $Z_2^M$ to $Z_2^N$, and 
\begin{equation} \label{eq:fz22}
Z_2^{N}(f(\eta))=f(Z_{2}^{M}(\eta)) \text{ for every }\eta \in H_{S^1}^{2}(M).
\end{equation}
Stating \eqref{eq:fz01} and \eqref{eq:fz2} in terms of the component Euler classes  gives \eqref{Euler100} with $\dim(F)=\dim\widehat{f}(F)$.

 Since $f$ is one-to-one, so is $\widehat{f}$. 
 Furthermore, by Lemma \ref{iso+fat}, we have $\iso_M+\fat_M=\iso_N+\fat_N$, and so the map $\widehat{f}$ is also onto.

We may use the
component Euler classes to determine the support (or equivalently the zero locus) of an arbitrary equivariant class:
$$
\alpha|_F = 0 \Longleftrightarrow \alpha\cup \varepsilon_F= 0.
$$
Because $f$ preserves cup product, this implies \eqref{eqzero} must hold.
\end{proof}

\begin{Proposition}\labell{cor:constant}
Let $M$ and $N$ be compact, connected, four-dimensional 
Hamiltonian $S^1$-manifolds and let 
$f \colon H_{S^1}^{2*}(M) \to H_{S^1}^{2*}(N)$ be a
weak algebra isomorphism.
For an isolated fixed point $p$, let $\delta^f_p \in \{+1,-1\}$ be the sign so that $f(\varepsilon_p)=\delta^f_p \varepsilon_{\widehat{f}(p)}$. Then $\delta^f_p$ is constant on the set of isolated fixed points.
\end{Proposition}
\begin{proof}
Let $p$ and $q$ be any two isolated fixed points in $M^{S^1}$. We aim to show that $\delta^f_p=\delta^f_q$.  We start by defining a class $\sigma \in H_{S^1}^{2}(M)$ whose support is $\{p,q\}$. The fixed points $p$ and $q$ correspond to vertices $v_{i,\alpha}$, $v_{i^*,\alpha^*}$ in chains $i,i^*$ of edges in an extended decorated graph. 
    If at least one of the vertices is extremal, we choose $i$ and $i^*$ to be equal.
    If $i=i^*$ and, without loss of generality, $\alpha<\alpha^*$, we set $$\sigma:=\sum_{j=\alpha}^{\alpha^*}m_{i,j}\sigma_{i,j}.$$ 
    By our conventions, if $i \neq i^*$ then $1<\alpha$ and $1<\alpha^*$; in this case,  we set $$\sigma:=\sum_{j=\alpha}^{\ell_i}m_{i,j}\sigma_{i,j}-\sum_{j=\alpha^{*}}^{\ell_{i^*}}m_{i,j}\sigma_{i,j}.
    $$
    The fact that this class $\sigma$ is non-zero precisely on $p$ and $q$
    follows by straight-forward, localization calculations using Tables
    \eqref{Table-rest-2},
    \eqref{Table-rest-0}, and
    \eqref{Table-rest-1}

We will now show that $\delta_p^f=\delta_q^f$.
Let $\varepsilon_M$ be the sum $\sum_{F}\varepsilon_F$ 
of component Euler classes,
taken over the connected components $F$ of $M^{S^1}$. This class satisfies 
$e_{S^1}(\nu(F\subset M))=\iota_{F}^{*}(\varepsilon_M)$. 
We define $\varepsilon_N$ analogously.
By the ABBV localization formula,  we have
$$
\frac{\iota_{p}^{*}\sigma}{\iota_{p}^{*}(\varepsilon_M)}+\frac{\iota_{q}^{*}\sigma}{\iota_{q}^{*}(\varepsilon_M)}=0.
$$
Extending $f$ and $\widehat{f}$ over $\Q[t^{-1}]$, and using \eqref{eqrestrict2}, this implies
\begin{equation}\label{eq:abbv1}
0= \frac{f(\iota_{p}^{*}\sigma)}{f(\iota_{p}^{*}(\varepsilon_M))} + 
\frac{f(\iota_{q}^{*}\sigma)}{f(\iota_{q}^{*}(\varepsilon_M))} 
=\frac{\iota_{\widehat{f}(p)}^{*}f(\sigma)}{\iota_{\widehat{f}(p)}^{*}f(\varepsilon_M)}
+ \frac{\iota_{\widehat{f}(q)}^{*}f(\sigma)}{\iota_{\widehat{f}(q)}^{*}f(\varepsilon_M)}. 
 \end{equation}
On the other hand, because $f(\sigma)$ is supported on the isolated fixed points $\{\widehat{f}(p), \widehat{f}(q)\}$, using ABBV on $N$, we must have
 \begin{equation}\label{eq:abbv2}
\frac{\iota_{\widehat{f}(p)}^{*}f(\sigma)}{\iota_{\widehat{f}(p)}^{*}(\varepsilon_N)}+
\frac{\iota_{\widehat{f}(q)}^{*}f(\sigma)}{\iota_{\widehat{f}(q)}^{*}(\varepsilon_N)}=0.
 \end{equation}
In Corollary~\ref{cor:Euler}, we established that
 $$
 \iota^*_{\widehat{f}(p)} f(\varepsilon_M) = \pm  \iota^*_{\widehat{f}(p)} \varepsilon_N.
 $$
Combining this fact with \eqref{eq:abbv1} and  \eqref{eq:abbv2}, we can then deduce that 
$$\Big( \iota_{\widehat{f}(p)}^{*}f(\varepsilon_M), \iota_{\widehat{f}(q)}^{*}f(\varepsilon_M)\Big)=\pm 
\Big(\iota_{\widehat{f}(p)}^{*}(\varepsilon_N),\iota_{\widehat{f}(q)}^{*}(\varepsilon_N)\Big).$$
That is, $\delta_p^f = \delta_q^f$, as desired.
\end{proof}

The following definition is key in determining when a weak algebra isomorphism 
$f$ preserves or reverses signs on component Euler classes:
$$
\varepsilon_F\mapsto \varepsilon_{\widehat{f}(F)}
\mbox{ versus }
\varepsilon_F \mapsto -\varepsilon_{\widehat{f}(F)}.
$$

\begin{Definition}\labell{def:orientation}
Let $M$ and $N$ be oriented, $d$-dimensional, compact $S^1$-manifolds 
that are equivariantly formal over $\Z$.
A weak isomorphism of algebras
 $f \colon H_{S^1}^{*}(M) \to H_{S^1}^{*}(N)$ is 
 {\bf or\-ien\-ta\-tion-pres\-erving} if the induced isomorphism 
in ordinary cohomology is or\-ien\-ta\-tion-pres\-erving.  That is,
the following diagram commutes:
 \begin{equation}\label{eq:orientation}
    \xymatrix{
    H_{S^1}^{d}(M) \ar[rr]^{f}\ar[d]_{I_M^*} & & H_{S^1}^{d}(N)\ar[d]^{I_N^*}\\
    H_{}^{d}(M) \ar[rr]^{\overline{f}} \ar[dr]_{\overline{\pi}_M^!} &  & H_{}^{d}(N) \ar[dl]^{\overline{\pi}_N^!}\\
    & \Z 
    }.
    \end{equation}
    Specifically we have
    $\overline{\pi}_N^{!} \circ {\overline{f}}|_{H^{d}(M)}=\overline{\pi}_M^{!}$.  
   We call $f$ {\bf orientation-reversing} if we have the alternative:     $\overline{\pi}_N^{!} \circ {\overline{f}}|_{H^{d}(M)}=- \overline{\pi}_M^{!}$.
    Here $\overline{\pi}_N^{!}$ and $\overline{\pi}_M^{!}$ are the (ordinary) pushforward maps for $\pi_N \colon N \to \pt$ and $\pi_M \colon M \to \pt$.

    Abusing notation in the context of compact, connected, four-dimensional 
    Hamiltonian $S^1$-manifolds, we use the commutativity of \eqref{eq:orientation}
    to define when a weak isomorphism $f:H^{2*}(M)\to H^{2*}(N)$ is orientation-preserving.
\end{Definition}

\begin{Remark} 
Note that the maps $\overline{\pi}_M^{!}$ and $\overline{\pi}_N^{!}$ are defined using Poincar\'e duality, 
which makes use of the orientations on $M$ and $N$.
The maps 
$\overline{\pi}_N^{!} \circ {\overline{f}}|_{H^{d}(M)}$ and $\overline{\pi}_M^{!}$ are both ring isomorphisms that identify $H^{d}(M)$ with $\Z$. In particular, a weak isomorphism $f$ must be either orientation-preserving or orientation-reversing.  
\end{Remark}

\begin{Remark} 
    In the symplectic category, we orient our manifolds using the top power of the symplectic form.
    If $f \colon H_{S^1}^{2*}(M) \to H_{S^1}^{2*}(N)$ is induced from an equivariant diffeomorphism $\check{f} \colon M \to N$, then $f$ is orientation preserving (reversing) if and only if $\check{f}$ preserves (reverses) orientation, with respect to the orientations on $M$ and $N$ induced by the symplectic forms.
\end{Remark}

\begin{Notation}
For $\alpha \in H_{S^1}^{p}(M)$ and $\beta \in H_{S^1}^{q}(M)$, denote the \emph{intersection form}
$$\alpha \cdot \beta =\pi^!_{M} (\alpha \cup \beta) \in H_{S^1}^{p+q-\dim M}(\pt),$$
where $\pi^!_{M}$ is the equivariant pushforward of $\pi_M \colon M \to \pt$, as defined in \eqref{eqpi}, and $\cup$ is the cup product in $H_{S^1}^{*}(M)=H^{*}(M \times ES^1/S^1)$.
The pushforward map can be identified with integration along the fiber, which gives an integration
formula 
$$\alpha \cdot \beta =\int_{M} (\alpha \cup \beta)$$
for the intersection form that may be more familiar to the reader.
\end{Notation}

\begin{Lemma}\labell{lem:inter1}
Let $M$ and $N$ be compact, connected, four-dimensional 
Hamiltonian $S^1$-manifolds and let 
$f \colon H_{S^1}^{2*}(M) \to H_{S^1}^{2*}(N)$ be a
weak algebra isomorphism.
If $f$ is or\-ien\-ta\-tion-pre\-ser\-ving, then it
preserves the intersection form
on classes in $H_{S^1}^2$.  If $f$ is orientation-reversing, then
it negates the intersection form.
\end{Lemma}

\begin{proof}
Let $\alpha, \, \beta \in H_{S^1}^{2}(M)$.
If $f$ is orientation-preserving, we have
 \begin{eqnarray*}
f(\alpha) \cdot f(\beta) 
&:=&\pi_N^{!} (f(\alpha) \cup f(\beta))=\pi_N^{!} (f(\alpha \cup \beta))\\
&=&\overline{\pi}_{N}^{!}(I_N^{*}(f(\alpha \cup \beta)))=\overline{\pi}_{N}^{!}(\overline{f}(I_M^{*}(\alpha \cup \beta))) \\
&=&
\overline{\pi}_M^{!}(I_M^{*}(\alpha \cup \beta))
= \pi_M^! (\alpha \cup \beta)\\
&=:&\alpha \cdot \beta.
\end{eqnarray*}
The first and last equalities are by definition of the intersection form. The second equality is since $f$ is a ring homomorphism. The third and penultimate equalities are since the equivariant pushforward of $\pi$ equals the non-equivariant one composed on $I^*$. The fourth equality is by definition of the induced isomorphism $\overline{f}$ in ordinary cohomology. The fifth equality is since $f$ is orientation-preserving. 

 Similarly, if $f$ is orientation-reversing, we have
\begin{eqnarray*}
f(\alpha) \cdot f(\beta) 
:=\pi_N^{!} (f(\alpha) \cup f(\beta))=\pi_N^{!} (f(\alpha \cup \beta))
= -\pi_M^! (\alpha \cup \beta)
=: -\alpha \cdot \beta.
\end{eqnarray*}
This completes the proof.
\end{proof}

\begin{Corollary} \label{cor:opor}
   Let $M$ and $N$ be compact, connected, four-dimensional 
Hamiltonian $S^1$-manifolds and let 
$f \colon H_{S^1}^{2*}(M) \to H_{S^1}^{2*}(N)$ be a
weak algebra isomorphism.
    \begin{enumerate}
   \item If $f$ is orientation-preserving then
    for any isolated fixed point $p$, 
    $$f(\varepsilon_p)= \varepsilon_{\widehat{f}(p)};$$
and if $f$ is orientation-reversing, then for any isolated fixed point $p$,  
    $$f(\varepsilon_p)= -\varepsilon_{\widehat{f}(p)}.$$
   \item  If $f$ is orientation-preserving, then $\widehat{f}$ sends extremal isolated fixed points to extremal ones, and interior isolated fixed 
points to interior ones.  On the other hand, if $f$   orientation-reversing, then $\widehat{f}$ interchanges extremal isolated fixed points and interior ones.
\item If $f$ is orientation-preserving then either 
\begin{itemize}
    \item $(f(\tau^M_0),f(\tau^M_\infty))=(\pm \tau^N_0,\pm \tau^N_\infty)$ and $e^M_{*}=e^N_{*}$ for $*=\min,\max$, or
    \item $(f(\tau^M_0),f(\tau^M_\infty))=(\pm \tau^N_\infty,\pm \tau^N_0)$ and $e^M_{\min}=e^N_{\max}$, $e^M_{\max}=e^N_{\min}$.
\end{itemize}
If $f$ is orientation-reversing then either 
\begin{itemize}
    \item $(f(\tau^M_0),f(\tau^M_\infty))=(\mp \tau^N_0,\mp \tau^N_\infty)$ and $e^M_{*}=-e^N_{*}$ for $*=\min,\max$, or
    \item $(f(\tau^M_0),f(\tau^M_\infty))=(\mp \tau^N_\infty,\mp \tau^N_0)$ and $e^M_{\min}=-e^N_{\max}$, $e^M_{\max}=-e^N_{\min}$.
\end{itemize}
\end{enumerate}
 \end{Corollary}

 \begin{proof}
We check each point in turn.
 
 \begin{enumerate}
 \item     
By Corollary \ref{cor:Euler}, 
for every isolated fixed point $p$, we have
$$f(\varepsilon_p)=\delta^f_p \varepsilon_{\widehat{f}(p)}$$
with $\delta^f_p=\pm 1$ and $\widehat{f}(p)$ an isolated fixed point in $N$.
Moreover, by Proposition~\ref{cor:constant}, the sign $\delta^f_p$ is constant on the set of isolated fixed points. We denote this constant by $\delta^f$. 

By the ABBV formula, 
$$\overline{\pi}_{M}^{!}(I_M^{*}(\varepsilon_p))=\pi_M^{!}(\varepsilon_p)=1=\pi_N^{!}(\varepsilon_{\widehat{f}(p)})=\overline{\pi}_N^{!}(I_N^{*}(\varepsilon_{\widehat{f}(p)})).$$
So 
$$\overline{\pi}_{M}^{!}(I_M^*(\varepsilon_{p}))=\overline{\pi}_N^{!}(I_N^{*}(\delta^f f(\varepsilon_p))=\delta^f \overline{\pi}_N^{!}(\overline{f}(I_M^*(\varepsilon_p))).$$
Therefore, by definition, whenever $f$ is orientation-preserving, we have 
$\delta^f=1$,
and when $f$ is orientation-reversing, we have $\delta^f=-1$.

\item  To analyze the impact of $\widehat{f}$ on the isolated fixed points,  
    we first note that a weak isomorphism must preserve the coefficient ring: it sends the constant class 
    $\mathbb{1}\in H^0_{S^1}(M;\Z)=\Z$ to the constant
$\mathbb{1}\in H^0_{S^1}(N;\Z)=\Z$. We now consider an isolated fixed point $p$.  By 
Fact~\ref{fact1}, its component Euler class
$\varepsilon_p$ has the property that $\iota_{p}^*(\varepsilon_p)= A_p\cdot {\pi^{*}(t)}^2$, where the integer coefficient $A_p$ 
is the product of the weights for the circle action $S^1\acts T_pM$. 
 If $p$ is an interior fixed point, then $A_p<0$, and if $p$ is extremal, $A_p>0$.
Using that $f(\mathbb{1})=\mathbb{1}$ and $f({\pi_M^{*}(t)}^2)={\pi_N^{*}(t)}^2$, we first compute
\begin{eqnarray*}
    f(\iota_{p}^*(\varepsilon_p)) & = & f( A_p\cdot {\pi_M^{*}(t)}^2) \\
    & = & f(A_p)\cdot f({\pi_M^{*}(t)}^2) \\ 
    & = & f(A_p)\cdot {\pi_N^{*}(t)}^2 \\
    & = & A_p \cdot {\pi_N^{*}(t)}^2.
\end{eqnarray*}
On the other hand, using the commutative diagram \eqref{diagramisorestrict1}, we also have
\begin{eqnarray*}
    f(\iota_{p}^*(\varepsilon_p)) & = & \iota_{\widehat{f}(p)}^*(f(\varepsilon_p))  \\
    & = & \iota_{\widehat{f}(p)}^*(\delta^f\cdot \varepsilon_{\widehat{f}(p)}) \\ 
    & = & \delta^f\cdot \iota_{\widehat{f}(p)}^*(\varepsilon_{\widehat{f}(p)}) \\
    & = &  \delta^f\cdot A_{\widehat{f}(p)} \cdot {\pi_N^{*}(t)}^2.
\end{eqnarray*}
By the previous item, when $f$ is orientation-preserving, $\delta^f=1$ and we may conclude that $A_p = A_{\widehat{f}(p)}$, so $p$ is 
an interior (resp.\ extremal) point if and only if $\widehat{f}(p)$ is interior (resp.\ extremal).  
When $f$ is orientation-reversing, $\delta^f=-1$ and we may conclude that $A_p = -A_{\widehat{f}(p)}$, so $p$ is 
an interior (resp.\ extremal) point if and only if $\widehat{f}(p)$ is  extremal (resp.\ interior). 
This establishes the desired conclusion that when $f$ is orientation-preserving
(orientation-reversing), $\widehat{f}$ will preserve (interchange) the interior and extremal labels.

\item  The last item follows immediately from 
\eqref{eq:fz2} and Lemma \ref{lem:inter1}.
\end{enumerate}
This completes the proof.
 \end{proof}

To further characterize an orientation-preserving, weak algebra isomorphism, we also need the following lemma, which is in the spirit of Lemma \ref{cl:ann2}.

\begin{Lemma}\label{lem:ann3}
When $\fat=2$,  a class $\eta \in H_{S^1}^{2}(M)$ is in $\sspan_{\Z}(\tau_h)$ if and only if
\begin{equation} \label{eq:tauh2}
 Z_0(\eta) =Z_0,  \mbox{ and } Z_1(\eta)=Z_1. 
 \end{equation}
When $\fat=1$,  a class $\eta \in H_{S^1}^{2}(M)$ is in  $\sspan_{\Z}(\tau_h) \smallsetminus \{0\}$ if and only if
\begin{equation} \label{eq:tauh1}
Z_0(\eta) =Z_0 
\, , \, |Z_1(\eta)|=\iso-1 \, ,  Z_2(\eta)=\emptyset\, , \mbox{ and }
\eta \cdot \eta >0. 
\end{equation}
\end{Lemma}

\begin{proof}
 The case $\fat=2$ follows from Table \ref{Table-ann-2-1a}.
 
 \vskip 0.1in
 
\noindent When $\fat=1$, a class $\eta$ listed in item  (3) of Lemma \ref{cl:ann2} has a positive self intersection only if it is a 
non-zero integer multiple of either $\sigma_{i,j}$ or $\sum_{\alpha}^{\beta}{m_{i,j} \sigma_{i,j}}$ with $\alpha=1$ 
or  $\beta=\ell_i$ (or both), by Table \ref{Table-zl-2}.
If it is also true that $Z_{2}(\eta)=\emptyset$, i.e., $\eta \cup \tau_\infty \neq 0$ we must have $\eta=\gamma \sum_{1}^{\ell_i}{m_{i,j}\sigma_{i,j}}=\gamma \tau_h$. Note that if $\fat=1$, we have $\sigma_{i,\ell_i} \cdot \sigma_{i,\ell_i} >0$ only if $\ell_i=1$; in this case, $m_{i,\ell_i}=1$. 
 \end{proof}

\begin{Corollary}\label{cl:tauh}
Let $M$ and $N$ be compact, connected, four-dimensional 
Hamiltonian $S^1$-manifolds and let 
$f \colon H_{S^1}^{2*}(M) \to H_{S^1}^{2*}(N)$ be an
orientation-preserving, weak algebra isomorphism.
Then $f$ maps 
the tuple $(\tau^M_0,\tau^M_{\infty},\tau^M_h,\pi_M^{*}(t))$
to one of the following:
$$
\begin{cases}
 (\tau^N_0,\tau^N_{\infty},\tau^N_h,\pi_N^{*}(t));\\
  (-\tau^N_\infty,-\tau^N_{0},-\tau^N_h,\pi_N^{*}(t)); \text{ or }\\
   (-\tau^N_0,-\tau^N_{\infty},-\tau^N_h,-\pi_N^{*}(t));\\
   (\tau^N_\infty,\tau^N_{0},\tau^N_h,-\pi_N^{*}(t))
         \end{cases}.
      $$  
If $f$ is a non-weak isomorphism,
one of the first two options holds 
and consequently
$f(\varepsilon_F) = \varepsilon_{\widehat{f}(F)}$ for fixed surfaces.
If $f$ is a strictly weak isomorphism, one
of the last two options holds and consequently
$f(\varepsilon_F) = -\varepsilon_{\widehat{f}(F)}$ for fixed surfaces. 

\end{Corollary}

\begin{proof}
Without loss of generality, up to composing with the strictly weak orientation-preserving isomorphism $H_{S^1}^{2*}(M) \to H_{S^1}^{2*}(M)$ that maps $\eta \to -\eta$ for every $\eta \in H_{S^1}^{2}(M)$, we may assume that $f(\pi_M^{*}(t))=\pi_N^{*}(t)$.   We now set out to prove
that $f$ sends the tuple $(\tau^M_0,\tau^M_{\infty},\tau^M_h,\pi_M^{*}(t))$
to one of the following:
\begin{equation}\labell{eq:all-plus-or-all-minus}
    \begin{cases}
 (\tau^N_0,\tau^N_{\infty},\tau^N_h,\pi_N^{*}(t)); \text{ or } \\
  (-\tau^N_\infty,-\tau^N_{0},-\tau^N_h,\pi_N^{*}(t)). \\
         \end{cases}.
\end{equation}    

We achieve this by examining the three cases, $\fat=2$, $\fat=1$, and $\fat=0$, in that order.
We first assume $\fat=2$. Then
    \begin{eqnarray*}
\eta \in \sspan_{\Z}(\tau^M_h) &\Leftrightarrow& Z^M_i(\eta) \bigcup -Z^M_i(\eta)=Z^M_i \bigcup -Z^M_i \text{ for }i=0,1 \\
&\Leftrightarrow& Z^N_i(f(\eta)) \bigcup -Z^N_i(f(\eta))=Z^N_i \bigcup -Z^N_i \text{ for }i=0,1 \\
&\Leftrightarrow & f(\eta) \in \sspan_{\Z}(\tau^N_h).
    \end{eqnarray*}
   The first and third equivalences follow from Lemma \ref{lem:ann3}, and the middle one is by 
  \eqref{eq:fz01}
   applied to $f$ and to $f^{-1}$.
    So
    $$f(\sspan_{\Z}(\tau^M_h))=\sspan_{\Z}(\tau^N_h).$$ Hence $f$ sends the generators $\pm \tau^M_h$ to the generators $\pm \tau^N_h$. Let $\delta_h \in \{-1,1\}$ be such that  $f(\tau^M_h)=\delta_h \tau^N_h$.

    Now, by   \eqref{eq:fz2}, for $*=0,\infty,(0,\infty)$, there are values 
    $\delta_* \in \{-1,1\}$ such that:
    \begin{itemize}
        \item[$\bullet$] either $\delta_{(0,\infty)}=1$ and 
    $f(\tau^M_0)=\delta_0 \tau^N_0, \, f(\tau^M_\infty)=\delta_\infty \tau^N_\infty$; 
    \item[$\bullet$] or $\delta_{(0,\infty)}=-1$ and 
    $f(\tau^M_0)=\delta_0 \tau^N_\infty, \, f(\tau^M_\infty)=\delta_\infty \tau^N_0$.
    \end{itemize}
    By Lemma \ref{lem:inter1}, since $f$ is orientation-preserving it preserves the intersection form; by calculation, $\tau_h \cdot \tau_0=1=\tau_h \cdot \tau_\infty$.  Therefore
    \begin{equation}\label{eq:h0}
    \delta_h \delta_0=1=\delta_h \delta_\infty.
    \end{equation}

 By Corollary \ref{cor:pit}, 
 $$\tau_h \cup \pi^{*}(t) =\tau_h \cup (\tau_\infty-\tau_0).$$
 Hence
 \begin{eqnarray*}
  \delta_h  (\tau^N_h \cup (\tau^N_\infty-\tau^N_0))&=&
\delta_h \tau^N_h \cup \pi_{N}^{*}(t)\\ 
&=&f(\tau^M_h \cup \pi_M^{*}(t))\\
&=&f(\tau^M_h \cup (\tau^M_\infty -\tau^M_0))\\
& = &\delta_h \tau^N_h \cup \delta_{(0,\infty)}(\delta_\infty \tau^N_\infty-\delta_0 \tau^N_0).
   \end{eqnarray*}
   Since $\tau^N_h \cup \tau^N_\infty$ and $\tau^N_h \cup \tau^N_0$ are linearly independent, we deduce that 
   \begin{equation}\label{eq:h1}
       \delta_{(0,\infty)} \delta_\infty=1=\delta_{(0,\infty)} \delta_0.
   \end{equation}
   Combining \eqref{eq:h0} and \eqref{eq:h1}, we get that either $\delta_{(0,\infty)}=\delta_h=\delta_0=\delta_\infty=1$  and we have the first
   possibility in \eqref{eq:all-plus-or-all-minus}; or we get that $\delta_{(0,\infty)}=\delta_h=\delta_0=\delta_\infty=-1$  and we have the second
   possibility in \eqref{eq:all-plus-or-all-minus}.

Next, we turn to the case $\fat=1$. 
 Then
    \begin{eqnarray*}
\eta \in \sspan_{\Z}(\tau^M_h) &\Leftrightarrow& Z^M_0(\eta) \bigcup -Z^M_0(\eta)=Z^M_0 \bigcup -Z^M_0,\\
& & |Z^M_{1}(\eta)|=\iso+\fat-2,\\
& & Z^M_2(\eta)=\emptyset, \text{ and }\eta {\cdot}\eta >0  \\
&\Leftrightarrow& 
Z^N_0(f(\eta)) \bigcup -Z^N_0(f(\eta))=Z^N_0 \bigcup -Z^N_0,  \\
& & |Z^N_{1}(f(\eta))|=\iso+\fat-2,  \\
& & Z^N_2(f(\eta))=\emptyset, \text{ and }f(\eta) {\cdot}f(\eta) >0   \\
&\Leftrightarrow & f(\eta) \in \sspan_{\Z}(\tau^N_h).
    \end{eqnarray*}
    The first and third equivalences are by Lemma \ref{lem:ann3}, and the middle one is by  \eqref{eq:fz01} and \eqref{eq:fz22} applied to $f$ and to $f^{-1}$, and by the assumption that $f$ is orientation preserving and Lemma \ref{lem:inter1}.
    So
    $$f(\sspan_{\Z}(\tau^M_h))=\sspan_{\Z}(\tau^N_h).$$ Hence $f$ sends the generators $\pm \tau^M_h$ to the generators $\pm \tau^N_h$. Let $\delta_h \in \{-1,1\}$ be such that  $f(\tau^M_h)=\delta_h \tau^N_h$.

    Without loss of generality, we assume that the fixed surface in $M$ is the maximum, so $\tau^M_{\infty} \neq 0$. By  \eqref{eq:fz2} and Corollary \ref{cor:opor},
    for $*=\infty,(0,\infty)$, there are values $\delta_* \in \{-1,1\}$ such that
    \begin{itemize}
        \item  either $\delta_{(0,\infty)}=1$ and $f(\tau^M_\infty)=\delta_\infty \tau^N_\infty, \, f(\varepsilon^M_{v_0})=\varepsilon^N_{v_0}$,
        \item or $\delta_{(0,\infty)}=-1$ and 
    $f(\tau^M_\infty)=\delta_\infty \tau^N_0, \,f(\varepsilon^M_{v_0})= \varepsilon^N_{v_\infty}$.
    \end{itemize}
By Lemma \ref{lem:inter1} and since $\tau^N_h \cdot \tau^N_{*}=1$ for $*\in \{0,\infty\}$ 
such that $\tau^N_* \neq 0$, we get that 
\begin{equation} \label{eq:h01}
\delta_h \delta_\infty=1.
\end{equation} 

    By Corollary \ref{cor:pit}, 
 $$\tau^N_{h} \cup \pi^{*}_ N(t)=\begin{cases}
 \tau^N_h \cup \tau^N_\infty -\varepsilon^N_{v_0} & \text{ if } \tau^N_\infty \neq 0 \\
\varepsilon^N_{v_\infty}- \tau^N_h \cup \tau^N_0 & \text{ if } \tau^N_0 \neq 0.
\end{cases}
$$
Applying the isomorphism $f$ to both sides of 
$$\tau^M_{h} \cup \pi_M^{*}(t)=
 \tau^M_h \cup \tau^M_\infty -\varepsilon^M_{v_0},$$  we get that 
\begin{eqnarray*}
\delta_h \tau^N_h \cup \pi_N^{*}(t)&=&f(\tau^M_h) \cup \pi_N^{*}(t)\\
&=&f(\tau^M_h \cup \pi_M^{*}(t))\\
&=&f(\tau^M_h \cup \tau^M_\infty-\varepsilon^M_{v_0}).
\end{eqnarray*}
and either $\delta_{(0,\infty)}=1$ and the last term equals $\delta_h \tau^N_h \cup \delta_{\infty}\tau^N_\infty -\varepsilon^M_{v_0}$ or $\delta_{(0,\infty)}=-1$ and the last term equals $\delta_h \tau^N_h \cup \delta_\infty \tau^N_0 -\varepsilon^N_{v_\infty}$. Using \eqref{eq:h01}, we conclude 
$$
    \delta_{(0,\infty)}=\delta_{\infty}=\delta_h.
$$    
So either they are all $1$  and we have the first
   possibility in \eqref{eq:all-plus-or-all-minus}; or they are all $-1$
   and we have the second
   possibility in \eqref{eq:all-plus-or-all-minus}.

Finally, we consider the case $\fat=0$.  In this case, $\tau_0=\tau_\infty=0$ in $M$ and $N$
and we only need to check that $\tau^M_h$ is sent to $\pm \tau_h^N$.
    By Corollary \ref{cor:pit}, 
$\tau^M_h \cup \pi_M^{*}(t)=\varepsilon^M_{v_\infty}-\varepsilon^M_{v_0}$.
By Corollary \ref{cor:opor}, since $f$ is orientation preserving, $\{f(\varepsilon^M_{v_\infty}),f(\varepsilon^M_{v_0})\}=\{\varepsilon^N_{v_\infty},\varepsilon^N_{v_0}\}$.
So
\begin{eqnarray*}
f(\tau^M_h) \cup \pi_N^{*}(t)&=&f(\tau^M_h \cup \pi_M^{*}(t))\\
&=&f(\varepsilon^M_{v_\infty}-\varepsilon^M_{v_0})
= \pm (\varepsilon^N_{v_\infty}-\varepsilon^N_{v_0})\\&=& \pm (\tau^N_h \cup \pi_N^{*}(t))=\pm \tau^N_h \cup \pi_N^{*}(t).
\end{eqnarray*}
Since $H_{S^1}^{2*}(N)$ is a free module over $H_{S^1}^{2*}(\pt)$ (see \cite[Theorem~1.1(A)]{KesslerHolm1}), we deduce that $f(\tau^M_h)=\pm \tau^N_h$. This completes the proof.
\end{proof}

\section{Relating Weak  algebra isomorphisms and  isomorphisms of dull graphs} \labell{sec5b}

In this section, we list explicitly the algebra isomorphisms of the even part of 
the equivariant cohomology
of a compact, connected, four-dimensional Hamiltonian $S^1$-manifold.
We show that orientation-preserving isomorphisms are induced by isomorphisms of dull
graphs.  We determine which isomorphisms preserve the equivariant first Chern class.
The section concludes with a detailed examination of an example of manifolds whose decorated graphs
differ by a chain flip.

Recall that a ring isomorphism $f \colon H_{S^1}^{2*}(M) \to H_{S^1}^{2*}(N)$ 
is a weak isomorphism of algebras over $H_{S^1}^{*}(\pt)$ if there is 
an automorphism $\gamma$ of $S^1$ such that $f(uw)=\gamma^{*}(u)f(w)$ for any $u \in H_{S^1}^{*}(\pt)$ and $w \in H_{S^1}^{2*}(M)$.
Note that the circle admits only two automorphisms: the trivial one and the non-trivial one. 
If $\gamma$ is the identity automorphism then $f$ is an \emph{isomorphism of algebras}.
On the other hand, if $\gamma$ is the non-trivial automorphism
then $f$ is a \emph{strictly weak isomorphism}.
Note that composing a strictly weak isomorphism $H_{S^1}^{2*}(M) \to H_{S^1}^{2*}(N)$ with a strictly weak isomorphism 
$H_{S^1}^{2*}(M) \to H_{S^1}^{2*}(M)$ gives an isomorphism.

We now describe several  weak isomorphisms of the equivariant cohomology
$ H_{S^1}^{2*}(M)$, as an $H_{S^1}^*(pt)$-algebra, that result from  
shuffling the data of the $S^1$-action, the symplectic
form, and the moment map.

\begin{Example}\labell{perm}
In Notation~\ref{not:gens}, we fixed an ordering of the chains $C_1,\dots,C_k$ so that the labels of the bottom
edges $m_{1,1},\dots,m_{k,1}$ satisfy $m_{1,1}\geq m_{2,1}\geq\cdots\geq m_{k,1}\geq 1$. 
Whenever two chains have the same length, $\ell_i=\ell_j$, and same labels $m_{i,s}=m_{j,s}$ for all $1\leq s\leq \ell_i$,
we say that the
two such chains $C_i$ and $C_j$ are {\bf isomorphic}.
Note that these chains are isomorphic in the dull graph.  Their moment map labels could differ, so they
may not be identical in the extended decorated graph (with respect to a generic metric).
We may permute isomorphic chains when fixing the data that determine our generators-and-relations presentation.
Let $\mathfrak{S}_M$ denote the subgroup of the permutation group $\mathfrak{S}_k$ corresponding to permutations of
the isomorphic chains.  Note that $\mathfrak{S}_M$ is a Young subgroup of $\mathfrak{S}_k$.
For any $\varrho\in\mathfrak{S}_M$, there is an induced map  $H_{S^1}^{2*}(M) \to H_{S^1}^{2*}(M)$ that sends
$$
\begin{array}{ll}
\tau_0\mapsto\tau_{0};& \\ 
\tau_{\infty}\mapsto\tau_{\infty}; & \text{ and }\\
  \sigma_{i,j}\mapsto \sigma_{\varrho(i),j} & \text{ for every } i \text{ and }1\leq j \leq \ell_i.
  \end{array}
  $$
  Such a map sends $\pi^{*}(t)$ to $\pi^{*}(t)$.
  This yields an orientation-preserving
isomorphism $f_{\varrho}$ of the equivariant cohomology ring as an algebras that we call a \emph{permutation}.
  It induces a map on the fixed point set that sends
\begin{equation} \labell{permutations-fixed}
\begin{array}{rll}
\Sigma_0 & \mapsto\  \Sigma_{0}\ ; & \\
\Sigma_{\infty} & \mapsto\  \Sigma_\infty\ ; & \\
v_0 & \mapsto\   v_{0}\ ; & \\
v_{\infty}&  \mapsto\   v_\infty\ ; & \text{ and }\\
 v_{i,j} & \mapsto\   v_{\varrho(i),j} &  \forall\ 1 \leq j < \ell_i.
 \end{array}
\end{equation}
A permutation of isomorphic chains is induced by an equivariant diffeomorphism of the
underlying manifold.  To whit, if the chains are isomorphic but not identical in the extended decorated
graph, there is a multi-rescaling of the graph so that the chains become identical.  By Proposition~\ref{scale1},
the multi-rescaling is induced by an equivariant diffeomorphism.  Karshon's uniqueness result \cite[Theorem~4.1]{karshon}
implies that there is an equivariant symplectomorphism inducing the permutation of identical chains in the extended decorated graph.
Composing these gives an equivariant diffeomorphism inducing the permutation of isomorphic chains. 
This diffeomorphism preserves the orientation induced by the symplectic form and an $S^1$-invariant,
compatible almost-complex structure.
 \hfill $\diamondsuit$
\end{Example}

\begin{Example}\labell{symp-flip}
Replacing $\omega$ with $-\omega$, without changing the circle action, 
sends the Hamiltonian manifold $S^1 \acts (M,\omega)$ with moment map $\Phi$ 
to a Hamiltonian manifold $S^1 \acts (M,-\omega)$ with moment map $-\Phi$.
These two manifolds have different decorated graphs: the 
decorated graph for $S^1 \acts (M,-\omega)$ is 
obtained by flipping the decorated graph of $S^1 \acts (M,\omega)$.
These two graphs give
different cohomology presentations.
We will let $\tau_*$ and $\sigma_{*,*}$ denote the generators defined using the graph
for $S^1 \acts (M,\omega)$; 
and $\widetilde\tau_*$ and $\widetilde\sigma_{*,*}$ denote the generators 
for the flipped graph.
The induced map  $f_\omega: H_{S^1}^{2*}(M) \to H_{S^1}^{2*}(M)$ 
is an isomorphism that sends 
\begin{equation} \labell{eq:sympflip}
\begin{array}{ll}
\tau_0 \mapsto -\widetilde\tau_{\infty}; & \\
\tau_{\infty} \mapsto -\widetilde\tau_0; & \text{ and }\\
 \sigma_{i,j} \mapsto -\widetilde\sigma_{i,\ell_i-j+1} &  \forall\ 1\leq j \leq \ell_i.
 \end{array}
\end{equation}
 The induced 
map on $H_{S^1}^{*}(\pt)$ is the identity map. 
This yields an orientation-preserving isomorphism of the equivariant cohomology rings as algebras that we call the \emph{symplectic flip}. 
This map induces the correspondence $\widehat{f}_\omega$ on the fixed set:
\begin{equation} \labell{sympflip-fixed}
\begin{array}{rll}
\Sigma_0 & \mapsto\  \widetilde\Sigma_{\infty}\ ; & \\
\Sigma_{\infty} & \mapsto\  \widetilde\Sigma_0\ ; & \\
v_0 & \mapsto\  \widetilde v_{\infty}\ ; & \\
v_{\infty}&  \mapsto\  \widetilde v_0\ ; & \text{ and }\\
 v_{i,j} & \mapsto\  \widetilde v_{i,\ell_i-j} &  \forall\ 1\leq j < \ell_i.
 \end{array}
\end{equation}
We have abused notation in \eqref{eq:sympflip} and \eqref{sympflip-fixed}: not each type of fixed set exists in
each type of fixed set ($\fat=0,1,2$): when one does not, the corresponding line is 
irrelevant in that case.
Because of the sign changes in \eqref{eq:sympflip} and the judicious choice of sign convention in 
Definition~\ref{def:comp euler}, the algebra isomorphism $f_\omega$ satisfies
$$
f_\omega(\varepsilon_F) = \varepsilon_{\widehat{f}_{\omega}(F)}.
$$
 \hfill $\diamondsuit$
\end{Example}

\begin{Example}\labell{action-flip}
The non-trivial automorphism of the circle sends $S^1 \acts (M,\omega)$ to a Hamiltonian circle action whose decorated graph is 
obtained by flipping the decorated graph of $S^1 \acts (M,\omega)$ (together with a vertical shift of the moment map).
As above we will let
$\tau_*$, $\sigma_{*,*}$
and $\widetilde\tau_*$, $\widetilde\sigma_{*,*}$ denote the two different 
sets of generators corresponding to the two graphs.
It
induces  
a map $f_{\acts}:H_{S^1}^{2*}(M) \to H_{S^1}^{2*}(M)$ 
that sends 
$$
\begin{array}{ll}
\tau_0\mapsto\widetilde\tau_{\infty};& \\ 
\tau_{\infty}\mapsto\widetilde\tau_0; & \text{ and }\\
  \sigma_{i,j}\mapsto \widetilde\sigma_{i,\ell_i-j+1} & \text{ for every } i \text{ and }1\leq j \leq \ell_i.
  \end{array}
  $$
 The induced 
map $H_{S^1}^{*}(\pt) \to H_{S^1}^{*}(\pt)$ is the automorphism $
t \mapsto -t$.
We get an orientation-preserving strictly weak isomorphism of the equivariant cohomology rings as algebras; we call it the \emph{action flip}.  Note that $\widehat{f}_{\acts}= \widehat{f}_\omega$.  \hfill $\diamondsuit$
\end{Example}

\begin{Example}
Composing the action flip with the symplectic flip gives
a map  
$$
{f}_{\acts}\circ {f}_\omega =  {f}_\omega\circ {f}_{\acts}: H_{S^1}^{2*}(M) \to H_{S^1}^{2*}(M)
$$ 
that sends 
\begin{equation} \labell{minusid}
\begin{array}{ll}
\tau_0 \mapsto -\tau_{0}; & \\
\tau_{\infty} \mapsto -\tau_\infty; & \text{ and }\\
 \sigma_{i,j} \mapsto -\sigma_{i,j} &  \forall\ 1\leq j \leq \ell_i,
 \end{array}
\end{equation}
and $\pi^{*}(t)$ to $-\pi^{*}(t)$.
This map is an orientation-preserving strictly weak isomorphism of the equivariant cohomology ring as an algebra,
which negates classes in $H^2_{S^1}(M)$.
This map induces the identity $\widehat{{f}_{\acts}\circ {f}_\omega } = \id$ on 
the fixed set and more generally on the decorated graph (up to a vertical
translation).  \hfill $\diamondsuit$
\end{Example}

Given a strictly weak isomorphism of algebras, we may compose with 
${f}_{\acts}\circ {f}_\omega$ to produce a (non-weak) isomorphism.  In the rest of the
section, we focus on non-weak isomorphisms.
To give a complete characterization of 
isomorphisms 
between equivariant cohomology rings 
(as algebras), we need one more isomorphism.
This isomorphism appears when the extended decorated graphs
of two Hamiltonian $S^1$-manifolds differ by turning 
a single chain upside down.
An example is illustrated in Figure~\ref{fig:chain flip}, where
the two extended decorated graphs differ in their leftmost chain.

Let $M$ and $N$ be compact, connected, Hamiltonian $S^1$-manifolds and 
consider the associated extended decorated graphs with 
respect to generic compatible K\"ahler metrics.
Suppose that the extended decorated graphs of $M$ and $N$
differ by a single flip of a chain that begins and ends
with the label $1$.  
We assume that the flipped chain is the $i^{\mathrm{th}}$
chain in $M$ and in $N$.
Then we have the following:
\begin{enumerate}
\item The number and lengths of chains for $M$ and $N$ agree: 
$k_M=k_N$ and $\ell^M_r=\ell^N_r$ for each $1\leq r \leq k_M$.

\item The $i^{\mathrm{th}}$ chain satisfies $\ell_{i}>1$, and
 \begin{equation} \labell{prd}
m^N_{i,1}=m^N_{i,\ell_{i}}=m^M_{i,\ell_{i}}= m^M_{i,1}=1.
 \end{equation}
 \item For $1<j<\ell_{i}$,
\begin{equation} \labell{eqmij}
m^N_{i,\ell_{i}-j+1}=m^M_{i,j}.
\end{equation}
\item For all $r \neq i$, we have  $$m^M_{r,j}=m^N_{r,j} \text{ for all } 1 \leq j \leq \ell_r.$$
\item If $\tau_0^M \neq 0 \,\, (\tau_{\infty}^M \neq 0)$ then $\tau_0^N \neq 0 \,\,(\tau_\infty^N \neq 0)$, 
the genus of $\Sigma_0^M \, (\Sigma_{\infty}^M)$ equals the genus of $\Sigma_0^N \, (\Sigma_{\infty}^N)$,
 and the self intersection $e_{\min}^M \, (e_{\max}^M)$ equals the self intersection $e_{\min}^N \, (e_{\max}^N)$.
\end{enumerate}

\begin{Definition} \labell{defpartial} 
For the manifolds $M$ and $N$ as just described,
the \emph{chain flip of the $i^{\mathrm{th}}$ chain}
is the map  defined on the generators of $H_{S^1}^{2*}(M)$ and mapping
to  $H_{S^1}^{2*}(N)$ 
that sends
$$
\tau^M_0\mapsto\tau^N_{0}\ ;\  \tau^M_{\infty}\mapsto\tau^N_{\infty} \ ; \
\sigma^M_{r,j}\mapsto \sigma^N_{r,j} \mbox{ for all } r\neq i \ ;
$$
and on the generators corresponding to the $i^{\mathrm{th}}$ chain has the effect
\begin{eqnarray}
\sigma^M_{i,1} &\mapsto &\displaystyle{\sum_{s=1}^{\ell_{i}-1}m^N_{i,s}\sigma^N_{i,s}\ } ; \labell{pra}\\
\sigma^M_{i,\ell_i} &\mapsto &\displaystyle{\sum_{s=2}^{\ell_{i}}m^N_{i,s}\sigma^N_{i,s}\ } ; \mbox{ and }  \labell{prb}\\
\sigma^M_{i,j} & \mapsto & -\sigma^N_{i,\ell_i-j+1} \mbox{ for } 1<j<\ell_i. \labell{prc}
\end{eqnarray}
\end{Definition}

We now verify that the map in Definition~\ref{defpartial} extends to 
an isomorphism from $H^{2*}_{S^1}(M)$ to $H_{S^1}^{2*}(N)$.

\begin{Proposition} \labell{rem:iso}
The chain flip of the $i^{\mathrm{th}}$ chain
induces a well-defined, orientation-preserving isomorphism  of algebras $f_\chi \colon H_{S^1}^{2*}(M)\to H_{S^1}^{2*}(N)$.  Moreover, $\widehat{f}_\chi$ is the following correspondence on the fixed set
\begin{equation} \labell{chainflip-fixed}
\renewcommand{\arraystretch}{1.5}
\begin{array}{rll}
\Sigma^M_0 & \mapsto\  \Sigma^N_{0}\ ; & \\
\Sigma_{\infty}^M & \mapsto\  \Sigma_{\infty}^N \ ; & \\
v_0^M & \mapsto\  v_0^N\ ; & \\
v_{\infty}^M &  \mapsto\  v_{\infty}^N \ ; &\\
 v_{i,j}^M & \mapsto\   v_{i,\ell_i-j}^N &  \forall\ 1\leq j < \ell_i \  ; \text{ and }\\
 v_{k,j}^M & \mapsto\ v_{k,j}^N & \forall\ 1\leq j < \ell_k\ . 
 \end{array}
 \renewcommand{\arraystretch}{1}
\end{equation}

\end{Proposition}

We also call the induced isomorphism $f$ of algebras 
a \emph{chain flip}
of the $i^{\mathrm{th}}$ chain.

\begin{proof}
By Theorem \ref{ThNew},  and the definition of a chain flip, 
to verify that this is a well-defined isomorphism,
it is enough to show that
\begin{equation} \labell{eqchain}
f_\chi\left(\sum_{j=1}^{\ell_{i}}m^M_{i,j}\sigma^M_{i,j}\right)=\sum_{j=1}^{\ell_{i}}m^N_{i,j}\sigma^N_{i,j}
 \quad \mbox{ and } \quad
f_\chi\left(\sum_{j=1}^{\ell_{i}}b^M_{i,j}\sigma^M_{i,j}\right)=\sum_{j=1}^{\ell_{i}}b^N_{i,j}\sigma^N_{i,j}.
\end{equation}
We have
\begin{eqnarray*}
f_\chi\left(\sum_{j=1}^{\ell_{i}}m^M_{i,j}\sigma^M_{i,j}\right)&=&\sum_{j=1}^{\ell_{i}}
m^M_{i,j}f_\chi\left(\sigma^M_{i,j}\right)\\
&=&\sum_{r=1}^{\ell_i-1}m^N_{i,r}\sigma^N_{i,r}-\sum_{j=2}^{\ell_i-1}m^M_{i,j}\sigma^N_{i,\ell_i-j+1}+\sum_{r=2}^{\ell_i}m^N_{i,r}\sigma^N_{i,r}\\
&=& \sum_{r=1}^{\ell_i-1}m^N_{i,r}\sigma^N_{i,r}-\sum_{j=2}^{\ell_i-1}m^N_{i,\ell_i-j+1}\sigma^N_{i,\ell_i-j+1}+\sum_{r=2}^{\ell_i}m^N_{i,r}\sigma^N_{i,r}\\
&=&\sum_{j=1}^{\ell_i}m^N_{i,j}\sigma^N_{i,j}.
\end{eqnarray*}
The first equality is by definition, the second by \eqref{pra}, \eqref{prb}, \eqref{prc}, and \eqref{prd}, and the third by \eqref{eqmij}.

Similarly, using the convention $b_{i,1}=0, \, b_{i,\ell_i}=1$, justified in Lemma \ref{lem:comb},
\begin{eqnarray*}
f_\chi\left(\sum_{j=1}^{\ell_{i}}b^M_{i,j}\sigma^M_{i,j}\right)&=&b^M_{i,1}\sum_{r=1}^{\ell_i-1}m^N_{i,r}\sigma^N_{i,r}-\sum_{j=2}^{\ell_i-1}b^M_{i,j}\sigma^N_{i,\ell_i-j+1}+b^M_{i,\ell_i}\sum_{r=2}^{\ell_i}m^N_{i,r}\sigma^N_{i,r}\\
&=&-\sum_{j=2}^{\ell_i-1}b^M_{i,\ell_i-j+1}\sigma^N_{i,j}+\sum_{r=2}^{\ell_i}m^N_{i,r}\sigma^N_{i,r}\\
&=&\sum_{j=2}^{\ell_i}(m^N_{i,j}-b^M_{i,\ell_i-j+1})\sigma^N_{i,j}.
\end{eqnarray*}
Set $b^N_{i,1}=0$; note that $b^N_{i,1}=0=m^N_{i,1}-b^M_{i,\ell_i-1+1}$ since $m^N_{i,1}=1=b^M_{i,\ell_i}$.
We claim that $$b^N_{i,j}=m^N_{i,j}-b^M_{i,\ell_i-j+1} \, \, \text{ for }1 < j \leq \ell_i.$$
By Lemma \ref{lem:comb}, it is enough to show that 
$(m^N_{i,j}-b^M_{i,\ell_i-j+1})m^N_{i,j-1}-(m^N_{i,j-1}-b^M_{i,\ell_i-j+2})m^N_{i,j}=1$ for $1<j \leq \ell_i$.
Indeed, 
\begin{eqnarray*}
(m^N_{i,j}-b^M_{i,\ell_i-j+1})m^N_{i,j-1}-(m^N_{i,j-1}-b^M_{i,\ell_i-j+2})m^N_{i,j} 
&=&m^N_{i,j}m^N_{i,j-1}-b^M_{i,\ell_i-j+1}m^N_{i,j-1}\\
&&-m^N_{i,j-1}m^N_{i,j} +b^M_{i,\ell_i-j+2}m^N_{i,j}\\
&=&b^M_{i,\ell_i-j+2}m^N_{i,j}-b^M_{i,\ell_i-j+1}m^N_{i,j-1}\\
&=&b^M_{i,\ell_i-j+2}m^M_{i,\ell_i-j+1}-b^M_{i,\ell_i-j+1}m^M_{i,\ell_i-j+2}\\
&=&1
\end{eqnarray*}
where the third equality is by  \eqref{eqmij}, and the fourth is by Lemma \ref{lem:comb}.

Finally, we verify that the map $f_\chi$ is orientation-preserving.  Recall the notation
that a map $g$ in equivariant cohomology restricts to a map $\overline{g}$ in ordinary cohomology.
Let  $r\neq i$.
 If $\tau^M_0\neq 0$ or $\tau_\infty^M\neq 0$, then we have
$$
\overline{\pi}_N^{!} \circ {\overline{f}_\chi} (\sigma^M_{r,1} \cup \tau_0^M) = 
\overline{\pi}_N^{!} (\sigma^N_{r,1}\cup \tau_0^N) = 1
=\overline{\pi}_M^{!}(\sigma^M_{r,1}\cup \tau_0^M)
$$
or
$$
\overline{\pi}_N^{!} \circ {\overline{f}_\chi} (\sigma^M_{r,\ell_r}\cup \tau_\infty^M) = 
\overline{\pi}_N^{!} (\sigma^N_{r,\ell_r}\cup \tau_\infty^N) = 1
=\overline{\pi}_M^{!}(\sigma^M_{r,\ell_r}\cup\tau_\infty^M).
$$
Otherwise, 
$$
\overline{\pi}_N^{!} \circ {\overline{f}_\chi} (\sigma^M_{r,1} \cup \sigma^M_{i,1}) = 
\overline{\pi}_N^{!} (\sigma^N_{r,1}\cup \displaystyle{\sum_{s=1}^{\ell_{i}-1}m^N_{i,s}\sigma^N_{i,s}}) = m^N_{i,1}=1=m^M_{i,1}
=\overline{\pi}_M^{!}(\sigma^M_{r,1}\cup \sigma^M_{i,1}).
$$
Since we must have either $\overline{\pi}_N^{!} \circ {\overline{f}\textbf{}}=\overline{\pi}_M^{!}$ or $\overline{\pi}_N^{!} \circ {\overline{f}\textbf{}}=-\overline{\pi}_M^{!}$, we deduce that $f$ is orientation-preserving.
\end{proof}

\begin{Proposition}\label{prop:newiso}
    Suppose that $f \colon H_{S^1}^{2*}(M) \to H_{S^1}^{2*}(N)$ is an orientation-preserving isomorphism of  $H_{S^1}^{*}(\pt)$-algebras.
    Then, after possibly precomposing with a
symplectic flip, the isomorphism $f$ is a composition of finitely many permutations of isomorphic chains and finitely many 
 chain flips.
    \end{Proposition}

\begin{proof}
 By Corollary \ref{cl:tauh}, 
 we can assume that, up to  precomposing with the symplectic flip,
$$f(\tau^M_h)=\tau^N_h, f(\tau^M_0)=\tau^N_0, \, f(\tau^M_\infty)=\tau^N_\infty, \, \widehat{f}({\min^M})={\min^N}, \, \widehat{f}({\max^M})={\max^N}.$$

We claim that for any chain $i^*$ in the extended decorated graph of $S^1 \acts (M,\omega_M)$, there exists a chain $i$ in the extended decorated graph of $S^1 \acts (N,\omega_N)$ such that
either
\begin{itemize}
\item[a.] we have $$f(\sigma^M_{i^*,j})=\sigma^N_{i,j}, \text{ and } m^N_{i,j}=m^M_{i^*,j} \text{ for all } 1 \leq j \leq \ell_{i^*},$$
$$ \widehat{f}(v^M_{i^*,j})=v^N_{i,j}
\text{ for all } 1 \leq j < \ell_{i^*},$$
$$\ell_i=\ell_{i^*};
$$ 
{\bf or}
\item[b.] we have 
$$f(\sigma^M_{i^*,1})=\sum_{j=1}^{\ell_{i}-1}m^N_{i,j}\sigma^N_{i,j}, \mbox{ and } f(\sigma^M_{i^{*},\ell_{i^*}})=\sum_{j=2}^{\ell_i}m^N_{i,j} \sigma^N_{i,j},
$$

$$
f(\sigma^M_{i^*,j})=-\sigma^N_{i,\ell_i-j+1} \text{ for all }1<j<\ell_{i^*}, $$ 
$$\widehat{f}(v^M_{i^*,j})=v^N_{i,\ell_i-j+1} \text{ for all }1\leq j<\ell_{i^*}, $$
$$m^N_{i,1}=m^M_{i^*,1}=1=m^M_{i^*,\ell_{i^*}}=m^N_{i,\ell_i} \text{ and } m^N_{i,j}=m^M_{i^*,\ell_i-j+1} \text{ for all }1<j<\ell_{i^*},$$
$$\ell_i=\ell_{i^*}>1.$$
\end{itemize}

Since $f$ is an isomorphism, sending  $i^* \mapsto i$ gives a 
bijection on the set of chains.
We prove the claim by a
case-by-case analysis of the possible images of the generators.
We first assume $\ell_{i^*}>1$. 
 By Proposition \ref{pro:ann} and item (1) of Lemma  \ref{cl:ann2}, the possible options for the image $f(\sigma^M_{i^*,j})$  are the classes in $H_{S^1}^{2}(N)$ listed in item (3) of Lemma \ref{cl:ann2}. The assumption that $\widehat{f}(\min^M)=\min^N$ and $\widehat{f}(\max^M)=\max^N$  implies, by Corollary \ref{cor:Euler}, that
$$
{f(\sigma^M_{i^*,j})}\big|_{\min^N}=0\    \Longleftrightarrow \
{\sigma^M_{i^*,j}}\big|_{\min^M}=0 \text{ and }{f(\sigma^M_{i^*,j})}\big|_{\max^N}=0\ \Longleftrightarrow \
{\sigma^M_{i^*,j}}\big|_{\max^M}=0.
$$
Since ${\sigma^M_{i^*,1}}|_{\min^M} \neq 0$ and ${\sigma^M_{i^*,1}}|_{\max^M} = 0$, some of the options for $f(\sigma^M_{i^*,1})$ are eliminated. Explicitly, 
$$
f(\sigma^M_{i^*,1})=\begin{cases} \delta \sigma^N_{i,1} & \text{ or} \\ \delta \sum_{j=1}^{\beta}m^N_{i,j}\sigma^N_{i,j} & \text{ with }\beta<\ell_i \end{cases}.$$
By assumption, $f$ preserves the cup product. By Lemma~\ref{lem:inter1}, $f$ also preserves the intersection form on classes in $H_{S^1}^{2}$.
Hence, if $\tau^N_0 \neq 0$, then $\tau^N_h \cup \tau^N_0 \neq 0$, and
\begin{eqnarray*}
\delta \tau^N_h \cup \tau^N_0&=& \delta \sigma^{N}_{i,1} \cup \tau^N_0=f(\sigma^M_{i^*,1}) \cup \tau^N_0 
\\&=&f(\sigma^M_{i^*,1} \cup \tau^M_0)=f(\tau^M_{h} \cup \tau^M_0)=f(\tau^M_h) \cup f(\tau^M_0)\\
&=&\tau^N_h \cup \tau^N_0,
\end{eqnarray*}
whence $\delta=1$.
If $\tau^N_0=0$, then
$f(\sigma^M_{i^*,1}) \cdot \tau^N_h =\sigma^M_{i^*,1} \cdot \tau^M_h> 0$, so $1 \leq \delta \in \N$.
Moreover, by Corollary \ref{cor:Euler}, the induced map $\widehat{f}$ is a bijection that maps a component in the support of a class (in $H^{2}_{S^1}(M)$ or in $H^{4}_{S^1}(M)$) to a component in the support of its image.
We note that for two classes $\eta_1,\eta_2 \in H_{S^1}^{2}(M)$, if $\eta_1 \neq \gamma \eta_2$ for any $\gamma \in \Z$,  the
product and intersection of $\eta_1$ and $\eta_2$ is non-zero 
only if the intersection of their supports is non-empty: 
there is a component of the fixed point set on which the restriction of each class is not zero.
Therefore, when $f(\sigma^M_{i^*,1})=\delta \sigma^N_{i,1}$, we have $\widehat{f}(v^M_{i^*,1})=v^N_{i,1}$; and when $f(\sigma^M_{i^*,1})=\delta \sum_{j=1}^{\beta}m^N_{i,j}\sigma^N_{i,j}$, we have $\widehat{f}(v^M_{i^*,1})=v^N_{i,\beta}$.

Similarly,  looking at $\sigma^M_{i^*,\ell^{i^*}}$, we have
$$f(\sigma^M_{i^*,\ell_{i^*}})=\begin{cases} \delta' \sigma^N_{i',\ell_{i'}} & \text{ or} \\\delta' \sum_{j=\alpha}^{\ell_{i'}}m^N_{i',j}\sigma^N_{i',j} & \text{ with }1 <\alpha\ \ \ , \end{cases}$$
 and in both cases $1 \leq \delta' \in \N$. In the first case,  $\widehat{f}(v^M_{i^*,\ell_{i^*}-1})=v^N_{i',\ell_{i'}-1}$  ; in the second case,   $\widehat{f}(v^M_{i^*,\ell_{i^*}-1})= v^N_{i',\alpha-1}$.

Now we make the  further assumption that $\ell_{i^*}>2$. Then we can analyze the possibilities as follows.

\begin{enumerate}
\item If $f(\sigma^M_{i^*,1})=\delta \sigma^N_{i,1}$, then because $1<2<\ell_{i^*}$, we conclude 
that 
$${f(\sigma^M_{i^*,2})}|_{\max^N}=0={f(\sigma^M_{i^*,2})}|_{\min^N}.$$ 
Moreover,
since $f(\sigma^M_{i^*,1}) \cdot f(\sigma^M_{i^*,2})=1$, we know that the restriction of $f(\sigma^M_{i^*,2})$ to $v^N_{i,1}$ is not zero.
By \eqref{eq:fz01} and \eqref{eq:fz22}, the image of $\sigma^M_{i^*,2}$ is one of the classes listed in item (3) in Lemma \ref{cl:ann2}. By Lemma \ref{lem:intersums}, the fact that $\delta \sigma^N_{i,1} \cdot f(\sigma^M_{i^*,2})=1$ further restricts the possibilities for that image to 
$\gamma \sigma^N_{i,2}$. Moreover, since $\sigma^N_{i,1} \cdot \sigma^N_{i,2}=1$ and $\delta \geq 1$, we have
 $f(\sigma^M_{i^*,2})=\sigma^N_{i,2}$ and $f(\sigma^M_{i^*,1})= \sigma^N_{i,1}$.
So $\widehat{f}(v^M_{i^*,1})=v^N_{i,1}$  and $\widehat{f}(v^M_{i^*,2})=v^N_{i,2}$.
 We proceed by induction: for $2<j \leq \ell_{i^*}-1$,  the induction hypothesis is that for $1 \leq s <j$, we have $f(\sigma^M_{i^*,s})=\sigma^N_{i,s}$ and $\widehat{f}(v^M_{i^*,s})=v^N_{i,s}$. Because
 $$\sigma^N_{i,s} \cdot f(\sigma^M_{i^*,j})=f(\sigma^M_{i^*,s}) \cdot f(\sigma^M_{i^*,j})=\sigma^M_{i^*,s} \cdot \sigma^M_{i^*,j}=\begin{cases} 1 & \text{ if }s=j-1 \\ 0 & \text{ if } 1 \leq s<j-1  \end{cases},$$
we deduce that $f(\sigma^M_{i^*,j})$ is supported on $v^N_{i,j-1}$ and $v^N_{i,j}$.  By Lemma \ref{lem:intersums},  
we have $f(\sigma^M_{i^*,j})=\sigma^N_{i,j}$, so  $\widehat{f}(v^M_{i^*,j})=v^N_{i,j}$.
  In particular $\widehat{f}(v^M_{i^*,\ell_i-1})=v^N_{i,\ell_{i^*-1}}$, hence $f(\sigma^M_{i^*,\ell_{i^*}})$ must have non-zero
  restriction to $v^N_{i,\ell_{i^*}-1}$.

\item  Similarly, if  $f(\sigma^M_{i^*,\ell_{i^*}})=\delta' \sigma^N_{i',\ell_{i'}}$ then $f(\sigma^M_{i^*,\ell_{i^*}-j+1})=\sigma^N_{i',\ell_{i'}-j+1}$ for $1 < j \leq \ell_{i^*}$.
 
\item  If $f(\sigma^M_{i^*,1})=\delta\sum_{j=1}^{\beta}m^N_{i,j}\sigma^N_{i,j}$ with $\beta < \ell_{i}$, then $f(\sigma^M_{i^,2})$ 
must have non-zero
  restriction to $v^N_{i,\beta}$ since it restricts to zero on $\min^N$
and $\max^N$.
Therefore,  by  Lemma \ref{lem:intersums}, the fact that
 $\delta\sum_{j=1}^{\beta}m^N_{i,j}\sigma^N_{i,j} \cdot f(\sigma^M_{i^*,2})=1$ implies that
$f(\sigma^M_{i^*,2})=\gamma \sigma^N_{i,\beta}$
and $\beta=\ell_i-1$, the label $m^N_{i,\ell_i}=1$, and that $\delta=1=-\gamma$.
Therefore $f(\sigma^M_{i^*,2})=\sum_{j=1}^{\ell_i-1}m^N_{i,j}\sigma^N_{i,j}$ and $f(\sigma^M_{i^*,2})=-\sigma^N_{i,\ell_i-1}$. 
So $\widehat{f}(v^M_{i^*,2})=v^N_{i,\ell_i-1}$. We continue as in the previous item and deduce that for 
$1<j<\ell_{i^*}$ we have $f(\sigma^M_{i^*,j})=-\sigma^N_{i,\ell_i-j+1}$ and $\widehat{f}(v^M_{i^*,j})=v^N_{i,\ell_i-j+1}$. 
In particular $\widehat{f}(v^M_{i^*,\ell_{i^*}-1})=v^N_{i,\ell_i-\ell_{i^*}+2}$, hence $f(\sigma^M_{i^*,\ell_{i^*}})$ 
must have non-zero restriction to $v^N_{i,\ell_i-\ell_{i^*}+2}$.

 \item Similarly, if $f(\sigma^M_{i^*,\ell_{i^*}})=\delta'\sum_{j=\alpha}^{\ell_i}m^N_{i',j}\sigma^N_{i',j}$ with $1<\alpha$ then 
$$f(\sigma^M_{i^*,\ell_{i^*}})=\sum_{j=2}^{\ell_i}m^N_{i',j}\sigma^N_{i',j},$$  and $f(\sigma^M_{i^*,\ell_{i^*}-j+1})=-\sigma^N_{i',j}$ for all $1<j<\ell_{i^*}$, and 
the label $m^N_{i,1}=1$.

\item 
In the  case that $f(\sigma^M_{i^*,1})=\sigma^N_{i,1}$, we must have $f(\sigma^M_{i^*,\ell_{i^*}})=\sigma^N_{i',\ell_{i'}}$; otherwise, by items (1) and (4), for $1<j<\ell_{i^*}$, e.g., $j=2$, we will have
both $f(\sigma^M_{i^*,j})=\sigma^N_{i,j}$ and $f(\sigma^M_{i^*,j})=-\sigma^N_{i',\ell_{i'}-j+1}$, 
which results in a contradiction, since no generator $\sigma^N_{i,j}$ is the additive inverse of a generator $\sigma^N_{i',j'}$. 
Moreover we must have $i=i'$, since $\sigma^N_{i,j}=f(\sigma^M_{i^*,j})=\sigma^N_{i',j}$ for $1<j<\ell_{i^*}$.
In particular, $v^N_{i,\ell_{i}-1}=\widehat{f}(v^M_{i^*,\ell_{i^*}-1})=v^N_{i,\ell_{i^*}-1}$, and hence $\ell_i=\ell_{i^*}$. Also, 
\begin{equation} \labell{eq:tauhagain}
\sum_{j=1}^{\ell_i}m^N_{i,j}\sigma^N_{i,j}=\tau^N_h=f(\tau^M_h)=f\Big(\sum_{j=1}^{\ell_{i^*}} m^M_{i^*,j} \sigma^M_{i^*,j}\Big)=\sum_{j=1}^{\ell_{i}}m^{M}_{i,j} \sigma^N_{i,j}, 
\end{equation}
where the last equality is a direct consequence of the fact that $f$ is a homomorphism.
Using the fact that the $\sigma^N_{i,j}$ for $1\leq j \leq \ell_i$ are independent over $\Z$, 
we deduce that $m^N_{i,j}=m^M_{i^*,j}$ for all $j$.

\item Similarly, in the case when $f(\sigma^M_{i^*,1})=\sum_{j=1}^{\beta}m^N_{i,j}\sigma^N_{i,j}$ 
we must have $f(\sigma^M_{i^*,\ell_{i^*}})=\sum_{j=2}^{\ell_i}m^N_{i,j}\sigma^N_{i,j}$ and $\ell_i=\ell_{i^*}$.
Note that in this case  we showed that $m^N_{i,1}=1=m^N_{i,\ell_i}$. By looking at $f^{-1}$ on $\sigma^N_{i,j}$, 
we deduce that 
$m^M_{i^*,1}=1=m^M_{i^*,\ell_{i^*}}$.
Furthermore, since
 \begin{eqnarray*}
\sum_{j=1}^{\ell_i}m^N_{i,j}\sigma^N_{i,j}&=&\tau^N_h=f(\tau^M_h)=f\Big(\sum_{j=1}^{\ell_{i^*}} m^M_{i^*,j} \sigma^M_{i^*,j}\Big)\\
&=&
\sum_{j=1}^{\ell_i-1}m^N_{i,j}\sigma^N_{i,j}-\sum_{j=2}^{\ell_i-1}m_{i^*,j}^M\sigma^N_{\ell_{i}-j+1}+\sum_{j=2}^{\ell_i}m^N_{i,j}\sigma^N_{i,j}\\
&=&m^N_{i,1}\sigma^N_{i,1}+\sum_{j=2}^{\ell_i-1}(2m^N_{i,j}-m^M_{i^*,\ell_i-j+1})\sigma^N_{i,j}+m^N_{i,\ell_i}\sigma^N_{i,\ell_i},
 \end{eqnarray*}
and $\sigma^N_{i,j} \, 1\leq j \leq \ell_i$ are independent over $\Z$, we get that
$2m^N_{i,j}-m^M_{i^*,\ell_i-j+1}=m^N_{i,j}$ and hence
 $m^N_{i,j}=m^M_{i^*,\ell_i-j+1}$, for all $1<j<\ell_i$.
 \end{enumerate}

Now we return to the case $\ell_{i^*}=2$. 
 If $f(\sigma^M_{i^*,1})=\delta \sigma^N_{i,1}$ for $\delta \in \N$
and $f(\sigma^M_{i^*,\ell_{i^*}})=\delta' \sigma^N_{i',\ell_{i'}}$ for $\delta' \in \N$ then,
since $f(\sigma^M_{i^*,1}) \cdot f(\sigma^M_{i^*,\ell_{i^*}})=1$ and $f(\sigma^M_{i^*,1})|_{\max^N}=0=f(\sigma^M_{i^*,\ell_{i^*}})|_{\min^N}$, we have $i=i'$, $\delta=\delta'=1$, and $\ell_i=2$. 
Moreover, \eqref{eq:tauhagain} holds and we must have   $m^N_{i,j}=m^M_{i^*,j}$ for $j=1,2$.

If  $f(\sigma^M_{i^*,1})=\delta \sum_{r=1}^{\beta}m^N_{i,r}\sigma^N_{i,r}$ with $1<\beta<\ell_i$ then $\ell_i>2$.  We apply 
the previous case $\ell_{i} >2$  to $f^{-1}$ and deduce that 
$f^{-1}(\sigma^N_{i,\beta})= \delta \sigma^M_{i'',j''}$ for $i'',j''$ such that $$\sigma^M_{i'',j''}|_{\min^M}=0=\sigma_{i'',j''}|_{\max^M}.$$ 
That is, $1<j''<\ell_{i''}$,
and $$\delta \sigma^M_{i'',j''} \cdot \sigma^M_{i^*,1}=f^{-1}(\sigma^N_{i,\beta}) \cdot f^{-1}\left(\sum_{r=1}^{\beta}m^N_{i,r}\sigma^N_{i,r}\right)=m^N_{i,\beta-1}+m^N_{i,\beta}\sigma^N_{i,\beta} \cdot \sigma^N_{i,\beta}=-m^N_{i,\beta+1}\neq 0,$$ 
hence $i''=i^*$ and $j''=2$, 
contradicting the assumption $\ell_{i^*}=2$. Similarly, we cannot have $f(\sigma^M_{i^*,\ell_{i^*}})=\sum_{r=\alpha}^{\ell_{i'}}m^N_{i',r}\sigma^N_{i',r}$ with $1<\alpha<\ell_i$.

 Finally, assume $\ell_{i^*}=1$. So ${\sigma^M_{i^*,1}}|_{\min^M}\neq 0$ and ${\sigma^M_{i^*,1}}|_{\max^M}\neq 0$.
Hence either $$f(\sigma^M_{i^*,1})=\delta \sigma^N_{i,1} \text{ with }\ell_i=1  \text{ or }f(\sigma^M_{i^*,1})=\delta \tau^N_h.$$ 
In the first case, $$m^M_{i^*,1}\delta \sigma^N_{i,1}=f(m^M_{i^*,1}\sigma^M_{i^*,1})=f(\tau^M_h)=\tau^N_h=m^N_{i,1} \sigma^N_{i,1,}$$ hence 
$\delta m^M_{i^*,1}=m^N_{i,1}$.
In the second case, $$m^M_{i^*,1} \delta \tau^N_h=f(m^M_{i^*,1}\sigma^M_{i^*,1})=f(\tau^M_h)=\tau^N_h,$$ hence $m^M_{i^*,1}=1=\delta$. Moreover $\ell_i=1$, otherwise the case $\ell_i >1$ applied to $f^{-1}$ gives a contradiction.

We have two cases to consider: $m^M_{i^*,1} > 1$ and $m^M_{i^*,1} = 1$.  We start with the case that  $m^M_{i^*,1} > 1$. Then we must have $f(\sigma^M_{i^*,1})=\delta \sigma^N_{i,1}$ with $\ell_i=1$, and $\delta m^M_{i^*,1}=m^N_{i,1}$.
Moreover, in this case,
$\fat^M =0$ and hence $\fat^N=0$ and there are precisely two chains in the graph. 
By reviewing the possible minimal models with zero fat vertices, we can deduce that
the other chain in the graph must have 
more than one edge.
Without loss of generality we may assume $i=i^*=1$. 
The case $\ell_{i^*}>1$ above implies 
that generators in the second chain in the graph of $S^1 \acts M$ are sent to sums of generators in the second chain in the graph of $S^1 \acts N$ and $m^M_{2,1}=m^N_{2,1}, \, m^M_{2,\ell_2}=m^N_{2,\ell_2}$. 
Since $f$ preserves self-intersection, we deduce that 
$$\frac{m^M_{2,1}+m^M_{2,\ell_2}}{m^M_{1,1}}=\sigma^M_{1,1} \cdot \sigma^M_{1,1}=\delta\sigma^N_{1,1} \cdot \delta \sigma^N_{1,1}=\delta^2 \frac{m^N_{2,1}+m^N_{2,\ell_2}}{m^N_{1,1}}=\delta^2 \frac{m^M_{2,1}+m^M_{2,\ell_2}}{\delta m^M_{1,1}}$$ hence $\delta=1$ and $m^N_{1,1}=m^M_{1,1}$, as needed.

For the other case, $m^M_{i^*,1}=1$, we then have $\sigma^M_{i^*,1}=\tau^M_h$. 
By the conventions laid out in \S \ref{decorated} and \S \ref{decorated1}, either 
\begin{itemize}

\item $\fat^M=1$ and there are $2$ chains, each of exactly one edge in the graph of $S^1 \acts M$. In this case, $\fat^N=1$ (by Lemma \ref{iso+fat}) and there are $2$ chains of one edge in the graph of $S^1 \acts N$, otherwise we get a contradiction to the previous case $\ell_{i}>1$ above applied to $f^{-1}$, note that
the labels of the edges must both equal $1$;
or
\item $\fat^M<2$ and there is exactly one chain $i$ of one edge of label $1$ and one more chain which has more than one edge. This holds for the graph $S^1 \acts N$ as well, otherwise we  get a contradiction to either the case $\ell_{i^*}>1$ or the case $\ell_{i^*}=1$ and $m^M_{i^*,1} > 1$; or
\item $\fat^M=2$ and $\iso^M=0$. Then, by Lemma \ref{iso+fat}, $\fat^N=2$ and $\iso^N=0$. 
\end{itemize}
So, in all of these cases there is and index $i$ such that $\sigma^N_{i,1}=\tau^N_h$, $m^N_{i,1}=1$ and $\ell_i=1$.
Since $f(\tau^M_h)=\tau^N_h$ we get $f(\sigma^M_{i^*,1})=\sigma^N_{i,1}$.

This completes our case-by-case analysis and the proposition now follows.
\end{proof}

Proposition \ref{prop:newiso} and the description of $\widehat{f}$ in Example \ref{symp-flip}, Example \ref{perm}, and Proposition \ref{rem:iso} imply the following result.

 \begin{Corollary} \label{cor:hatf2}
    Let $f,g \colon H_{S^1}^{2*}(M) \to H_{S^1}^{2*}(N)$ be orientation-preserving algebra isomorphisms.
\begin{enumerate}
  \item   If $f \neq g$, then $\widehat{f} \neq \widehat{g}$.
  \item The map $\widehat{f}$ induces a bijection
        \begin{equation}
            \widehat{\widehat{f}} \colon \bigcup_{k >1}\{S\ | \ S \text{ is a }\Z_k \text{-sphere in } M\} \to \bigcup_{k >1}\{S \ |\ S \text{ is a }\Z_k \text{-sphere in } N\},
        \end{equation}
        defined by $$\widehat{\widehat{f}}(S_M)=S_N \text{  exactly if }\{\widehat{f}(p_M),\widehat{f}(q_M)\}=\{p_M,q_M\},$$ where $p_M,q_M$ ($p_N,q_N$) are the fixed poles of $S_M$ ($S_N$). 
        The map $ \widehat{\widehat{f}}$ satisfies the following properties:
        \begin{itemize}
            \item $S$ is a $\Z_k$-sphere in $M$ if and only if  $\widehat{\widehat{f}}(S)$ is a $\Z_k$-sphere in $M$, for the same $k$, and
            \item $S'$ and $S$ intersect at a fixed point $p$ if and only if $\widehat{\widehat{f}}(S)$ and $\widehat{\widehat{f}}(S')$ intersect at $\widehat{f}(p)$.
        \end{itemize}
        \end{enumerate}
    \end{Corollary}

We saw in Example~\ref{action-flip} that $\widehat{f}_{\acts}= \widehat{f}_\omega$.  This doesn't contradict 
Corollary~\ref{cor:hatf2}(1) because  $\widehat{f}_{\acts}$ is a strictly weak isomorphism.

\begin{Theorem} \label{thm:strong}
Let  $S^1 \acts (M,\omega_M)$ and $S^1 \acts (N,\omega_N)$ be compact, connected, 
Hamiltonian $S^1$-manifolds of 
dimension four, with dull graphs $\O_M$ and $\O_N$ 
respectively.

Suppose first that there is an isomorphism $\psi: \O_M \to \O_N$ of labeled graphs.
Then there is a unique, orientation-preserving isomorphism 
$f: H_{S^1}^{2*}(M;\Z) \to H_{S^1}^{2*}(N;\Z)$, as algebras over $H_{S^1}^{*}(\pt)$,
so that the induced map $\widehat{f}$ between the fixed components of $M$ and of $N$ agrees with
$\psi$ on the corresponding vertices of $\O_M$ and $\O_N$.
The dull graph isomorphism also guarantees that $M$ and $N$ have equal genus $g_M=g_{N}$.

Conversely, if we have $g_M=g_N$, then given an orientation-preserving isomorphism 
$f: H_{S^1}^{2*}(M) \to H_{S^1}^{2*}(N)$, as algebras over $H_{S^1}^{*}(\pt)$, there 
is a unique isomorphism  $\psi: \O_M \to \O_N$ of labeled graphs 
 whose restriction to the vertices coincides with $\widehat{f}$.
\end{Theorem}

 \begin{proof}
Let $\psi: \O_M \to \O_N$ be an isomorphism of dull graphs as labeled graphs.
Consider the extended decorated graphs with respect to a generic compatible metric associated to $S^1 \acts (M,\omega_M)$ and $S^1 \acts (N,\omega_N)$. 
We fix orderings of the chains in the extended decorated graphs that satisfy $m^M_{1,1}\geq m^M_{2,1} \geq \ldots \geq m^M_{k_M,1}$ for $M$ and $m^N_{1,1}\geq m^N_{2,1} \geq \ldots \geq m^N_{k_N,1}$  for $N$, and also with $m^M_{r,j}=m^N_{r,j}$. 
By Lemma~\ref{equiv}, the map $\psi$ between the dull graphs is induced by a composition of finitely many chain flips in the extended decorated graph, up to possibly precomposing with the flip of the extended decorated graph and a multi-rescaling of the graph.  Since we fix orderings of the chains, we may need to precompose with a permutation of isomorphic chains.

As explained in Example \ref{perm}, a permutation of isomorphic chains induces an orientation-preserving isomorphism $H^{2*}_{S^1}(M)\to H^{2*}_{S^1}(M)$.
Precomposing with the flip of the extended decorated graph 
(and possibly a vertical translation) induces the algebra-map 
that we described in Example~\ref{symp-flip}, and called the symplectic flip.
The symplectic flip is an orientation-preserving isomorphism $H^{2*}_{S^1}(M)\to H^{2*}_{S^1}(M)$.
Following the conventions in Notation~\ref{not:gens}, the  algebra-map 
induced by a multi-rescaling of the extended decorated graph is 
$$
\tau^M_0\mapsto\tau^M_{0}\ ;\  \tau^M_{\infty}\mapsto\tau^M_{\infty} \ ; \
\sigma^M_{r,j}\mapsto \sigma^M_{r,j} \mbox{ for all } r \ ;
$$
this map extends to the identity map, which is an orientation-preserving isomorphism $H^{2*}_{S^1}(M)\to H^{2*}_{S^1}(M)$.
Finally, a chain flip on the extended decorated graphs induces
a chain flip in equivariant cohomology, 
as in Definition~\ref{defpartial}.  By Proposition~\ref{rem:iso},
this yields an orientation-preserving isomorphism $H^{2*}_{S^1}(M)\to H^{2*}_{S^1}(N)$.

It is straight-forward to check that
composing the algebra isomorphisms  induced by a permutation of isomorphic chains, a flip the extended decorated 
graph (if needed), a multi-rescaling, and  chain flips 
yields an orientation-preserving isomorphism $$f \colon H^{2*}_{S^1}(M)\to H^{2*}_{S^1}(N),$$ such that
$f(\varepsilon_F) = \varepsilon_{\psi(F)}$ for every fixed component $F$.
This means that
the map $\widehat{f}$ between the sets of fixed components of $M$ and of $N$ agrees with $\psi$ on the corresponding vertices of 
the dull graphs. By item (1) of Corollary \ref{cor:hatf2}, $f$ is unique.
Lastly, since the genera $g_M$ and $g_N$ are either labels in $\O_M$ and $\O_N$ or both zero, we must have $g_M=g_N$.

In the other direction, assume that $f \colon H_{S^1}^{2*}(M) \to H_{S^1}^{2*}(N)$ 
is an orientation-preserving algebra isomorphism and that $g_M=g_N$. By 
Corollary \ref{cor:Euler}, the map $\widehat{f}$ induces a bijection 
between the set of vertices of the dull graph of $M$ and the set of vertices 
of the dull graph of $N$. By items (2) and (3) of Corollary \ref{cor:opor}, 
and the assumption on the genus, this map preserves the extremal labels, 
the self-intersection labels, and the genus labels. Moreover, by item (2) 
of Corollary \ref{cor:hatf2}, the map induced by $\widehat{f}$ 
from the set of edges of the dull graph of $M$ 
to the set of edges of the dull graph of $N$ is a well-defined bijection that
sends an edge with end points $p,q$ to the edge with end 
points $\widehat{f}(p),\widehat{f}(q)$.  This map
preserves the edge-labels and the adjacency relation. Therefore, the map on 
dull graphs induced by $\widehat{f}$ is an isomorphism.  It is the unique 
isomorphism that agrees with $\widehat{f}$ on the vertices.
 \end{proof}

The last part of Theorem~\ref{thm:strong} has the following immediate corollary.

\begin{Corollary}  \label{cor:sue}
Let  $S^1 \acts (M,\omega_M)$ and $S^1 \acts (N,\omega_N)$ be compact, connected, 
four-dim\-en\-sion\-al Hamiltonian $S^1$-manifolds 
and let $\Lambda: H_{S^1}^{*}(M) \to H_{S^1}^{*}(N)$
be an orientation-preserving algebra isomorphism.  Then there is a unique isomorphism $\psi$ of the 
dull graphs associated to $M$ and $N$  such that $\Lambda(\varepsilon_F) = \varepsilon_{\psi(F)}$ for all fixed components $F$ of $M$.
 \end{Corollary}

In order to prove Theorem~\ref{thm:unique},
the remaining step is to understand
orientation-reversing isomorphisms.
In the following lemma, we explicitly construct
orientation-reversing isomorphisms
when $b_2(M)=2$.
This is a key ingredient in the proof of
Proposition~\ref{prop:inver}, 
which establishes exactly when orientation-reversing isomorphisms exist.

\begin{Lemma}\label{2reversing}
Let $S^1 \acts (M,\omega)$ be a  compact, connected, four-dimensional Hamiltonian $S^1$-manifold with $b_2(M) = 2$.
Then there exists an orientation-reversing isomorphism $H_{S^1}^{2*}(M) \to H_{S^1}^{2*}(M)$ as algebras over $H_{S^1}^{*}(\pt)$.
\end{Lemma}

\begin{proof}
When $b_2(M)=2$, we are in a base case where the (even-degree)
equivariant cohomology is explicitly described in Proposition~\ref{Prohirz} or Proposition~\ref{Proruled}.
We reproduce the extended decorated graphs, first shown in Figure~\ref{fig:toric projection1}(iii)-(vi) here, now labeled with
cohomology generators in $H^2_{S^1}(M;\Z)$.

\newpage 
\begin{center}
  \begin{figure}[ht]
\begin{overpic}[
scale=0.8,unit=1mm]{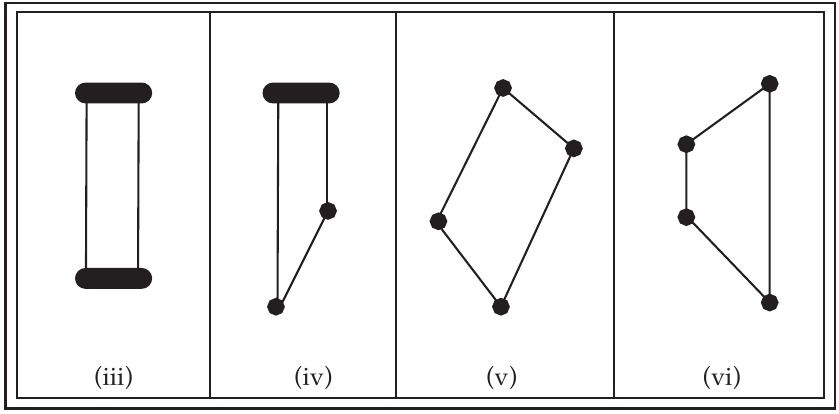}
   \put(14,13){\small{$\tau_0$}}
   \put(14,47){\small{$\tau_\infty$}}
   \put(5,30){\small{$\sigma_{1,1}$}}
   \put(20,30){\small{$\sigma_{2,1}$}}
   \put(39,47){\small{$\tau_\infty$}}
   \put(31,30){\small{$\sigma_{2,1}$}}
   \put(43,20){\small{$\sigma_{1,1}$}}
   \put(46,33){\small{$\sigma_{1,2}$}}
   \put(57,18){\small{$\sigma_{1,1}$}}
   \put(57,36){\small{$\sigma_{1,2}$}}
   \put(73,42){\small{$\sigma_{2,2}$}}
   \put(74,22){\small{$\sigma_{2,1}$}}
   \put(92,18){\small{$\sigma_{1,1}$}}
   \put(86,30){\small{$\sigma_{1,2}$}}
   \put(92,42){\small{$\sigma_{1,3}$}}
   \put(105,30){\small{$\sigma_{2,1}$}}
\end{overpic}
\caption{The possible extended decorated graphs, labeled with generators
in $H^2_{S^1}(M;\Z)$, for manifolds with $b_2(M)=2$.  See Figure~\ref{fig:toric projection1} for 
the decorated labels on these graphs.}
  \label{fig:classes for b2=2}
\end{figure}
\end{center}

\noindent For each case, we indicate the map on the generators
in $H^2_{S^1}(M;\Z)$, which extends to an orientation-reversing
isomorphism
$H_{S^1}^{2*}(M) \to H_{S^1}^{2*}(M)$ as algebras over $H_{S^1}^{*}(\pt)$. 
To verify that the isomorphism is orientation-reversing, 
it is enough to check that it negates the intersection form on the generators.

\vskip 0.1in

\noindent {\bf Case (iii).}
This case includes all ruled surfaces with $g>0$, as well as Hirzebruch surfaces
endowed with a circle action with two fixed surfaces. For the
generators $\tau_0$ ,$\tau_\infty$, $\sigma_{1,1}$, 
and $\sigma_{2,1}$, we define a correspondence
\begin{eqnarray*}
\tau_0 & \mapsto & \tau_\infty \\
\tau_\infty & \mapsto & \tau_0 \\
\sigma_{1,1} & \mapsto & -\sigma_{1,1}\\
\sigma_{2,1} & \mapsto & -\sigma_{2,1}
\end{eqnarray*}
Note that the classes $\sigma_{1,1}$ and $\sigma_{2,1}$ are equal
and may be identified with $\tau_h$.

\vskip 0.1in

\noindent {\bf Case (iv).}
This is the case of Hirzebruch surfaces endowed with a circle
action with a single fixed surface.
For the
generators $\tau_\infty$, $\sigma_{1,1}$, $\sigma_{1,2}$,
and $\sigma_{2,1}$, we define a correspondence
\begin{eqnarray*}
\tau_\infty & \mapsto & -\tau_\infty \\
\sigma_{1,1} & \mapsto & -\sigma_{1,1}\\
\sigma_{1,2} & \mapsto & \sigma_{2,1}\\
\sigma_{2,1} & \mapsto & \sigma_{1,2}
\end{eqnarray*}

\vskip 0.1in

\noindent {\bf Case (v).}
This is the case of Hirzebruch surfaces endowed with a circle
action with four isolated fixed points and two chains
with $\ell_1=\ell_2=2$.  
For the
generators $\sigma_{1,1}$, $\sigma_{1,2}$,
$\sigma_{2,1}$, and $\sigma_{2,2}$ we define a correspondence
\begin{eqnarray*}
\sigma_{1,1} & \mapsto & -\sigma_{1,1}\\
\sigma_{1,2} & \mapsto & \sigma_{2,1}\\
\sigma_{2,1} & \mapsto & \sigma_{1,2}\\
\sigma_{2,2} & \mapsto & -\sigma_{2,2}
\end{eqnarray*}

\vskip 0.1in

\noindent {\bf Case (vi).}
This is the case of Hirzebruch surfaces endowed with a circle
action with four isolated fixed points, one chain of length $3$, and the
other chain of length $1$.  Without loss of generality,
we assume $\ell_1=3$ and $\ell_2=1$.
For the
generators $\sigma_{1,1}$, $\sigma_{1,2}$, $\sigma_{1,3}$,
and $\sigma_{2,1}$, we define a correspondence
\begin{eqnarray*}
\sigma_{1,1} & \mapsto & -\sigma_{1,1}\\
\sigma_{1,2} & \mapsto & \sigma_{2,1}\\
\sigma_{1,3} & \mapsto &  -\sigma_{1,3}\\
\sigma_{2,1} & \mapsto & \sigma_{1,2}
\end{eqnarray*}
Having identified the isomorphism in each case completes the proof.
\end{proof}

\begin{Proposition} \label{prop:inver}
Let $S^1 \acts (M,\omega_M)$ and $S^1 \acts (N,\omega_N)$ be   
compact, connected, four-dim\-en\-sion\-al Hamiltonian $S^1$-manifolds with equal
genus $g_M=g_{N}$.
If there is an orientation-reversing isomorphism 
$H_{S^1}^{2*}(M) \to H_{S^1}^{2*}(N)$ as algebras over $H_{S^1}^{*}(\pt)$,
then $b_2(M) = b_2(N) = 2$
and there also exists an
orientation-preserving isomorphism $H_{S^1}^{2*}(M) \to H_{S^1}^{2*}(N)$ as algebras over $H_{S^1}^{*}(\pt)$.
\end{Proposition}

   \begin{proof}
   Let $f \colon H_{S^1}^{2*}(M) \to H_{S^1}^{2*}(N)$ be
 an orientation-reversing algebra-isomorphism. 
By Corollary \ref{cor:Euler}, since $f$ is an algebra-isomorphism, it induces 
a bijection $\widehat{f}$ between the sets of connected components of the fixed points sets of $M$ and $N$ that sends isolated points to isolated points and surfaces to surfaces. In particular, the decorated graphs of $S^1 \acts (M,\omega_M)$ and of $S^1 \acts (N,\omega_N)$ have the same number of fat vertices and the same number of thin vertices. 
Since every graph has exactly two extremal vertices, and since every fat vertex
is extremal, this implies that they have the same number of extremal thin vertices. 
Moreover, by item (2) of Corollary \ref{cor:opor}, since $f$ is orientation-reversing, $\widehat{f}$ interchanges extremal isolated fixed points and interior ones. Therefore, the number of extremal thin
vertices equals the number of interior thin vertices in each graph. Hence, in each graph
either there are exactly four thin vertices, or there is a fat vertex and exactly two thin
vertices, or there are two fat vertices and no thin vertices. Hence, by Morse theory, $b_{2}(M)=2=b_2(N)$. Lemma~\ref{2reversing} now guarantees that there is an orientation-reversing algebra-isomorphism $g \colon H_{S^1}^{2*}(M) \to H_{S^1}^{2*}(M)$. The composition $f \circ g$ is an orientation-preserving algebra-isomorphism $H_{S^1}^{2*}(M) \to H_{S^1}^{2*}(N)$.
\end{proof}

\begin{Remark}
For compact, oriented four-manifolds $M$ and $N$, an 
orientation-reversing isomorphism in cohomology $H^{*}(M) \to H^{*}(N)$ automatically negates the
intersection form.  Thus, such an isomorphism can exist only if the signature of the intersection
form is zero (see \cite{milnor husemoller} for the more general story).
Every compact, connected, symplectic four-manifold that admits a Hamiltonian $S^1$-action is
a blowup of $\C P^2$ or a blowup of a ruled surface, implying that $b_2^+=1$.
Zero signature then implies that $b_2^{-}=1$. So for there possibly to be
an orientation-reversing isomorphism, we must be in the case that $b_{2}=b_2^++b_2^-=2$.

Turning to the manifolds themselves,
M\"ullner has investigated orientation-reversing automorphisms of compact, oriented manifolds, and in
four dimensions has established the existence of orientation-reversing automorphisms 
in the (non-equivariant) topological category for the zero signature,
simply connected case \cite[Theorem~B]{mullner}.
In  \cite{hkt}, we provide explicit, equivariant, orientation-reversing diffeomorphisms
for compact, connected Hamiltonian $S^1$-manifolds of dimension four with $b_{2}=2$, which need not be simply connected.
\end{Remark}

\begin{proof}[Proof of Theorem \ref{thm:unique}]
The implication (1) $\Longrightarrow$ (2) is by the first part of Theorem \ref{thm:strong} and since, by \eqref{eq:poincare}, the odd degree ranks of $H_{S^1}^*$ are determined by the genus $g$ of the fixed surface if exists, and are all $0$ otherwise.

The implication (2) $\Longrightarrow$ (3) is immediate.

To show that (3)  $\Longrightarrow$ (1), assume that $2g_M=\rank H_{S^1}^{1}(M)=\rank H_{S^1}^{1}(N)=2g_N$ and that there is a 
weak algebra-isomorphism
$H_{S^1}^{2*}(M) \to H_{S^1}^{2*}(N)$. 
Up to pre-composing with the strictly weak algebra-isomorphism induced from 
$-\id \colon H_{S^1}^{2}(M) \to H_{S^1}^{2}(M)$ if necessary, we can assume that 
the weak isomorphism is an isomorphism. 
Moreover, by Proposition~\ref{prop:inver} we can assume that there is an 
orientation-preserving algebra-isomorphism $H_{S^1}^{2*}(M) \to H_{S^1}^{2*}(N)$. 
Therefore, by the second part of Theorem~\ref{thm:strong}, there is an isomorphism 
between the dull graphs of $M$ and $N$. 
\end{proof}

  Now, we check which of the isomorphisms preserves (or negates) the first equivariant Chern class. Recall that the \emph{equivariant Chern classes} $c_{\ell}^{S^1}(E)$ of an equivariant complex vector bundle $E$ are the Chern classes of the vector bundle $\widetilde{E}$ on $(M \times ES^1)/S^1$ whose pull back to $M \times ES^1$ is $E \times ES^1$. 
   To consider $TM \to M$ as an equivariant complex vector bundle, equip it with an $S^1$-invariant complex structure $J \colon TM \to TM$ that is compatible with $\omega$; 
   the Chern class of the complex vector bundle $(TM,J)$ is independent of the choice of invariant compatible almost complex structure.

  \begin{Proposition}\labell{pro:equivchern}
  Given Hamiltonian $S^1 \acts(M^4,\omega)$, we have 
  \begin{equation} \labell{eqc1}
c_1^{S^1}(TM)=\tau_0+\tau_{\infty}+\left(\sum_{i=1}^{k}\sum_{j=1}^{\ell_i}\sigma_{i,j}\right)-(2g+k-2)\tau_h;
\end{equation}
\begin{equation} \labell{eqc2}
c_{1}^{S^1}(TM)^2-2c_2^{S^1}(TM)=\tau_0^2+\tau_{\infty}^2+\left(\sum_{i=1}^{k}\sum_{j=1}^{\ell_i}\sigma_{i,j}^2\right)-(k-2)\tau_h^2;
\end{equation}
\begin{equation} \labell{eqc3}
(k-2) \tau_h^2-\sum_{i=3}^{k}{\sigma_i^2}=0.
\end{equation}
  \end{Proposition}
The number $k$, the classes  $\tau_h,\, \sigma_{i,j}, \, \tau_0, \, \tau_\infty$, and the lables $m_{i,j}$ 
are defined in Notation \ref{not:gens}; the genus $g$ is defined in Notation \ref{not:genus}. For $3 \leq i \leq k$, denote $$\sigma_{i}=
\begin{cases}
\sigma_{i,1} & \text{ if there is exactly one fixed surface and it is maximal}\\
\sigma_{i,\ell_i} & \text{ if there is exactly one fixed surface and it is minimal}\\
0 & \text{ otherwise}.
\end{cases}$$
In case the top (bottom) fat vertex does not exist we take $\tau_{\infty}$ to be $0$ ($\tau_0$ to be $0$).

\begin{proof}
The map 
$$
i^{*} \colon H_{S^1}^{*}(M) \to H_{S^1}^{*}(M^{S^1})=\bigoplus_{F \subset M^{S^1}} H^{*}(F)[t]
$$  
is injective by \cite[Theorem 1.1(A)]{KesslerHolm2}. Hence it is enough to show that equalities \eqref{eqc1} and \eqref{eqc2} hold when both sides are restricted to $F$, for every connected component $F$ of $M^{S^1}$.
By \cite[Corollary 4.9]{KesslerHolm2}, relying on \cite[Appendix~C]{GGK},
\begin{itemize}
\item if $F=\pt$ is an isolated fixed point, and  $w_1$ and $w_2$ are the
 weights of the $S^1$-action at $\pt$, with respect to $J$, then 
$${c_{1}^{S^1}(TM)}|_{\pt}=(-w_1-w_2)t \in H_{S^1}^{2}(\pt),$$ $${c_{2}^{S^1}(TM)}|_{\pt}=w_1 w_2 t^2 \in H_{S^1}^{4}(\pt),$$ and $$({{c_1^{S^1}})^2-2c_2^{S^1}}|_{\pt}=(w_1^2+w_2^2) t^2.$$   

\item If $F=\Sigma$ is a fixed surface, then
$$c_1^{S^1}(TM)|_{\Sigma_{*}}=(2-2g)[\Sigma]\otimes 1 + e_{*}[\Sigma] \otimes 1+(-1)^{\delta_{*=\min}}	 \otimes t,$$
$$c_2^{S^1}(TM)|_{\Sigma_{*}}=(-1)^{\delta_{*=\min}}(2-2g)[\Sigma]\otimes t,$$ 
and 
$$({{c_1^{S^1}})^2-2c_2^{S^1}}|_{\Sigma_{*}}=1 \otimes t^2+2(-1)^{\delta_{*=\min}}e_{*}[\Sigma] \otimes t,$$
where $[\Sigma]$ is the Poincar\'e dual of the class of a point in $H_{0}(\Sigma;\Z)$.
\end{itemize}

Note that if $\pt =v_{i,j}$ (the vertex between the $j$ edge and the $j+1$ edge on the $i$ chain) then the weights are $-m_{i,j}$ and $m_{i,j+1}$; if $\pt =v_0$ (isolated minimum) then the weights are $m_{1,1}$ and $m_{2,1}$; if $\pt =v_{\infty}$ (isolated maximum) then the weights are $-m_{1,\ell_1}$ and $-m_{2,\ell_2}$.
(Recall that if $\fat=1$ we assume, without loss of generality, that the fat vertex is the top vertex and that $m_{1,1} \geq m_{2,1} \geq 1=m_{3,1}=\ldots=m_{k,1}$.)
 
 Now, the equality of these classes to the restrictions to $F$ of the right-hands in \eqref{eqc1}, \eqref{eqc2} and \eqref{eqc3} follows directly from the restrictions of $\tau_0, \tau_{\infty},\tau_h,\sigma_{i,j}$ to $F$ listed in Tables \ref{Table-rest-2}, \ref{Table-rest-0}, \ref{Table-rest-1}, and justified in \S \ref{NoTitinj}.
\end{proof}

\begin{Proposition} \labell{cl:c1}
If $\ell_{i^*}>2$, for the map $f \colon H_{S^1}^{2*}(M) \to H_{S^1}^{2*}(N)$ induced by 
 a chain flip of the $(i^*)^{\mathrm{th}}$ chain, we have 
 $$f(c_1^{S^1}(TM))\neq \pm c_1^{S^1}(TN).$$ 
 However, 
 $$f(c_1^{S^1}(TM))|_{p} =  {c_1^{S^1}(TN)}|_{p}
 $$
 whenever $p=\max^N$, $p=\min^N$, or $p$ is an isolated vertex not on the flipped chain.
\end{Proposition}

\begin{proof}
Without loss of generality $i^*=1$.
By definition of the chain flip and Proposition \ref{pro:equivchern}, 
\begin{eqnarray*}
f(c_1^{S^1}(TM)) &=& f(\tau^M_0+\tau^M_{\infty}+\left(\sum_{i=1}^{k}\sum_{j=1}^{\ell_i}\sigma^M_{i,j}\right)-(2g+k-2)\tau^M_h)\\
&=&\tau^N_0+\tau^N_\infty+\sum_{j=1}^{\ell_1-1}m^N_{1,j}\sigma^N_{1,j}-\sum_{j=2}^{\ell_1-1}\sigma^N_{1,j}+\sum_{j=2}^{\ell_1}m^N_{1,j}\sigma^N_{1,j}\\
& & +\sum_{i=2}^{k}\sum_{j=1}^{\ell_i}\sigma^N_{i,j}-(2g+k-2)\tau^N_h\\
&=& c_1^{S^1}(TN)+\sum_{j=2}^{\ell_1-1}(2m^N_{1,j}-2)\sigma^N_{1,j}\\
&=&c_1^{S^1}(TN)+2(\tau^N_h-\sum_{j=1}^{\ell_1}\sigma^N_{1,j}) \neq c_1^{S^1}(TN).
\end{eqnarray*}
The last step follows from $\tau^N_h=\sum_{j=1}^{\ell_1}m_{1,j} \sigma^N_{1,j}$  and 
the fact that there are no other linear relations among $\tau_h,\sigma_{1,1},\ldots,$ and $\sigma_{1,\ell_1}$. The term $\sum_{j=1}^{\ell_1}m_{1,j} \sigma^N_{1,j}$ does not equal $\sum_{j=1}^{\ell_1}\sigma^N_{1,j}$ because we assumed $\ell_1>2$ and by Proposition 
\ref{label1}. 
Moreover, we get
$$f(c_1^{S^1}(TM))|_{\max^N}=c_1^{S^1}((TN))|_{\max^N} \neq 0$$
and similarly at $\min^N$ and at $p$ for every isolated vertex $p$ that is not on the first chain,
hence $f(c_1^{S^1}(TM))$ cannot equal $-c_1^{S^1}(TN)$. 
\end{proof}

We now turn to a key Corollary of Theorem~\ref{thm:unique}.  This will allow us to detect when there cannot
be a diffeomorphism preserving a compatible complex structure.

\begin{Corollary}\labell{cor:uniquec1}
Let  $S^1 \acts (M,\omega_M)$ and $S^1 \acts (N,\omega_N)$ be compact, connected, four-di\-men\-sion\-al 
Hamiltonian $S^1$-manifolds. 
Equip $M$ and $N$ with the orientations induced from the symplectic forms.
Let  $f \colon H_{S^1}^{2*}(M) \to H_{S^1}^{2*}(N)$ be an  
orientation-preserving isomorphism of algebras. 
Then 
\begin{itemize}
\item
$f(c_1^{S^1}(TM))=c_1^{S^1}(TN)$ if and only if $f$ is induced from an orientation-preserving 
equivariant diffeomorphism $\widecheck{f} \colon M \to N$ such that for any $S^1$-invariant, 
$\omega_M$-compatible almost complex structure $J_M$ on $M$, the structure 
$$\widecheck{f}_{*}{J_M}:=d\widecheck{f} \circ J_M \circ d\widecheck{f}^{-1}$$ 
is $S^1$-invariant and $\omega_N$-compatible;
\item 
$f(c_1^{S^1}(TM))=-c_1^{S^1}(TN)$ if and only if the isomorphism 
$f$ is induced from an 
orientation-preserving, equivariant diffeomorphism $\widecheck{f} \colon M \to N$ 
such that for any $S^1$-invariant, $\omega_M$-compatible almost complex structure 
$J_M$ on $M$, the structure $-\widecheck{f}_{*}{J_M}$ is $S^1$-invariant and $\omega_N$-compatible.
\end{itemize}
 \end{Corollary}

\begin{proof}
By Proposition \ref{prop:newiso},  an orientation-preserving isomorphism 
$$f \colon H_{S^1}^{2*}(M) \to  H_{S^1}^{2*}(N)$$
can be presented as a composition of finitely many chain flips, followed by
finitely many permutations of isomorphic chains, and possibly a symplectic flip.
Let $n$ be the minimal number of chain flips in such a presentation. 
It follows from formula \eqref{eqc1} in Proposition \ref{pro:equivchern} that a permutation of 
isomorphic chains sends $c_1^{S^1}$ to $c_1^{S^1}$, and a symplectic flip sends $c_1^{S^1}$ to 
$-c_1^{S^1}$. Thus, by Proposition \ref{cl:c1}, the isomorphism $f$ sends $c_1^{S^1}(TM)$ to 
either $c_1^{S^1}(TN)$ or to $-c_1^{S^1}(TN)$
if and only if the number $n$ is $0$.
Moreover, if $f(c_1^{S^1}(TM))=-c_1^{S^1}(TN)$, then there is a symplectic flip 
in any presentation of $f$ in which $n=0$; and if $f(c_1^{S^1}(TM))=c_1^{S^1}(TN)$, 
then there is a presentation of $f$ with $n=0$ and with no symplectic flip.

On the other hand, by Proposition \ref{scale1}, a permutation of isomorphic chains is induced by an
orientation-preserving equivariant diffeomorphisms that sends any invariant, compatible almost 
complex structure to an invariant, compatible almost complex structure. By Remark \ref{flipandscale}, 
the symplectic flip is induced by an orientation-preserving equivariant diffeomorphism that sends 
any invariant, compatible almost complex structure to minus an invariant, compatible almost 
complex structure.
\end{proof}

Recall that every compact, connected, four-dimensional Hamiltonian 
$S^1$-manifold of admits an $S^1$-invariant, integrable almost complex 
structure that is compatible with the symplectic form \cite[Theorem~7.1]{karshon}. 
An almost complex structure is integrable if and only if the 
Nijenhuis tensor vanishes (see \cite[Appendix 8]{370} and \cite[Chapter 2]{160}). 
Thus the naturality of the Nijenhuis tensor (see \cite[Proposition 4.2.1]{MS})
implies that the pushforward by a diffeomorphism of an integrable almost 
complex structure is again integrable. 
Therefore, Corollary \ref{cor:uniquec1} has the following consequence.

\begin{Corollary}\label{cor:uniquec2}
Let  $S^1 \acts (M,\omega_M)$ and $S^1 \acts (N,\omega_N)$ be compact, connected, four-di\-men\-sion\-al 
Hamiltonian $S^1$-manifolds. 
Equip $M$ and $N$ with the orientations induced from the symplectic forms.
If there is an orientation-preserving isomorphism of algebras from $H_{S^1}^{*}(M)$ to $H_{S^1}^{*}(N)$ that
 takes $c_1^{S^1}(TM)$ to $c_1^{S^1}(TN)$, then
there exist
invariant complex 
structures on $M$ and $N$, compatible with the symplectic forms, so that the manifolds are equivariantly biholomorphic.
\end{Corollary}

\begin{Remark} \label{rem:orient-rev}
In the case of an orientation-reversing isomorphism  
$$f \colon H_{S^1}^{2*}(M) \to H_{S^1}^{2*}(N),$$
the equivariant first Chern class is neither preserved nor negated.  
This is straight-forward to see by examining
the orientation-reversing algebra isomorphisms listed in the proof of Lemma~\ref{2reversing}
and using the formula \eqref{eqc1}.
\end{Remark}

We conclude this section by revisiting Example~\ref{first pass}.  We are now equipped with the tools
necessary to complete a full analysis of these $S^1$-manifolds. 
\begin{Example} \labell{counterE}
We 
consider the two compact Hamiltonian $S^1$-manifolds $S^1 \acts (M,\omega_M)$ and $S^1 \acts (N,\omega_N)$ with extended decorated graphs shown in Figure~\ref{fig:chain flip2}. 


\begin{center}
\begin{figure}[h]
\includegraphics[height=5cm]{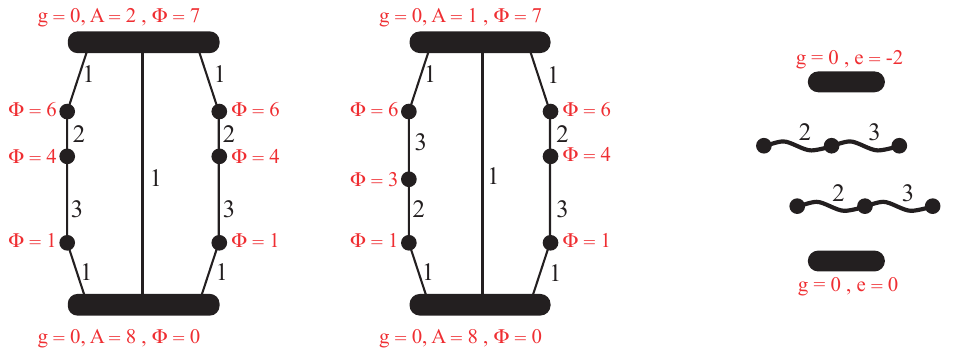} 
\caption[.]{On the left are two extended decorated graphs that differ by a chain flip. 
On the right is the dull graph
that is the dull graph of each.
}
\label{fig:chain flip2}
\end{figure}
\end{center}

The $S^1$-manifolds $S^1 \acts (M,\omega_M)$ and $S^1 \acts (N,\omega_N)$ are obtained by precomposing the inclusion $S^1 \hookrightarrow (S^1)^2$ sending $s \mapsto (1,s)$ on the toric actions in Figure \ref{fig:toricchainflip}.
\begin{center}
\begin{figure}[h]
\includegraphics[height=5cm]{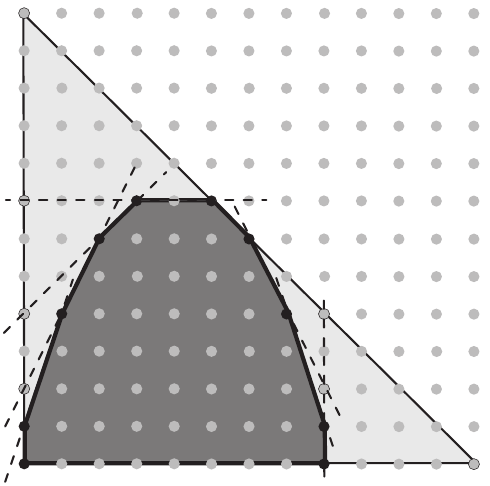} \,\,\,\,\,\,\,\,\,\,\,\,\,\,\,\,\,\,\,\,\,\,\,\,\,\,\,\,\,\,\,\,\,\,\,\,\,\,\,\,\,\,
\includegraphics[height=5cm]{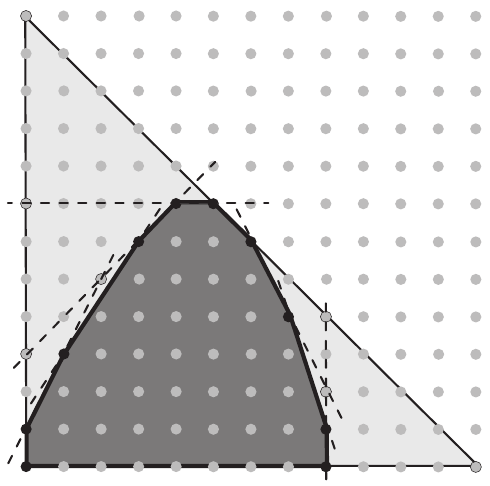}
\caption[.]{The moment map  (Delzant) polytopes of the toric actions extending the $S^1$-actions.}
\label{fig:toricchainflip}
\end{figure}
\end{center}
Both toric actions are obtained from the toric action on $(\CP^2, 12 \omega_{\FS})$ by a sequence of $7$ equivariant blowups,  of sizes $(5,4,3,2,2,1,1)$ in the left and  $(5,4,4,2,2,1,1)$ in the right. 
The moment map polytopes define different fans, so the toric manifolds are different as toric varieties. 

The decorated graphs in Figure \ref{fig:chain flip2} are not isomorphic, so the two manifolds 
are not $S^1$-equivariantly symplectomorphic. 
The manifolds $M$ and $N$ are diffeomorphic: they are both 
$7$-fold blowups of $\CP^2$ (or equivalently $6$-fold blowups of the non-trivial $S^2$-bundle over $S^2$).  The blowup forms determining the symplectic structures have reduced forms
$(12;5,4,3,2,2,1,1)$ and $(11;4,3,3,2,2,1,1)$ respectively.  Because these are not equal,
there is no diffeomorphism taking the symplectic form of one to the symplectic form of the other
\cite[Thm.~1.8]{KK:blowups}.

We show that the manifolds are not isomorphic by a weaker isomorphism: an equivariant diffeomorphism $\widecheck{f} \colon M \to N$ 
such that for an $S^1$-invariant $\omega_M$-compatible complex structure $J_M \colon TM \to TM$,  the complex 
structure $J_N:=\widecheck{f}_{*}J_M$ on $TN$ is $\omega_N$-compatible.
Under such an isomorphism, we would have  
\begin{equation} \labell{eq:c11}
\begin{array}{lcl}
\widecheck{f}_{*}(c_1^{S^1}(TM)) & =& \widecheck{f}_{*}(c_1^{S^1}(TM,J_M))\\
& =& c_1^{S^1}(TN,\widecheck{f}_{*}J_M)\\
& =&  c_1^{S^1}(TN,J_N)\\
&=& c_1^{S^1}(TN).
\end{array}
\end{equation}
The map $\widecheck{f}$ must send fixed points to fixed points. Moreover, since $d\widecheck{f}$ is an invariant $(J_M,J_N)$-complex linear isomorphism at each fixed point $p \in M$, the weights of the $J_N$-complex representation at $\widecheck{f}(p)$ 
equal either the weights of  the $J_M$-complex representation or their negation.
Furthermore, the map $\widecheck{f}$ sends any invariant $J_M$-holomorphic embedded sphere $S$ to an invariant $J_N$-holomorphic embedded sphere; it preserves or negates simultaneously the $J_M$-weights at fixed points that are the 
poles of $S$.
In particular, it sends the embedded spheres that correspond to the fat vertices and edges in the decorated graph of $S^1 \acts M$ to the ones that correspond to the fat vertices and edges in the decorated graph of $S^1 \acts N$.

The compatibility with the symplectic form implies that the weights of the complex representation can be read from the graph, as in \cite[\S 2]{karshon}:
$$\big\{ {\color{blue} \{-1,3\}, \{-3,2\},\{-2,1\}}, \{-1,3\}, \{-3,2\},\{-2,1\}\big\}$$ at the isolated fixed points in $M$ and 
$$\big\{ {\color{red} \{-1,2\}, \{-2,3\},\{-3,1\},} \{-1,3\}, \{-3,2\},\{-2,1\}\big\}$$ at the isolated fixed points in $N$. 
We conclude that $\widecheck{f}$ sends the isolated fixed points that are in a chain in the decorated graph of $M$ to the isolated fixed points that are in a chain in the decorated graph of $N$, but for one chain, the order 
(according to the moment map value) remains the same, and for the other chain, the order is reversed.
We call the latter a \emph{flipped chain}.

By \cite[Corollary 4.9]{KesslerHolm2}, relying on \cite[Appendix~C]{GGK}, at each isolated fixed point $\pt$, we have ${c_1^{S^1}}|_{\pt}=(-w_1-w_2)t$, where $w_1, \, w_2$ are the weights of the complex representation at $\pt$.
Hence, for every fixed point $p$ in the flipped chain, we have 
$${c_1^{S^1}(TN)}|_{F(p)}=-{c_1^{S^1}(TM)}|_{p}$$ 
and for every isolated fixed point $q$ in the other chain, 
$${c_1^{S^1}(TN)}|_{F(q)}={c_1^{S^1}(TM)}|_{q}$$ contradicting \eqref{eq:c11}.
Therefore, no such equivariant diffeomorphism $\widecheck{f}$ can exist.

 Still, the two $S^1$-manifolds in Figure \ref{fig:chain flip2} have the same dull graph, thus, by Theorem \ref{thm:unique},  their equivariant cohomology rings are isomorphic as algebras.
  (Note that $g=0$ in both, hence $H_{S^1}^{2*+1}(M)=0=H_{S^1}^{2*+1}(N)$ and an isomorphism of the even part is an isomorphism of $H_{S^1}^{*}$.)
  Explicitly, the map $f:H_{S^1}^{2*}(M;\Z)\to H_{S^1}^{2*}(N;\Z)$ 
induced by the chain flip of the left chain is an isomorphism, see Proposition~\ref{rem:iso}.
 \hfill $\diamondsuit$
\end{Example}

\section{Finitely many inequivalent maximal Hamiltonian torus 
actions\\ on a compact symplectic four-manifold} \labell{sec8}

We now use our understanding of the cohomology ring and invariants of a Hamiltonian
$S^1\acts (M^4,\omega)$ to provide a soft proof that there are finitely many maximal torus actions
on a fixed, compact, connected, four-dimensional symplectic manifold $(M,\omega)$.
If  we do not restrict to maximal torus actions, then there are infinitely many 
Hamiltonian tours actions on any four-dimensional symplectic manifold that admits a toric
action.
For example, for every relatively prime positive integers $(m,n)$, there is a 
Hamiltonian circle action on $(\CP^2,\omega_{\FS})$ whose graph has three 
edges of labels $m,n,m+n$, as in Figure~\ref{fig:circle on cp2}(c).
These actions are non-equivalent, and they all extend to the unique 
toric action on $(\CP^2,\omega_{\FS})$.

In her monograph, Karshon established that a Hamiltonian circle action extends 
to a toric action if and only if each fixed surface has genus zero and each 
non-extremal level set for the moment map contains at most two non-free orbits 
\cite[Proposition 5.21]{karshon}. In particular,  
every maximal circle action on $(M,\omega)$ has either 
\begin{itemize}[leftmargin=1in]
\item[(Case I)] two fixed surfaces; or
\item[(Case II)] exactly one fixed surface. 
\end{itemize}
For a maximal circle action, then, if $g(\Sigma)=0$, then the decorated graph 
with respect to any metric must have at least $3$ chains. 
Karshon also showed that if $M=\CP^2$ or is obtained from an $S^2$-bundle over 
$S^2$ by at most one symplectic blowup, then every effective circle action 
extends to a toric action.

In Case I above, the Hamiltonian $S^1$-manifold can be obtained  
by a sequence of $r$ equivariant symplectic blowups from a manifold with 
two fixed surfaces and no interior fixed points, that is from a symplectic $S^1$-ruled 
surface \cite[Theorem 6.1]{karshon}. If $g(\Sigma)=0$, then by the comments above, we
must have $r>1$.
If $g(\Sigma)>0$, then every Hamiltonian circle action falls into Case I.

For Case II, the $S^1$-manifold can be obtained by a sequence of 
$r>1$ equivariant symplectic blowups from either a minimal model of 
four isolated fixed points or from a minimal model of one fixed surface 
and two isolated fixed points. Both minimal models are projections of 
Delzant polytopes.

\begin{Notation} \labell{notm}
Fix $(M,\omega)$ that is a symplectic manifold obtained  by a sequence of 
$r$ symplectic blowups of sizes $\varepsilon_1,\ldots,\varepsilon_r$ from a ruled symplectic $S^2$-bundle $(M_{\Sigma},\omega_{\Sigma})$ over a surface $\Sigma$,
using Notation~\ref{basic notions}. 
As a smooth manifold, $M$ is the $r$-fold complex blowup of $M_{\Sigma}$, equipped with a complex structure such that each fiber is a holomorphic sphere. Let $E_1,\ldots,E_r$ denote the homology classes of the exceptional divisors.

Let 
$$N=\max\left\{\quad \dim H^{2}(M;\R)+2\quad ,
\quad 2\dim H^{2}(M;\R)-2\quad \right\}.
$$ 
Then $N$ is the maximal possible number of edges + fat vertices in the extended decorated graph of a Hamiltonian circle action on $(M,\omega)$, and in particular bounds from above the number of chains in such a graph. 
\end{Notation}

\begin{Remark}
We do not know {\em a priori} that every Hamiltonian circle action on $(M,\omega)$ is obtained by preforming the blowups of sizes $\varepsilon_1,\ldots,\varepsilon_r$ equivariantly. 
It might be obtained by $S^1$-equivariant blowups of different sizes, or starting from 
a different minimal model.
In fact, one can use ``hard" pseudo-holomorphic tools to 
every Hamiltonian circle action is obtained by preforming the 
blowups of sizes $\varepsilon_1,\ldots,\varepsilon_r$ equivariantly.
For the case  $g(\Sigma)=0$, this is in \cite{kkp}.  For positive genus, this is in
\cite{KesslerHolm1}.  In particular, these results imply
the finiteness of inequivalent maximal Hamiltonian circle actions on 
the symplectic manifold. In this section we give a soft proof of the
finiteness result, so in this section, we will not use \cite{KesslerHolm1,kkp}.
\end{Remark}

We now describe how to recover the decorated graph, 
without the height and area labels,  given the set of robust 
generators (without indicating the indices) and pointers to the minimal 
and maximal elements. 
We deduce the algorithm from the correlation between the intersection 
form on the generators and the labels and the adjacency relation in 
the decorated graph, explained in \S \ref{rem:int}.
We will denote $$\tau_{\max}=\begin{cases} \tau_{\infty} & \text{ if }\Sigma_{\infty} \text{ exists}\\
                                                            \sigma_{1,\ell_1} \cup \sigma_{2,\ell_2} & \text{ otherwise}
                                     \end{cases},$$     
                  $$\tau_{\min}=\begin{cases} \tau_{0} & \text{ if }\Sigma_0 \text{ exists}\\
                                                            \sigma_{1,1} \cup \sigma_{2,1} & \text{ otherwise}
                                     \end{cases}.$$     
Then we have the following.

\begin{Proposition} \labell{orderetc2}
Given the set  $(\tau_{\min},\tau_{\max}) \times S\times (g)$, where 
$$S:=\{\sigma_{i,j}\,| \, 1 \leq j \leq \ell_i  , \, 1 \leq i \leq k\} \smallsetminus  \{\sigma_{i,1}\, |\, i>2\},
$$ 
and the intersection form on the elements $\{\tau_{\min},\tau_{\max}\} \cup S$,
the 
decorated graph associated to the action, minus the height values and the 
area labels, is determined as follows.
\begin{enumerate}
\item First we recover the indices of the generators recursively.

\noindent Set $S_{\last}:=\{x \in S \, | \, x \cdot \tau_{\max} \neq 0\} \smallsetminus \{\tau_{\max}\}$ and
$$S_{\last-1}=\{x \in S \, | \, x \cdot y \neq 0 \text{ for some }y \in S_{\last}\} \smallsetminus (S_{\last} \cup \{\tau_{\max}\}).$$ 
For $1<j \leq \last-1$, set 
$$S_{\last-j}=\{x \in S \, | \, x \cdot y \neq 0 \text{ for some }y \in S_{\last-j+1}\} \smallsetminus \bigcup_{r=0}^{j-1} S_{\last-r}.$$ 
For $x,x' \neq \tau_{\max}$ we say $x \equiv x'$ if there are $y_1,\ldots,y_m$ in $S \smallsetminus \{\tau_{\max}\}$ such that $x \cdot y_1  \neq 0, \, y_1 \cdot y_2 \neq 0,  \ldots, y_{m-1} \cdot y_{m}\neq 0, y_{m} \cdot x'\neq 0$.  
Set $x=\sigma_{i,\ell_i-j}$ if $x$ is in the $i^{\mathrm{th}}$ equivalence class and in $S_{\last-j}$; the index $i$ is such that
$$\sigma_{1,\ell_1} \cdot \tau_{\max} \geq \sigma_{2,\ell_2} \cdot \tau_{\max} \geq \sigma_{3,\ell_3} \cdot \tau_{\max} \geq \ldots \ge \sigma_{k,\ell_k} \cdot \tau_{\max}.$$ 
The number of elements in the $i^{\mathrm{th}}$ equivalence class equals $\ell_i$ if $i=1,2$ and equals $\ell_i-1$ if $i \geq 2$.
\item  If $\tau_{\max} \in H_{S^1}^{2}(M) \smallsetminus \{0\}$ there is a $\max$ fat vertex, with self-intersection 
 $\tau_{\max} \cdot \tau_{\max}$.
Otherwise the maximal vertex is isolated. 

\noindent If $\tau_{\min} \in H_{S^1}^{2}(M) \smallsetminus \{0\}$ there is a $\min$ fat vertex, with self intersection 
 $\tau_{\min} \cdot \tau_{\min}$. 

\noindent Otherwise the minimal vertex is isolated. For every $i,j$ such that $\sigma_{i,j}\in S$, there is an edge that corresponds to $\sigma_{i,j}$. The $(i,j+1)$ edge is adjacent to the $(i,j)$ edge, the common vertex $v_{i,j}$ is isolated. The $(i,\ell_i)$ edge is adjacent to the top vertex, the $(i,1)$ edge is adjacent to the bottom vertex.
\item The labels are determined recursively.

\noindent If $\tau_{\max} \in H_{S^1}^{2}(M) \smallsetminus \{0\}$ then for every $1\leq i \leq k$, the label $m_{i,\ell_i}$ of the $(i,\ell_i)$ edge equals $1$. Otherwise $m_{1,\ell_1}$ and $m_{2,\ell_2}$ are
$\sigma_{2,\ell_2} \cdot \tau_{\max}$ and  $\sigma_{1,\ell_1} \cdot \tau_{\max}$ respectively, and for $3 \leq i \leq k$, $m_{i,\ell_i}=1$.
 
\noindent  For $j<\ell_i$, the label of the $(i,j-1)$ edge $m_{i,j-1}$ is 
$(m_{i,j}\sigma_{i,j} \cdot \sigma_{i,j}+m_{i,j+1})$.  
\item The genus label of each of the fat vertices is $g$.
\end{enumerate}
\end{Proposition}

\begin{Remark} \labell{rem:project}
If $\Sigma_{\infty}$ exists, hence $\tau_{\max}=\tau_{\infty}$, the algorithm works, as is, when we replace the elements of the set $\{\tau_{\max}(=\tau_{\infty})\} \cup S \subset H_{S^1}^{2}(M)$ with their images in $H^{2}(M)$ under the map $I^{*} \colon H_{S^1}^{2}(M) \to H^{2}(M)$ of \eqref{surjfiber}, and replace $\tau_{\min}$ with the image of $\tau_0$ under this map (setting $\tau_0$ to be zero if $\Sigma_0$ does not exist). 
This is because the intersection form on the elements of $\{\tau_{\infty}\} \cup S \subset H_{S^1}^{2}(M)$ coincides with the intersection form on their images under $I^*$ in $H^{2}(M)$, by naturality of equivariant and ordinary cup products.
\end{Remark}

\begin{Notation}\labell{notxi}
Consider a Hamiltonian $S^1$-action on $(M,\omega)$. Associate the extended decorated graph  with respect to a generic metric as in \S \ref{decorated1}.
Denote by $k$ the number of chains in the extended decorated graph, and by $\ell_i$ the number of edges in the $i^{\mathrm{th}}$ chain. 
If there is a maximal fixed surface we order the chains such that $m_{1,1} \geq m_{2,1}\geq 1=m_{3,1}=\ldots=m_{k,1}$, where $m_{i,1}$ is the label of the first edge from the bottom in the $i^{\mathrm{th}}$ chain. 
Otherwise, if there is a minimal fixed surface we order the chains such that $m_{1,\ell_1} \geq m_{2,\ell_2}\geq 1=m_{3,\ell_3}=\ldots=m_{k,\ell_k}$, where $m_{i,\ell_i}$ is the label of the first edge from the top in the $i^{\mathrm{th}}$ chain. If $\sharp \text{fixed surfaces} = 0$ then $k=2$.
 If $\sharp \text{fixed surfaces} =2$, then $m_{i,1}=1=m_{i,\ell_i}$ for all $i$, so any order of $\{1,\ldots,k\}$ will work, see Proposition \ref{label1}. If $\sharp \text{fixed surfaces} = 1$ there is such an order, see Proposition \ref{claim:eph}.

Denote by $$x_0,x_{\infty},x_h,x_{i,j}, \, 1\leq i \leq k, \, 1 \leq j \leq \ell_i$$  the images of $\tau_0,\tau_{\infty},\tau_h,\sigma_{i,j} \in  H_{S^1}^{2}(M)$, defined in Notation \ref{not:gens}, under the map 
\begin{equation}\labell{eq:s1 to ordinary}
I^{*} \colon H_{S^1}^{2}(M) \to H^{2}(M)
\end{equation}
restricting to ordinary cohomology.
 If there is no minimal (maximal) fixed surface $\Sigma_0$ ($\Sigma_{\infty}$), we take $x_0$ ($x_{\infty}$) to be the $0$ class. 
 Note that $x_h=\sum_{j=1}^{\ell_i}{m_{i,j}x_{i,j}}$ for all $i$; as a consequence of Theorem~\ref{ThNew},
 this is well defined. 
  For $3 \leq i \leq k$, denote $$z_{i}=
\begin{cases}
x_{i,1} & \text{ if there is exactly one fixed surface and it is maximal}\\
x_{i,\ell_i} & \text{ if there is exactly one fixed surface and it is minimal}\\
0 & \text{ otherwise}.
\end{cases}$$
Denote the set 
$$X=\{x_{i,j}, \, 1\leq i \leq k, \, 1 \leq j \leq \ell_i \} \smallsetminus \{z_i, \, 3\leq i \leq k\}.$$

Since the equivariant and ordinary push forward maps  commute with the map \eqref{surjfiber}, see \eqref{dic},
the classes $x_0$ and $x_{\infty}$ are the Poincar\'e duals of the moment map-preimages of the bottom and top fat vertices; if $\sigma_{i,j}$ corresponds to a robust edge, the class $x_{i,j}$ is the Poincar\'e dual of the invariant embedded sphere whose image is that edge.
By \cite[Theorem 1.1]{KesslerHolm2}, the map \eqref{eq:s1 to ordinary} is surjective. 
Thus Theorem~\ref{ThNew} implies that the classes in $X$ generate the ordinary even-dimensional cohomology of $M$.
\end{Notation}

\begin{Definition}
Define a map $\xi$ from the set of Hamiltonian circle actions on $(M,\omega)$ 
to $(H^{2}(M;\R))^2 \times 2^{H^{2}(M;\R)}$ by
\begin{equation} \labell{tau}
\xi([S^1 \acts (M,\omega)])= (x_0,x_{\infty}) \times X, 
\end{equation}
where the classes $x_0,x_{\infty}$ and the set $X$ are as in Notation \ref{notxi}.
\end{Definition}

\begin{Proposition}\labell{lemunique}
When restricted to maximal Hamiltonian  circle actions on $(M,\omega)$, or to the larger subset of Hamiltonian circle actions on $(M,\omega)$ with $\sharp \text{fixed surfaces} \geq 1$, 
the map \eqref{tau} is one-to-one. 
\end{Proposition}

\begin{proof}
Since  the decorated graph determines the Hamiltonian $S^1$-manifold \cite[Proposition 4.1]{karshon},  it is enough to show that the image of \eqref{tau} determines the decorated graph. Assume without loss of generality that if  $\sharp \text{fixed surfaces} = 1$ then the fixed surface is maximal.
By Remark \ref{rem:project}, 
 the decorated graph sans the height values and the area labels is determined by this image, and we can recover the indices of the elements of $X$. The genus label $g=g(\Sigma)$ is determined by $M$. 
 
 The height values and area labels are determined as follows, as we get directly from Notation \ref{notxi} and \S \ref{decorated}.
The height of the minimal vertex is $0$, the height difference between the minimal vertex and $v_{i,1}$ is $m_{i,1}\frac{1}{2\pi}  \int_{M} x_{i,1} \omega$, the height difference between $v_{i,j}$ and $v_{i,j+1}$ is $m_{i,j+1} \frac{1}{2\pi} \int_{M} x_{i,j+1} \omega$, the height of the maximal vertex is $\frac{1}{2\pi}\int_{M}x_h \omega$.
The area label of the top fat vertex is $\frac{1}{2\pi}\int_{M}x_\infty \omega$. If $x_0 \neq 0$ the area label of the bottom fat vertex is $\frac{1}{2\pi}\int_{M}x_0 \omega$.
\end{proof}

In proving the finiteness theorem, we will use the following
 formul\ae\ for the first Chern class and first Pontryagin class.
To consider $TM \to M$ as a complex vector bundle, equip it with a complex structure $J \colon TM \to TM$  that is compatible with $\omega$.

\begin{Lemma}\labell{c1c2}
For every Hamiltonian circle action on $(M,\omega)$ and $x_0,x_{\infty},x_h,x_{i,j},z_i$ as in Notation \ref{notxi}, we have
\begin{enumerate}
\item $\displaystyle{c_1(TM)=x_0+x_{\infty}+\left(\sum_{i=1}^{k}\sum_{j=1}^{\ell_i}x_{i,j}\right)-(2g+k-2)x_h,}$

\noindent where $k$ is the number of chains in the extended decorated graph and $\ell_i$ is the number of edges 
in the $i^{\mathrm{th}}$ chain; and
\item $\displaystyle{c_{1}(TM)^2-2c_2(TM)=x_0^2+x_{\infty}^2+\left(\sum_{i=1}^{k}\sum_{j=1}^{\ell_i}x_{i,j}^2\right)-(k-2)x_h^2}$
and 

\noindent $\displaystyle{(k-2) x_h^2-\sum_{i=3}^{k}z_i^2=0.}$
\end{enumerate}
\end{Lemma}

The lemma follows from its equivariant counterpart in Proposition \ref{pro:equivchern}, by appying the map $I^* \colon H_{S^1}^{*}(M) \to H^{*}(M)$ of \eqref{surjfiber} to both sides of equations \eqref{eqc1}, \eqref{eqc2} and \eqref{eqc3}.
As discussed in \cite[\S5]{Tu-char:2010}, the equivariant characteristic classes are equivariant extensions 
  of the ordinary characteristic classes.

In what follows, we will need the following characterization
of the image of the fiber class
under the inclusion from the cohomology of the minimal model $M_\Sigma$ into the cohomology of its $k$-fold blowup.

\begin{Lemma} \labell{lemf}
For a class $A \in H_{2}(M_{\Sigma};\Z)$, assume that its image under the inclusion into $H_{2}(M;\Z)$ satisfies the following conditions.
\begin{itemize}
\item Its self intersection number is zero.
\item Its coupling with the first Chern class $c_{1}(TM)$ equals two.
\item Its $\omega$-symplectic area  is positive.
\end{itemize}
Then 
\begin{enumerate}
\item if $g(\Sigma)>0$, then $A=F$.
\item if $g(\Sigma)=0$ then either $A=F$ or $M_{\Sigma}=S^2 \times \Sigma$ and $A$ is the base class $B$. 
\end{enumerate}
\end{Lemma}
Part (1) is proven in \cite[Lemma 4.7]{KesslerHolm1}; part (2) is a slight modification of that proof. For completeness, we prove the lemma here.

\begin{proof}
Denote $\widehat{B}=B$ if  $M=(\Sigma \times S^2)_{k}$ and 
$\widehat{B}=B_{1}$ if $M=(M_{\Sigma})_{k}$.
Write $A\in H_{2}(M_{\Sigma};\Z)$ as $A=p\widehat{B}+qF$ for $p,q \in \Z$. 
Since $A \cdot A=0$, $\widehat{B} \cdot F=1$, $F \cdot F=0$ and 
$$\widehat{B} \cdot \widehat{B}=\begin{cases} 0 & \text{ if }\widehat{B}=B\\ 1 & \text{ if }\widehat{B}=B_{1}\end{cases},$$ 
we get that  if $M_{\Sigma}=\Sigma \times S^2$ then $2pq=0$, i.e., either $p=0$ or $q=0$, and  if $M_{\Sigma}$ is the non-trivial $S^2$-bundle then $0=p^2+2pq=p(p+2q)$, i.e.,  either $p=0$ or $p+2q=0$. 
In the trivial bundle case, if $q=0$ then by the second property of $A$ we have  $2=c_{1}(TM)(A)=(2-2g)p$. Similarly, in the non-trivial bundle case, if $p+2q=0$ then $2=c_{1}(TM)A=(2-2g)p+p+2q=(2-2g)p$. 

If $g$ is a positive integer, the equality $2=(2-2g)p$ holds only if $g=2$ and $p=-1$, however, if $M_{\Sigma}=\Sigma \times S^2$ this (and $q=0$) yield that  $\omega(A)= -\omega(\widehat{B})<0$ contradicting the third condition; if $M_{\Sigma}$ is the non-trivial bundle, this (and $p+2q=0$) yield that $2q=1$ contradicting the fact that $q$ is an integer.
 We conclude that $p=0$ hence $2q=c_{1}(TM)(A)=2$, i.e., $q=1$. 
 
 If $g=0$ then $2=(2-2g)p$ implies that $p=1$. If $M_{\Sigma}$ is the trivial bundle then
 this (and $q=0$)  yield that  $A=B$; if $M_{\Sigma}$ is the non-trivial bundle we get $1+2q=p+2q=0$ contradicting $q \in \Z$. Otherwise $p=0$ hence (since $2q=c_{1}(TM)(A)=2$) $q=1$.
\end{proof}

We next turn to a key estimate that will allow us to show that the generators must lie in a bounded region in the (equivariant) 
cohomology algebra with real coefficients, thought of as a normed vector space.

\begin{Lemma}\labell{xhbounded}
There exists a positive constant $C_h(=C_h(M,\omega))$ such that for every maximal Hamiltonian circle action on $(M,\omega)$, for the class $x_h$ and the $k-2$ classes $z_3,\ldots,z_{k}$ associated to the action we have 
\begin{equation} \labell{eq:xhbound}
\int_{M}((2g+k-2)x_h-\sum_{i=3}^{k}z_i)\omega \leq C_h.
\end{equation}
\end{Lemma}

\begin{proof}
We will prove \eqref{eq:xhbound} in each of the two possible cases 
of maximal Hamiltonian circle actions described in at the beginning of the section.

In Case I, the circle action is obtained from a ruled circle action on  a 
symplectic $S^2$-bundle $(M'_{\Sigma},\omega'_{\Sigma})$ with fiber 
class $F'$ by a sequence of $r$ equivariant blowups at either a fixed 
surface or a non-extremal isolated fixed point. Such blowups do not 
affect the fiber edge and the height difference between the fat vertices. 
Therefore $x_h=F'$ and
$$
\omega(x_h)=\omega(F')=2 \pi \cdot \left(\text{ the height difference between the fat 
vertices }\right)=\omega'_{\Sigma} (F').
$$
The surface $\Sigma$ is determined by the genus $g$.
Moreover, we can assume that as 
smooth, K\"ahler
$S^2$-bundles $M_{\Sigma}=M'_{\Sigma}$. 
If $\dim H_2(M)>2$ then by \cite[Lemma 5.1]{KesslerHolm1},
 we can assume that the sequence of equivariant 
 blow downs along invariant embedded symplectic spheres 
 in $E'_k$, then in $E'_{k-1}$, and so on, results in a ruled  
 $S^1$-action on $S^2 \times \Sigma=M'_{\Sigma}$, otherwise 
 replace the last blow down with a blow down along an invariant 
 embedded symplectic sphere in $F'-E'_1$; we can similarly assume 
 that $M_{\Sigma}=S^2 \times \Sigma$.
  If $\dim H_{2}(M)=2$ then if $M$ is the non trivial bundle there are classes 
  of odd self intersection in $H_{2}(M)$, while in the trivial bundle there are not.
  By Lemma \ref{lemf},
 if $g>0$ then $F'=F$; if $g=0$ then either $F'=F$ or $F'$ is the base class 
 $B$ in $H_{2}(M_{\Sigma})=H_{2}(S^2 \times S^2)$, (note that, since the 
 action is maximal, if $g=0$ then $\dim H_{2}(M)>2$ so we can 
 assume $M_{\Sigma}=S^2 \times \Sigma$).

 We conclude that 
$$\text{if }g>0, \, \int_{M}{x_h} \omega = \omega(F); \, \, \int_{M}(2g+k-2){x_h} \omega \leq (2g+N)\omega(F)=:C_h;$$
$$\text{if }g=0, \, \int_{M}{x_h} \omega \leq \max\{\omega(F),\omega(B)\};  \, \, \int_{M}(k-2){x_h} \omega \leq N\max\{\omega(F),\omega(B)\}=:C_h,$$
where $N$ is as in Notation \ref{notm}.
(In this case $z_i=0$ for all $i$.)

Now turning to Case II, 
we must have  exactly one fixed sphere and the action is obtained by a 
sequence of more than one equivariant blowups from an $S^1$-action 
that extends to a toric action. without loss of generality, the fixed sphere 
is maximal. We now apply ``reverse induction": blow up (by a small enough size) 
at the isolated minimum with weights $m_{1,1},m_{2,1}$ such that 
$m_{1,1}>m_{2,1}$ to get a new minimum with weights 
$m_{1,1},m_{1,1}-m_{2,1}$.  We repeat until we get a minimal point with 
both weights equal $1$; in the next blowup it is replaced by a minimal 
sphere.
We get a circle action of 
Case I on the symplectic manifold $(\tilde{M},\tilde{\omega})$  
obtained from $(M,\omega)$ by the sequence of symplectic blowups.

Note that all the new vertices created in this process are on the first two chains.
The effect of each blowup in the process on each of the $k-2$ chains is by reducing the size of the first edge (that is ephemeral unless the label of the first edge on the second chain equals $1$). Such an edge has label $1$ through the process, and is robust at the end of the process. 
Therefore, for the class $\tilde{x}_h$ associated to the obtained circle action on the obtained $(\tilde{M},\tilde{\omega})$ we have $$\int_{M} x_h \omega-\int_{\tilde{M}}\tilde{x_h}\tilde{\omega}=\int_{M}{z_i}\omega-\int_{\tilde{M}}\tilde{x}_{i,1} \tilde{\omega}$$
for $1 \leq i \leq k-2$.
Therefore
\begin{eqnarray*}
\int_{M}{((k-2)x_h-\sum_{i=3}^{k}z_i)\omega} &=& \int_{\tilde{M}}{((k-2)\tilde{x}_h-\sum_{i=3}^{k}\tilde{x}_{i,1})\tilde{\omega}}\\
 &\leq & \int_{\tilde{M}}{((k-2)\tilde{x}_h)\tilde{\omega}} \\
 & \leq&  N\max\{\tilde{\omega}(F),\tilde{\omega}(B)\}\\
 & =&  N\max\{\omega(F),\omega(B)\}=:C_h.\\
\end{eqnarray*}
The first inequality is a consequence of the fact that the $\tilde{x}_{i,1}$s are 
Poincar\'e dual to embedded $\tilde{\omega}$-symplectic spheres and hence 
their coupling with $\tilde{\omega}$ is positive. The second inequality
follows from Case I. 
Now because $(\tilde{M},\tilde{\omega})$ is obtained from $(M,\omega)$ 
by symplectic blowups, hence is obtained from $(M_{\Sigma},\omega_{\Sigma})$ 
by symplectic blowups, we have that $\tilde{\omega}(F)=\omega_{\Sigma}(F)=\omega(F)$ and $\tilde{\omega}(B)=\omega_{\Sigma}(B)=\omega(B)$.
\end{proof}

\begin{Remark}
Lemma \ref{xhbounded} is not true if we do not restrict to maximal actions. For example, for every $m \in \N$ there is a non-maximal Hamiltonian circle action on $(\CP^2,\omega_{\FS})$ whose graph has three edges of labels $m,1,m+1$, and each edge is the image of an invariant embedded symplectic sphere with symplectic size $1$; the symplectic size $\frac{1}{2\pi}\int_M x_h \omega=m+1$.
\end{Remark}

We now turn to the main theorem of this section, that there are only finitely many maximal $S^1$ and
$S^1\times S^1$ actions on a four-dimensional symplectic manifold $(M,\omega)$.
The main idea in the proof is to use the Hodge Index Theorem, as in McDuff and Borisov's proof for
the finiteness of toric actions \cite[Theorem 1.2]{mcduff:finite}. 
We need the additional estimates described above to implement this method for circle actions on
four-dimensional manifolds.

\begin{proof}[Proof of Theorem \ref{finite}]

The finiteness of the toric actions is by
 \cite[Proposition 3.1]{mcduff:finite}. 
So 
 it is enough to show that the number of maximal Hamiltonian circle actions is finite.
By Proposition~\ref{lemunique},
it is enough to show that there is a finite subset $B$ of $H^{2}(M,\R)$, such that for every maximal Hamiltonian circle action on $(M,\omega)$, the set $X \cup \{x_0,x_{\infty}\}$ (as in Notation \ref{notxi}) that corresponds to the action is a subset of $B$.
Since these are integral classes, it is enough to show that they are all contained in a bounded subset of  $H^{2}(M,\R)$.

Consider a maximal Hamiltonian $S^1 \acts (M,\omega)$. There 
is an integrable  complex structure $J$ on $M$ such that $(M,\omega,J)$ is K\"ahler, the action is holomorphic, and the associated metric $\omega(\cdot,J\cdot)$ is generic \cite[Theorem 7.1]{karshon}. This allows us to apply the  Hodge index theorem, and get that the Hodge-Riemann form
$$\langle \alpha,\beta\rangle:=\int_{M}\alpha \, \wedge \, \beta $$ on $H^{1,1}(M,J) \cap H^{2}(M;\R)$
is nondegenerate of type $(1,-1,\ldots,-1)$, i.e.,
is negative definite on the orthogonal complement to $[\omega]$.

The spheres and surfaces whose images are fat vertices and robust edges of the extended decorated graph  with respect to a generic metric are holomorphic curves in $(M,J)$, see Remark \ref{rem:spheres}. Hence the set $X_s$ of their Poincar\'e  duals 
  is contained in $H^{1,1}(M,J) \cap H^{2}(M;\R)$, see \cite[pp.\ 162--163]{GH}.
Enumerate the elements of the set $X \cup \{x_0,x_{\infty}\}$ that is associated with the $S^1$-action as $\{x_n\}$, with $n=1,\ldots,2+|X|$; write 
$$x_n=y_n+r_n[\omega], \text{ where }\langle y_n,\omega\rangle =0 \text{ and }r_n \in \R.$$
By the Hodge index theorem
$\langle y_n,y_n \rangle\leq 0.$
Since each of the $x_n$s is the Poincar\'e dual to the class of a symplectic sphere or surface
 $S_n$, 
$$r_n\langle \omega, \omega\rangle=\langle x_n,\omega \rangle=\int_{S_n}\omega  \geq 0.$$ 
By item (1) of Lemma \ref{c1c2} and Lemma \ref{xhbounded}, 
$$\sum_{n}\int_{M}x_n \omega = \int_{M}c_{1}(TM) \omega + \int_M ((2g+k-2) x_h -\sum_{i=3}^{k}z_i) \omega \leq  \int_{M}c_{1}(TM) \omega +C_h=:C,$$
where $x_h, z_3,\ldots,z_{k}$ are as in Notation \ref{notxi}.
So
$\sum r_n\langle \omega,\omega \rangle$ is bounded from above by the constant $C$. 
We now assume, without loss of generality, that $\langle \omega,\omega \rangle =\int_{M} \omega^2 = 1$, normalizing $\omega$ and the constants if necessary.
Hence $0\leq r_n\leq C$ for all $n$, and $\sum {r_n}^2\leq  (2+|X|) C^2 \leq N C^2$, where $N$ is as in Notation \ref{notm}. By item (2) of Lemma \ref{c1c2},
$$\sum_{n}\int_{M}{x_n}{x_n}=\int_{M}({c_1}^2-2c_2)=:A,$$
so
 $\sum{r_n}^2+\sum\langle y_n,y_n \rangle=\sum\langle x_n,x_n \rangle$ equals the constant $A$, hence $0 \leq -\sum \langle y_n,y_n \rangle \leq NC^2-A$. 

We conclude that for every maximal Hamiltonian $S^1$-action, each of the elements in the associated set $X \cup \{x_0,x_{\infty}\}$ is in the bounded subset 
$$\{y+r[\omega]\, :\, 0\leq - \langle y,y \rangle \leq NC^2-A \text{ and }0\leq r \leq C \}$$
of $H^{2}(M;\R)$.
\end{proof}

In fact, we have proven the following result, analogous to the McDuff-Borisov result  
\cite[Theorem~1.2]{mcduff:finite}.

\begin{Theorem}
Let $R$ be a commutative ring of finite rank with even grading, and write $R_{\R} := R \otimes _{\Z} \R$. Fix elements $[\omega] \in R_{\R}$ and $c_1,c_2 \in R$ 
of degrees $2$, $2$ and $4$ respectively. 
Then, for any non-negative integer $g$, there are, up to equivariant symplectomorphism, 
at most finitely many maximal Hamiltonian $S^1$-actions on a fixed, compact, connected, four-dimensional  symplectic manifold $(M,\omega)$ for which 
\begin{itemize}
\item there is a ring isomorphism $\Psi \colon H^{2*}(M;\Z) \to R$ that takes the symplectic class and the Chern classes $c_{i}(M)$, $i=1,2$, to the given elements $[\omega] \in R_{\R}, \, c_i \in R$, 
\item the genus of the fixed surfaces in $S^1 \acts M$, if they exist, is $g$.
\end{itemize}
\end{Theorem}

\newpage
\appendix

\section{Preliminaries on equivariant cohomology} \labell{app:A}

In the case of circle actions, we  consider the \emph{classifying bundle}
$ES^1 := S^{\infty}$ as the unit sphere in an infinite
dimensional complex Hilbert space $\C^\infty$.  This space is contractible 
and equipped with a free $S^1$-action by
coordinate multiplication.  We define
$$
H_{S^1}^*(M) = H_{S^1}^{*}(M;\Z):=H^*((M\times ES^1)/S^1; \Z),
$$
where $S^1\acts (M\times ES^1)$ diagonally. 
The classifying space is $BS^1 = ES^1/S^1 = \CP^\infty$.
The equivariant cohomology of a point is
\begin{equation} \label{eqhs1pt}
H^*_{S^1}(\pt) = H^*(BS^1;\Z) = H^*(\CP^\infty;\Z) = \Z[t], \,\,\,\deg(t)=2.
\end{equation}

\begin{Remark} \labell{finiteappro}
We can interpret $ES^1=S^{\infty}$ as the direct limit of odd-dimensional spheres $S^{2k+1} \subset {\C}^{k+1}$ with respect to the natural inclusions, and $BS^1=\CP^{\infty}= \lim \limits_{\longrightarrow} {\CP^k}$.
Then $(M \times ES^1) / S^1$ is a direct limit of $(M \times S^{2k+1})/ S^1 $. For every degree $q$ we have 
$H_{S^1}^{q}(M)=H^{q}((M \times S^{2k+1})/ S^1;\Z)$ for all sufficiently large $k$. See, e.g., \cite[Example C.1]{GGK}.
\end{Remark}

If we endow a point $\pt$ with the trivial $S^1$-action, then the constant map
$$\pi:M\to \pt$$ is equivariant.  This induces a map
in equivariant cohomology 
\begin{equation} \labell{pi*}
\pi^*:H_{S^1}^*(\pt) \to H_{S^1}^*(M).
\end{equation}

Consider the map \begin{equation} \labell{surjfiber}
 I^{*} \colon H_{S^1}^{*}(M;\Z) \to H^{*}(M;\Z),
\end{equation}
induced by the
 fiber inclusion $I \colon M \to (M\times ES^1)/S^1$.
By \cite[Theorem 1.1]{KesslerHolm2}, in case $M$ is  a compact connected symplectic four-manifold, and $S^1\acts M$ is Hamiltonian, the map
\eqref{surjfiber}
is a surjection.
Moreover, the sequence of maps
\begin{equation*} 
0 \to M \to (M\times ES^1)/S^1 \to BS^1 \to 0
\end{equation*}
induces
 a short exact sequence

\begin{equation*} 
0 \leftarrow H^{*}(M) \leftarrow H_{S^1}^{*}(M) \leftarrow \langle \pi^{*}t \rangle \leftarrow 0.
\end{equation*}
Hence 
\begin{equation*} 
H^{*}(M)=H_{S^1}^{*}(M)/\langle \pi^{*}t \rangle.
\end{equation*}
Moreover, \eqref{surjfiber} gives a natural map between ``ordinary" invariants, e.g., Chern numbers, to ``equivariant" ones.
The equivariant Euler class is denoted $e_{S^1}$ and is defined to be the ordinary Euler class of $ES^1$.

Recall that the map \eqref{pi*} endows
$H_{S^1}^*(M)$ with a $H_{S^1}^*(\pt)=H^*(BS^1;\Z)$-algebra structure. 
Due to \eqref{eqhs1pt}, 
one can find the algebra-structure of $H_{S^1}^{*}(M)$ over $H^{*}(BS^1;\Z)$ if one knows how elements in $H^{2}(BS^1;\Z)$ map to $H_{S^1}^{2}(M)$ by $\pi^{*}$,
 and in particular if one knows $\pi^{*}(t)$.

An  $S^1$-equivariant continuous map of compact oriented $S^1$-manifolds, $f \colon N \to M$  induces 
 the {\bf equivariant pushforward map}
$$f^{!} \colon H_{S^1}^{*}(N) \to H_{S^1}^{*-n+m}(M),$$
where $n=\dim N, \, m=\dim M$, as follows.
 For $q, \, k \in \N$,  
 we have the push-forward homomorphism 
$ H^{q}(N \times S^{2k+1}/S^1) \to H^{q-n+m}(M \times S^{2k+1}/S^1)$
defined by
$$
\xymatrix{
    H^{q}(N \times S^{2k+1}/S^1) \ar[d]  & H^{q+m-n}(M \times S^{2k+1}/S^1) \\
    H_{n-q}(N \times S^{2k+1}/S^1)\ar[r] & H_{n-q}(M \times S^{2k+1}/S^1) \ar[u]
}$$
where the vertical maps are the Poincar\'e duality homomorphisms and the horizontal one is the map induced by $f$ on homology.
To define the equivariant push-forward map $f^{!}$ take $k$ large enough such that these cohomology spaces are equal to the equivariant cohomology of $M$ and $N$, see Remark \ref{finiteappro}. The push-forward is independent of $k$.
This map is sometimes called the {\bf equivariant Gysin homomorphism}. 
We have the following commutative diagram
 \begin{equation} \labell{dic}
\begin{CD}
H_{S^1}^{*}(N) @> f^{!} >> H_{S^1}^{*-n+m}(M)\\
@V \pr VV @V \pr VV\\
H^{*}(N) @> f^{!} >> H^{*-n+m}(M)
\end{CD}
\end{equation}
where the top pushforward map is equivariant and the bottom one is standard, and the $\pr$ map is as in \eqref{surjfiber}.

We denote 
\begin{equation} \labell{eqpi}
\int_{N}:={\pi}^{!} \colon H_{S^1}^{*}(N) \to H_{S^1}^{*-\dim N}(\pt), \text{ where }\pi \colon N \to \pt
\end{equation}
is the constant map. We will denote the push forward of $\pi$ in standard cohomology by the same notation. We similarly define the equivariant pushforward map induced by $(S^1)^k$-equivariant maps.

 For an $S^1$-invariant embedded surface $\iota_{\Sigma} \colon \Sigma \to M$ in a four-dimensional $M$, 
the Poincar\'e dual of $\Sigma$ as an equivariant cycle in $M$, i.e.,
$\iota_{\Sigma}^{!}(\mathbb{1})$, where $\mathbb{1} \in H_{S^1}^{0}(\Sigma)$, is a class in $H_{S^1}^{2}(M)$. Its pullback under $\iota_{\Sigma}$ is  the equivariant Euler class  of the normal $S^1$-vector bundle of $\Sigma$ in $M$:
$$
\iota_{\Sigma}^*\left(\iota_{\Sigma}^{!}(\mathbb{1})\right) = e_{S^1}(\nu(\Sigma \subset M)).
$$

For an action of a torus $T=(S^1)^k$ on a compact manifold $M$, the \emph{Atiyah-Bott/ Berline-Vergne (ABBV) localization formula}
\cite{AB, BV:1982} expresses 
the equivariant pushforward of any class $\alpha\in H_T^*(M;\Q)$
under $\pi \colon M \to \pt$
as a sum 
\begin{equation}\label{eq:abbv}
\pi^{!}(\alpha) = \sum_{{F\subseteq M^T}}  \left({\pi|_{F}}\right)^{!}\left(\frac{\alpha|_F}{e_T(\nu(F\subseteq M)}\right).
\end{equation}
of the equivariant pushforwards under $\pi_F$ over
the connected components $F$ of the fixed point set $M^T$.
Here
$\alpha|_F$ is the restriction of $\alpha$ to $F$, and $e_T(\nu(F\subseteq M))$ is the equivariant Euler class of the normal bundle of $F$.
(We must use $\Q$ coefficients because Euler classes are inverted.)

\newpage

\pagestyle{empty}
\begin{landscape}

\section{Data on equivariant cohomology read from the decorated graph}\labell{app:fixed}

\subsection{Restrictions to connected components of the fixed point set}
 In the following three tables, we list
all {\bf non-zero} restrictions of each generator to the fixed point set.  Restrictions to any other vertices not listed in the table are zero. 
For these restriction calculations, see \S \ref{NoTitinj}.

\subsection*{The case of two fixed surfaces.}
Here, $\sigma_{i,j}$ is supported on the isolated vertices $v_{i,j-1}$ and $v_{i,j}$ for $j=2,\dots,\ell_i-1$.  
When $j=1$ or $j=\ell_i$, $\sigma_{i,j}$ is supported on one isolated vertex and on the minimum or maximum fixed surface 
$\Sigma_0$ and $\Sigma_\infty$ respectively.
In those cases, the $\nexists$ in the table indicates when there is no isolated maximum or minimum fixed point.

\renewcommand{\arraystretch}{1}

\begin{equation} \labell{Table-rest-2}
\begin{array}{c|c|c|c|c|}
& F=\Sigma_{\infty} & F=\Sigma_{0} &F=v_{i,j}&
F=v_{i,j-1} \\ \hline\hline
 \varepsilon_F|_F & -1\otimes t - e_{\max}[\Sigma]\otimes 1 & -1\otimes t + e_{\min}[\Sigma]\otimes 1 & - m_{i,j}m_{i,j+1}\cdot t^2
 &  - m_{i,j-1}m_{i,j}\cdot t^2 \\ \hline
 \varepsilon_F^{-1}|_F & -1\otimes \frac1t + e_{\max}[\Sigma]\otimes \frac{1}{t^2} & -1\otimes \frac1t - e_{\min}[\Sigma]\otimes \frac{1}{t^2} & - \frac{1}{m_{i,j}m_{i,j+1}\cdot t^2}
 &  -\frac{1}{m_{i,j-1}m_{i,j}\cdot t^2} \\ \hline
 \tau_0|_F & 0 & -1\otimes t + e_{\min}[\Sigma]\otimes 1 & 0& 0\\ \hline
 \tau_\infty|_F &  1\otimes t + e_{\max}[\Sigma]\otimes 1 & 0 & 0& 0\\ \hline
 \tau_h|_F & [\Sigma]\otimes 1 & [\Sigma]\otimes 1 & 0& 0\\ \hline
\mathrm{\begin{tabular}{c}$\sigma_{i,1}|_F$\\if $1 <\ell_i$ \end{tabular}} & 0 & [\Sigma]\otimes 1 & -m_{i,2}\cdot t & \nexists \\ \hline
\mathrm{\begin{tabular}{c}$\sigma_{i,\ell_i}|_F$\\if $1 <\ell_i$ \end{tabular}}& [\Sigma]\otimes 1 & 0 &  \nexists & m_{i,\ell_i-1}\cdot t \\ \hline
\mathrm{\begin{tabular}{c}$\sigma_{i,1}|_F$\\if $1 = \ell_i$ \end{tabular}} & [\Sigma]\otimes 1 & [\Sigma]\otimes 1 & 0& 0\\ \hline
\mathrm{\begin{tabular}{c}all other\\ $\sigma_{i,j}|_F$\end{tabular}} & 0 & 0 & -m_{i,j+1}\cdot t & m_{i,j-1}\cdot t \\ \hline
\end{array}
\end{equation}

\newpage

\subsection*{The case of no fixed surfaces.}
Here, $\sigma_{i,j}$ is supported on the isolated vertices $v_{i,j-1}$ and $v_{i,j}$ for $j=2,\dots,\ell_i-1$.  
When $j=1$ or $j=\ell_i$, $\sigma_{i,j}$ is supported on one interior isolated vertex and on the minimum or maximum isolated fixed 
point $v_0$ or $v_\infty$ respectively.
In those cases, the $\mathrm{n.a.}$  in the table indicates when the support includes $v_0$ or $v_\infty$.

\begin{equation} \labell{Table-rest-0}
\begin{array}{c|c|c|c|c|}
& F=v_\infty & F=v_0 & F=v_{i,j} & F=v_{i,j-1} \\ \hline\hline
\phantom{\mbox{\Large I}} \varepsilon_F|_F \phantom{\mbox{\Large I}} & m_{1,\ell_1}m_{2,\ell_2}\cdot t^2 & m_{1,1}m_{2,1}\cdot t^2& - m_{i,j}m_{i,j+1}\cdot t^2
 &  - m_{i,j-1}m_{i,j}\cdot t^2 \\ \hline
\phantom{\mbox{\Large I}} \varepsilon_F^{-1}|_F  \phantom{\mbox{\Large I}}& \frac{1}{ m_{1,\ell_1}m_{2,\ell_2}\cdot t^2} & \frac{1}{ m_{1,1}m_{2,1}\cdot t^2} & - \frac{1}{m_{i,j}m_{i,j+1}\cdot t^2}
 &  -\frac{1}{m_{i,j-1}m_{i,j}\cdot t^2} \\ \hline
\tau_h|_F & m_{1,\ell_1}m_{2,\ell_2} \cdot t &  - m_{1,1}m_{2,1} \cdot t & \mathrm{n.a.} & \mathrm{n.a.} \\ \hline
\mathrm{\begin{tabular}{c}$\sigma_{1,1}|_F$\\if $1 <\ell_1$ \end{tabular}}  & 0 &  -m_{2,1}\cdot t & -m_{1,2}\cdot t & \mathrm{n.a.} \\ \hline
\mathrm{\begin{tabular}{c}$\sigma_{1,\ell_1}|_F$\\if $1 <\ell_1$ \end{tabular}}  & m_{2,\ell_2}\cdot t  & 0 & \mathrm{n.a.} & m_{1,\ell_1-1}\cdot t \\ \hline
\mathrm{\begin{tabular}{c}$\sigma_{2,1}|_F$\\if $1 <\ell_2$ \end{tabular}}  & 0& -m_{1,1}\cdot t  & -m_{2,2}\cdot t & \mathrm{n.a.} \\ \hline
 \mathrm{\begin{tabular}{c}$\sigma_{2,\ell_2}|_F$\\if $1 <\ell_2$ \end{tabular}}   &  m_{1,\ell_1}\cdot t &0& \mathrm{n.a.} & m_{2,\ell_2-1}\cdot t \\ \hline
 \mathrm{\begin{tabular}{c}$\sigma_{1,1}|_F$\\if $1 =\ell_1$ \end{tabular}}  & m_{2,\ell_2}\cdot t  & -m_{2,1}\cdot t & \mathrm{n.a.} & \mathrm{n.a.} \\ \hline
  \mathrm{\begin{tabular}{c}$\sigma_{2,1}|_F$\\if $1 =\ell_2$ \end{tabular}}  & m_{1,\ell_1}\cdot t  & -m_{1,1}\cdot t & \mathrm{n.a.} & \mathrm{n.a.} \\ \hline
 \mathrm{\begin{tabular}{c}all other\\ $\sigma_{i,j}|_F$\end{tabular}} & 0 & 0  & -m_{i,j+1}\cdot t & m_{i,j-1}\cdot t \\ \hline
\end{array}
\end{equation}

\newpage

\subsection*{The case of one fixed surface.}
Here, $\sigma_{i,j}$ is supported on the isolated vertices $v_{i,j-1}$ and $v_{i,j}$ for $j=2,\dots,\ell_i-1$.  
The minimum vertex of the labeled graph is denoted $v_0$, and for $\sigma_{i,1}$, its support is $v_0$ and $v_{i,1}$; we say that ``$v_{i,j-1}$ is not applicable" (denoted $\mathrm{n.a.}$ in the table).  The support of $\sigma_{i,\ell_i}$ is $v_{i,\ell_i-1}$ and $\Sigma_\infty$; in this case,
the isolated vertex $v_{i,\ell_i}$ does not exist (denoted $\nexists$ in the table).

\begin{equation} \labell{Table-rest-1}
\begin{array}{c|c|c|c|c|}
& F=\Sigma_{\infty} & F=v_0 &  F=v_{i,j-1} & F=v_{i,j}  \\ \hline\hline
\phantom{\mbox{\Large I}} \varepsilon_F|_F \phantom{\mbox{\Large I}} & -1\otimes t - e_{\max}[\Sigma]\otimes 1 & m_{1,1}m_{2,1}\cdot t^2 & - m_{i,j}m_{i,j+1}\cdot t^2
 &  - m_{i,j-1}m_{i,j}\cdot t^2 \\ \hline
\phantom{\mbox{\Large I}}  \tau_\infty|_F \phantom{\mbox{\Large I}} &  1\otimes t + e_{\max}[\Sigma]\otimes 1 & 0&  0& 0\\ \hline
\phantom{\mbox{\Large I}}  \tau_h|_F \phantom{\mbox{\Large I}} & [\Sigma]\otimes 1 & -m_{1,1}m_{2,1}\cdot t & 0 & 0\\ \hline
\mathrm{\begin{tabular}{c}$\sigma_{1,1}|_F$\\if $1 <\ell_1$ \end{tabular}} & 0 & -m_{2,1}\cdot t & \mathrm{n.a.} & -m_{1,2}\cdot t\\ \hline
\mathrm{\begin{tabular}{c}$\sigma_{2,1}|_F$\\if $1 <\ell_2$ \end{tabular}}& 0 & -m_{1,1}\cdot t & \mathrm{n.a.} & -m_{2,2}\cdot t \\ \hline
\mathrm{\begin{tabular}{c}$\sigma_{i,1}|_F$\\if $i>2$ \end{tabular}}  & 0 &  -m_{1,1}m_{2,1}\cdot t  &\mathrm{n.a.} &  -m_{i,2}\cdot t  \\ \hline
\mathrm{\begin{tabular}{c}$\sigma_{i,\ell_i}|_F$\\if $1 <\ell_i$ \end{tabular}}  & [\Sigma]\otimes 1 & 0 & m_{i,\ell_i-1}\cdot t &  \nexists \\ \hline
\mathrm{\begin{tabular}{c}$\sigma_{1,1}$\\if $1 =\ell_1$ \end{tabular}}  & [\Sigma]\otimes 1 &  -m_{2,1}\cdot t  & \mathrm{n.a.}  &  \nexists \\ \hline
\mathrm{\begin{tabular}{c}$\sigma_{2,1}|_F$\\if $1 =\ell_2$ \end{tabular}}  & [\Sigma]\otimes 1 &  -m_{1,1}\cdot t  & \mathrm{n.a.}  &  \nexists \\ \hline
 \mathrm{\begin{tabular}{c}all other\\ $\sigma_{i,j}|_F$\end{tabular}} & 0&   0 & m_{i,j-1}\cdot t & -m_{i,j+1}\cdot t \\ \hline
\end{array}
\end{equation}

\newpage

\subsection{Ranks of  annihilators in degrees $2$ and $4$}
We list linearly independent generators of the annihilators in degrees $2$ and $4$, as algebras over $\Z$,  for classes in $H_{S^1}^{2}(M)$ and in $H_{S^1}^{4}(M)$. To determine if a class is in the annihilator, we use 
 Tables \ref{Table-rest-2}, \ref{Table-rest-0} and \ref{Table-rest-1}.
 In particular, in determining the annihilators of linear combinations of the $\sigma_{i,j}$s, we use 
 \eqref{eq:labelproducts}.

To determine linear dependence we rely on the fact that the linear relations in $H_{S^1}^{2*}(M)$ are given by 
$\tau_h=\sum_{j=1}^{\ell_i} m_{i,j}\sigma_{i,j},$
where the index $i$ varies over the chains in the extended decorated graph of the Hamiltonian $S^1 \acts M$.

\renewcommand{\arraystretch}{1}

\subsubsection*{The case $\fat=2$.}

\begin{equation} \labell{Table-ann-2-1a}
\begin{array}{c|c|c}
\alpha & \text{ generators for  }\Ann^2(\alpha) & \rank \Ann^2(\alpha) \\  \hline\hline

\tau_0, \, \, \tau_0 \cup \tau_0 &\mathrm{\begin{tabular}{c} $\tau_\infty, \, \sigma_{i,j} \, j>1$
\end{tabular}} &\iso+1 \\ \hline

 \tau_\infty, \, \, \tau_\infty \cup \tau_\infty &\tau_0,\,\sigma_{i,j} \, j<\ell_i &\iso+1\\ \hline

\tau_h  & \tau_h, \, \sigma_{i,j} \, \, j<\ell_i &\iso+1 \\ \hline

\tau_h \cup \tau_0 & \tau_\infty,\, \tau_h, \,\sigma_{i,j} \,\, j>1 &\iso+2 \\ \hline

\tau_h \cup \tau_\infty &  \tau_0, \,\tau_h, \,\sigma_{i,j} \,\, j<\ell_i &\iso+2 \\ \hline

 \sigma_{i,j}, \, \, \sigma_{i,j} \cup \sigma_{i,j} \,\,\, 1<j<\ell_i  &\mathrm{\begin{tabular}{c}$\tau_0,\,\tau_\infty,\,\tau_h$, \\ $\sigma_{k,m} \,\, k \neq i \text{ and }m<\ell_k$,\\ $\sigma_{i,m} \,\, m\neq j-1,j,j+1$ \end{tabular}} &\iso+1 \\ \hline
 
 \sigma_{i,1} \,\,\, \ell_i\neq 1  &\mathrm{\begin{tabular}{c}$\tau_\infty,\,\tau_h$, \\ $\sigma_{k,m} \,\, k \neq i \text{ and }m<\ell_k$,\\ $\sigma_{i,m} \,\, m\neq 1,2$ \end{tabular}}&\iso+1\\ \hline

  \sigma_{i,\ell_i}  \,\,\, \ell_i \neq 1 &\mathrm{\begin{tabular}{c}$\tau_0,\, \tau_h$, \\ $\sigma_{k,m} \,\, k \neq i \text{ and }m<\ell_k$,\\ $\sigma_{i,m} \,\, m\neq \ell_i-1,\ell_i$ \end{tabular}}&\iso+1 \\ \hline  
  
    \sigma_{i,j} \cup \sigma_{i,j+1} \, \,\, j<\ell_i   &\mathrm{\begin{tabular}{c}$\tau_0,\, \tau_\infty,\, \tau_h$, \\ $\sigma_{k,m} \,\, k \neq i \text{ and }m<\ell_k$,\\ $\sigma_{i,m} \,\, m\neq j,j+1$ \end{tabular}} &\iso+2 \\ \hline

 a_0 \tau_0 + a_h \tau_h \,\,\, a_0, \, a_h \neq 0&  \sigma_{i,j} \,\, j>1
&\iso \\ \hline

a_\infty \tau_\infty + a_h \tau_h \,\,\, a_\infty, \, a_h \neq 0 &\sigma_{i,j} \, \, j<\ell_i&\iso \\ \hline

\mathrm{\begin{tabular}{c} $a_0 \tau_0 + a_\infty \tau_\infty \,\,\, a_0, \, a_\infty \neq 0$ or \\
$a_0\tau_0 +a_h \tau_h+a_\infty \tau_\infty \,\,\, a_0, \, a_h, \, a_\infty \neq 0$
\end{tabular}}& \mathrm{\begin{tabular}{c}$\sigma_{i,j} \,\, 1< j<\ell_i,$\\
 $ \sigma_{i^*,1}-\sigma_{i,1}$ for $i^* \neq i$ \\
 with $\ell_{i^*} \neq 1 \neq \ell_i$\end{tabular}}
& \iso-1 \\ \hline
\end{array}
\end{equation}

\begin{equation} \labell{Table-ann-2-1b}
\begin{array}{c|c|c}
\alpha & \text{ generators for  }\Ann^4(\alpha)  & \rank \Ann^4(\alpha)
\\  \hline\hline

\tau_0, \, \, \tau_0 \cup \tau_0 &
\mathrm{\begin{tabular}{c} $\tau_\infty^2,\,\tau_h \tau_\infty$\\$\sigma_{i,j}^2 \,\, j>1, \sigma_{i,j}  \sigma_{i,j+1} \,\, j<\ell_i$
\end{tabular}} &2 \iso+2 \\ \hline

 \tau_\infty, \, \, \tau_\infty \cup \tau_\infty & 
 \mathrm{\begin{tabular}{c} $\tau_0^2,\,\tau_h \tau_0$\\$\sigma_{i,j}^2 \,\, j<\ell_i, \sigma_{i,j}  \sigma_{i,j+1} \,\, j<\ell_i$ \end{tabular}} 
 & 2\iso+2\\ \hline

\tau_h  & \mathrm{\begin{tabular}{c}$\tau_h \tau_0, \,\tau_h \tau_\infty$ \\$\sigma_{i,j}^2 \,\, j<\ell_i, \sigma_{i,j}  \sigma_{i,j+1} \,\, j<\ell_i$ \end{tabular}} &2 \iso+2 \\ \hline

 \sigma_{i,j}, \, \, \sigma_{i,j} \cup \sigma_{i,j} \,\,\, 1<j<\ell_i  & \mathrm{\begin{tabular}{c}$\tau_0^2,\, \tau_\infty^2,\, \tau_h \tau_0, \, \tau_h \tau_\infty$, \\ $\sigma_{k,m}^2 \,\, k \neq i \text{ and }m<\ell_k$,\\ $\sigma_{i,m}^2 \,\, m\neq j-1,j,j+1$, \\ $\sigma_{k,m} \sigma_{k,m+1} \,\, (k,m) \neq (i,j), (i,j-1) \,\,m<\ell_k$ \end{tabular}} &2\iso \\ \hline
 
 \sigma_{i,1} \,\,\, \ell_i\neq 1  & \mathrm{\begin{tabular}{c}$\tau_\infty^2,\, \tau_h \tau_0, \, \tau_h \tau_\infty$, \\ $\sigma_{k,m}^2 \,\, k \neq i \text{ and }m<\ell_k$,\\ $\sigma_{i,m}^2 \,\, m\neq 1,2$, \\ $\sigma_{k,m} \sigma_{k,m+1} \,\, (k,m) \neq (i,1) \,\, m<\ell_k$ \end{tabular}} &2\iso+1\\ \hline
 
  \sigma_{i,\ell_i}  \,\,\, \ell_i \neq 1 & \mathrm{\begin{tabular}{c}$\tau_0^2,\, \tau_h \tau_0, \, \tau_h \tau_\infty$, \\ $\sigma_{k,m}^2 \,\, k \neq i \text{ and }m<\ell_k$,\\ $\sigma_{i,m}^2 \,\, m\neq \ell_i-1,\ell_i$, \\ $\sigma_{k,m} \sigma_{k,m+1} \,\, (k,m) \neq (i,\ell_i-1)\,\,m<\ell_k$ \end{tabular}}&2\iso+1 \\ \hline  
 
\end{array}
\end{equation}

\subsubsection*{The case $\fat=1$.}
\begin{equation} \labell{Table-ann-1-1a}
\begin{array}{c|c|c}
\alpha &  \text{ generators for  }\Ann^2(\alpha) & \rank \Ann^2(\alpha) 
\\  \hline\hline

 \tau_\infty, \, \, \tau_\infty \cup \tau_\infty &\sigma_{i,j} \,\, j<\ell_i &\iso-1\\ \hline

\tau_h  & \sigma_{i,j} \, \, j>1&\iso-1\\ \hline

\tau_h \cup \tau_\infty&  \tau_h, \,\sigma_{i,j} \,\,j>1 & \iso\\ \hline

\tau_h \cup \tau_h =\sigma_{1,1} \cup \sigma_{2,1} &  \tau_\infty, \,\sigma_{i,j} \,\, j>1 & \iso\\ \hline

 \sigma_{i,j}, \,\, \sigma_{i,j} \cup \sigma_{i,j} \,\,\, 1<j<\ell_i  &\mathrm{\begin{tabular}{c}$\tau_\infty,\,\tau_h$, \\ $\sigma_{k,m} \,\, k \neq i \text{ and }m<\ell_k$,\\ $\sigma_{i,m} \,\, m\neq j-1,j,j+1$ \end{tabular}} &\iso-1 \\ \hline
 
 \sigma_{i,1} \,\, \ell_i \neq 1  &\mathrm{\begin{tabular}{c}$\tau_\infty$, \\ $\sigma_{k,m} \,\, k \neq i \text{ and }m>1$,\\ $\sigma_{i,m} \,\, m\neq 1,2$ \end{tabular}}&\iso-1 \\ \hline

   \sigma_{i,\ell_i}  \, \,\, \ell_i \neq 1&\mathrm{\begin{tabular}{c}$ \tau_h$, \\ $\sigma_{k,m} \,\, k \neq i \text{ and }m<\ell_k$,\\ $\sigma_{i,m} \,\, m\neq \ell_i-1,\ell_i$ \end{tabular}}&\iso-1\\ \hline

  \sigma_{i,j} \cup \sigma_{i,j+1} \, \,\, j<\ell_i   &\mathrm{\begin{tabular}{c}$\tau_\infty,\, \tau_h$, \\ $\sigma_{k,m} \,\, k \neq i \text{ and }m<\ell_k$,\\ $\sigma_{i,m} \,\, m\neq j,j+1$ \end{tabular}}  & \iso\\ \hline

\end{array}
\end{equation}

\newpage

\begin{equation} \labell{Table-ann-1-1b}
\begin{array}{c|c|c}
\alpha & \text{ generators for  }\Ann^4(\alpha)   & \rank \Ann^4(\alpha)
\\  \hline\hline

 \tau_\infty, \, \, \tau_\infty \cup \tau_\infty & \mathrm{\begin{tabular}{c} $\tau_h^2$\\$\sigma_{i,j}^2 \,\, j<\ell_i, \sigma_{i,j}  \sigma_{i,j+1} \,\, j<\ell_i$ \end{tabular}} &2\iso-1\\ \hline

\tau_h  & \mathrm{\begin{tabular}{c}$\tau_h \tau_\infty$ \\$\sigma_{i,j}^2 \,\, j>1, \sigma_{i,j}  \sigma_{i,j+1} \,\, j<\ell_i$ \end{tabular}} &2\iso-1\\ \hline

 \sigma_{i,j}, \,\, \sigma_{i,j} \cup \sigma_{i,j} \,\,\, 1<j<\ell_i  & \mathrm{\begin{tabular}{c}$ \tau_\infty^2,\,  \tau_h^2, \, \tau_h \tau_\infty$, \\ $\sigma_{k,m}^2 \,\, k \neq i \text{ and }m<\ell_k$,\\ $\sigma_{i,m}^2 \,\, m\neq j-1,j,j+1$, \\ $\sigma_{k,m} \sigma_{k,m+1} \,\, (k,m) \neq (i,j), (i,j-1) \,\, m<\ell_k$ \end{tabular}} &2\iso-3\\ \hline
 
 \sigma_{i,1} \,\, \ell_i \neq 1  & \mathrm{\begin{tabular}{c}$\tau_\infty^2,\, \tau_h \tau_\infty$, \\ $\sigma_{k,m}^2 \,\, k \neq i \text{ and }m>1$,\\ $\sigma_{i,m}^2 \,\, m\neq 1,2$, \\ $\sigma_{k,m} \sigma_{k,m+1} \,\, (k,m) \neq (i,1) \,\, m<\ell_k$ \end{tabular}} &2\iso-2 \\ \hline

  \sigma_{i,\ell_i}  \, \,\, \ell_i \neq 1& \mathrm{\begin{tabular}{c}$ \tau_h^2, \, \tau_h \tau_\infty$, \\ $\sigma_{k,m}^2 \,\, k \neq i \text{ and }m<\ell_k$,\\ $\sigma_{i,m}^2 \,\, m\neq \ell_i-1,\ell_i$, \\ $\sigma_{k,m} \sigma_{k,m+1} \,\, (k,m) \neq (i,\ell_i-1)\,\,m<\ell_k$ \end{tabular}} &2\iso-2\\ \hline

\end{array}
\end{equation}
\newpage

\subsubsection*{The case $\fat=0$.}
\begin{equation} \labell{Table-ann-0}
\begin{array}{c|c|c}
\alpha & \text{ generators for  }\Ann^2(\alpha) & \rank \Ann^2(\alpha)\\  \hline\hline

  \sigma_{i,j} \cup \sigma_{i,j+1} \, \,\, j<\ell_i   &\mathrm{\begin{tabular}{c} $\sigma_{k,m} \,\, k \neq i$,\\ $\sigma_{i,m} \,\, m\neq j,j+1$
   \end{tabular}} &\iso-2 \\ \hline

\sigma_{1,1} \cup \sigma_{2,1}  &\mathrm{\begin{tabular}{c}$\sigma_{1,j} \,\,j>1$,\\ $\sigma_{2,j} \,\,j>1$ \end{tabular}} &\iso-2 \\ \hline

  \sigma_{1,\ell_i} \cup \sigma_{2,\ell_i}  &\mathrm{\begin{tabular}{c}$\sigma_{1,j} \,\,j<\ell_1$,\\ $\sigma_{2,j} \,\,j<\ell_2$ \end{tabular}} &\iso-2\\ \hline

   \sigma_{i,j} \cup \sigma_{i,j} \,\,\, 1<j<\ell_i &\mathrm{\begin{tabular}{c} $\sigma_{k,m} \,\, k \neq i$,\\ $\sigma_{i,m} \,\, m\neq j-1,j,j+1$\end{tabular}} &\iso-3\\ \hline

\sigma_{i,1} \cup \sigma_{i,1}  &\mathrm{\begin{tabular}{c}$\sigma_{k,m} \,\, k \neq i \text{ and } m>1$,\\ $\sigma_{i,m} \,\, m\neq 1,2$ \end{tabular}} &\iso-3\\ \hline

  \sigma_{i,\ell_i} \cup \sigma_{i,\ell_i} &\mathrm{\begin{tabular}{c}$\sigma_{k,m} \,\, k \neq i \text{ and }m<\ell_k$,\\ $\sigma_{i,m} \,\, m\neq \ell_i-1,\ell_i$ \end{tabular}} &\iso-3\\ \hline
\end{array}
\end{equation}

\newpage

\subsection{Intersection form and zero lengths}

\begin{Definition}
For a class $\eta\in H^{2*}_{S^1}(M)$, we define
$$Z(\eta)=\left\{F \text{ connected component of } M^{S^1} \, \big| \, \eta|_{F}=0\right\}.$$ 
We call $|Z(\eta)|$ the \emph{zero length} of $\eta$.  Note that the zero-length is related as follows
to the sets from Notations~\ref{not:z012} and \ref{Not:zieta}:
$$
|Z(\eta)| = |Z_0(\eta)|+ |Z_1(\eta)| = |Z_1(\eta)|+ |Z_2(\eta)|.
$$
\end{Definition}

Note, for example, that a component Euler class $e_{S^1}(\nu(F\subset M))$
has zero length $\iso + \fat -1$;
 the component Euler class
is non-zero on precisely one component of the fixed point locus: $F$. 
The ABBV formula \eqref{eq:abbv} implies that any class supported on a single fixed component
must be a multiple of the component Euler class.

In \cite[Appendix A]{KesslerHolm2}, 
we show that 
if $\dim M=4$, and 
$\alpha=\iota_{S_{a}}^{!}(\mathbb{1})$, $\beta=\iota_{S_{b}}^{!}(\mathbb{1}) $, with $S_a, \, S_b$ invariant embedded spheres, then 
\begin{equation} \labell{eq:inter3}
\alpha \cdot \beta =[S_a] \cdot [S_b],
\end{equation}
 where the intersection form in the left-hand side is in equivariant cohomology and  in the right-hand side is in standard homology. 
For the right-hand, the spheres should be positively oriented with respect to the orientation induced by the symplectic form; the intersection will then be non-negative when $S_a\neq S_b$.
We may use \eqref{eq:inter3} and the ABBV formula \eqref{eq:abbv} to show that
for $S_{i,j}$ and $S_{i',j'}$, invariant embedded symplectic spheres whose images are the $(i,j)$ and $(i',j')$ edge in the extended decorated graph, we can read the intersection number from the graph:
\begin{equation}\label{rem:int}
[S_{i,j}] \cdot [S_{i',j'}]=\begin{cases}
1 & \text{if }(i,j)=(i',j' \pm 1) \\
1 & \text{if }\Sigma_0  \text{ does not exist and }i \neq i' \text{ and }j=1=j'\\
1 & \text{if }\Sigma_\infty \text{ does not exist and }i \neq i' \text{ and }j=\ell_i, \,j'=\ell_{i'}\\
\frac{-m_{i,j-1}-m_{i,j+1}}{m_{i,j}} & \text{if }(i,j)=(i',j')\\
    0 & \text{ otherwise}\\
\end{cases}.
\end{equation}
Note that for these calculations, we do not assume that if $\fat=1$ then $\tau_\infty \neq 0$.
The labels $m_{i,0}$ and $m_{i,\ell_i+1}$ depend on the number of fixed surfaces and are defined 
as follows.
We set
$$
\sigma_{i,0}=\begin{cases}
\tau_0 & \text{ if }\fat=2\\
\sigma_{i^*,1}  & \text{ if } \fat=1, 0 \\
   \end{cases};
$$
\begin{equation} \labell{eq:label0}
m_{i,0}=\begin{cases}
0 & \text{ if }\fat=2 \\
-m_{i^*,1} & \text{ if } \fat=1,0 \\
\end{cases},
\end{equation}
and
$$
\sigma_{i,\ell_i+1}=\begin{cases}
\tau_\infty & \text{ if }\fat=2,1 \\
    \sigma_{i^*,\ell_{i^*}} & \text{ if } \fat=0 \\ 
\end{cases};
$$
\begin{equation} \labell{eq:label+1}
m_{i,\ell_i+1}=\begin{cases}
0 & \text{ if }\fat=2,1 \\
    -m_{i^*,\ell_{i^*}} & \text{ if } \fat=0 \\ 
\end{cases}.
\end{equation}
Using this notation, the result \eqref{rem:int} is consistent with \eqref{eq:labelproducts}.
Similarly
$$[\Sigma_{0}] \cdot [S_{i,j}]=\begin{cases}
1 & \text{ if }j=1 \text{ and }\Sigma_0 \text{ exists}\\
0 & \text{ otherwise}
\end{cases},
$$
$$[\Sigma_{\infty}] \cdot [S_{i,j}]=\begin{cases}
1 & \text{ if }j=\ell_i \text{ and }\Sigma_\infty \text{ exists}\\
0 & \text{ otherwise}
\end{cases}.
$$

In the following three tables we list the intersection form and zero lengths. The zero lengths are computed using 
Tables~\ref{Table-rest-2}, \ref{Table-rest-0} and \ref{Table-rest-1}. The intersection numbers are justified in \S
\ref{rem:int}. Here $m_{i,0}$ and $m_{i,\ell_i+1}$ are as above \eqref{eq:label0} and \eqref{eq:label+1}.

\subsubsection*{The case $\fat=2$.}
\begin{equation} \labell{Table-inter-2}
\begin{array}{c|c|c|c|c|c|c|c}
& \cdot \tau_0 & \cdot \tau_{\infty} & \cdot \tau_h & \cdot \sigma_{i,j} 
 & \mathrm{\begin{tabular}{c}
zero \\ length \end{tabular}} \\  \hline\hline
\tau_0 & e_{\min} & 0 & 1 & \mathrm{\begin{tabular}{c}$1$ if $j=1$ \\ $0$ if $j \neq 1$ \end{tabular}} & \iso+1 \\ \hline
\tau_{\infty} & 0 &e_{\max} &1&  \mathrm{\begin{tabular}{c}$1$ if $j=\ell_i$ \\ $0$ if $j \neq \ell_i$ \end{tabular}} & \iso+1 \\ \hline
\tau_h & 1 & 1 & 0& 0& \iso \\ \hline
\mathrm{\begin{tabular}{c}$\sigma_{i',j'}$\\if $1<\ell_{i'}$ \end{tabular}} &  \mathrm{\begin{tabular}{c}$1$ if $j'=1$ \\ $0$ if $j '\neq 1$ \end{tabular}} & \mathrm{\begin{tabular}{c}$1$ if $j'=\ell_{i'}$ \\ $0$ if $j' \neq \ell_{i'}$ \end{tabular}} & 0 & \mathrm{\begin{tabular}{c}$1$ if $(i,j)=(i',j'\pm1)$\\$-\frac{m_{i,j-1}+m_{i,j+1}}{m_{i,j}}$ if $(i,j)=(i',j')$ \\ 
$0$ otherwise
 \end{tabular}} & \iso \\ \hline
\end{array}
\end{equation}

\subsubsection*{The case $\fat=1$.} Assume without loss of generality that the extremal fixed surface is a 
maximum, and that the $k$ chains are indexed such that $m_{1,1} \geq m_{2,1} \geq \ldots m_{k,1}$. 
(For $3 \leq i \leq k$, we have $m_{i,1}=1$, see Proposition \ref{claim:eph}.) Here,  $i^{*} =1+( i \text{ mod }2)$.

\begin{equation} \labell{Table-int-1}
\begin{array}{c|c|c|c|c|c|c|c|c}
& \cdot \tau_\infty & \cdot \tau_h & \cdot \sigma_{i,j} 
 & \mathrm{\begin{tabular}{c}
zero \\ length \end{tabular}} \\  \hline\hline
\tau_\infty & e_{\max} & 1  & \mathrm{\begin{tabular}{c}$1$ if $j=\ell_i$ \\ $0$ if $j \neq \ell_i$ \end{tabular}} & \iso \\ \hline
\tau_h & 1 &
m_{1,1} m_{2,1}
 &  \mathrm{\begin{tabular}{c}$m_{i^*,1}$ if $j=1$ and $i \in \{1,2\}$ \\ 
$m_{1,1}m_{2,1}$ if $j=1$ and $i \geq 3$\\
$0$ if $1<j$ \end{tabular}}   & \iso-1 \\ \hline
\sigma_{i',j'} &  &
 &
  \mathrm{\begin{tabular}{c} $1$ if $(i,j)=(i',j'\pm1)$;\\
$1$  if $i=1$, $i'=2$ and $j=1=j'$;\\
$m_{i^*,1}$ if $i \in\{1,2\}$, $i' \geq 3$ and $j=1=j'$\\
$m_{1,1}m_{2,1}$ if $i \neq i'$, $i,i' \geq 3$ and $j=1=j'$;\\
$-\frac{m_{i,j-1}+m_{i,j+1}}{m_{i,j}}$ if $(i,j)=(i',j')$\\
$0$ otherwise 
\end{tabular}}
 & \iso-1 \\ \hline
\end{array}
\end{equation}

\subsubsection*{The case $\fat=0$.}
\begin{equation} \labell{Table-int-0}
\begin{array}{c|c|c|c|c|c|c|c}
 & \cdot \tau_h& \cdot \sigma_{i,j} 
 & \mathrm{\begin{tabular}{c}
zero \\ length \end{tabular}} \\  \hline\hline
\tau_h & m_{1,1}m_{2,1}+m_{1,\ell_1}m_{2,\ell_{2}}
 & \mathrm{\begin{tabular}{c} $0$ if $j \neq 1,\ell_i$ \\$m_{i^*,\ell_{i^*}}$  if $j=1<\ell_i$ 
 \\$m_{i^*,1}$  if $j=\ell_i>1$\\
 $m_{i^*,\ell_{i^*}}+m_{i^*,1}$  if $j=1=\ell_i$\\
 where $i^* \neq i$\\  
 \end{tabular}} & \iso-2 \\ \hline
\sigma_{i',j'} &  \mathrm{\begin{tabular}{c} $0$ if $j' \neq 1,\ell_{i'}$ \\$m_{i^*,\ell_{i^*}}$  if $j'=1<\ell_{i'}$ 
 \\$m_{i^*,1}$  if $j'=\ell_{i'}>1$\\
 $m_{i^*,\ell_{i^*}}+m_{i^*,1}$  if $j'=1=\ell_{i'}$\\
 where $i^* \neq i'$\\  
 \end{tabular}} & \mathrm{\begin{tabular}{c}$1$ if $(i,j)=(i',j'\pm1)$\\
$1$  if $i\neq i'$ and $j=\ell_i, \, j'=\ell_{i'}$\\
$1$ if $i\neq i'$ and $j=j'=1$\\
$-\frac{m_{i,j-1}+m_{i,j+1}}{m_{i,j}}$ if $(i,j)=(i',j')$ \\ 
$0$ otherwise \end{tabular}} & \iso-2 \\ \hline
\end{array}
\end{equation}

\newpage

\begin{noTitle} \labell{Other}
{\bf Other classes with zero length $\iso+\fat-2$.}
\begin{itemize}
\item[a)] Any non-zero integer multiple of each of the classes $\sigma_{i,j}$ and $\tau_h$. 
\item[b)] $a_0 \tau_0+a_h \tau_h+a_\infty \tau_\infty$ with $a_0,a_h,a_{\infty} \in \Z$ and at least two of the coefficients of non-zero elements are not zero. 
\item[c)]  $a_0 \tau_0 +\gamma \sum_{j=1}^{\beta}m_{i,j}\sigma_{i,j}$ for $a_0, \gamma \in \Z \smallsetminus \{0\}$ and $1\leq \beta <\ell_i$; $a_{\infty} \tau_{\infty} + \gamma \sum_{j=\alpha}^{\ell_i}m_{i,j}\sigma_{i,j}$ for $a_{\infty}, \gamma \in \Z \smallsetminus \{0\}$ and $1<\alpha \leq \ell_i$.

\item[d)] For $\gamma$ a non-zero integer and $m_{i,0}, \, m_{i,\ell_i+1}$ as in \eqref{eq:label0} and \eqref{eq:label+1}, according to $\fat$, we have the following classes.
\begin{equation} \labell{Table-zl-2}
\begin{array}{c|c|c|c|c|c|c|clclclc}
\sigma & \sigma \cdot \sigma & \sigma \cdot \tau_0 & \sigma \cdot \tau_{\infty} & \sigma \cdot \tau_h 
\\  \hline\hline

 \mathrm{\begin{tabular}{c}$\displaystyle{\gamma \sum_{j=\alpha}^{\beta}m_{i,j}\sigma_{i,j}}$ \\  with $1 \leq \alpha < \beta \leq \ell_i$\end{tabular}} &  \mathrm{\begin{tabular}{c}$ -\gamma^2(m_{i,\beta} m_{i,\beta+1}$\\$+m_{i,\alpha}m_{i,\alpha-1})$ 
  \end{tabular}} &  \mathrm{
 \begin{tabular}{c}$\left\{ \begin{array}{ll} \gamma & 
 \mbox{if (*)}
 \\ 0 & \mbox{else}\end{array}\right.$.\end{tabular} 
}
  & \mathrm{
 \begin{tabular}{c}$\left\{ \begin{array}{ll} \gamma & 
 \mbox{if (**)}
 \\ 0 & \mbox{else}\end{array}\right.$.\end{tabular} 
}& \mathrm{\begin{tabular}{c}$\gamma(\delta_{\tau_0=0}\delta_{\alpha=1} m_{1,1}m_{2,1}$\\ $+ \delta_{\tau_\infty=0}\delta_{\beta=\ell_i}m_{1,\ell_1}m_{2,\ell_2})$ \end{tabular}}
   \\ \hline

       \mathrm{\begin{tabular}{c}$\displaystyle{\gamma\left(\sum_{j=1}^{\beta_i}m_{i,j}\sigma_{i,j}-\sum_{j=1}^{\beta_{i'}} m_{i',j}\sigma_{i',j}\right)}$  \\  with $ i\neq i', \, \beta_i < \ell_i, \, \beta_{i'} < \ell_{i'}$\end{tabular}} &\mathrm{\begin{tabular}{c} {$ -\gamma^2(m_{i,\beta_{i}}m_{i,\beta_{i}+1}$}\\{$+m_{i',\beta_{i'}}m_{i',\beta_{i'}+1})$} \\
       \end{tabular}} & 0 &0 &0 
       \\ \hline
       
 \mathrm{\begin{tabular}{c}$\displaystyle{\gamma\left(\sum_{j=\alpha_i}^{\ell_i}m_{i,j}\sigma_{i,j}-\sum_{j=\alpha_{i'}}^{\ell_{i'}}m_{i',j}\sigma_{i',j}\right)}$ \\  with $i \neq i',  \, 1<\alpha_i,\,1< \alpha_{i'}$\end{tabular}}&\mathrm{\begin{tabular}{c} { $-\gamma^2 (m_{i,\alpha_{i}}m_{i,\alpha_{i}-1}$}\\{$+ m_{i',\alpha_{i'}}m_{i',\alpha_{i'}-1})$}\end{tabular}} &0& 0&0 
  \\ \hline
\end{array}
\end{equation}
The condition (*) in Table~\ref{Table-zl-2} is that
$\alpha=1$ and $\tau_0 \neq 0$; in this case, we have $m_{i,\alpha}=1$.  The condition (**) is that
$\beta=\ell_i$ and $\tau_\infty \neq 0$; in this case $m_{i,\beta}=1$.

\item[e)] For every non-zero class in $H_{S^1}^{2}$ that is not 
one of the above classes, the zero length is strictly smaller than $\iso+\fat-2$.
\end{itemize}
\end{noTitle}

\newpage

\begin{Lemma}\label{lem:intersums}
Let $\sigma=\sum_{j=\alpha}^{\beta}m_{i,j}\sigma_{i,j}$, with 
$1\leq \alpha < \beta \leq \ell_i$.
\begin{enumerate}
\item $\sigma \cdot  \sigma=-m_{i,\alpha}m_{i,\alpha-1}-m_{i,\beta}m_{i,\beta+1}.$
\item Let $\gamma, \delta \in \Z$. For $\sigma_{r,s}$ with $1 \leq s \leq \ell_r$, 
\begin{equation*}
\gamma \sigma \cdot \delta \sigma_{r,s}=\begin{cases}
-\gamma \delta m_{i,\alpha-1} &  \text{ if $r=i$ and $s=\alpha$}\\
-\gamma \delta m_{i,\beta+1} &  \text{ if $r=i$ and $s=\beta$}\\
\gamma \delta m_{i,\alpha}& \text{ if $(r,s)=(i,\alpha-1)$ or ($r \neq i$, $s=1=\alpha$, $\tau_0=0$ and $\ell_r >1$)} \\
\gamma \delta m_{i,\beta} & \text{ if $(r,s)=(i,\beta+1)$ or ( $r \neq i$, $s=\ell_r$, $\beta=\ell_i$, $\tau_\infty=0$ and $\ell_r >1$)}\\
\gamma \delta m_{i,\alpha}+\gamma \delta m_{i,\beta} & \text{ if  $r \neq i$, $s=\ell_r=1$, $\alpha=1$, $\beta=\ell_i$, and $\tau_0=0=\tau_\infty$}\\
0 & \text{ otherwise}
\end{cases}.
\end{equation*}
For $1\leq c<d \leq \ell_r$, 
\begin{equation*}
\gamma \sigma \cdot \delta \sum_{j=c}^{d}m_{r,j}\sigma_{r,j}=
\begin{cases}
\gamma \delta(-m_{i,\alpha}m_{i,\alpha-1}-m_{i,\beta}m_{i,\beta+1}) & \text{ if $r=i$, $c=\alpha$ and $d=\beta$}\\
\gamma \delta(-m_{i,\alpha}m_{i,\alpha-1})& \text{ if ($r=i$, $c=\alpha$ and $d\neq \beta$) or ($r=i$ and $d=\alpha$)}\\
\gamma \delta(-m_{i,\beta}m_{i,\beta+1})& \text{ if ($r=i$, $c\neq\alpha$ and $d=\beta$) or ($r=i$ and $c=\beta$)}\\
\gamma \delta m_{i,\beta}m_{i,\beta+1}& \text{ if $r=i$ and $c=\beta+1$}\\
\gamma \delta m_{i,\alpha}m_{i,\alpha-1}& \text{ if $r=i$ and $d=\alpha-1$}\\
\gamma \delta m_{i,1} m_{r,1}& \text{ if  $r \neq i$, $c=1=\alpha$, $\tau_0=0$ and $d \neq \ell_r$ or $\beta \neq \ell_i$ or $\tau_\infty \neq 0$}\\
\gamma \delta m_{i,\ell_i}m_{r,\ell_r}& \text{ if  $r \neq i$, $d=\ell_r$, $\beta=\ell_i$, $\tau_\infty=0$ and $c \neq 1$ or $\alpha \neq 1$ or $\tau_0 \neq 0$}\\
\gamma \delta m_{i,1}m_{r,1}+\gamma \delta m_{i,\ell_i} m_{r,\ell_r}& \text{ if  $r \neq i$, $c=1=\alpha$, $d=\ell_r$, $\beta=\ell_i$ and $\tau_0=0=\tau_\infty$}\\
0 & \text{ otherwise}
\end{cases}.
\end{equation*}

For $r\neq r', \, d_r < \ell_r, \, d_{r'} < \ell_{r'}$, 
\begin{equation*}
\gamma \sigma \cdot \delta (\sum_{j=1}^{d_r}m_{r,j}\sigma_{r,j}-\sum_{j=1}^{d_{r'}} m_{r',j}\sigma_{r',j})=
\begin{cases}
\gamma \delta(-m_{i,\beta}m_{i,\beta+1}) & \text{ if $(r,d_r)=(i,\beta)$}\\
\gamma \delta(m_{i,\beta}m_{i,\beta+1}) & \text{ if $(r',d_{r'})=(i,\beta)$}\\
\gamma \delta(-m_{i,\alpha}m_{i,\alpha-1}) & \text{ if $(r,d_r)=(i,\alpha)$}\\
 \gamma \delta(m_{i,\alpha}m_{i,\alpha-1}) & \text{ if $(r',d_{r'})=(i,\alpha)$}\\
\gamma \delta(m_{i,\alpha}m_{i,\alpha-1}) & \text{ if $(r,d_r)=(i,\alpha-1)$}\\
\gamma \delta(-m_{i,\alpha}m_{i,\alpha-1}) & \text{ if $(r',d_{r'})=(i,\alpha-1)$}\\
0 & \text{ otherwise}
\end{cases}.
\end{equation*}

For $r \neq r',  \, 1<c_r,\,1< c_{r'}$, 
\begin{equation*}
\gamma \sigma \cdot \delta (\sum_{j={c_r}}^{\ell_r}m_{r,j}\sigma_{r,j}-\sum_{j=c_{r'}}^{\ell_{r'}} m_{r',j}\sigma_{r',j})=
\begin{cases}
\gamma \delta(-m_{i,\alpha}m_{i,\alpha-1}) & \text{ if $(r,c_r)=(i,\alpha)$}\\
\gamma \delta(m_{i,\alpha}m_{i,\alpha-1}) & \text{ if $(r',c_{r'})=(i,\alpha)$}\\
\gamma \delta(-m_{i,\beta}m_{i,\beta+1}) & \text{ if $(r,c_r)=(i,\beta)$}\\
 \gamma \delta(m_{i,\beta}m_{i,\beta+1})& \text{ if $(r',c_{r'})=(i,\beta)$}\\
 \gamma \delta(m_{i,\beta}m_{i,\beta+1})& \text{ if $(r,c_r)=(i,\beta+1)$}\\
\gamma \delta(-m_{i,\beta}m_{i,\beta+1}) & \text{ if $(r',c_{r'})=(i,\beta+1)$}\\

0 & \text{ otherwise}
\end{cases}.
\end{equation*}

\item The intersection of $\gamma \sigma$ and $\delta \sigma_{r,s}$ equals $1$ exactly in the following cases:
\begin{itemize}
\item[i.] $r=i$, $s=2=\alpha$, $m_{i,1}=1$ and $\gamma \delta=-1$;
\item[ii.] $r=i$, $s=\ell_i-1=\beta$, $m_{i,\ell_i}=1$ and $\gamma \delta=-1$;
\item[iii.] $r \neq i$, $s=1=\alpha$, $m_{i,1}=1$, $\tau_0=0$, $\ell_r>1$ and $\gamma \delta=1$;
\item[iv.] $r \neq i$, $s=\ell_r$, $\beta=\ell_i$, $m_{i,\ell_i}=1$, $\tau_\infty=0$, $\ell_r>1$ and $\gamma \delta=1$.
\end{itemize}

\bigskip

\noindent The intersection of $\gamma \sigma$ and  $\delta \sum_{j=c}^{d}m_{r,j}\sigma_{r,j}$ equals $1$ exactly in the following cases:
\begin{itemize}
\item[i.]    $r \neq i$, $c=1=\alpha$, $m_{i,1}=1=m_{r,1}$, $\gamma \delta=1$, $\tau_0=0$ and $d \neq \ell_r$ or $\beta \neq \ell_i$ or $\tau_\infty \neq 0$;
\item[ii.]  $r \neq i$, $d=\ell_r$, $\beta=\ell_i$, $m_{i,\ell_i}=1=m_{r,\ell_r}$, $\gamma \delta=1$, $\tau_\infty=0$ and $c \neq 1$ or $\alpha \neq 1$ or $\tau_0 \neq 0$.
\end{itemize}

\bigskip

\noindent  Neither the intersection of $\gamma \sigma$ and $\delta (\sum_{j=1}^{d_r}m_{r,j}\sigma_{r,j}-\sum_{j=1}^{d_{r'}} m_{r',j}\sigma_{r',j})$; nor the intersection of $\gamma \sigma$ with 
$$ \delta (\sum_{j={c_r}}^{\ell_r}m_{r,j}\sigma_{r,j}-\sum_{j=c_{r'}}^{\ell_{r'}} m_{r',j}\sigma_{r',j})$$ can equal $1$.

\bigskip

\noindent The intersection of  $\gamma \sigma_{i,s}$ and $\delta (\sum_{j=1}^{d_r}m_{r,j}\sigma_{r,j}-\sum_{j=1}^{d_{r'}} m_{r',j}\sigma_{r',j})$ is $1$ only if either 
\begin{itemize}
\item[i.]  $i=r$, $s=\ell_i-1=d_r$, $m_{i,\ell_i}=1$ and $\gamma \delta=-1$, or
\item[ii.] $i=r'$, $s=\ell_i-1=d_{r'}$, $m_{i,\ell_i}=1$ and $\gamma \delta=1$.
\end{itemize}

\bigskip

\noindent The intersection of  $\gamma \sigma_{i,s}$ and  $ \delta (\sum_{j={c_r}}^{\ell_r}m_{r,j}\sigma_{r,j}-\sum_{j=c_{r'}}^{\ell_{r'}} m_{r',j}\sigma_{r',j})$ is $1$ only if 
\begin{itemize}
\item[i.]  $i=r$, $s=2=c_r$, $m_{i,1}=1$ and $\gamma \delta=-1$, or
\item[ii.] $i=r'$, $s=2=c_{r'}$, $m_{i,1}=1$ and $\gamma \delta=1$.
\end{itemize}
\end{enumerate}
\end{Lemma}

\noindent Here the labels $m_{i,0}$ and $m_{i,\ell_i+1}$ are defined in \eqref{eq:label0} and \eqref{eq:label+1}.

\begin{proof}
By the intersection form listed in Tables \ref{Table-inter-2}, \ref{Table-int-1}, and \ref{Table-int-0},
\begin{eqnarray*}
 \sum_{j=\alpha}^{\beta}m_{i,j}\sigma_{i,j} \cdot \sum_{j=\alpha}^{\beta}m_{i,j}\sigma_{i,j}&=&\sum_{j=\alpha}^{\beta}m_{i,j}^2 \sigma_{i,j} \cdot \sigma_{i,j}+\sum_{j=\alpha+1}^{\beta}m_{i,j-1}m_{i,j}\sigma_{i,j-1} \cdot \sigma_{i,j}\\ & &+\sum_{j=\alpha}^{\beta-1}m_{i,j+1}m_{i,j} \sigma_{i,j+1} \cdot \sigma_{i,j}\\
&=&m_{i,\alpha}^2 \frac{-m_{i,\alpha-1}-m_{i,\alpha+1}}{m_{i,\alpha}}+m_{i,\beta}^2 \frac{-m_{i,\beta-1}-m_{i,\beta+1}}{m_{i,\beta}}\\& &+\sum_{j=\alpha+1}^{\beta-1}(m_{i,j}^2(-\frac{m_{i,j-1}+m_{i,j+1}}{m_{i,j}})+m_{i,j-1}m_{i,j}+m_{i,j+1}m_{i,j})\\& & + \, m_{i,\beta-1}m_{i,\beta}+m_{i,\alpha+1}m_{i,\alpha}\\
&=&-m_{i,\alpha}m_{i,\alpha-1}-m_{i,\beta}m_{i,\beta+1}.\\
\end{eqnarray*}
Explicitly, $\sum_{j=\alpha+1}^{\beta-1} m_{i,j}\sigma_{i,j} \cdot  \sum_{j=\alpha}^{\beta}m_{i,j}\sigma_{i,j} =0$, $m_{i,\beta}\sigma_{i,\beta} \cdot  \sum_{j=\alpha}^{\beta}m_{i,j}\sigma_{i,j}=-m_{i,\beta}m_{i,\beta+1}$, and $m_{i,\alpha}\sigma_{i,\alpha} \cdot  \sum_{j=\alpha}^{\beta}m_{i,j}\sigma_{i,j}=-m_{i,\alpha}m_{i,\alpha-1}$. 
This proves item (1). Similar calculations, and the intersection tables, will allow the reader to verify item (2). 
Item (3) follows from item (2), since for $1<j<\ell_i$, we have $m_{i,j}>1$, by
Proposition \ref{label1}.
\end{proof}

\begin{Remark}
By similar calculations, we get that
or $\sigma$ as in Table \eqref{Table-zl-2}, the numbers $\sigma \cdot \sigma$, $\sigma \cdot \tau_0$, $\sigma \cdot \tau_{\infty}$, $\sigma \cdot \tau_h$ 
 are as listed in the table.
\end{Remark}

\end{landscape}

\newpage

\pagestyle{empty}

\end{document}